%% file: main.tex
\documentclass[reqno]{amsart}

\usepackage{etoolbox}
\patchcmd{\abstract}{\null\vfil}{}{}{}

\input{preamble/packages}
\input{preamble/theorems}

\input{preamble/definitions}

\pgfplotsset{compat=1.18}

\title{Functional limit theorems for edge counts in dynamic random connection hypergraphs}

\author[C. Hirsch]{Christian Hirsch $^{1,2}$}
\email{hirsch@math.au.dk}

\author[B. Jahnel]{Benedikt Jahnel $^{3,4}$}
\email{benedikt.jahnel@tu-braunschweig.de}

\author[P. Juhasz]{Peter Juhasz $^1$}
\email{peter.juhasz@math.au.dk}

\address{$^1$Department of Mathematics, Aarhus University, Aarhus, Denmark}
\address{$^2$DIGIT Center, Aarhus University, Aarhus, Denmark}
\address{$^3$Department of Mathematics, Technische Universit\"at Braunschweig, Braunschweig, Germany}
\address{$^4$Weierstrass Institute for Applied Analysis and Stochastic, Berlin, Germany}

\begin{document}

\input{abstract.tex}

\maketitle

\subsection*{Keywords} Poisson point process; random graph; functional limit theorem; stable process; hypergraph; bipartite graph

\thispagestyle{empty}

\input{introduction.tex}
\input{model.tex}
\input{outline.tex}

\input{proofs.tex}

\input{appendix.tex}


\input{acknowledgement.tex}

\bibliographystyle{abbrv}
\bibliography{main.bbl}

\end{document}

%% file: preamble/packages.tex
\usepackage[foot]{amsaddr}
\usepackage{amsmath}
\usepackage{amssymb}
\usepackage{amsthm}
\usepackage[dvipsnames]{xcolor}
\usepackage{tikz}
\usepackage{bm}
\usepackage{bbm}
\usepackage{color}
\usepackage{comment}
\usepackage{dsfont}
\usepackage{enumitem}
\usepackage{etoolbox}
\usepackage{fullpage}
\usepackage[utf8]{inputenc}
\usepackage{mathtools}
\usepackage{multirow} 
\usepackage{pgf}
\usepackage{pgfplots}
\usepackage{placeins} 
\usepackage{tcolorbox}
\usepackage[final]{pdfpages}
\usepackage{subcaption}

\usetikzlibrary{patterns}
\usepgfplotslibrary{groupplots}
\usepgfplotslibrary{fillbetween}

\usepackage[hidelinks]{hyperref}

%% file: preamble/theorems.tex
\theoremstyle{plain}
\usepackage{amsthm}
\newtheorem{theorem}{Theorem}
\newtheorem{proposition}[theorem]{Proposition}

\newtheorem{lemma}[theorem]{Lemma}

\theoremstyle{definition}
\newtheorem{definition}[theorem]{Definition}

\theoremstyle{remark}
\newtheorem{remark}[theorem]{Remark}

\numberwithin{equation}{section}
\numberwithin{theorem}{section}

%% file: preamble/definitions.tex
\def\bs{\boldsymbol}
\def\mrm{\mathrm}
\def\msf{\mathsf}
\def\mcl{\mathcal}
\def\mbb{\mathbb}
\def\mbf{\mathbf}

\let\P\undefined
\let\exp\undefined
\let\log\undefined
\def\d{\mathop{}\!\mrm d}
\def\Leb{\mathop{}\!\msf{Leb}}
\def\deg{\mathop{}\!\msf{deg}}
\DeclarePairedDelimiterXPP\P[1]{\mathop{}\!\mbb{P}}{(}{)}{}{#1}
\DeclarePairedDelimiterXPP\E[1]{\mathop{}\!\mbb{E}}{[}{]}{}{#1}
\DeclarePairedDelimiterXPP\1[1]{\mathop{}\!\mathbbm{1}}{\{}{\}}{}{#1}
\DeclarePairedDelimiterXPP\Poi[1]{\mathop{}\!\msf{Poi}}{(}{)}{}{#1}
\DeclarePairedDelimiterXPP\Var[1]{\mathop{}\!\msf{Var}}{(}{)}{}{#1}
\DeclarePairedDelimiterXPP\Cov[1]{\mathop{}\!\msf{Cov}}{(}{)}{}{#1}
\DeclarePairedDelimiterXPP\exp[1]{\mathop{}\!\msf{exp}}{(}{)}{}{#1}
\DeclarePairedDelimiterXPP\log[1]{\mathop{}\!\msf{log}}{(}{)}{}{#1}

\def\f{\frac}
\def\e{\mrm{e}}
\def\w{\wedge}
\def\any{\,\bs\cdot\,}
\def\es{\varnothing}
\def\su{\subseteq}
\def\co{\colon}
\def\sm{\setminus}
\def\ff{\infty}
\def\tff{\uparrow\ff}
\def\given{\nonscript\:\vert\nonscript\:\mathopen{}}
\def\biggiven{\nonscript\:\big\vert\nonscript\:\mathopen{}}
\def\Biggiven{\nonscript\:\Big\vert\nonscript\:\mathopen{}}
\def\bigggiven{\nonscript\:\bigg\vert\nonscript\:\mathopen{}}
\def\ti{\times}
\def\AA{\mcl{A}}
\def\BB{\mcl{B}}

\def\DD{\mcl{D}}
\def\FF{\mcl{F}}
\def\HH{\mcl{H}}

\def\PP{\mcl{P}}

\def\UU{\mcl{U}}
\def\XX{\mcl{X}}
\def\wt{\widetilde}

\def\le{\leqslant}
\def\ge{\geqslant}
\DeclarePairedDelimiter{\norm}{\lVert}{\rVert}
\DeclarePairedDelimiter{\abs}{\lvert}{\rvert}
\DeclarePairedDelimiter{\floor}{\lfloor}{\rfloor}

\DeclarePairedDelimiter{\set}{\{}{\}}

\def\R{\mbb R}
\def\J{\mbb J}

\def\N{\mbb N}
\def\S{\mbb S}
\def\T{\mbb T}
\def\V{\mbb V}

\def\Z{\mbb Z}

\def\a{\alpha}
\def\b{\beta}
\def\g{\gamma}
\def\Ga{\Gamma}
\def\de{\delta}
\def\De{\Delta}
\def\eps{\varepsilon}
\def\k{\kappa}
\def\la{\lambda}
\def\La{\Lambda}
\def\om{\omega}

\def\r{\rho}
\def\t{\tau}

\def\k{\kappa}
\def\m{\mu}
\def\s{\sigma}

\def\Ns{N_{\msf{s}}}
\def\Nt{N_{\msf{t}}}

\def\Nl{\mbf{N}_{\msf{loc}}}

\def\So{\overline{S}}
\def\pp{\bs{p}}
\def\pw{\wt{p}}

\def\Gb{G^{\msf{bip}}}

\def\PPP{\mathop{}\!\msf{PPP}}

\newcommand{\dell}[1][]{\ifthenelse{\equal{#1}{}}{\mathop{}\!\mbb{P}_L(\mrm{d} \ell)}{\mathop{}\!\mbb{P}_L(\mrm{d} \ell_{\mathnormal{#1}})}}

\newcommand{\dbdell}[1][]{\ifthenelse{\equal{#1}{}}{\mathop{}\!\m_\msf{t}(\mrm{d}(b, \ell))}{\mathop{}\!\m_\msf{t}(\mrm{d}(b_{#1}, \ell_{#1}))}}

\newcommand{\mdp}[1][]{\ifthenelse{\equal{#1}{}}{\mathop{}\!\m(\mrm{d} p)}{\mathop{}\!\m(\mrm{d} p_{\mathnormal{#1}})}}

\newcommand{\ps}[1][]{\ifthenelse{\equal{#1}{}}{p_\mathnormal{\msf{s}}}{p_{{\mathnormal{\msf s}},#1}}}

\newcommand{\pt}[1][]{\ifthenelse{\equal{#1}{}}{p_\mathnormal{\msf{t}}}{p_{{\mathnormal{\msf s}},#1}}}

\newcommand{\bsps}[1][]{\ifthenelse{\equal{#1}{}}{\bs{p}_\msf{s}}{\bs{p}_{\msf{s},#1}}}

\newcommand{\bspt}[1][]{\ifthenelse{\equal{#1}{}}{\bs{p}_\msf{t}}{\bs{p}_{\msf{t},#1}}}

\newcommand{\dpt}[1][]{\mathop{}\!\m_\msf{t}(\mrm{d} \pt[#1])}

\newcommand{\Ps}[1][]{\ifthenelse{\equal{#1}{}}{P_\mathnormal{\msf{s}}}{P_{\mathnormal{\msf{s},#1}}}}

\newcommand{\Pt}[1][]{\ifthenelse{\equal{#1}{}}{P_\mathnormal{\msf{t}}}{P_{\mathnormal{\msf{t},#1}}}}

\DeclareRobustCommand{\quote}[1]{\textquotedblleft\ignorespaces#1\ignorespaces\textquotedblright}

%% file: abstract.tex
\begin{abstract}
    We introduce a dynamic random hypergraph model constructed from a bipartite graph.
    In this model, both vertex sets of the bipartite graph are generated by marked Poisson point processes.
    Vertices of both vertex sets are equipped with marks representing their weight that influence their connection radii.
    Additionally, we assign the vertices of the first vertex set a birth-death process with exponential lifetimes and the vertices of the second vertex set a time instant representing the occurrence of the corresponding vertices.
    Connections between vertices are established based on the marks and the birth-death processes, leading to a weighted dynamic hypergraph model featuring power-law degree distributions.
    We analyze the edge-count process in two distinct regimes.
    In the case of finite fourth moments, we establish a functional central limit theorem for the normalized edge count, showing convergence to a Gaussian AR(2)-type process as the observation window increases.
    In the challenging case of the heavy-tailed regime with infinite variance, we prove convergence to a novel stable process that is not L\'evy and not even Markov.
\end{abstract}

%% file: introduction.tex
\section{Introduction}\label{sec:introduction}

In the last 20 years, network science has evolved into a vibrant field of research, with applications in many areas such as biology, sociology, computer science, and physics.
This is not surprising due to the fundamental nature of the idea that many systems are essentially described by a system of nodes together with a set of links indicating which nodes interact with one another~\cite{complex}.
The success of network science is based on the observation that a broad variety of networks share some fundamental features, such as being scale-free (existence of hubs) and exhibiting strong clustering.

One of the widely used models in network science is the \emph{preferential attachment model} where nodes arrive over time and connect more likely to nodes with an already high degree~\cite{barabasi}.
However, while this model has a compelling narrative and is able to reproduce the scale-free property, its most basic versions lack the crucial feature of clustering.
While several solutions have been proposed to address this weakness, one of the most elegant and powerful ones is the idea to embed the components in some ambient space and make the connection probabilities distance dependent.
Then, the triangle inequality from the ambient space implies the desired clustering, and this idea led to the development of the spatial preferential attachment models~\cite{jacMor1,jacMor2}.

While spatial preferential attachment exhibits many of the qualitative features of real complex networks, they are notoriously hard to analyze from a rigorous mathematical perspective.
This is because the existence of edges can only be understood through following the complex time-dynamics of the evolving networks.
Due to these disadvantages, recently the class of \emph{spatial inhomogeneous random graphs} or \emph{weight-dependent random connection models} has gained popularity~\cite{komjathy,gracar2022recurrence,glm,glm2}.
Loosely speaking, the connection probabilities are here not driven by the actual node degrees but rather by their conditional expected values given their birth time.
This is a dramatic simplification, as now the edge connection event can be determined entirely from the knowledge of the position and the appearance times of the nodes involved in the pair.

Despite the success of network science, with the rise of the amount and complexity of data available, it is now apparent that in many settings a simple network with binary interactions is insufficient, and we need to take into account \emph{higher-order interactions}.
This urgent need has been a main driver of the recent research effort in \emph{hypergraphs} and \emph{simplicial complexes}~\cite{hon,bianconi,edHar,yvinec}.
For instance, in economics, several companies might simultaneously be exposed to the same risk, or in neuroscience, entire groups of neurons need to fire to create an effect~\cite{scaff,curto}.
Finally, in social sciences, a prototypical example is collaboration networks where several researchers work together on a paper~\cite{scolamiero,carstens,patania}.

Due to the explosion of complexities, modeling such higher-order networks is much harder than modeling binary networks.
Inspired by ensembles from statistical physics, a powerful class of such models is proposed in~\cite{flavor}, but it is challenging to analyze from a rigorous mathematical perspective~\cite{sulzbach}.
The generic model is the \emph{multiparameter simplicial complex} that is a higher-order extension of the Erd\H{o}s--R\'enyi model~\cite{nanda}, but it lacks the scale-free property.
However, as noted before, as in the case of binary networks, enforcing clustering properties can be elegantly implemented by embedding nodes in an ambient space.
In an earlier work, we investigated such systems when higher-order networks are given by cliques in scale-free networks~\cite{hj23}, and, in~\cite{juh}, we introduced a novel spatial model called \emph{random connection hypergraphs}, which is based on a bipartite Poisson point process.

A particular challenge in analyzing higher-order networks is that they change over time.
For instance, the topology of the neuronal network changes through the process of learning, while in scientific collaboration networks, PhD students enter the network when completing their first paper and leave it when they retire.
There are currently highly active research streams in topological data analysis with the goal of providing a toolkit to track the topological changes over time, e.g., vines and zigzag persistence~\cite{vine1, vine2, vine3}.
The goal in this paper is to endow the random connection model from~\cite{juh} with a time dynamics and to prove functional central and stable limit theorems in this setting.
Here, we focus our attention on the most fundamental quantity, namely the edge count of the underlying bipartite graph.

Before we describe more precisely our results, we comment on the existing literature of functional limit results in time-varying models.
The earliest investigations in this context consider higher-order networks defined on a discrete point set, such as the clique complex associated with a dynamic Erd\H{o}s--R\'enyi graph or the multiparameter simplicial complex~\cite{clique,gugan}.
In these works, topological quantities such as Betti numbers and Euler characteristics are considered.
However, since the networks do not exhibit the occurrence of hubs, we only see light-tailed contributions, and therefore the limits are Gaussian processes, or more precisely Ornstein--Uhlenbeck processes in the case of~\cite{clique}.
Additionally, a major disadvantage of these works is that the node set is entirely discrete.
More recently, however, spatial models were also considered~\cite{onaran,fclt}.
Here, the dynamics are both of movement and of birth-death type, and the statistics are of local nature.
However, again, since the models do not allow for the occurrence of hubs, all limits are Gaussian.

In our setting, we are able to go beyond this limiting behavior and prove two main results.
First, a functional central limit theorem (CLT) for the normalized edge count in growing domains when the fourth moment of the degree distribution is finite.
Second, a convergence of the normalized edge count to a certain stable process.
While this process is not L\'evy and not even Markov, we nevertheless provide a specific representation.
In both limits, the main methodological challenge is to deal with the long-range correlations and heavy tails coming from the hub nodes.
We now discuss briefly the proof ideas with further details provided in Sections~\ref{sec:proof_univariate_normal}, \ref{sec:proof_multivariate_normal}, \ref{sec:edge_count_gaussian}, and~\ref{sec:edge_count_stable} below.

We start with the setting of the finite fourth moments.
Here, we proceed via the classical two steps, namely convergence of the finite-dimensional marginals and tightness.
Already for the finite-dimensional marginals, the methodology from~\cite{onaran} breaks down as our functionals are not local, and we cannot rely on the stabilization CLTs from~\cite{yukCLT}.
Hence, our approach is to use instead the Malliavin--Stein normal approximation from~\cite{normalapproximation}, which involves delicate higher-order moment computations.
For the tightness, we essentially use the moment-based criterion from~\cite{billingsley}.
The fourth moments are approached using cumulant computations.
However, since our functional can be represented as a difference of monotone functionals, a refinement due to~\cite{davydov1996weak} allows us to simplify the tightness proof.

In the case of infinite variance, we use a different approach.
We distinguish between lightweight and heavyweight nodes.
First, we show that the lightweight nodes are negligible.
The difficulty here is that the negligibility needs to be established in the process sense, and the involved processes are not martingales.
Hence, we refine our tightness computations to achieve this goal.
Then, it remains to deal with the heavy nodes.
Since the edge count is Poisson conditioned on the weight, it concentrates around its conditional mean.
This allows us to pass from a space-time process to a classical time-varying stochastic process.
While the limit process is not L\'evy, the classical proof machinery for L\'evy-type convergence from~\cite{heavytails} is flexible enough to also be applicable in our more involved setting.
More precisely, we first show that the edge count converges to a stable process with index $\a \in (1, 2)$.
The proof of the convergence is based on the method of moments, and the tightness is shown using the method of characteristic functions.
To summarize, the main contributions of our work are as follows.
\begin{enumerate}
    \item
        We propose a dynamic version of the random connection hypergraph model from~\cite{juh} where nodes are born and die over time.
        This model therefore features both higher-order interactions, spatial effects, and scale-free degree distributions.
    \item
        In the case of finite fourth moments of the degree distribution, we establish the convergence of the edge count to a Gaussian process of AR(2)-type.
        Here, the mechanism of the random connection hypergraphs induces long-range dependencies that are absent in the models considered in the literature so far.
        This requires us to consider normal approximation techniques that are completely different from the ones used in earlier works.
    \item
        Our work is the first to establish a functional limit theorem for the edge count in the particularly challenging case of infinite variance of the degree distribution.
        The process limit is a stable process that is not L\'evy and not even Markov.
        Loosely speaking, it can be considered as a variant of a L\'evy process where the noise variables have a finite lifetime.
\end{enumerate}

While the focus of the investigation in this paper is on the case of the edge count, we are convinced that the methods from our work will also be useful to study more complex functionals such as simplex counts or Betti numbers.
We leave these topics for future investigations.

The rest of the paper is organized as follows.
In Section~\ref{sec:model}, we introduce the considered model and present our main results.
In Sections~\ref{sec:proof_univariate_normal},~\ref{sec:proof_covariance_function}, and~\ref{sec:proof_multivariate_normal} we give the proofs of Propositions~\ref{prop:univariate_normal},~\ref{prop:covariance_function_of_Snt}, and~\ref{prop:multivariate_normal}, respectively.
Next, in Sections~\ref{sec:edge_count_gaussian} and~\ref{sec:edge_count_stable} we present the outlines of the proofs of Theorems~\ref{thm:functional_normal} and~\ref{thm:functional_stable}, respectively.
In the final part of the paper, we exhibit the proofs of the minor propositions and lemmas used in the main proofs presented in the previous part of the paper.
In Sections~\ref{sec:minor_lemmas} we present the preliminary lemmas used frequently in the later part of the paper.
Section~\ref{sec:proofs_univariate_multivariate} contains the proofs of the lemmas used for Propositions~\ref{prop:univariate_normal},~\ref{prop:covariance_function_of_Snt}, and~\ref{prop:multivariate_normal}.
Section~\ref{sec:proofs_functional_normal} gives the proofs of the lemmas used to show Theorem~\ref{thm:functional_normal}.
In Section~\ref{sec:proofs_functional_stable}, we present the proofs of the lemmas used to prove Theorem~\ref{thm:functional_stable}.
Finally, in Section~\ref{sec:proofs_minor_lemmas} we conclude the paper by proving the lemmas presented in Section~\ref{sec:minor_lemmas}.

%% file: model.tex
\section{Model and main results}\label{sec:model}

In this section, we present the \emph{dynamic random connection hypergraph model} (DRCHM) and state our main results.
First, consider two independent Poisson point processes
\[ \PP \su \S \ti \T := (\R \ti [0, 1]) \ti (\R \ti \R_+) \qquad \text{and} \qquad \PP' \su \S \ti \R, \]
that we refer to as \emph{vertices} and \emph{interactions}, respectively.
We will refer to the four random variables characterizing a vertex $P := (X, U, B, L) \in \PP$ as the \emph{position}, \emph{weight}, \emph{birth time}, and \emph{lifetime} of the vertex $P$, respectively.
Similarly, in the case of an interaction $P' := (Z, W, R) \in \PP'$, the three real random variables denote the \emph{position}, \emph{weight}, and \emph{time} of the interaction $P'$.
If the point $P$ or $P'$ is decorated with some indices, its coordinates $X, U, B, L$ or $Z, W, R$ receive the same indices as well.
Whenever we refer to a nonrandom point, we will use the notations $p := (x, u, b, \ell)$ and $p' := (z, w, r)$.
We emphasize that the above symbols are reserved for denoting the coordinates of vertices and interactions, and will be used throughout the remainder of the paper without further explicit definition.
A pair of vertices $P \in \PP$ and $P' \in \PP'$ are connected in the \emph{bipartite graph} $\Gb := \Gb(\PP, \PP')$ if and only if the following two conditions hold:
\begin{equation}
    \abs{X - Z} \le \b U^{-\g} W^{-\g'} \qquad \textrm{and} \qquad B \le R \le B + L,
    \label{eq:wrchm}
\end{equation}
where $\b > 0$ and $\g, \g' \in (0, 1)$ are real parameters.
In words, vertices and interactions share an edge based on the smallness of their spatial distance compared to the product of their associated weights, with small weights favoring connections, and under the constraint that the time of the interaction is within the lifetime of the vertex.
Note that the parameter $\b$ influences the expected number of edges in $\Gb$, and, as it will be shown later, the parameters $\g, \g'$ determine the power-law exponents of the degree distribution of the vertices and interactions.
In the following, the parameter $\b > 0$ is chosen arbitrarily, therefore we will not mention it explicitly throughout the presentation of the results.
The connection condition~\eqref{eq:wrchm} is a natural extension of the age-dependent random connection models from~\cite{glm2,glm}.
The spatial connection condition in~\eqref{eq:wrchm} is visualized in Figure~\ref{fig:spatial_connection_condition}.
Additionally, Figure~\ref{fig:temporal_connection_condition} illustrates the temporal connection condition of the DRCHM, where the horizontal axis represents the time, and the vertical axis represents the position.
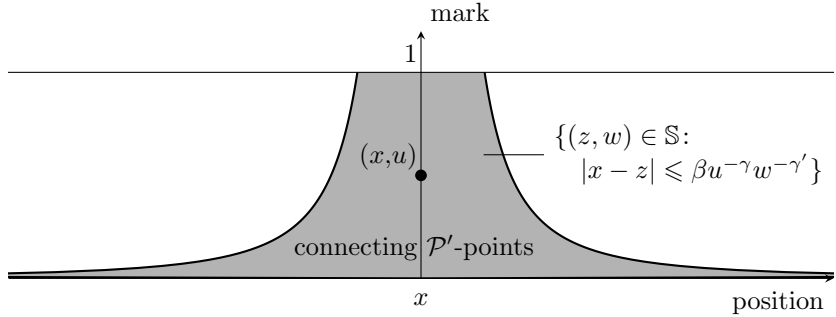
\begin{figure} [h] \centering
    \begin{tikzpicture}
        \begin{axis}[
                scale=1.6,
                axis lines=middle, 
                axis on top=true, 
                unit vector ratio*=1 1 1, 
                xmin=-2, xmax=2,
                ymin=0, ymax=1.2,
                tick style={draw=none},
                xtick={0.001},
                xticklabels={$x$},
                ytick={0.5, 1},
                yticklabels={$(x{,} u)$, $1$},
                xlabel=position,
                x label style={yshift=-0.6cm},
                ylabel=mark,
                y tick label style={
                        xshift=+1.5mm,
                        yshift=+2.5mm
                    },
                y label style={yshift=0.5cm},
            ]
            \addplot [black] {1};
            \addplot [black, only marks, mark size=2pt] coordinates {(0, 0.5)};
            \addplot [black,   thick, samples=100, domain=0.0:1.0, name path=left   ] (+0.1 * x^(-0.5) * 0.2^(0.3 - 1), x); 
            \addplot [black,   thick, samples=100, domain=0.0:1.0, name path=right  ] (-0.1 * x^(-0.5) * 0.2^(0.3 - 1), x); 
            \addplot [black!30]   fill between [of=left and right,   reverse=true,];
            \draw[-, shorten >=0.1mm, shorten <=0.1mm] (axis cs:+0.6   , 0.60) node[right, align=left  , anchor=west ] {
                $\set{ (z, w) \in \S \co$ \\ \quad $\abs{x - z} \le \b u^{-\g} w^{-\g'}}$} -- (axis cs:+0.3, 0.6);
            \draw[-, shorten >=0.1mm, shorten <=0.1mm] (axis cs:-0.66  , 0.15) node[right, align=left  , anchor=west ] {connecting $\PP'$-points};
        \end{axis}
    \end{tikzpicture}
    \caption{
        The spatial connection condition of the DRCHM.
        The horizontal axis represents the position of a vertex $P \in \PP$, and the vertical axis represents the mark of $P$.
        The shaded area indicates the set of points $(z, w) \in \S$ that connect to $P$ based on the connection Condition~\eqref{eq:wrchm}.}
    \label{fig:spatial_connection_condition}
\end{figure}
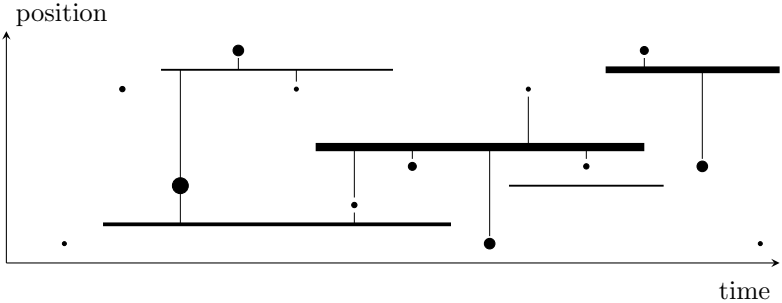
\begin{figure} [h] \centering
    \begin{tikzpicture}
        \begin{axis}[
                scale=1.5,
                axis lines=middle, 
                axis on top=true, 
                unit vector ratio*=1 1 1, 
                xmin=0, xmax=4,
                ymin=0, ymax=1.2,
                tick style={draw=none},
                xtick=\empty,
                ytick=\empty,
                xlabel=time,
                x label style={yshift=-0.6cm},
                ylabel=position,
                y tick label style={
                        xshift=+1.5mm,
                        yshift=+2.5mm
                    },
                y label style={yshift=0.5cm},
            ]
            \addplot [black, domain=0.8:2.0, line width=0.25mm] {1};
            \addplot [black, domain=3.1:4.0, line width=0.90mm] {1};
            \addplot [black, domain=1.6:3.3, line width=1.10mm] {0.6};
            \addplot [black, domain=2.6:3.4, line width=0.25mm] {0.4};
            \addplot [black, domain=0.5:2.3, line width=0.50mm] {0.2};
            \addplot [black, only marks, mark size=0.75pt] coordinates {
                (0.3, 0.1)
                (1.5, 0.9)
                (2.7, 0.9)
                (3.9, 0.1)
            };
            \addplot [black, only marks, mark size=1pt] coordinates {
                (0.6, 0.9)
                (1.8, 0.3)
                (3.0, 0.5)
            };
            \addplot [black, only marks, mark size=3pt] coordinates {
                (0.9, 0.4)
            };
            \addplot [black, only marks, mark size=1.5pt] coordinates {
                (2.1, 0.5)
                (3.3, 1.1)
            };
            \addplot [black, only marks, mark size=2pt] coordinates {
                (1.2, 1.1)
                (2.5, 0.1)
                (3.6, 0.5)
            };
            \draw[shorten <=1mm, line width=0.1mm] (axis cs:0.9, 0.4) -- (axis cs:0.9, 0.2);
            \draw[shorten <=1mm, line width=0.1mm] (axis cs:0.9, 0.4) -- (axis cs:0.9, 1.0);
            \draw[shorten <=1mm, line width=0.1mm] (axis cs:1.2, 1.1) -- (axis cs:1.2, 1.0);
            \draw[shorten <=1mm, line width=0.1mm] (axis cs:1.5, 0.9) -- (axis cs:1.5, 1.0);
            \draw[shorten <=1mm, line width=0.1mm] (axis cs:1.8, 0.3) -- (axis cs:1.8, 0.2);
            \draw[shorten <=1mm, line width=0.1mm] (axis cs:1.8, 0.3) -- (axis cs:1.8, 0.6);
            \draw[shorten <=1mm, line width=0.1mm] (axis cs:2.1, 0.5) -- (axis cs:2.1, 0.6);
            \draw[shorten <=1mm, line width=0.1mm] (axis cs:2.5, 0.1) -- (axis cs:2.5, 0.6);
            \draw[shorten <=1mm, line width=0.1mm] (axis cs:2.7, 0.9) -- (axis cs:2.7, 0.6);
            \draw[shorten <=1mm, line width=0.1mm] (axis cs:3.0, 0.5) -- (axis cs:3.0, 0.6);
            \draw[shorten <=1mm, line width=0.1mm] (axis cs:3.3, 1.1) -- (axis cs:3.3, 1.0);
            \draw[shorten <=1mm, line width=0.1mm] (axis cs:3.6, 0.5) -- (axis cs:3.6, 1.0);
        \end{axis}
    \end{tikzpicture}
    \caption{
        The temporal connection condition of the DRCHM based on~\eqref{eq:wrchm}.
        The horizontal axis represents the time, and the vertical axis represents the position.
        $\PP$-vertices are represented by intervals of different lengths and widths, where the width of the interval corresponds to the weight of the vertex.
        $\PP'$-vertices are represented by points, and the size of the point corresponds to the weight of the vertex.
        Vertical lines mark if a pair of $\PP$- and $\PP'$-vertices is connected.
    }
    \label{fig:temporal_connection_condition}
\end{figure}

We assume that the Poisson point process $\PP$ is stationary in the first and third components, the second component is uniformly distributed, and the fourth component is distributed according to $\mbb{P}_L$, i.e., the intensity measure is given by
\[ \mdp = \m(\d x, \d u, \d b, \d \ell) := \d (x, u, b) \dell, \]
and assume that $\PP'$ is also stationary with intensity $\d (z, w, r)$.

Next, for a point $p := (x, u, b, \ell) \in \PP$, we define the spatial part $\ps := (x, u)$ and the temporal part $\pt := (b, \ell)$ of~$p$, together with their corresponding measures $\d \ps := \d (x, u)$ and $\dpt := \dbdell = \d b \dell$, respectively.
Similarly, for a point $p' := (z, w, r) \in \PP'$, we define its spatial part $\ps' := (z, w)$ with the corresponding measure $\d \ps' := \d (z, w)$.

Furthermore, we consider the \emph{spatial neighborhood} $\Ns(\ps)$, the \emph{temporal neighborhood} $\Nt(\pt; t)$, and the general \emph{neighborhood} $N(p; t)$ as
\[ \begin{aligned}
    \Ns(\ps) &:= \set{(z, w) \in \S \co \abs{x - z} \le \b u^{-\g} w^{-\g'}} \\
    \Nt(\pt, t) &:= \set{r \in \R \co b \le r \le t \le b + \ell} \\
    N(p; t) &:= \Ns(\ps) \ti \Nt(\pt; t).
\end{aligned} \]
The \emph{degree} of $p$ at time $t$ is then defined as the number of interactions connecting to $p$ at time $t$, i.e.,
\[ \deg(p; t) := \sum_{P' \in \PP'} \1{P' \in N(p; t)}. \]
Using this, we define our central quantity of interest, the \emph{edge count} and the \emph{normalized edge count} as
\[ S_n(\any) := \sum_{P \in \PP \cap(\S_n \ti \T)} \deg(P; \any) \qquad \text{and} \qquad \So_n(\any) := n^{-1/2} (S_n(\any) - \E{S_n(\any)}), \]
respectively, where $\S_n := [0, n] \ti [0, 1]$ is an \emph{observation window} of size $n$.

Our first statement is about the univariate normal limit of the edge count for a fixed time point.
\begin{proposition}[Univariate normal limit of edge count]\label{prop:univariate_normal}
    Let $\g, \g' < 1/2$.
    Then, for all $t \in \R$, the normalized edge count $\So_n(t)$ converges weakly to a normal distribution, as $n \to \ff$.
\end{proposition}

The following statement provides the limiting covariance function in the thin-tailed case.
\begin{proposition}[Limiting covariance function of~$\So_n$] \label{prop:covariance_function_of_Snt}
    Let $\g, \g' < 1/2$.
    Then, for all $t_1 < t_2$, the limiting covariance function of the edge count $\So_n$ is given by
    \begin{equation}
        \lim_{n \tff} \Cov[\big]{\So_n(t_1), \So_n(t_2)} = \big( c_1 + c_3 + c_2 (2 + t_2 - t_1) \big) \e^{-(t_2 - t_1)},
        \label{eq:cov_func_total}
    \end{equation}
    where $c_1 = 2 \b / ((1 - \g) (1 - \g'))$, $c_2 = (2 \b)^2 / ((1 - 2 \g) (1 - \g')^2)$ and $c_3 = (2 \b)^2 / ((1 - \g)^2 (1 - 2 \g'))$.
\end{proposition}
\noindent
The limiting covariance function is an appropriate parameterization of the Mat\'ern covariance function~\cite{matern2013spatial} with smoothness parameter $\bar\nu = 3/2$, which is the covariance function of a continuous-time AR(2) process.

Next, we turn our attention to the multivariate normal limit of the edge counts.
\begin{proposition}[Multivariate normal limit of the edge count]\label{prop:multivariate_normal}
    Let $\g < 1/2$ and $\g' < 1/3$.
    Then, for all $k \in \Z_+$ and $t_1, \dots, t_k \in \R$, the vector of normalized edge counts $(\So_n(t_1), \dots, \So_n(t_k))$ converges weakly to a multivariate normal distribution, as $n \to \ff$.
\end{proposition}

Functional limit theorems play a key role in understanding the dynamic behavior of complex random systems.
In the context of dynamic higher-order networks, edge counts and their evolution in time describe dependencies between the heavy-tailed degree distribution and the temporal evolution of the vertices.
Thus, establishing functional limit theorems for the edge counts provides insights into the interplay between the spatial distribution and the temporal variability of these systems.
Let $X := \set{X(t) \co t \ge 0}$ denote a Gaussian process with \emph{AR(2)} covariance function
\[ \Cov{X(t_1), X(t_2)} = \big( c_1 + c_3 + c_2 (2 + \abs{t_2 - t_1}) \big) \e^{-\abs{t_2 - t_1}}, \]
where the constants are presented in Proposition~\ref{prop:covariance_function_of_Snt}.
Then, the first main result of this paper is the following.
\begin{theorem}[Functional normal limit of edge count]\label{thm:functional_normal}
    Let $\g, \g' < 1/4$.
    Then,
    \[ \So_n(\any) \xlongrightarrow[n \tff]{d} X(\any) \]
    in the Skorokhod space $D([0, 1], \R)$.
\end{theorem}
\noindent
For the precise definition of the Skorokhod metric $d_\msf{Sk}$, we refer the reader to Definition~\ref{def:skorokhod_metric}.
\begin{remark}
    We conjecture that Theorem~\ref{thm:functional_normal} could be proved for $\g \in (0, 1/2)$ instead of $\g \in (0, 1/4)$, using Lemma~\ref{lem:low_mark_edge_count_negligible} showing that the low-mark edge count converges to $0$ in probability if $\g < 1/2$.
    We refrain from presenting the proof in this version, since this would considerably lengthen the manuscript.
\end{remark}

When $\g > 1/2$, the variance of the edge count $\So_n(t)$ diverges as $n \to \ff$ for any $t \in \R$.
This means that a few vertices with very high degrees dominate the edge count, leading to a heavy-tailed distribution.
Thus, if $\g > 1/2$, the limiting process is not a Gaussian process.

To state our second main result, let $\nu$ be the measure on $\J := [0, \ff)$ defined by $\nu([\eps, \ff)) := ((2 \b) / (1 - \g'))^{1 / \g} \eps^{-1 / \g}$ and let $\PP_\ff$ denote the Poisson point process on $\J \ti \T$ with intensity measure $\nu \ti \Leb \ti \mbb{P}_L$.
Note that the component $\J$ corresponds to the limit of the appropriately scaled size of the spatial neighborhoods of the heaviest vertices in the sense detailed in Lemma~\ref{lem:convergence_to_levy} later.
Furthermore, define
\begin{equation}
    S_\eps^\ast(\any) := \sum_{(J, B, L) \in \PP_\ff} J (\any - B) \1{J \ge \wt{c} \eps^\g} \1{B \le \any \le B + L} \quad \text{and} \quad \So_\eps^\ast(\any) := S_\eps^\ast(\any) - \E{S_\eps^\ast(\any)},
    \label{eq:definition_S_eps_ast}
\end{equation}
where $\wt{c} = 2 \b / (1 - \g')$ and $\E{S_\eps^\ast(\any)} = \wt{c}^{1 / \g} \eps^{-(1/\g - 1)} / (1 - \g)$.
Then, we have the following stable limit theorem.
\begin{theorem}[Functional stable limit of edge count]\label{thm:functional_stable}
    Let $\g > 1/2$ and $\g' < 1/4$.
    Then, the limit $\So(\any) := \lim_{\eps \downarrow 0} \So_\eps^\ast(\any)$ exists in the Skorokhod space $D([0, 1], \R)$, and the centered edge-count process $\So_n(\any)$ converges weakly to the process $\So(\any)$ in the Skorokhod space $D([0, 1], \R)$.
\end{theorem}
\begin{remark}
    In Section 1, we mentioned that our limiting process $\So(\any)$ fails to be a L{\'e}vy process or even to be Markov.
    After having derived a precise mathematical expression, we now elaborate on these properties in further detail.
    \begin{itemize}
        \item
            The processes $\So_\eps^\ast(\any)$ are not Markov.
            Loosely speaking, knowing the edge count $\So_\eps^\ast(\any)$ is insufficient to predict the evolution of the process.
            To see this, suppose that there is a single $\PP$-vertex with $2$ edges at time $t$.
            If this vertex dies, both its edges simultaneously vanish.
            In contrast, if there are $2$ $\PP$-vertices each with a single edge, the process evolves more smoothly.
            Thus, the edge count $\So_\eps^\ast(\any)$ alone does not capture enough information to predict the future of the process.
        \item
            Since the processes $\So_\eps^\ast(\any)$ are not Markov, they cannot be a L{\'e}vy processes either.
            Moreover, the increments of $\So_\eps^\ast(\any)$ are not independent.
            Consider two adjacent time intervals $[t, t + \eps_1]$ and $[t + \eps_1, t + \eps_1 + \eps_2]$ for some $\eps_1, \eps_2 > 0$.
            $\PP$-vertices alive at time $t + \eps_1$ are likely to remain alive at time $t + \eps_1 + \eps_2$ as well.
            Thus, the increments are not independent.
    \end{itemize}
    Note that the above properties carry over to the limiting process $\So(\any)$ as well when taking the limit $\eps \to 0$.
\end{remark}
\noindent
Before turning to the proofs, we briefly describe the main techniques used in each of the following sections.
\begin{itemize}
    \item \textbf{Proof of Proposition~\ref{prop:univariate_normal}}.
        We prove the convergence of the normalized edge count to a Gaussian distribution using the theory of associated random variables.
        The key tool is a CLT for sums of associated Poisson functionals.
    \item \textbf{Proof of Proposition~\ref{prop:covariance_function_of_Snt}}.
        With $\g < 1/2$, we calculate the covariance function of the normalized edge count $\So_n(t)$ by decomposing the edge count to a sum of three terms and then calculating the covariance function for each part.
    \item \textbf{Proof of Proposition~\ref{prop:multivariate_normal}}.
        We establish the multivariate normal limit by applying the Malliavin--Stein approximation.
        More precisely, we bound the so-called $d_3$~distance of the distribution of the edge counts and the multivariate normal distribution.
        This is done by applying~\cite[Theorem~1.1]{normalapproximation2}, which involves the examination of cost operators.
    \item \textbf{Proof of Theorem~\ref{thm:functional_normal}}.
        We show functional convergence in the Skorokhod space by decomposing the edge count into differences of monotone functionals and applying a refined tightness criterion due to Davydov~\cite[Theorem~2]{davydov1996weak}.
        The main challenge of this proof is to show the refined tightness criterion required by the theorem.
    \item \textbf{Proof of Theorem~\ref{thm:functional_stable}}.
        In the proof of this theorem, the main challenge is that the variance of the edge count is infinite.
        For infinite-variance cases, we split the edge count into contributions from light- and heavy-tailed nodes.
        We show the light-node contribution is negligible and the heavy-node contribution converges to a stable process that is not Markov.
        The proof combines conditional expectation approximations, Poisson point process convergence, and vague convergence of L\'evy measures.
\end{itemize}

\medskip
Before entering proofs, let us also introduce the \emph{common neighborhood} of a set of $m$ points $\pp_m := (p_1, \dots, p_m) \in \PP^m$, i.e.,
\[ \Ns(\bsps[m]) := \bigcap_{i = 1}^m \Ns(\ps[i]) \qquad \Nt(\bspt[m]; t) := \bigcap_{i = 1}^m \Nt(\pt[i]; t) \qquad N(\pp_m; t) := \bigcap_{i = 1}^m N(p_i; t), \]
and the associated degree
\[ \deg(\pp_m; t) := \sum_{P' \in \PP'} \1{P' \in N(\pp_m; t)}. \]
Next, for Borel sets $\XX, \UU, \BB, \DD \su \R$, we introduce the notations
\[ \begin{array}{lll}
    \S_\XX := \set{(x, u) \in \S \co x \in XX} & \S^\UU := \set{(x, u) \in \S \co u \in \UU} & \S_\XX^\UU := \S_\XX \cap \S^\UU \\
    \T_\BB := \set{(b, \ell) \in \T \co b \in \BB} & \T^\DD := \set{(b, \ell) \in \T \co b + \ell \in \DD} & \T_\BB^\DD := \T_\BB \cap \T^\DD
\end{array} \]
and for the edge counts restricted to a Borel measurable spatial domain $\AA$, we introduce
\[ S_\AA(t) := \sum_{P \in \PP \cap (\S_\AA \ti \T)} \deg(P; t). \]
For the coordinates of the points $\pp_m \in \PP^m$ we use the notation $\bs{x}_m$, $\bs{u}_m$, $\bs{b}_m$ and $\bs{\ell}_m$.
As the Poisson point processes are stationary, we neglect the time arguments $t$ in the notations whenever a formula contains only one time argument.
We note that for readability, we keep the indices of the constants unique within each part of the proofs.

%% file: outline.tex
\input{proofs/univariate_normal/prop_univariate_normal.tex}
\input{proofs/covariance_function/prop_limiting_covariance_function.tex}
\input{proofs/multivariate_normal/prop_multivariate_normal.tex}
\input{proofs/functional_normal/thm_functional_normal.tex}
\input{proofs/functional_stable/thm_stable.tex}

%% file: proofs/univariate_normal/prop_univariate_normal.tex
\section{Proof of Proposition~\ref{prop:univariate_normal}}\label{sec:proof_univariate_normal}

In this section, we present the proof of the univariate normal limit of the normalized edge count for a fixed time.
After calculating the mean and variance of the edge count, we apply a CLT for associated random variables to show that the edge count converges to a Gaussian distribution.
We begin with a lemma that gives the asymptotic behavior of the mean and variance of the edge count.
\begin{lemma}[Asymptotic behavior of mean and variance of $S_n(t)$]\label{lem:mean_variance_of_Snt}
    Let $t \in \R$ and $\g' \in (0, 1)$.
    If $\g \in (0, 1)$, then $\lim_{n \tff} n^{-1} \E{S_n(t)} < \ff$.
    Furthermore, if $\g < 1/2$, then also $\lim_{n \tff} n^{-1} \Var{S_n(t)} < \ff$.
\end{lemma}
We present the proofs of this and the following statement in Section~\ref{sec:proofs_univariate_multivariate}.
The next lemma gives the covariance of the edge counts of different spatial domains at a fixed time.
\begin{lemma}[Covariance of edge counts at fixed times]\label{lem:covariance_of_Snt}
    Let $t \in \R$, $\g, \g' < 1/2$, and let~$\AA_1, \AA_2$ be two disjoint Borel sets.
    Then,
    \[ \Cov{S_{\AA_1}(t), S_{\AA_2}(t)} \le \f{\abs{\AA_1}}{2 (1 - 2 \g')} \Big( \f{2 \b}{1 - \g} \Big)^2. \]
    In particular, the right-hand side is independent of the time~$t$ and of the set $\AA_2$.
\end{lemma}
To show the Gaussian limit, we apply the CLT for associated random variables~\cite[Theorem~4.4.3]{whitt}, which states that if $\bs{T} := T_1, T_2, \dots$ is a sequence of associated iid random variables with $\sum_{k \ge 1} \Cov{T_1, T_k} < \ff$, then the centered and normalized sum converges to a Gaussian distribution.
Let us recall that a sequence of random variables $\bs{T}$ is associated if and only if $\Cov{f(T_1, \dots, T_k), g(T_1, \dots, T_k)} \ge 0$ for all non-decreasing functions~$f$, $g$ for which $\E{f(T_1, \dots, T_k)}$, $\E{g(T_1, \dots, T_k)}$, and $\E{f(T_1, \dots, T_k) g(T_1, \dots, T_k)}$ exist~\cite[Definition~1.1]{associatedrvs}.
In order to see this, we partition the spatial coordinates of the edges into intervals of length one and define
\[ T_i := \sum_{P \in \PP \cap (\S_{[i - 1, i]} \ti \T)} \deg(P; t). \]
Since $\deg(P; t)$ is increasing in the Poisson process $\PP'$, we conclude from the Harris-FKG inequality~\cite[Theorem~20.4]{poisBook} that the sequence $T_1, \dots, T_k$ is associated.
For the covariance condition, note that we have $T_i \in O(1)$ for each $i \in \N$ by Lemma~\ref{lem:mean_variance_of_Snt}, where we require that $\g < 1/2$.
By Lemma~\ref{lem:covariance_of_Snt}, we have that
\[ \sum_{k \ge 2} \Cov{T_1, T_k} = \Cov[\Big]{T_1, \sum_{k \ge 2} T_k} < \ff. \]
Thus, $\sum_{k \ge 1} \Cov{T_1, T_k}$ is finite, and Proposition~\ref{prop:univariate_normal} is proved.

%% file: proofs/covariance_function/prop_limiting_covariance_function.tex
\section{Proof of Proposition~\ref{prop:covariance_function_of_Snt}}\label{sec:proof_covariance_function}

The calculation of the limiting covariance function of $\So_n$ relies on a time-interval based decomposition of the edge count.
Without loss of generality, we may assume that $t_1 \le t_2$.
First, note that
\[ \Cov{\So_n(t_1), \So_n(t_2)} = n^{-1} \Cov{S_n(t_1), S_n(t_2)}. \]
Now, to simplify the calculations, we decompose the edge count into three parts as
\[ \begin{aligned}
    S_n^\msf{A}(t_1, t_2) &:= \sum_{P \in \PP \cap (\S_n \ti \T_{\le t_1}^{t_2 \le})} \deg(P; t_1) \\
    S_n^\msf{B}(t_1, t_2) &:= \sum_{P \in \PP \cap (\S_n \ti \T_{\le t_1}^{[t_1, t_2]})} \deg(P; t_1) \\
    S_n^\msf{C}(t_1, t_2) &:= \sum_{P \in \PP \cap (\S_n \ti \T_{\le t_2}^{t_2 \le})} \sum_{P' \in \PP'} \1{P' \in N(P; t_2)} \1{t_1 \le R},
\end{aligned} \]
where we abbreviate $\T_{\le t_1}^{t_2 \le} = \T_{(-\ff, t_1]}^{(t_2, \ff)}$, $\T_{\le t_1}^{[t_1, t_2]} = \T_{(-\ff, t_1]}^{[t_1, t_2]}$ and $\T_{\le t_2}^{t_2 \le} = \T_{(-\ff, t_2]}^{(t_2, \ff)}$ and present a visualization in Figure~\ref{fig:covariance_intervals}.
\begin{figure}[htb] \centering
    \begin{tikzpicture}[scale = 0.4]
        \draw[->,line width=0.07cm] (0,-3.75) -- (30,-3.75);
        \node[draw=none, anchor=north] at (30, -3.75) {$\R$};
        \node[draw=none, anchor=north] at (10, -3.75) {$t_1$};
        \node[draw=none, anchor=north] at (20, -3.75) {$t_2$};
        \node[draw=none] at (-2, 6.25) {\Large{$A$}};
        \draw[line width=0.025cm] ( 0,6.75) -- (10,6.75);
        \draw[line width=0.025cm] (20,6.75) -- (30,6.75);
        \draw[dashed, line width=0.025cm] (0,6) -- (10,6);
        \draw[fill=white] (10,6) circle (0.3);
        \draw[fill=white] (10,6.75) circle (0.3);
        \draw[fill=white] (20,6.75) circle (0.3);
        \node[draw=none, anchor=south] at ( 5, 6.75) {$b$};
        \node[draw=none, anchor=south] at (25, 6.75) {$b+\ell$};
        \node[draw=none, anchor=north] at ( 5, 6) {$r$};
        \node[draw=none] at (-2, 2.5) {\Large{$B$}};
        \draw[line width=0.025cm] (0,2.75) -- (10,2.75);
        \draw[line width=0.025cm] (10,2.75) -- (20,2.75);
        \draw[dashed, line width=0.025cm] (0,2) -- (10,2);
        \draw[fill=white] (10,2) circle (0.3);
        \draw[fill=white] (10,2.75) circle (0.3);
        \draw[fill=white] (20,2.75) circle (0.3);
        \node[draw=none, anchor=south] at ( 5, 2.75) {$b$};
        \node[draw=none, anchor=south] at (15, 2.75) {$b+\ell$};
        \node[draw=none, anchor=north] at ( 5, 2) {$r$};
        \node[draw=none] at (-2, -1.25) {\Large{$C$}};
        \draw[line width=0.025cm] ( 0,-1) -- (20,-1);
        \draw[line width=0.025cm] (20,-1) -- (30,-1);
        \draw[dashed, line width=0.025cm] (10,-1.75) -- (20,-1.75);
        \draw[fill=white] (10,-1.75) circle (0.3);
        \draw[fill=white] (20,-1.75) circle (0.3);
        \draw[fill=white] (20,-1) circle (0.3);
        \draw[fill=white] (20,-1) circle (0.3);
        \node[draw=none, anchor=south] at (10, -1   ) {$b$};
        \node[draw=none, anchor=south] at (25, -1   ) {$b+\ell$};
        \node[draw=none, anchor=north] at (15, -1.75) {$r$};
    \end{tikzpicture}
    \caption{Decomposition of the covariance function}\label{fig:covariance_intervals}
\end{figure}
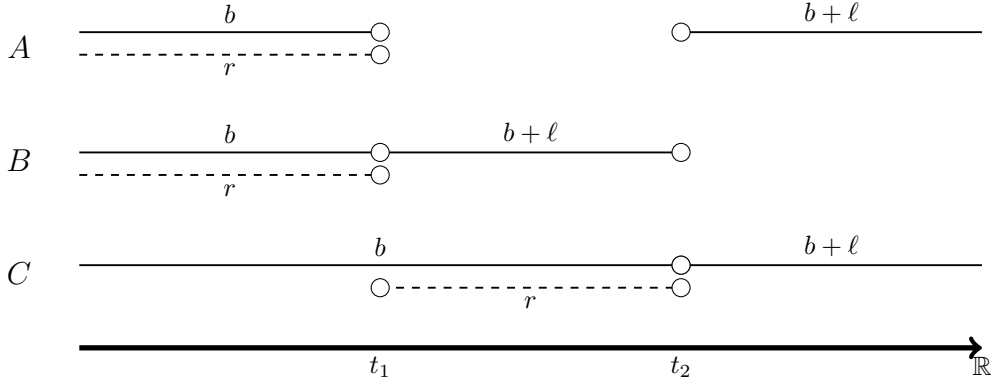
With these notations, we have that
\[ S_n(t_1) = S_n^\msf{A}(t_1, t_2) + S_n^\msf{B}(t_1, t_2) \qquad \text{and} \qquad S_n(t_2) = S_n^\msf{A}(t_1, t_2) + S_n^\msf{C}(t_1, t_2). \]
Since, by the independence of the Poisson process, $\Cov{S_n^\msf{B}(t_1, t_2), S_n^\msf{C}(t_1, t_2)} = 0$, the covariance function is given by
\[ \Cov{\So_n(t_1), \So_n(t_2)} = \Var[\big]{\So_n^\msf{A}(t_1, t_2)} + \Cov[\big]{\So_n^\msf{A}(t_1, t_2), \So_n^\msf{B}(t_1, t_2)} + \Cov[\big]{\So_n^\msf{A}(t_1, t_2), \So_n^\msf{C}(t_1, t_2)}, \]
and we have the following limiting behavior.
\begin{lemma}[Terms of the limiting covariance function of $\So_n(t)$]\label{lem:covariance_function_terms}
    Let $\g, \g' < 1/2$ and $t_1 \le t_2$.
    Then,
    \[ \begin{aligned}
        \lim_{n \tff} \Var[\big]{\So_n^\msf{A}(t_1, t_2)} &= (c_1 + 2 c_2) \e^{-(t_2 - t_1)} + c_3 \e^{-2 (t_2 - t_1)} \\
        \lim_{n \tff} \Cov[\big]{\So_n^\msf{A}(t_1, t_2), \So_n^\msf{B}(t_1, t_2)} &= c_3 \big( \e^{-(t_2 - t_1)} - \e^{-2 (t_2 - t_1)} \big) \\
        \lim_{n \tff} \Cov[\big]{\So_n^\msf{A}(t_1, t_2), \So_n^\msf{C}(t_1, t_2)} &= c_2 (t_2 - t_1) \e^{-(t_2 - t_1)},
    \end{aligned} \]
    where $c_1 = 2 \b / ((1 - \g) (1 - \g'))$, $c_2 = (2 \b)^2 / ((1 - 2 \g) (1 - \g')^2)$ and $c_3 = (2 \b)^2 / ((1 - \g)^2 (1 - 2 \g'))$.
\end{lemma}
\noindent
The proof of the lemma is given in Section~\ref{sec:proofs_univariate_multivariate} in a slightly more complex form, since the lemma is used in a very similar scenario in the proof of Lemma~\ref{lem:covariance_function_of_Snt_ge}.

Finally, by summing the above terms, we have that
\[ \lim_{n \tff} \Cov[\big]{\So_n(t_1), \So_n(t_2)} = \big( c_1 + c_3 + c_2 (2 + t_2 - t_1) \big) \e^{-(t_2 - t_1)}, \]
which finishes the proof of Proposition~\ref{prop:covariance_function_of_Snt}.

%% file: proofs/multivariate_normal/prop_multivariate_normal.tex
\section{Proof of Proposition~\ref{prop:multivariate_normal}}\label{sec:proof_multivariate_normal}

Next, we show that if $\g < 1/2$ and $\g' < 1/3$, the finite-dimensional distributions of the normalized edge count converge to a multivariate normal distribution.

To this end, we would like to apply the normal approximation result in~\cite[Theorem~1.1]{normalapproximation2}, which uses Malliavin--Stein approximation to bound the so-called $d_3$ distance between the distribution of Poisson functionals and the normal distribution.
Applying this theorem requires controlling the error terms $E_1(n), E_2(n), E_3(n)$ detailed below, and establishing that they converge to zero in probability as $n \to \ff$.
However, this can only be achieved under the stricter condition $\g < 1/3$, since it involves showing that an integral of the form $\int_0^1 u^{-3 \g} \d u$ is finite, which diverges for $\g \ge 1/3$.
To circumvent this restriction, we apply a low-mark/high-mark decomposition of the edge count $S_n$.
We first show that the contribution from low-mark edges is negligible, and then analyze the high-mark edge count separately.
This strategy allows us to establish convergence to a multivariate normal distribution under the milder condition $\g < 1/2$, rather than the stricter $\g < 1/3$.
We begin by setting a mark $u_n := n^{-2/3}$ as a function of the window size~$n$, and we decompose the edge count~$S_n$ to the sum of the high-mark edge count~$S_n^\ge$ and the low-mark edge count~$S_n^\le$ as $S_n(\any) = S_n^\ge(\any) + S_n^\le(\any)$, with
\begin{equation}
    S_n^\ge(\any) := \sum_{P \in \PP \cap (\S_n^{u_n \le} \ti \T)} \deg(P; \any) \qquad \text{and} \qquad S_n^\le(\any) := \sum_{P \in \PP \cap (\S_n^{\le u_n} \ti \T)} \deg(P; \any),
    \label{eq:high_low_mark_decomposition}
\end{equation}
where $\S_n^{\le u_n} = \S_n^{[0, u_n)}$ and $\S_n^{u_n \le} = \S_n^{[u_n, 1]}$.
This decomposition is illustrated in Figure~\ref{fig:decomposition_of_edge_count}.
\begin{figure} \centering
    \begin{tikzpicture}[>=latex]
        \begin{axis}[
                scale=0.8,
                axis x line=center,
                axis y line=center,
                axis equal image,
                xtick={0.01,6},
                xticklabels={$0$, $n$},
                ytick={0.01,3,5},
                yticklabels={$0$, $n^{-2/3}$, $1$},
                xlabel={location},
                ylabel={mark},
                xlabel style={above left},
                ylabel style={above right},
                xmin=0, xmax=8.5,
                ymin=0, ymax=5.5,
            ]
            \addplot [mark=none, dashed, domain=0:6] {3};
            \addplot [mark=none, dashed, domain=0:6] {5.0};
            \draw [dashed] (6,0) -- (6,5);
            \node [fill=white] at (3,4) {$S_n^\ge$};
            \node [fill=white] at (3,1.5) {$S_n^\le$};
        \end{axis}
    \end{tikzpicture}
    \caption{
        Decomposition of the edge count $S_n$ to a sum of the low-mark edge count $S_n^\le$ and the high-mark edge count $S_n^\ge$.
    }
    \label{fig:decomposition_of_edge_count}
\end{figure}
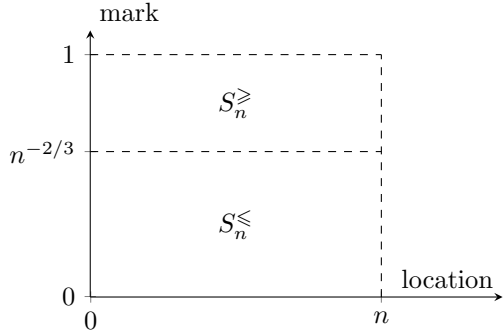
The next lemma shows that the low-mark edge count $S_n^\le$ is negligible in the sense that its normalized version
\[ \So_n^\le := n^{-1/2} \big( S_n^\le - \E{S_n^\le} \big) \]
converges to zero in probability for all times~$t$.
\begin{lemma}[The low-mark edge count is negligible]
\label{lem:low_mark_edge_count_negligible}
    Let $\g, \g' < 1/2$.
    Then, for any $t$, in probability,
    \[ \So_n^\le(t) \xlongrightarrow[n \tff]{\mbb{P}} 0.\]
\end{lemma}
\noindent
Next, we apply~\cite[Theorem~1.1]{normalapproximation2} bounding the $d_3$~distance between the distribution of $S_n^\ge$ and the normal distribution.
The definition of the $d_3$~distance, a metric on the space of random vectors is given below.
\begin{definition}[$d_3$ distance]\label{def:d3_distance}
    Let $\HH_m^{(3)}$ be the set of all $C^3$ functions $h \co \R^m \to \R$ such that the absolute values of the second and third partial derivatives of~$h$ are bounded by~$1$ and $\abs{h(x) - h(y)} \le \norm{x - y}_E$ for all $x, y \in \R^m$, where $\norm{\any}_E$ denotes the Euclidean norm.
    Then, the $d_3$ distance of two $m$-dimensional random vectors $X, Y$ with $\E{\norm{X}_E}, \E{\norm{Y}_E} < \ff$ is defined by
    \[ d_3(X, Y) := \sup_{h \in \HH_m^{(3)}} \abs[\big]{\E{h(X)} - \E{h(Y)}}. \]
\end{definition}
\noindent
Note that as convergence in $d_3$~distance implies convergence in distribution~\cite[Proposition~20.A.2]{doucmarkov}, it is enough to show convergence in $d_3$~distance to prove Proposition~\ref{prop:multivariate_normal}.
Next, we introduce the cost operators that will be used to bound the $d_3$~distance.
\begin{definition}[Cost operators]\label{def:cost_operators}
    For the edge count $S_n^\ge$, the \emph{add-one cost operators} $D_p$ and $D_{p'}$ are defined by
    \[ D_p S_n^\ge(t) := \sum_{P' \in \PP'} \1{P' \in N(p; t)} \qquad \text{and} \qquad D_{p'} S_n^\ge(t) := \sum_{P \in \PP \cap (\S_n^{u_n \le} \ti \T)} \1{p' \in N(P; t)}. \]
    Note that the cost operator $D_p S_n^\ge(t)$ is Poisson distributed since it counts the number of interactions in the neighborhood of $p$.
    Similarly, the cost operator $D_{p'} S_n^\ge(t)$ is also Poisson distributed, as it is the number of vertices such that $\ps' \in \Ns(\Ps)$ and $B \le r$ on the one hand, and this restricted Poisson process is thinned with a $B$-dependent probability $\exp{-(t - B)}$.

    Turning our attention to the \emph{add-two cost operators} $D^2_{p_1, p_2}$, $D^2_{p_1, p'_2}$, $D^2_{p'_1, p_2}$ and $D^2_{p'_1, p'_2}$, we have
    \[ \begin{aligned}
        D^2_{p_1, p_2} S_n^\ge(t) &= S_n^\ge(t) + D_{p_1} S_n^\ge(t) + S_n^\ge(t) + D_{p_2} S_n^\ge(t) + S_n^\ge(t) \\
        &\qquad - \big( S_n^\ge(t) + D_{p_1} S_n^\ge(t) + S_n^\ge(t) + D_{p_2} S_n^\ge(t) \big) = 0 \\
        D^2_{p'_1, p'_2} S_n^\ge(t) &= S_n^\ge(t) + D_{p'_1} S_n^\ge(t) + S_n^\ge(t) + D_{p'_2} S_n^\ge(t) \\
        &\qquad - \big( S_n^\ge(t) + D_{p'_1} S_n^\ge(t) + S_n^\ge(t) + D_{p'_2} S_n^\ge(t) \big) = 0 \\
        D^2_{p, p'} S_n^\ge(t) &= \1{p' \in N(p; t)} + S_n^\ge(t) + D_p S_n^\ge(t) + S_n^\ge(t) + D_{p'} S_n^\ge(t) \\
        &\qquad - \big( S_n^\ge(t) + D_p S_n^\ge(t) + S_n^\ge(t) + D_{p'} S_n^\ge(t) \big) = \1{p' \in N(p; t)},
    \end{aligned} \]
\end{definition}
\noindent
where we note that the add-two cost operators are much simpler than the add-one cost operators.

To apply~\cite[Theorem~1.1]{normalapproximation2}, we need to have a single Poisson process, which we obtain by merging the two Poisson processes $\PP$ and $\PP'$ into a single Poisson process $\wt{\PP} := \PP \sqcup \PP'$, where $\sqcup$ denotes the disjoint union of the two Poisson processes.
The point process $\wt{\PP}$ is defined on the space $\S_n^{u_n \le} \ti \T \sqcup \S \ti \R$ using appropriate marks, and its intensity measure is given by $\wt{\m}$.
In other words, if $\pw \in \wt{\PP}$, then
\[ \wt{\m}(\d \pw) = \left\{ \begin{array}{ll} \m(\d \pw) & \text{if } \pw \in \S_n^{u_n \le} \ti \T \\ \d \pw & \text{if } \pw \in \S \ti \R. \end{array} \right. \]

Now, to show the multivariate convergence, we apply~\cite[Theorem~1.1]{normalapproximation2}, which we restate here for convenience.
\begin{proposition}[Multivariate CLT]\label{prop:multivariate_normal_limit}
    Let $N_\Sigma$ be an $m$-dimensional centered multivariate normal distribution with a positive semi-definite covariance matrix $\Sigma \in \R^{m \ti m}$ with elements $\s_{i, j} \in \R$.
    Further, let $\pw \in \wt{\PP}$ and let $F := \big(\So_n^\ge(t_1), \dots, \So_n^\ge(t_m)\big)$ denote an $m$-dimensional random vector of Poisson functionals.
    Then, the distance $d_3(F, N_\Sigma)$ is upper bounded by
    \begin{equation}
        d_3(F, N_\Sigma) \le \f{m}{2} \sum_{i, j = 1}^m \abs[\big]{\s_{ij} - \Cov{\So_n^\ge(t_1), \So_n^\ge(t_2)}} + m E_1(n) + \f{m}{2} E_2(n) + \f{m^2}{4} E_3(n),
        \label{eq:d3_distance}
    \end{equation}
    where the terms $E_1(n), E_2(n), E_3(n)$ are defined as
    \[ \begin{aligned}
        E_1(n) &:= \bigg( \sum_{i, j = 1}^m \int \big( \E[\big]{(D^2_{\pw_1, \pw_3} \So_n^\ge(t_1))^2 (D^2_{\pw_2, \pw_3} \So_n^\ge(t_1))^2} \E[\big]{(D_{\pw_1} \So_n^\ge(t_2))^2 (D_{\pw_2} \So_n^\ge(t_2))^2} \big)^{1/2} \d (\pw_1, \pw_2, \pw_3) \bigg)^{1/2} \\
        E_2(n) &:= \bigg( \sum_{i, j = 1}^m \int \Big( \E[\big]{(D^2_{\pw_1, \pw_3} \So_n^\ge(t_1))^2 (D^2_{\pw_2, \pw_3} \So_n^\ge(t_1))^2} \\
        &\hspace{2.2cm} \ti \E[\big]{(D^2_{\pw_1, \pw_3} \So_n^\ge(t_2))^2 (D^2_{\pw_2, \pw_3} \So_n^\ge(t_2))^2} \Big)^{1/2} \d (\pw_1, \pw_2, \pw_3) \bigg)^{1/2} \\
        E_3(n) &:= \sum_{i = 1}^m \int \E[\big]{\abs{D_{\pw} \So_n^\ge(t_1)}^3} \d \pw.
    \end{aligned} \]
    In the above, the cost operators involving $\pw$ are defined as the number of edges of the point $\pw$, in alignment with Definition~\ref{def:cost_operators}.
    For notational convenience, we do not introduce further notation for the cost operators involving $\pw$.
\end{proposition}
\noindent
We begin by setting the elements~$\s_{ij}$ of the covariance matrix $\Sigma$ in~\eqref{eq:d3_distance}.
To do so, we show the below lemma, which is very similar to the covariance function of the total edge count introduced in Proposition~\ref{prop:covariance_function_of_Snt}.
\begin{lemma}[Limiting covariance function of~$\So_n^\ge$] \label{lem:covariance_function_of_Snt_ge}
    Let $\g, \g' < 1/2$.
    Then, the limiting covariance function of the edge count $\So_n^\ge$ is given by~\eqref{eq:cov_func_total}.
\end{lemma}
\noindent
Then, we see that the first term in the right-hand side of~\eqref{eq:d3_distance} converges to zero as $n \to \ff$.
Next, we show that the error terms $E_1(n), E_2(n), E_3(n)$ converge to zero, as $n \to \ff$.
This result is summarized in the following lemma, which we prove in Section~\ref{sec:proofs_univariate_multivariate} by examining the cost operators $D_{\pw}$, $D^2_{\pw_1, \pw_3}$ and $D^2_{\pw_2, \pw_3}$.
\begin{lemma}[Bounds of error terms]\label{lem:bounds_of_error_terms}
    Let $\g < 1/2$, $\g' < 1/3$.
    Then, the error terms defined in~\eqref{eq:d3_distance} satisfy
    \[ \lim_{n \tff} (E_1(n) + E_2(n) + E_3(n)) = 0. \]
\end{lemma}
\medskip
\noindent
In particular, this proves Proposition~\ref{prop:multivariate_normal}.

%% file: proofs/functional_normal/thm_functional_normal.tex
\section{Proof of Theorem~\ref{thm:functional_normal}}\label{sec:edge_count_gaussian}

Broadly speaking, apart from the convergence of finite-dimensional distributions showed in Proposition~\ref{prop:multivariate_normal}, we need to show that the sequence of the normalized edge counts is tight when $t \in [0, 1]$.
For this, we would like to apply~\cite[Theorem~2]{davydov1996weak}, which holds for non-decreasing processes.
The key advantage of this theorem is that to show the moment bound for tightness, it is enough to consider time increments of size larger than a minimum $n$-dependent threshold, which makes the proof more manageable.

We would like to write the edge count as the difference $S_n = S_n^+ - S_n^-$ of two non-decreasing processes.
To define $S_n^+$ and $S_n^-$, we first introduce the \quote{plus} and \quote{minus} neighborhoods of a point $p \in \S \ti \T$ as
\[ \begin{aligned}
    &N^+(p; t) := \Ns(\ps) \ti \Nt^+(\pt; t), \qquad \text{where} \qquad \Nt^+(\pt; t) := \left\{ \begin{array}{ll} \set{r \in \R \co b \le r \le b + \ell} & \text{if } b + \ell \le t \\ \set{r \in \R \co b \le r \le t} & \text{if } b + \ell > t \end{array} \right. \\
    &N^-(p; t) := \Ns(\ps) \ti \Nt^-(\pt; t), \qquad \text{where} \qquad \Nt^-(\pt; t) := \left\{ \begin{array}{ll} \set{r \in \R \co b \le r \le b + \ell} & \text{if } b + \ell \le t \\ \es & \text{if } b + \ell > t. \end{array} \right.
\end{aligned} \]
Note that the spatial parts of the neighborhoods $N^+(p; t)$ and $N^-(p; t)$ are the same as the spatial part of $N(p; t)$ and the difference is only in the temporal part.
Considering that $t \in [0, 1]$, we also define
\[ \begin{aligned}
    S_n^+(t) &:= \sum_{P \in \PP \cap( \S_n \ti \T^{0 \le})} \deg^+(P; t), \qquad \text{where} \qquad \deg^+(P; t) := \sum_{P' \in \PP'} \1{P' \in N^+(P; t)} \text{ and} \\
    S_n^-(t) &:= \sum_{P \in \PP \cap (\S_n \ti \T^{0 \le})} \deg^-(P; t), \qquad \text{where} \qquad \deg^-(P; t) := \sum_{P' \in \PP'} \1{P' \in N^-(P; t)},
\end{aligned} \]
and $\T^{0 \le} := \T^{[0, \ff)}$ denotes the set for which the death time $b + \ell \ge 0$.
In words, both $S_n^+$ and $S_n^-$ consider only points from $\PP$ whose lifetime $[B, B + L]$ intersects the temporal interval $[0, 1]$, and $\deg^+(P)$ counts the $\PP'$-points that
\begin{itemize}
    \item{are in the spatial neighborhood $\Ns(\Ps)$ of the vertex $P$ and}
    \item{have a birth time $R$ in the interval $[B, (B + L) \w t]$, regardless of the lifetime $L$.}
\end{itemize}
On the other hand, $\deg^-(P; t)$ considers a point $P$ if its lifetime ends before the time $t$, and counts those interactions in $\PP'$ that
\begin{itemize}
    \item{are in the spatial neighborhood $\Ns(\Ps)$ of the vertex $P$ and}
    \item{have a birth time $R$ in the interval $[B, B + L]$, regardless of the birth time $B$.}
\end{itemize}
We show in the following lemma that the \quote{plus-minus decomposition} holds.
\begin{lemma}[Plus-minus decomposition of edge count]
    \label{lem:plus_minus_decomposition}
    We have that $S_n = S_n^+ - S_n^-$.
\end{lemma}
\noindent
Since the difference of tight sequences is tight, it is enough to show that~\cite[Theorem~2]{davydov1996weak} holds for the normalized edge counts
\[ \So_n^+ := n^{-1/2} (S_n^+ - \E{S_n^+}) \qquad \text{and} \qquad \So_n^- := n^{-1/2} (S_n^- - \E{S_n^-}). \]
From now on, we will use the index $\any^\pm$ whenever a formula is valid for both $\any^+$ and $\any^-$.
Unless stated otherwise, all the indices of a formula are either $+$ or~$-$.
For ease of reference, we summarize the statement of~\cite[Theorem~2]{davydov1996weak} in our context.
\begin{theorem}[Specialized version of Davydov's theorem]\label{thm:davydov}
    Let $a_n \to 0$ be a sequence of positive numbers converging to $0$, and set $t_k = k a_n$ for $k = 0, 1, \dots, k_n$ with $k_n = \floor{1 / a_n}$ and $t_{k_n + 1} = 1$.
    Then, if $\So_n^\pm(\any)$ are non-decreasing processes defined on the interval $[0, 1]$ such that
    \begin{enumerate}
        \item
            the finite-dimensional distributions of the processes $\So_n^\pm(\any)$ converge to the finite-dimensional distribution of a limiting process $\So_\ff^\pm(\any)$, \label{condition:finite_dimensional}
        \item
            there exists some constants $\chi_1, \chi_2 > 1$ such that
            \[ \E[\big]{\abs[\big]{\So_n^\pm(t) - \So_n^\pm(s)}^{\chi_1}} \in O(\abs{t - s}^{\chi_2}) \]
            for all $n$ if $\abs{t - s} \ge a_n$, and \label{condition:cumulant}
        \item
            for the limit of the expected increments, we have
            \[ \lim_{n \tff} \max_{k \le \floor{1 / a_n}} \abs[\big]{\E[\big]{\So_n^\pm(t_{k + 1})} - \E[\big]{\So_n^\pm(t_k)}} = 0, \] \label{condition:third}
    \end{enumerate}
    then $\So_\ff^\pm(\any)$ is almost surely continuous and the sequence $\So_n^\pm(\any)$ converges in distribution to $\So_\ff^\pm(\any)$.
\end{theorem}

Condition~(\ref{condition:finite_dimensional}) of Theorem~\ref{thm:davydov} is fulfilled for $\So_n^\pm(\any)$, which is stated in the next proposition.
\begin{proposition}[Convergence of the finite-dimensional distributions of $\So_n^\pm$]\label{prop:davydov_condition_1}
    If $\g < 1/2$ and $\g' < 1/3$, then the finite-dimensional distributions of $\So_n^\pm(\any)$ converge to a multivariate normal distribution.
\end{proposition}

\noindent
The next proposition shows that Condition~(\ref{condition:cumulant}) holds for $\So_n^\pm$.
We set $\chi_1 := 4$, $\chi_2 := 1 + \eta$ and $a_n := n^{-1/(1 + \eta)}$ for some $\eta > 0$.
\begin{proposition}[Tightness of the sequence $\So_n^\pm$]\label{prop:tightness_So_n_pm}
    Let $\g, \g' < 1/4$.
    Then, there exists an $\eta > 0$ such that if $\abs{t - s} \ge n^{-1/(1 + \eta)}$, then
    \[ \E[\big]{\abs[\big]{\So_n^\pm(t) - \So_n^\pm(s)}^4} \in O(\abs{t - s}^{1 + \eta}) \]
    for all $n$.
\end{proposition}

\noindent
Finally, we turn our attention to Condition~(\ref{condition:third}).
Note that the absolute value can be dropped due to the monotonicity of $S_n^\pm$.
Then, the condition is fulfilled if the following lemma holds.
\begin{lemma}[Bound on the expectation of the increments of $\So_n^\pm(t)$]\label{lem:expectation_of_increments}
    Let $t_k := k n^{1 / (1 + \eta)}$ for any $k \in \N$.
    Then, for all $\eps > 0$,
    \[ \lim_{n \tff} \max_{k \le \floor{n^{1 / (1 + \eta)}}} \big( n^{-1/2} \E[\big]{\De_n^\pm(t_k, t_{k + 1})} \big) = 0. \]
\end{lemma}

The previous propositions and lemmas establish all the conditions of Theorem~\ref{thm:davydov} and hence we conclude that the sequence $\So_n^\pm(t)$ is tight, and thus the functional CLT holds for the edge count~$\So_n(t)$.

%% file: proofs/functional_stable/thm_stable.tex
\section{Proof of Theorem~\ref{thm:functional_stable}}\label{sec:edge_count_stable}

In this section, we outline the proof of the convergence of the edge count to a stable limit when $\g > 1/2$.
We divide the proof of Theorem~\ref{thm:functional_stable} into two parts.

Before delving into the proof, recall the high-mark low-mark decomposition from~\eqref{eq:high_low_mark_decomposition} and Figure~\ref{fig:decomposition_of_edge_count}, where we set $u_n = n^{-2/3}$.
To promote consistency with the rest of the proof, we introduce the notation~$S_n^{(1)} := S_n^\le$ for the low-mark edge count.
The normalized edge counts are defined using $n^{-\g}$ in place of $n^{-1/2}$ as in the case of $\g < 1/2$, i.e.,
\[ \So_n^\ge(\any) := n^{-\g} \big( S_n^\ge(\any) - \E{S_n^\ge(\any)} \big) \qquad \text{and} \qquad \So_n^{(1)}(\any) := n^{-\g} \big( S_n^{(1)}(\any) - \E{S_n^{(1)}(\any)} \big). \]
Next, we define the Skorokhod metric $d_\msf{Sk}$ used in the proof of Theorem~\ref{thm:functional_stable}.
\begin{definition}[Skorokhod metric]\label{def:skorokhod_metric}
    Let $f, g \in D([0, 1], \R)$.
    The Skorokhod metric $d_\msf{Sk}(f, g)$ is defined as
    \[ d_\msf{Sk}(f, g) := \inf_\la \Big( \norm{\la - I} \vee \norm{f \circ \la^{-1} - g} \Big), \]
    where the infimum is over all homeomorphisms $\la$ from $[0, 1]$ to itself, $I$ is the identity map and $\norm{\any}$ is the supremum norm on $[0, 1]$.
\end{definition}

The main steps of the first part of the proof of Theorem~\ref{thm:functional_stable} are as follows.
In Step~1, we show that the normalized high-mark edge count is negligible, which is the following statement whose proof, together with the proofs of the remaining statements in this section, is given in Section~\ref{sec:proofs_functional_stable}.
\begin{proposition}[High-mark edge count is negligible]
\label{prop:high_mark_edge_count}
    Let $\g > 1/2$ and $\g' < 1/4$.
    Then,
    \[ \So_n^\ge(\any) \xlongrightarrow[n \tff]{d} 0 \]
    in the Skorokhod space $D([0, 1], \R)$,
\end{proposition}
\noindent
In Step~2, after neglecting the high-mark edges, we approximate the low-mark edge count~$S_n^{(1)}$ by a sum of conditional expectations~$S_n^{(2)}$ of the neighborhoods of the points, defined as
\[ S_n^{(2)}(\any) := \sum_{P \in \PP \cap (\S^{\le u_n}_n \ti \T)} \E[\big]{\deg(P; \any) \biggiven P} \qquad \text{and} \qquad \So_n^{(2)}(\any) := n^{-\g} \big( S_n^{(2)}(\any) - \E[\big]{S_n^{(2)}(\any)} \big). \]
\begin{remark}
    We crucially note that $\So_n^{(2)}$ is devoid of the spatial correlations of the neighborhoods.
\end{remark}
\noindent
The following proposition states that $\So_n^{(2)}$ is indeed an approximation in the sense that the difference between the two edge counts is negligible.
\begin{proposition}[Approximation of the low-mark edge count]
\label{prop:low_mark_edge_count}
    Let $\g > 1/2$ and $\g' \in (0, 1)$.
    Then,
    \[ \norm[\big]{\So_n^{(1)} - \So_n^{(2)}} \xrightarrow[n \tff]{\mbb{P}} 0, \]
    in probability.
\end{proposition}
\noindent
This part of the proof is concluded after approximating $\So_n$ by $\So_n^{(2)}$.

\medskip
In the second part of the proof, we show that the approximated low-mark edge count~$\So_n^{(2)}$ converges in distribution to $\So$ as $n \to \ff$, where $\So$ is defined in the statement of Theorem~\ref{thm:functional_stable}.
Before presenting the remaining steps of the proof, recall the definition of the measure~$\nu$, the Poisson point process~$\PP_\ff$, and that of~$\So(\any)$ from the statement of Theorem~\ref{thm:functional_stable}.
Also recall that
\[ \E{\deg(P; t) \given P} = \abs{N(P; t)} = \abs{\Ns(\Ps)} \abs{\Nt(\Pt; t)} = \wt{c} \, U^{-\g} (t - B) \1{B \le t \le B + L}, \]
where $\wt{c} = 2 \b / (1 - \g')$.
Then, for all window sizes~$n$ and~$\eps > 0$, we define the edge count $S_{n, \eps}^{(3)}$ together with its centered, normalized version as
\[ S_{n, \eps}^{(3)}(\any) := \sum_{P \in \PP \cap (\S_n \ti \T)} \E{\deg(P; \any) \given P} \1{U \le 1 / (\eps n)} \qquad \text{and} \qquad \So_{n, \eps}^{(3)}(\any) := n^{-\g} \big( S_{n, \eps}^{(3)}(\any) - \E{S_{n, \eps}^{(3)}(\any)} \big). \]
Furthermore, recall the definition of the edge count~$S_\eps^\ast$ and its centered version~$\So_\eps^\ast$ from~\eqref{eq:definition_S_eps_ast}.
Note that $S_{n, \eps}^{(3)}$ and $S_\eps^\ast$ are defined on different probability spaces.
The second part of the proof is divided into three steps, Step 3, 4 and 5, which are illustrated in Figure~\ref{fig:steps_of_convergence_proof}, and based on arguments presented in~\cite[Sections~5.5 and~7.2]{heavytails}.
\begin{figure} \centering
    \begin{subfigure}[b]{0.45\textwidth} \centering
        \begin{tikzpicture}[>=latex]
            \begin{axis}[
                    scale=0.8,
                    axis x line=center,
                    axis y line=center,
                    axis equal image,
                    xtick={0.01,6},
                    xticklabels={$0$, $n$},
                    ytick={0.01, 2, 5},
                    yticklabels={$0$, $(\eps n)^{-1}$, $1$},
                    xlabel={location},
                    ylabel={mark},
                    xlabel style={above left},
                    ylabel style={above right},
                    xmin=0, xmax=8.5,
                    ymin=0, ymax=5.5,
                ]
                \addplot [mark=none, dashed, domain=0:6] {2};
                \addplot [mark=none, dashed, domain=0:6] {5};
                \draw [dashed] (6, 0) -- (6, 5);
                \node [fill=white] at (3, 1) {$S_{n, \eps}^{(3)}$};
            \end{axis}
        \end{tikzpicture}
    \end{subfigure}
    \begin{subfigure}[b]{0.45\textwidth} \centering
        \begin{tikzpicture}[>=latex, scale=0.325]
            \draw[thick, dashed, ->] ( 1,  0) -- ( 9,  0);
            \node at  ( 5  ,  1)             {goal};
            \draw[thick,->] ( 0,  9) -- ( 0,  1);
            \node at  (-2.5,  5) [rotate=90] {$\eps \to 0$};
            \node at  (-1  ,  5) [rotate=90] {($\forall n$) uniformly};
            \node at  ( 1  ,  5) [rotate=90] {Step 3};
            \draw[thick,->] ( 1, 10) -- ( 9, 10);
            \node at  ( 5  , 10.75)          {$n \to \ff$};
            \node at  ( 5  ,  8.75)          {Step 4};
            \draw[thick,->] (10,  9) -- (10,  1);
            \node at  ( 9  ,  5) [rotate=90] {Step 5};
            \node at  (11  ,  5) [rotate=90] {$\eps \to 0$};

            \node at  ( 0  ,  0) [circle,fill,inner sep=1.5pt]{};
            \node at  (-1.5, -1) {$\So_n^{(2)}$};
            \node at  ( 0  , 10) [circle,fill,inner sep=1.5pt]{};
            \node at  (-2  , 11) {$\So_{n, \eps}^{(3)}$};
            \node at  (10  ,  0) [circle,fill,inner sep=1.5pt]{};
            \node at  (11  , -1) {$\So$};
            \node at  (10  , 10) [circle,fill,inner sep=1.5pt]{};
            \node at  (11.5, 11) {$\So_\eps^\ast$};
        \end{tikzpicture}
    \end{subfigure}
    \caption{
        Main steps of the second part of the proof of Theorem~\ref{thm:functional_stable}.
        Step~3 shows that $\So_{n, \eps}^{(3)}$ converges to $\So_n^{(2)}$ as $\eps \to 0$ uniformly for all~$n$.
        Step~4 shows convergence of $\So_{n, \eps}^{(3)}$ to $\So_\eps^\ast$ as $n \to \ff$.
        Finally, Step~5 shows that $\So_\eps^\ast$ converges to $\So$ as $\eps \to 0$.
    }
    \label{fig:steps_of_convergence_proof}
\end{figure}
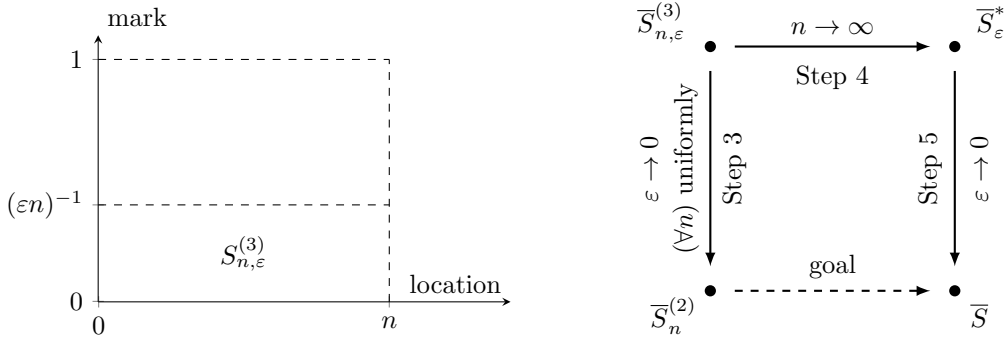
The approach is summarized as follows.
\begin{itemize}
    \item
        In Step~3, we show that $\So_{n, \eps}^{(3)}$ converges to $\So_n^{(2)}$ in $D([0, 1], \R)$ as $\eps \to 0$ uniformly for all~$n$.
        More specifically, we show that $\lim_{\eps \downarrow 0} \limsup_{n \tff} \P[\big]{d_\msf{Sk} \big( \So_{n, \eps}^{(3)}, \So_n^{(2)} \big) > \de} = 0$ for all $\de > 0$.
        After bounding the Skorokhod distance $d_\msf{Sk}$ with the supremum norm, we apply the plus-minus decomposition of the edge count~$S_{n, \eps}^{(3)}$ to show that the supremum norm of the difference converges to zero in probability.
        The main idea is to decompose of the edge count~$S_{n, \eps}^{(3)}$ into a sum of a martingale and a continuous process that can be written as a sum of integrals.
        Note that while other decompositions, such as the Doob--Meyer decomposition are possible, it is unclear if they satisfy the properties we need in the proof.
    \item
        In Step~4, we show that the edge count $\So_{n, \eps}^{(3)}$ with marks less than $(\eps n)^{-1}$ converges in distribution to $\So_\eps^\ast$ in $D([0, 1], \R)$ as $n \to \ff$.
        We follow the arguments in~\cite[Section~7.2]{heavytails} to show the above convergence.
        The main difference in our proof is that we need to show the continuity of the summation functional~$S_\eps^\ast$ with respect to the Skorokhod metric.
        This is done in Proposition~\ref{prop:continuity_of_summation}, and its proof differs from the arguments in the corresponding section~\cite[Section~7.2.3]{heavytails}.
    \item
        Finally, in Step~5, we will see that $\So_\eps^\ast \to \So$ almost surely as $\eps \to 0$ in $D([0, 1], \R)$.
        In this step, we follow the proof of~\cite[Proposition~5.7]{heavytails}.
        While the proof presented in~\cite[Proposition~5.7]{heavytails} relies on the continuous version of Kolmogorov's inequality~\cite[Lemma~5.3]{heavytails}, we need to apply Lemma~\ref{lem:So_eps_cauchy_probability}, which shows that the sequence $\So_{\eps_n}^\ast$ is Cauchy in probability with respect to the supremum norm.
\end{itemize}

\medskip
We begin with Step~3.
Since the Skorokhod metric $d_\msf{Sk}$ on $[0, 1]$ is bounded by the supremum metric, we have that
\begin{equation}
    \lim_{\eps \downarrow 0} \limsup_{n \tff} \P[\big]{d_\msf{Sk} \big( \So_{n, \eps}^{(3)}, \So_n^{(2)} \big) > \de} \le \lim_{\eps \downarrow 0} \limsup_{n \tff} \P[\big]{ \norm[\big]{\So_{n, \eps}^{(3)} - \So_n^{(2)}} > \de}.
    \label{eq:step_4_convergence_result}
\end{equation}
To show that the right-hand side is zero, we turn again to the \quote{plus-minus decomposition} $S_{n, \eps}^{(3)} = S_{n, \eps}^{(3), +} - S_{n, \eps}^{(3), -}$, i.e.,
\[ \begin{aligned}
    S_{n, \eps}^{(3), \pm}(\any) &:= \sum_{P \in \PP \cap (\S_n^{\le 1 / (\eps n)} \ti \T^{0 \le})} \E{\deg^\pm(P; \any) \given P} = \sum_{P \in \PP \cap (\S_n^{\le 1 / (\eps n)} \ti \T^{0 \le})} \abs{N^\pm(P; \any)} \text{ and} \\
    \So_n^{(2), \pm}(\any) &:= n^{-\g} \big( S_{n, \eps}^{(3), \pm}(\any) - \E{S_{n, \eps}^{(3), \pm}(\any)} \big).
\end{aligned} \]
To show that the right-hand side of~\eqref{eq:step_4_convergence_result} is zero, we verify the convergence result for the~\quote{plus} and~\quote{minus} cases separately.
\begin{proposition}[Convergence of $S_{n, \eps}^{(3), \pm}$ to $S_n^{(2), \pm}$]
    \label{lem:convergence_of_Sn_eps_to_Sn}
    Let $\g > 1/2$ and $\g' \in (0, 1)$.
    Then,
    \[ \limsup_{n \tff} \norm[\Big]{\So_{n, \eps}^{(3), \pm} - \So_n^{(2), \pm}} \xrightarrow[\eps \downarrow 0]{\mbb{P}} 0 \]
    in probability.
\end{proposition}
\noindent
This proposition is the main result of Step~3, and it shows that the edge count of the heaviest vertices approximates the total edge count.

The proof of Step 4 begins with a lemma that shows the convergence of the scaled size $n^{-\g} \abs{\Ns(\Ps)}$ to a measure.
\begin{lemma}[Convergence to a measure]
    \label{lem:convergence_to_levy}
    Let $\g > 1/2$ and $\g' \in (0, 1)$.
    Then, we have that
    \[ n \P[\big]{n^{-\g} \abs{\Ns(\Ps)} \in \any} \xrightarrow[n \tff]{v} \nu(\any), \]
    where $\xrightarrow{v}$ denotes vague convergence, and~$\nu$ is the measure defined by $\nu([a, \ff)) := \wt{c}^{1 / \g} a^{-1 / \g}$, i.e., the right-tail probabilities are regularly varying with tail index $-1 / \g$.
\end{lemma}
We now turn to the total contributions of the points to~$S_{n, \eps}^{(3)}$.
The contributions $n^{-\g} (\abs{N(P; t)} - \E{\abs{N(P; t)}})$ of single points to the sum $S_{n, \eps}^{(3)}(t)$ are iid random variables due to the properties of the Poisson process.
By Lemma~\ref{lem:convergence_to_levy}, we have that
\[ \sum_{P \in \PP \cap (\S_n \ti \T)} \de_{(n^{-\g} \abs{\Ns(\Ps)}, B, L)} \xrightarrow[n \tff]{d} \PPP(\nu \otimes \Leb \otimes \mbb{P}_L), \]
where $\PPP(\nu \otimes \Leb \otimes \mbb{P}_L)$ is a Poisson point process on $\J \ti \T$ with intensity measure $\nu \otimes \Leb \otimes \mbb{P}_L$.
To examine the edge count $S_{n, \eps}^{(3)}$, we restrict the domain of the marks~$[0, 1]$ to $[0, 1 / (\eps n)]$, which requires that $n^{-\g} \abs{\Ns(\Ps)} \ge \wt{c} \eps^\g$.
Furthermore, we restrict the domain of the points $(J, B, L)$ to $K := \set{(J, B, L) \in \J \ti \T_{\le 1}^{0 \le}}$,
ensuring that the intervals $[B, B + L]$ intersect the interval $[0, 1]$.
The domain of the marks is visualized in Figure~\ref{fig:domain_of_marks}.
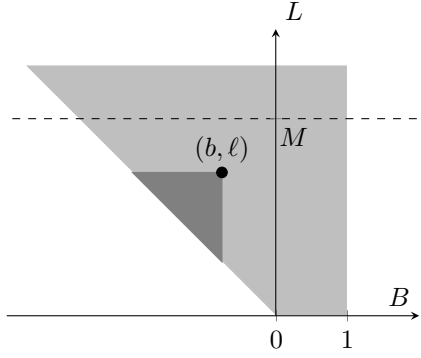
\begin{figure} \centering
    \begin{tikzpicture}[>=latex]
        \begin{axis}[
                scale=0.8,
                axis x line=center,
                axis y line=center,
                axis equal image,
                axis on top,
                xtick={0.01, 1},
                xticklabels={$0$, $1$},
                ytick={2.75},
                yticklabels={$M$},
                xlabel={$B$},
                ylabel={$L$},
                xlabel style={above left},
                ylabel style={above right},
                yticklabel style={below right},
                xmin=-3.75, xmax=2,
                ymin=0, ymax=4.0,
            ]
            \addplot[white, fill=gray!50!white] table {
              0 0
              1 0
              1 3.5
              -3.5 3.5
            } -- cycle;
            \addplot[gray, line width=0.001cm, fill=gray] table {
              -2    2
              -0.75 2
              -0.75 0.75
            } -- cycle;
            \addplot[only marks, mark=*, mark size=2pt] coordinates {(-0.75, 2)} node[above] {$(b, \ell)$};
            \addplot[dashed] {2.75};
        \end{axis}
    \end{tikzpicture}
    \caption{
        Domain of the endpoints of the intervals intersecting with the time interval~$[0, 1]$.
    }
    \label{fig:domain_of_marks}
\end{figure}
Then, we consider the summation functional
\begin{equation}
    \sum_{(J, B, L) \in \PP_\ff \cap ([\wt{c} \eps^\g, \ff) \ti \T)} \de_{(J, B, L)} \mapsto \sum_{(J, B, L) \in \PP_\ff} J (\any - B) \1{J \ge \wt{c} \eps^\g} \1{B \le \any \le B + L}.
    \label{eq:summation_functional}
\end{equation}
The next result shows the continuity of the above functional.
In order to state it, let $K_\eps := \set{(j, b, \ell) \in \J \ti \T_{\le 1}^{0 \le} \co j \ge \wt{c} \eps^\g}$ be the domain of the points contributing to the edge count, see again Figure~\ref{fig:domain_of_marks}.
As in~\cite{heavytails}, we consider $K_\eps$ to be endowed with the topology where in the first component, we use the one-point compactification at infinity.
That means, we set $\overline{\J} = (0, \ff]$, i.e., we compactify at the point $\ff$ in the first coordinate.
Then, the topology of the space $\overline{K}_\eps$ is the one of the space $\J \ti \T_{\le 1}^{0 \le}$ with $\J$ compactified at $\ff$ in the first coordinate.
Furthermore, let~$\Nl(\overline{\J} \ti \T)$ denote the family of point measures on locally finite subsets of $\overline{\J} \ti \T$ and consider the summation functional~$\chi \co \Nl(\overline{\J} \ti \T) \to D([0, 1], \R)$ defined as
\[ \begin{aligned}
    \chi(\eta)(t) &:= \chi \Big( \sum_{(j, b, \ell) \in \eta} \de_{(j, b, \ell)} \Big)(t) := \chi \Big( \sum_{(j, b, \ell) \in \eta} \de_{(j, b, \ell)} \big( \any \cap \overline{K}_\eps \big) \Big)(t) \\
    &:= \sum_{(j, b, \ell) \in \eta} j (t - b) \1{j \ge \wt{c} \eps^\g} \1{b \le t \le b + \ell}.
\end{aligned} \]

\begin{proposition}[Continuity of summation functional]
    \label{prop:continuity_of_summation}
    It holds that $\chi(\eta)(t)$ is almost surely continuous with respect to the distribution of $\PPP(\nu \otimes \Leb \otimes \mbb{P}_L)$.
\end{proposition}
\noindent
Note that if $j = n^{-\g} \abs{\Ns(\ps)}$ for a point $p \in \S_n \ti \T$ in Proposition~\ref{prop:continuity_of_summation}, then the first indicator is equivalent to $u \le 1 / (\eps n)$.
From Proposition~\ref{prop:continuity_of_summation}, we get the following convergence in $D([0, 1], \R)$,
\[ n^{-\g} S_{n, \eps}^{(3)}(\any) \xrightarrow[n \tff]{d} S_\eps^\ast(\any). \]
The following lemma shows that the expectation of the edge count $S_\eps^\ast$ is finite.
\begin{lemma}[Expectation of $S_\eps^\ast$]
    \label{lem:expectation_of_S_eps}
    Let $\g > 1/2$.
    Then,
    \[ \E[\big]{S_\eps^\ast} = \f{\wt{c}}{1 - \g} \eps^{-(1 - \g)}. \]
\end{lemma}
\noindent
Then, for the centered and scaled edge count $\So_{n, \eps}^{(3)}$, we have the convergence in $D([0, 1], \R)$,
\[ \So_{n, \eps}^{(3)}(\any) \xrightarrow[n \tff]{d} \So_\eps^\ast(\any).\]


Finally, in Step 5, we show that $\So_\eps^\ast \to \So$ in $D([0, 1], \R)$ as $\eps \to 0$ almost surely.
As a preliminary result, we first follow closely~\cite[Section~5.5.1]{heavytails} to show that almost sure convergence happens for all fixed time points~$t \in [0, 1]$ as $\eps \to 0$.
\begin{lemma}[Convergence of $\So_\eps^\ast(t)$ for fixed time points]\label{lem:convergence_of_So_eps_fixed_time}
    For a fixed $t \in [0, 1]$, we have almost surely that
    \[ \So_\eps^\ast(t) \xrightarrow[\eps \downarrow 0]{\msf{a.s.}} \So(t).\]
\end{lemma}
\noindent
Following the proof of~\cite[Proposition~5.7]{heavytails}, we aim to show that the convergence is almost surely uniform.
Before we do that, we need to show that the sequence $\So_{\eps_n}^\ast$ is Cauchy in probability with respect to the supremum norm in the sense stated below.
\begin{lemma}[$\So_{\eps_n}^\ast$ is Cauchy in probability]\label{lem:So_eps_cauchy_probability}
    Let~$\eps_k \to 0$ be a decreasing sequence as $k \to \ff$.
    Then,
    \[ \sup_{n, m \ge N} \norm[\big]{\So_{\eps_n}^\ast - \So_{\eps_m}^\ast} \xrightarrow[N \tff]{\mbb{P}} 0.\]
\end{lemma}
\noindent
Next, in Proposition~\ref{prop:So_cauchy_uniform}, we present the main result of this step.
We show that there exists~$\So$ with almost all paths in $D([0, 1], \R)$ such that $\lim_{\eps \downarrow 0} \norm{\So_\eps^\ast - \So} = 0$ almost surely, which is done by showing that the sequence $\So_{\eps_n}^\ast$ is almost surely Cauchy with respect to the supremum norm on $[0, 1]$.
\begin{proposition}[$\So_{\eps_n}^\ast$ is Cauchy almost surely]\label{prop:So_cauchy_uniform}
    Let~$\eps_k \to 0$ be a decreasing sequence as $k \to \ff$.
    Then,
    \[\sup_{n, m \ge N} \norm[\big]{\So_{\eps_n}^\ast - \So_{\eps_m}^\ast} \xrightarrow[N \tff]{\msf{a.s.}} 0. \]
\end{proposition}
\noindent
By Proposition~\ref{prop:So_cauchy_uniform}, there exists~$\So$ such that $\lim_{\eps \downarrow 0} \norm{\So_\eps^\ast - \So} = 0$ almost surely.
Furthermore, the limit $\So$ is in $D([0, 1], \R)$ by~\cite[Lemma~5.2]{heavytails}.

%% file: proofs.tex

\section{Preliminary lemmas}\label{sec:minor_lemmas}

Before presenting the proofs of the main propositions and lemmas, we collect some minor results that are used frequently in the proofs later.
We begin with a lemma which is about the size of the spatial neighborhoods of points.
\input{minor_lemmas/size_of_ns.tex}
The following lemma is about the integrals of the spatial neighborhoods.
\input{minor_lemmas/integrals_of_ns.tex}
The next lemma is about the size of the temporal neighborhood of points.
\input{minor_lemmas/size_of_nt.tex}
The next lemma characterizes some of the temporal integrals that will be used in the proof.
\input{minor_lemmas/integrals_of_nt.tex}
Next, we introduce notation for the differences of temporal neighborhoods of a point $p$ between two times $t_1 < t_2$,
\begin{equation} \begin{aligned}
    \de_{t_1, t_2}(\Nt^\pm(\pt)) &:= \Nt^\pm(\pt; t_2) \sm \Nt^\pm(\pt; t_1) \\
    \de_{t_1, t_2}(N^\pm(p)) &:= N^\pm(p; t_2) \sm N^\pm(p; t_1) = \Ns(\ps) \ti \de_{t_1, t_2}(\Nt^\pm(\pt; t)).
\end{aligned} \label{eq:def_of_de_Nt} \end{equation}
Then, the size of the temporal \quote{plus-minus neighborhoods} is given in the next lemma.
\input{minor_lemmas/size_of_nt_pm.tex}
The next lemma is about the integrals of the temporal \quote{plus-minus neighborhoods}.
\input{minor_lemmas/integrals_of_nt_pm.tex}


\section{Proofs of the lemmas used for Propositions~\ref{prop:univariate_normal},~\ref{prop:covariance_function_of_Snt} and~\ref{prop:multivariate_normal}}\label{sec:proofs_univariate_multivariate}

In this section, we assume that $\g < 1/2$, and we do not explicitly mention this condition in the proofs.

\subsection{Proofs of Lemmas \ref{lem:mean_variance_of_Snt} and \ref{lem:covariance_of_Snt}}

First, we present the proofs of the lemmas used in the proofs of Propositions~\ref{prop:univariate_normal}.
The proofs are based on the application of the Mecke formula \cite[Theorem~4.1]{poisBook} and the independence property \cite[Definition~3.1~(ii)]{poisBook} of the Poisson process.
\input{proofs/univariate_normal/lemmas/lem_mean_variance_of_Snt.tex}
\input{proofs/univariate_normal/lemmas/lem_covariance_of_Snt.tex}

\subsection{Proofs of Lemmas \ref{lem:low_mark_edge_count_negligible}, \ref{lem:covariance_function_of_Snt_ge}, \ref{lem:covariance_function_terms} and \ref{lem:bounds_of_error_terms}}

Finally, we present the proofs of the lemmas used in the proof of Proposition~\ref{prop:multivariate_normal}.
In Lemma~\ref{lem:low_mark_edge_count_negligible}, we show that the low-mark edge count converges to $0$ in probability.
\input{proofs/multivariate_normal/lemmas/lem_low_mark_edge_count_negligible.tex}
The covariance function of the high-mark edge count $\So_n^\ge$ was calculated in Lemma~\ref{lem:covariance_function_of_Snt_ge}, whose proof is similar to the proof of Proposition~\ref{prop:covariance_function_of_Snt}.
\input{proofs/multivariate_normal/lemmas/lem_covariance_function_Snt_ge.tex}
Next, we show the proof of the lemma used in the proofs of Proposition~\ref{prop:covariance_function_of_Snt} and Lemma~\ref{lem:covariance_function_of_Snt_ge}, which determined the terms of the limiting covariance functions of the edge counts $\So_n(t)$ and~$\So_n^\ge(t)$.
\input{proofs/covariance_function/lem_covariance_function_terms.tex}
In Lemma~\ref{lem:bounds_of_error_terms}, we bound the error terms to bound the $d_3$~distance in Proposition~\ref{prop:multivariate_normal_limit}.
We examine the cost operators, and then manipulate the error terms to show their bounds.
\input{proofs/multivariate_normal/lemmas/lem_bounds_of_error_terms.tex}


\section{Proofs of the lemmas used to prove Theorem~\ref{thm:functional_normal}}\label{sec:proofs_functional_normal}

This section presents the proofs of the lemmas used in the proof of Theorem~\ref{thm:functional_normal}.
In the proof of Theorem~\ref{thm:functional_normal}, we introduced the \quote{plus-minus decomposition} in the beginning of Section~\ref{sec:edge_count_gaussian} to decompose $S_n$ to the difference of two monotone increasing functions.
The first proof shows that the difference of the \quote{plus} and \quote{minus} parts of the edge count is indeed the total edge count.
\input{proofs/functional_normal/lemmas/lem_plus_minus_decomposition.tex}

We continue by showing that the conditions of Theorem~\ref{thm:davydov} hold for the processes $\So_n^\pm(\any)$.
The details follow after an overview of the proof of the conditions.
Condition~\ref{condition:finite_dimensional} of Davydov's theorem is about the convergence of the finite-dimensional distributions of $\So_n^\pm$, which we show in the following proof.
\input{proofs/functional_normal/propositions/prop_finite_dimensional_distributions.tex}
Condition~\ref{condition:cumulant} of Davydov's theorem is about the tightness of the processes $\So_n^\pm(\any)$, which is shown next.
\input{proofs/functional_normal/propositions/prop_tightness.tex}
Finally, we show that Condition~\ref{condition:third} of Theorem~\ref{thm:davydov} holds, which is about the convergence of the expected increments $\E[\big]{\De_n^\pm(t_k, t_{k + 1})}$.
\input{proofs/functional_normal/lemmas/lem_expectation_of_increments.tex}

Since the proofs of the Lemmas~\ref{lem:limiting_covariance_function_Snt_pm} and~\ref{lem:bounds_of_error_terms_pm} are very similar to the proofs of the corresponding lemmas in the proof of Proposition~\ref{prop:multivariate_normal_limit}, we postpone their proofs to Appendix~\ref{sec:appendix_proofs_cov_func_error_terms_pm}.

Next, we present the bounds of the variance and cumulant terms that were required by Condition~\ref{condition:cumulant} of Theorem~\ref{thm:davydov} to show the tightness of the sequence of random variables $\So_n^\pm$.
While bounding the variance term $\Var{\De_n^\pm(s, t)}$ is relatively straightforward after the application of the Poincar\'e inequality, bounding the cumulant term $\k_4{\De_n^\pm(s, t)}$ requires a more careful analysis.
\input{proofs/functional_normal/lemmas/lem_tightness_variance_term.tex}

To bound the cumulant term in the three cases introduced in Section~\ref{sec:proofs_functional_normal}, we employ~\cite[Proposition~3.2.1]{wiener_chaos} to rewrite the fourth cumulant $\k_4(\De_n^\pm)$ in a form that is easier to bound.
\input{proofs/functional_normal/lemmas/lem_tightness_cumulant_term.tex}
For~\ref{im:cumulant_case_all_in_one}, we bound the fourth moment of $\De_n^\pm(s, t)$ in the following proof.
\input{proofs/functional_normal/lemmas/lem_cumulant_fourth_moment_bound.tex}
For~\ref{im:cumulant_case_1_3}, we bound the covariance term appearing in \eqref{eq:complex_cumulant_formula}.
\input{proofs/functional_normal/lemmas/lem_cumulant_covariance_bound.tex}


\section{Proofs of the propositions and lemmas used to prove Theorem~\ref{thm:functional_stable}}\label{sec:proofs_functional_stable}

In this section, we first show the main propositions that were used in the proof of Theorem~\ref{thm:functional_stable}.

\subsection{Proofs of the propositions and lemmas used in Steps~1 and~2 of the proof of Theorem~\ref{thm:functional_stable}}\label{subsec:thm_stable_part_1}

We begin with proving that the high-mark edge count~$\So_n^\ge$ is negligible compared to the total edge count~$\So_n$, which was required by Step~1 of the proof.
The proof strategy is the same as in the proof of Theorem~\ref{thm:functional_normal} above, and we show convergence of $\So_n^\ge$ to $0$ in the Skorokhod topology using the plus-minus decomposition.
\input{proofs/functional_stable/propositions/prop_high_mark_edge_count.tex}

The proofs of the lemmas used to show that the finite-dimensional distributions of the high-mark edge count $\So_n^\ge$ converge to $0$ are similar to the proofs of the lemmas in the proof of Theorem~\ref{thm:functional_normal}.
Thus, we shift the proofs of the lemmas to Appendix~\ref{sec:appendix_proofs_part_1}.

Next, we show that we can approximate the low-mark edges by applying Chebyshev's inequality, which was the goal of Step~2 of the proof.
We aim to bound the difference of the low-mark edge count $\So_n^{(1)}$ and its approximation $\So_n^{(2)}$ by bounding the supremum of the difference.
\input{proofs/functional_stable/propositions/prop_low_mark_edge_count.tex}

\subsection{Proofs of propositions and lemmas used in Steps~3,~4, and~5 of the proof of Theorem~\ref{thm:functional_stable}}\label{subsec:thm_stable_part_2}

In Step~3, we introduced the edge count $S_{n, \eps}^{(3)}$ to represent the edges of vertices with marks less than $1/(n \eps)$.
The following proof show that its plus-minus decomposition $S_{n, \eps}^{(3), \pm}$ converges to $S_n^{(2), \pm}(\any)$ as $\eps \to 0$ uniformly for all $n$.
This is done by bounding the supremum norm of the difference of the two edge counts.
For the minus part, we recognize that the difference is the supremum of a martingale, and we apply Doob's inequality to bound the difference.
For the plus part, we write the $\So_{n, \eps}^{(3), +}(t)$ as a sum of the integrals of the neighborhoods to bound the supremum.
\input{proofs/functional_stable/propositions/prop_convergence_of_Sn_eps_to_Sn.tex}

In the following, we present the results required for Steps~4 and~5 in the proof of Theorem~\ref{thm:functional_stable}.
In the next proof, we show that the size of the spatial neighborhoods of the points converges to a measure, which is done by showing convergence to a distribution function.
\input{proofs/functional_stable/lemmas/lem_convergence_to_levy.tex}

In~\eqref{eq:summation_functional}, we defined the summation functional, which was used to define $S_\eps^\ast(\any)$.
To show convergence of $S_{n, \eps}^{(3)}(\any)$ to $S_\eps^\ast(\any)$, we need to show that the summation functional is continuous with respect to the Skorokhod metric $d_\msf{Sk}$.
As the proof is similar to the proof of the continuity of the summation functional in~\cite[Section~7.2.3]{heavytails}, we present the proof in Appendix~\ref{sec:appendix_proofs_part_2}.

In the last part of Step 4, we show that the expectation of $S_\eps^\ast(\any)$ is finite.
\input{proofs/functional_stable/lemmas/lem_expectation_of_So_eps.tex}

Step~5 is about showing the convergence of $\So_\eps^\ast(\any)$ to $\So(\any)$ as $\eps \to 0$.
We claimed that the edge count $\So_\eps^\ast(t)$ converges to $\So(t)$ for a fixed time point $t$, which is shown in the following proof, following the arguments presented in~\cite[Section~5.5.1]{heavytails}.
\input{proofs/functional_stable/lemmas/lem_convergence_of_So_eps_fixed_time.tex}

Finally, the proofs showing that $\So_{\eps_n}^\ast$ is Cauchy in probability and almost surely with respect to the uniform convergence is postponed to Appendix~\ref{sec:appendix_proofs_part_2}, since the proof of the Cauchy property in probability is very similar to the arguments presented in Step~3, and the proof of the Cauchy property almost surely follows the same approach as the proof of Property 2 in~\cite[Proposition~5.7]{heavytails}.


\section{Proofs of the minor lemmas}\label{sec:proofs_minor_lemmas}

Finally, we present the proofs of the minor lemmas stated at the beginning of this section.
\input{proofs/minor_lemmas/lem_size_of_ns.tex}
\input{proofs/minor_lemmas/lem_integrals_of_ns.tex}
\input{proofs/minor_lemmas/lem_size_of_nt.tex}
\input{proofs/minor_lemmas/lem_integrals_of_nt.tex}
\input{proofs/minor_lemmas/lem_size_of_nt_pm.tex}
\input{proofs/minor_lemmas/lem_integrals_of_nt_pm.tex}

%% file: minor_lemmas/size_of_ns.tex
\begin{lemma}[Size of spatial neighborhoods]\label{lem:size_of_ns}
 Let $\g, \g' \in (0, 1)$.
    \begin{enumerate}
        \item
            For all $\ps \in \S$, we have that
            \[ \abs{\Ns(\ps)} = \f{2 \b}{1 - \g'} u^{-\g}. \]
            \label{lem:size_of_ns:p}
        \item
            For all $\ps[1], \ps[2] \in \S$, we have that
            \[ \abs{\Ns(\ps[1], \ps[2])} \le 2 \f{(2 \b)^{1 / \g'}}{1 - \g'} \abs{x_1 - x_2}^{-  (1 / \g' - 1)} u_1^{- \g / \g'} u_2^{- \g / \g'}. \]
            \label{lem:size_of_ns:common_p}
        \item
            For all $u_- \in [0, 1]$, the size of the spatial neighborhood of a point $p'$ is given by
            \[ \int_{\S^{u_- \le}} \1{\ps' \in \Ns(\ps)} \d \ps = \f{2 \b}{1 - \g} w^{-\g'} \big( 1 - u_-^{1 - \g} \big). \]
            \label{lem:size_of_ns:p_prime}
    \end{enumerate}
\end{lemma}

%% file: minor_lemmas/integrals_of_ns.tex
\begin{lemma}[Integrals of spatial neighborhoods $\Ns(\ps)$]\label{lem:int_of_ns}
We have the following.
    \begin{enumerate}
        \item
            Let $\g, \g' \in (0, 1)$, $A \su \R$ a Borel set and $\a \ge 0$.
            Then, if $\g < 1/\a$,
            \[ \int_{\S_A} \abs{\Ns(\ps)}^\a \d \ps = \Big( \f{2 \b}{1 - \g'} \Big)^\a \f{\abs{A}}{1 - \a \g}. \]
            If $\g \in (0, 1)$ and $u_- \in (0, 1]$, then
            \[ \int_{\S_A^{u_- \le}} \abs{\Ns(\ps)}^\a \d \ps = \left\{ \begin{array}{ll} \Big( \f{2 \b}{1 - \g'} \Big)^\a \f{\abs{A}}{1 - \a \g} \Big( 1 - u_-^{1 - \a \g} \Big) & \text{if } \g \ne 1 / \a \\ \Big( \f{2 \b}{1 - \g'} \Big)^\a \abs{A} \log[\big]{u_-^{-1}} & \text{if } \g = 1 / \a. \end{array} \right. \]
            \label{lem:int_of_ns:power_p}
        \item
            Let $\g \in (0, 1)$ and $\g' < 1 / m$.
            Then,
            \[ \f{1}{n} \int_{\S_n^m} \abs{\Ns(\bsps[m])} \d \bsps[m] \nearrow \f{(2 \b)^m}{(1 - \g)^m (1 - m \g')} \qquad \text{as } n \tff, \]
            where $\nearrow$ denotes convergence from below.
            \label{lem:int_of_ns:common_p_0}
        \item
            Let $u_n = n^{-b}$ denote an $n$-dependent mark with $b \in (0, 1)$.
            Furthermore, let $\g \in (0, 1)$ and $\g' < 1 / m$.
            Then,
            \[ \f{1}{n} \int_{\big( \S_n^{u_n \le} \big)^m} \abs{\Ns(\bsps[m])} \d \bsps[m] \nearrow \f{(2 \b)^m}{(1 - \g)^m (1 - m \g')} \qquad \text{as } n \tff. \]
            \label{lem:int_of_ns:common_p_minus}
        \item
            Let $m > 0$ be a positive integer, let $\g \in (0, 1)$ and $\g' < 1 / m$.
            Then,
            \[ \int_\S \bigg( \int_{\S_n} \1{\ps' \in \Ns(\ps)} \d \ps \bigg)^m \d \ps' \le \bigg( \f{2 \b}{1 - \g} \bigg)^m \f{n}{1 - m \g'}. \]
            \label{lem:int_of_ns:power_p_prime}
        \item
            We set $m_1, m_2, m_3 \in \set{0, 1, 2, \dots}$, $\g < (1 + (m_1 \vee m_2) + m_3)^{-1}$ and $\g' < (2 + m_3)^{-1}$.
            Then,
            \[ \begin{aligned}
                &\iint_{\S_n^2} \abs{\Ns(\ps[1])}^{m_1} \abs{\Ns(\ps[1], \ps[2])} \abs{\Ns(\ps[2])}^{m_2} \abs{x_1 - x_2}^{m_3} \d (x_1, u_1) \d (x_2, u_2) \le c \big( c' + c'' \big) n, \text{ where} \\
                &\qquad c   = \f{(2 \b)^{2 + m_1 + m_2 + m_3}}{(1 + m_3) (1 - \g')^{m_1 + m_2} (1 - (2 + m_3) \g')} \\
                &\qquad c'  = \f{1}{(1 - (1 + m_1) \g) (1 - (1 + m_2 + m_3) \g)} \\
                &\qquad c'' = \f{1}{(1 - (1 + m_2) \g) (1 - (1 + m_1 + m_3) \g)}.
            \end{aligned} \]
            Note that the bound is always nonnegative.
            \label{lem:int_of_ns:power_cap_power_0}
        \item
            We set $m_1, m_2, m_3 \in \set{0, 1, 2, \dots}$, $\g > 1/2$, $\g' < (2 + m_3)^{-1}$ and $u_- \in (0, 1]$ a mark.
            Then,
            \[ \begin{aligned}
                &\iint_{\big( \S_n^{u_- \le} \big)^2} \abs{\Ns(\ps[1])}^{m_1} \abs{\Ns((x_1, u_1), (x_2, u_2))} \abs{\Ns((x_2, u_2))}^{m_2} \abs{x_1 - x_2}^{m_3} \d (x_1, u_1)\d (x_2, u_2) \\
                &\qquad \le c \Big( \abs{c'} u_-^{-((1 + m_2 + m_3) \g - 1)_+ - ((1 + m_1) \g - 1)_+} + \abs{c''} u_-^{- ((1 + m_2) \g - 1)_+ -((1 + m_1 + m_3) \g - 1)_+} \Big) n,
            \end{aligned} \]
            where $(\any)_+ := \any \vee 0$, the constants $c$, $c'$ and $c''$ are defined as in Part~(\ref{lem:int_of_ns:power_cap_power_0}) of this lemma.
            \label{lem:int_of_ns:power_cap_power_minus}
    \end{enumerate}
\end{lemma}

%% file: minor_lemmas/size_of_nt.tex
\begin{lemma}[Size of temporal neighborhoods]\label{lem:size_of_nt}
    Let $t \in \R$.
    \begin{enumerate}
        \item
            For all $\pt = (b, \ell) \in \T$, we have that
            \[ \abs{\Nt(\pt; t)} = (t - b) \1{b \le t \le b + \ell}. \]
            \label{lem:size_of_nt:p}
        \item
            For all $\pt = (b, \ell) \in \T$ and $r \in \R$, the size of the temporal neighborhood of a point $p'$ is given by
            \[ \int_\T \1{r \in \Nt(\pt; t)} \dpt = \1{r \le t} \e^{-(t - r)}. \]
            \label{lem:size_of_nt:p_prime}
    \end{enumerate}
\end{lemma}

%% file: minor_lemmas/integrals_of_nt.tex
\begin{lemma}[Integrals of temporal neighborhoods $\Nt(\pt; t)$]\label{lem:int_of_nt}
We have the following.
    \begin{enumerate}
        \item
           For all $\a > 0$ we have that
            \[ \int_\T \abs{\Nt(\pt; t)}^\a \dpt = \Ga(\a + 1). \]
            \label{lem:int_of_nt:power_p}
        \item
            For all $t \in \R$, the integral of the common temporal neighborhood of the points $\pp_m$ is given by
            \[ \int_{\otimes_{i = 1}^m \T_i} \abs[\Big]{\bigcap_{i = 1}^m N_\msf{t} (p_i; t_i)} \big( \d \bs{b}_m \mbb{P}_L^{\otimes m} (\d \bs{\ell}_m) \big) = \int_\R \prod_{i = 1}^m \bigg( \int_{\T_i} \1{r \in \Nt(\pt; t)} \dpt \bigg) \d r. \]
            Furthermore, if $\T_i = \T$ and $t_i = t$ for all indices $i \in \set{1, \dots, m}$, then the right-hand side equals $1 / m$.
            \label{lem:int_of_nt:common_p}
        \item
            For all $\a_1, \a_2 > 0$, we have that
            \[ \int_\T \int_\T \abs{\Nt(\pt[1]; t_1)}^{\a_1} \abs{\Nt(\pt[1]; t_1) \cap \Nt(\pt[2]; t_2)} \abs{\Nt(\pt[2]; t_2)}^{\a_2} \pt[2] \dpt[1] \le \Ga(\a_1 + 1) \Ga(\a_2 + 2). \]
            \label{lem:int_of_nt:power_cap_power}
        \item
            For all $\a > 0$, we have that
            \[ \int_\R \bigg( \int_\T \1{r \in \Nt(\pt; t)} \dpt \bigg)^\a \d r = \f{1}{\a}. \]
            \label{lem:int_of_nt:power_prime}
    \end{enumerate}
\end{lemma}

%% file: minor_lemmas/size_of_nt_pm.tex
\begin{lemma}[Size of $\Nt^\pm(\pt; t)$]\label{lem:size_of_Nt_pm}
    Let $t, t_1, t_2 \in [0, 1]$ be fixed.
    \begin{enumerate}
        \item
            The size of the temporal neighborhoods is given by
            \[ \abs{\Nt^+(\pt; t)} = (((b + \ell) \w t) - b) \1{b \le t} \qquad \text{and} \qquad \abs{\Nt^-(\pt; t)} = \ell \1{b + \ell \le t}. \]
            \label{lem:size_of_Nt_pm:size}
        \item
            The change of the temporal neighborhood of a point $p$ between $t_1 < t_2$ is given by
            \[ \begin{aligned}
                \abs[\big]{\de_{t_1, t_2}(\Nt^+(\pt; t_2))} &= \big( ((b + \ell) \w t) - (b \vee t_1) \big) \1{b \le t_2} \1{t_1 \le b + \ell} \text{ and} \\
                \abs[\big]{\de_{t_1, t_2}(\Nt^-(\pt; t_2))} &= (t_1 - b) \1{b \le t_1 \le b + \ell \le t_2}.
            \end{aligned} \]
            \label{lem:size_of_Nt_pm:size_difference}
        \item
            The size of the temporal neighborhood of a point $p'$ is given by
            \[ \begin{aligned}
                \int_{\T^{0 \le}} \1{r \in \Nt^+(\pt; t)} \dpt &= \1{r \le 0} \e^{- (- r)} + \1{0 \le r \le t} \text{ and} \\
                \int_{\T^{0 \le}} \1{r \in \Nt^-(\pt; t)} \dpt &= \1{r \le 0} \big( \e^r - \e^{- (t - r)} \big) + \1{0 \le r \le t} \big( 1 - \e^{- (t - r)} \big).
            \end{aligned} \]
            \label{lem:size_of_Nt_pm:size_p_prime}
    \end{enumerate}
\end{lemma}

%% file: minor_lemmas/integrals_of_nt_pm.tex
\begin{lemma}[Integrals of $\Nt^\pm(\pt; t)$]\label{lem:int_of_nt_pm}
    Let $t, t_1, t_2 \in [0, 1]$ be fixed.
    \begin{enumerate}
        \item
            For all integers $m \ge 1$, we have that
            \[ \int_{\T^{0 \le}} \abs{\Nt^+(\pt; t)}^m \dpt = m! (t + 1) \qquad \text{and} \qquad \int_{\T^{0 \le}} \abs{\Nt^-(\pt; t)}^m \dpt = m! t. \]
            \label{lem:int_of_nt_pm:power_p_specific}
        \item
            For all $\a \ge 0$, we have that $\int_{\T^{0 \le}} \abs{\Nt^\pm(\pt; t)}^\a \dpt \le 2 c t + \Ga(\a + 1)$, where $c := (2 \a)^\a \e^{-\a}$.
            \label{lem:int_of_nt_pm:power_p_bound}
        \item
            For all integers $m \ge 1$, we have that $\int_{\T^{0 \le}} \abs[\big]{\de_{t_1, t_2}(\Nt^\pm(\pt))}^m \dpt \le (t_2 - t_1)^m$.
            \label{lem:int_of_nt_pm:difference_p}
        \item
            For all integers $m \ge 1$, we have that $\int_\R \big( \int_{\T^{0 \le}} \1{r \in \de_{t_1, t_2}(\Nt^\pm(\pt))} \dpt \big)^m \d r = t_2 - t_1$.
            \label{lem:int_of_nt_pm:difference_p_prime}
        \item
            For all integers $m \ge 1$, the integral of the intersection of the neighborhoods is given by
            \[ \iint_{\big( \T^{0 \le} \big)^m} \abs{\Nt^+(\bspt[m]; t)} \, \m_\msf{t}^{\otimes m} (\d \bspt[m]) = \int_\R \bigg( \int_{\T^{0 \le}} \1{r \in \Nt^\pm(\pt; t)} \dpt \bigg)^m \d r \le 1 / m + t. \]
            \label{lem:int_of_nt_pm:cap}
        \item
            For all $\a_1, \a_2 \ge 0$, we have that
            \[ \iint_{\big( \T_{\le 1}^{0 \le} \big)^2} \abs{\Nt^\pm(\pt[1]; t_1)}^{\a_1} \abs{\Nt^\pm(p_1; t_1) \cap \Nt^\pm(\pt[2]; t_2)} \abs{\Nt^\pm(\pt[2]; t_2)}^{\a_2} \dpt[1] \dpt[2] < \ff. \]
            \label{lem:int_of_nt_pm:power_cap_power}
    \end{enumerate}
\end{lemma}

%% file: proofs/univariate_normal/lemmas/lem_mean_variance_of_Snt.tex
\begin{proof}[Proof of Lemma~\ref{lem:mean_variance_of_Snt}]
    The mean of the edge count $S_n(t)$ is given by
    \[ \E{S_n(t)} = \E[\Big]{\sum_{P \in \PP \cap (\S_n \ti \T)} \deg(P)} = \int_{\S_n \ti \T} \abs{N(p)} \mdp = \f{2 \b}{(1 - \g)(1 - \g')} n, \]
    where we used the Mecke formula in the second step and Lemmas~\ref{lem:int_of_ns}~(\ref{lem:int_of_ns:power_p}) and~\ref{lem:int_of_nt}~(\ref{lem:int_of_nt:power_p}) with $\a = 1$ in the third step.
    Next, we calculate the variance using the Mecke formula:
    \begin{equation} \begin{aligned}
        \Var{S_n(t)} &= \E[\Big]{\sum_{(P_1, P_2) \in (\PP \cap (\S_n \ti \T))_{\ne}^2} \deg(P_1) \deg(P_2)} - \E[\Big]{\sum_{P \in \PP \cap (\S_n \ti \T)} \deg(P)}^2 + \E[\Big]{\sum_{P \in \PP \cap (\S_n \ti \T)} \deg(P)^2} \\
        &= \int_{\S_n \ti \T} \E{\deg(p)^2} \mdp + \iint_{(\S_n \ti \T)^2} \Cov{\deg(p_1), \deg(p_2)} \mdp[1] \mdp[2].
    \end{aligned} \label{eq:variance_Snt} \end{equation}
    The above formula is the variance of a $U$-statistic, which was also calculated in a more general setting in~\cite[Lemma~3.5]{reitzner}.
    We can see that in the first step, the variance is determined by the last term, where $P = P_1 = P_2$, since the first two terms cancel.
    For the first term, we have
    \begin{equation} \begin{aligned}
        \E{\deg(p)^2} &= \E[\Big]{\Big( \sum_{P' \in \PP'} \1{P' \in N(p)} \Big)^2} \\
        &= \E[\Big]{\sum_{P' \in \PP'} \1{P' \in N(p)}} + \E[\Big]{\sum_{\set{P'_1, P'_2} \in \PP'^2_{\ne}} \1{P'_1 \in N(p)} \1{P'_2 \in N(p)}} \\
        &= \int_{\S \ti \R} \1{p' \in N(p)} \d p' + \bigg( \int_{\S \ti \R} \1{p' \in N(p)} \d p' \bigg)^2 = \abs{N(p)} + \abs{N(p)}^2.
    \end{aligned} \label{eq:expectation_of_degree_square} \end{equation}
    The first term is the integral of the above expression:
    \[ \begin{aligned}
        \lim_{n \tff} \f{1}{n} \int_{\S_n \ti \T} \E{\deg(p)^2} \mdp &= \sum_{k = 1}^2 \lim_{n \tff} \f{1}{n} \int_{\S_n} \abs{\Ns(\ps)}^k \d \ps \int_\T \abs{\Nt(\pt)}^k \dpt \\
        &= \f{2 \b}{(1 - \g)(1 - \g')} + \f{(2 \b)^2}{(1 - 2 \g)(1 - \g')^2},
    \end{aligned} \]
    where we applied Lemmas~\ref{lem:int_of_ns}~(\ref{lem:int_of_ns:power_p}) and~\ref{lem:int_of_nt}~(\ref{lem:int_of_nt:power_p}), and it is required that $\g < 1/2$.
    For the second term, we have
    \[ \begin{aligned}
        &\Cov{\deg(p_1), \deg(p_2)} = \E[\Big]{\sum_{P' \in \PP'} \1{P' \in N(p_1, p_2)}} + \E[\Big]{\sum_{(P'_1, P'_2) \in \PP_{\ne}'^2} \1{P'_1 \in N(p_1), P'_2 \in N(p_2)}} \\
        &\qquad - \E[\Big]{\sum_{P' \in \PP'} \1{P' \in N(p_1)}} \E[\Big]{\sum_{P' \in \PP'} \1{P' \in N(p_2)}} = \int_{\S \ti \R} \1{p' \in N(p_1, p_2)} \d p' = \abs{N(p_1, p_2)}.
    \end{aligned} \]
    Note that just as in the case of the variance, the covariance is determined by the term in which $P' = P'_1 = P'_2$.
    This insight will be used many times in the following calculations.
    As $\abs{N(p_1, p_2)}$, factorizes to the spatial and temporal parts, applying Lemmas~\ref{lem:int_of_ns}~(\ref{lem:int_of_ns:common_p_0}) and~\ref{lem:int_of_nt}~(\ref{lem:int_of_nt:common_p}) with $m = 2$, we have for the second term that
    \[ \begin{aligned}
        \lim_{n \tff} \f{1}{n} \iint_{(\S_n \ti \T)^2} \Cov{\deg(p_1), \deg(p_2)} \mdp[1] \mdp[2] &= \lim_{n \tff} \f{1}{n} \iint_{(\S_n \ti \T)^2} \abs{N(p_1, p_2)} \mdp[1] \mdp[2] \\
        &= \f{2 \b^2}{(1 - \g)^2 (1 - 2 \g')},
    \end{aligned} \]
    and we require that $\g' < 1/2$.
    Combining the two terms, we have that the variance is asymptotically equal to~$n$, more precisely,
    \[ \lim_{n \tff} \f{1}{n} \Var{S_n(t)} = \f{2 \b}{(1 - \g)(1 - \g')} + \f{(2 \b)^2}{(1 - 2 \g)(1 - \g')^2} + \f{2 \b^2}{(1 - \g)^2 (1 - 2 \g')}, \]
    as desired.
\end{proof}

%% file: proofs/univariate_normal/lemmas/lem_covariance_of_Snt.tex
\begin{proof}[Proof of Lemma~\ref{lem:covariance_of_Snt}]
    By the definition of the covariance,
    \[ \begin{aligned}
        &\Cov{S_{A_1}(t), S_{A_2}(t)} \\
        &\qquad = \E[\Big]{\sum_{\substack{
        P_1 \in \PP \cap (\S_{A_1} \ti \T), \\ P_2 \in \PP \cap (\S_{A_2} \ti \T)}} \deg(P_1; t) \deg(P_2; t)} - \E[\Big]{\sum_{P_1 \in \PP \cap (\S_{A_1} \ti \T)} \deg(P_1; t)} \E[\Big]{\sum_{P_2 \in \PP \cap (\S_{A_2} \ti \T)} \deg(P_2; t)}.
    \end{aligned} \]
    As $A_1, A_2$ are disjoint, $P_1 \ne P_2$, and thus an application of the Mecke formula yields
    \[ \Cov{S_{A_1}(t), S_{A_2}(t)} = \int_{\S_{A_1} \ti \T} \int_{\S_{A_2} \ti \T} \Cov{\deg(p_1; t), \deg(p_2; t)} \d p_2 \d p_1. \]
    For the integrand, by the definition of $\deg(p)$, we have that
    \[ \begin{aligned}
        \Cov{\deg(p_1; t), \deg(p_2; t)} &= \E[\Big]{\sum_{P' \in \PP'} \1{P' \in N(p_1, p_2; t)}} \\
        &\phantom{=} + \E[\Big]{\sum_{(P_1', P_2') \in (\PP')_{\ne}^2} \1{P_1' \in N(p_1; t)} \1{P_2' \in N(p_2; t)}} \\
        &\phantom{=} - \E[\Big]{\sum_{P_1' \in \PP'} \1{P_1' \in N(p_1; t)}} \E[\Big]{\sum_{P_2' \in \PP'} \1{P_2' \in N(p_2; t)}}.
    \end{aligned} \]
    The last two terms cancel by the independence of the Poisson process $\PP'$, and the Mecke formula gives
    \[ \Cov{\deg(p_1; t), \deg(p_2; t)} = \int_{\S \ti \R} \1{p' \in N(p_1, p_2; t)} \d p' = \abs{N(p_1, p_2; t)}. \]
    Then,
    \[ \begin{aligned}
        \Cov{S_{A_1}(t), S_{A_2}(t)} &= \int_{\S_{A_1} \ti \T} \int_{\S_{A_2} \ti \T} \abs{N(p_1, p_2; t)} \d p_2 \d p_1 \\
        &\le \int_{\S_{A_1} \ti \T} \int_{\S \ti \T} \abs{N(p_1, p_2; t)} \d p_2 \d p_1 \le \f{\abs{A_1}}{2 (1 - 2 \g')} \Big( \f{2 \b}{1 - \g} \Big)^2,
    \end{aligned} \]
    where we applied Lemmas~\ref{lem:int_of_ns}~(\ref{lem:int_of_ns:common_p_0}) and~\ref{lem:int_of_nt}~(\ref{lem:int_of_nt:common_p}).
\end{proof}

%% file: proofs/multivariate_normal/lemmas/lem_low_mark_edge_count_negligible.tex
\begin{proof}[Proof of Lemma~\ref{lem:low_mark_edge_count_negligible}]
    Our proof strategy is applying Chebyshev's inequality and bounding the variance of $S_n^\le(t)$.
    Let $\de > 0$ arbitrary.
    Then, applying Chebyshev's inequality yields
    \[ \P[\big]{\So_n^\le(t) > \de} \le \de^{-2} n^{-1} \Var[\big]{S_n^\le(t)}. \]
    Thus, it is enough to show that $\Var{S_n^\le(t)} = o(n)$.
    We bound the variance by applying the Mecke formula, which, similarly to the proof of Lemma~\ref{lem:mean_variance_of_Snt}, gives
    \[ \Var{S_n^\le(t)} = \int_{\S_n^{\le u_n} \ti \T} \E{\deg(p; t)^2} \mdp + \iint_{(\S_n^{\le u_n} \ti \T)^2} \Cov{\deg(p_1; t), \deg(p_2; t)} \mdp[1] \mdp[2]. \]
    Using~\eqref{eq:expectation_of_degree_square}, we have for the first term
    \[ \int_{\S_n^{\le u_n} \ti \T} \E{\deg(p; t)^2} \mdp = \int_{\S_n^{\le u_n} \ti \T} \abs{N(p; t)}^2 + \abs{N(p; t)} \mdp = c_1 n u_n^{1 - 2\g} + c_2 n u_n^{1 - \g}, \]
    where $c_1 = (2 \b)^2 / ((1 - 2\g) (1 - \g')^2)$ and $c_2 = 2 \b /((1 - \g) (1 - \g'))$.
    In the second step, we used Lemma~\ref{lem:int_of_nt}~(\ref{lem:int_of_nt:power_p}) to bound the temporal part by $1$ for both terms and used Lemma~\ref{lem:int_of_ns}~(\ref{lem:int_of_ns:power_p}) twice for the spatial parts, once with $u_- = 0$ and once with $u_- = u_n$.
    Substituting $u_n = n^{-2/3}$, we have that the first term is of order $O(n^{1 - 2/3 (1 - 2 \g)}) \subset o(n)$, and the second term is of order $O(n^{1 - 2/3 (1 - \g)}) \subset o(n)$ as well.

    Next, we calculate the covariance term by following the proof of Lemma~\ref{lem:mean_variance_of_Snt}:
    \[ \iint_{(\S_n^{\le u_n} \ti \T)^2} \Cov{\deg(p_1; t), \deg(p_2; t)} \mdp[1] \mdp[2] = \iint_{(\S_n^{\le u_n} \ti \T)^2} \abs{N(p_1, p_2; t)} \mdp[1] \mdp[2]. \]
    We bound the temporal part by $2$ using Lemma~\ref{lem:int_of_nt}~(\ref{lem:int_of_nt:common_p}).
    Requiring $\g' < 1/2$, we bound the spatial part using Lemma~\ref{lem:int_of_ns}~(\ref{lem:int_of_ns:common_p_minus}) and~(\ref{lem:int_of_ns:common_p_0}).
    Then,
    \[ \iint_{(\S_n^{\le u_n} \ti \T)^2} \Cov{\deg(p_1; t), \deg(p_2; t)} \mdp[1] \mdp[2] \le c_3 n u_n^{2(1 - \g)} + c_4 n u_n^{1 - \g}, \]
    where $c_3 = 2 (2 \b)^2 / ((1 - \g)^2 (1 - 2 \g'))$ and $c_4 = 4 \b / ((1 - \g) (1 - \g'))$.
    The first term is of order $O(n^{1 - 4/3 (1 - \g)}) \subset o(n)$, and the second term is of order $O(n^{1 - 2/3 (1 - \g)}) \subset o(n)$ as well.
    Thus, the covariance term is of order $O(n)$.
\end{proof}

%% file: proofs/multivariate_normal/lemmas/lem_covariance_function_Snt_ge.tex
\begin{proof}[Proof of Lemma~\ref{lem:covariance_function_of_Snt_ge}]
    The statement and the proof of Lemma~\ref{lem:covariance_function_of_Snt_ge} is almost identical to the proof of Proposition~\ref{prop:covariance_function_of_Snt}.
    The only difference is that we consider the high-mark edge count $\So_n^\ge$ instead of the total edge count $\So_n$, which entails that we need to consider $\S_n^{u_n \le}$ in place of $\S$.

    Using the same time interval-based decomposition of the edge count $\So_n^\ge$ as we defined in the proof of Proposition~\ref{prop:covariance_function_of_Snt}, the covariance function can again be decomposed to three terms.
    Employing Lemma~\ref{lem:covariance_function_terms} finishes the proof of Lemma~\ref{lem:covariance_function_of_Snt_ge}.
\end{proof}

%% file: proofs/covariance_function/lem_covariance_function_terms.tex
\begin{proof}[Proof of Lemma~\ref{lem:covariance_function_terms}]
    We exhibit the terms of the covariance function of the high-mark edge count $\So_n^{\ge}$ when $\g < 1/2$.
    However, using Lemma~\ref{lem:int_of_ns}~(\ref{lem:int_of_ns:common_p_0}) instead of Lemma~\ref{lem:int_of_ns}~(\ref{lem:int_of_ns:common_p_minus}), the below calculations are also valid for the limiting covariance function of the total edge count $\So_n$.
    As Lemma~\ref{lem:int_of_ns}~(\ref{lem:int_of_ns:common_p_0}) gives a specific limit instead of an upper bound, the results of this proof are also accurate when we apply the steps for $\So_n$.

    \medskip
    \noindent
    \textbf{Term $\Var{\So_n^\msf{A}(t_1, t_2)}$.}
    First, note that $\Var{\So_n^\msf{A}(t_1, t_2)} = n^{-1} \Var{S_n^\msf{A}(t_1, t_2)}$.
    Then,
    \[ \begin{aligned}
        \Var{S_n^\msf{A}(t_1, t_2)} &= \int_{\S_n^{u_n \le} \ti \T_{\le t_1}^{t_2 \le}} \E[\big]{\deg(p; t_1)^2} \mdp \\
        &\phantom{=} + \iint_{\big( \S_n^{u_n \le} \ti \T_{\le t_1}^{t_2 \le} \big)^2} \Cov[\big]{\deg(p_1; t_1), \deg(p_2; t_1)} \mdp[1] \mdp[2],
    \end{aligned} \]
    where, similarly as for the variance term in Section~\ref{sec:proof_univariate_normal}, we kept only the term in which the points are the same in the expansion of the variance, and used Mecke's formula.
    Using~\eqref{eq:expectation_of_degree_square}, we have that
    \[ \begin{aligned}
        &\int_{\S_n^{u_n \le} \ti \T_{\le t_1}^{t_2 \le}} \E[\big]{\deg(p; t_1)^2} \mdp = \int_{\S_n^{u_n \le} \ti \T_{\le t_1}^{t_2 \le}} \sum_{k = 1}^2 \abs{N(p; t_1)}^k \mdp \\
        &\qquad = \sum_{k = 1}^2 \int_{\S_n^{u_n \le}} \abs{\Ns(\ps)}^k \d \ps \int_{\T_{\le t_1}^{t_2 \le}} \abs{\Nt(\pt; t_1)}^k \dpt \\
        &\qquad = c_1 n \big( 1 - u_n^{1 - \g} \big) \int_{-\ff}^{t_1} \int_{t_2 - b}^\ff (t_1 - b) \dell \d b + c_2 n \big( 1 - u_n^{1 - 2 \g} \big) \int_{-\ff}^{t_1} \int_{t_2 - b}^\ff (t_1 - b)^2 \dell \d b \\
        &\qquad = n \big( c_1 (1 - u_n^{1 - \g}) + 2 c_2 (1 - u_n^{1 - 2 \g}) \big) \e^{-(t_2 - t_1)},
    \end{aligned} \]
    where $c_1 = 2 \b / ((1 - \g) (1 - \g'))$, $c_2 = (2 \b)^2 / ((1 - 2 \g)(1 - \g')^2)$ are positive constants.
    Substituting $u_n = n^{-2/3}$, we have that
    \[ \lim_{n \tff} \f{1}{n} \int_{\S_n^{u_n \le} \ti \T_{\le t_1}^{t_2 \le}} \E[\big]{\deg(p; t_1)^2} \mdp = \big( c_1 + 2 c_2 \big) \e^{-(t_2 - t_1)}. \]
    The covariance term is calculated the same as in the variance calculation of the edge count $\Cov[\big]{\deg(p_1; t_1), \deg(p_2; t_1)} = \abs{N(p_1, p_2; t_1)}$.
    After applying Lemma~\ref{lem:int_of_ns}~(\ref{lem:int_of_ns:common_p_minus}) and Lemma~\ref{lem:int_of_nt}~(\ref{lem:int_of_nt:common_p}) with $m = 2$, we have that
    \[ \begin{aligned}
        &\lim_{n \tff} \f{1}{n} \iint_{\big( \S_n^{u_n \le} \ti \T_{\le t_1}^{t_2 \le} \big)^2} \Cov[\big]{\deg(p_1; t_1), \deg(p_2; t_1)} \mdp[1] \mdp[2] = I_\msf{s} \ti I_\msf{t} \\
        &\qquad I_\msf{s} := \lim_{n \tff} \f{1}{n} \iint_{(\S_n^{u_n \le})^2} \abs{\Ns(\ps[1], \ps[2])} \d \ps[1] \d \ps[2] = \lim_{n \tff} c_3 \big( 1 - u_n^{1 - \g} \big)^2 = c_3 \\
        &\qquad I_\msf{t} := \int_\R \bigg( \int_{\T_{\le t_1}^{t_2 \le}} \1{r \in \Nt(\pt; t_1)} \dpt \bigg)^2 \d r = \f{1}{2} \e^{-2 (t_2 - t_1)},
    \end{aligned} \]
    where $c_3 = (2 \b)^2 / ((1 - \g)^2 (1 - 2 \g'))$ as defined in the statement of the lemma, and the application of Lemma~\ref{lem:int_of_ns}~(\ref{lem:int_of_ns:common_p_minus}) requires that $\g' < 1/2$.

    \medskip
    \noindent
    \textbf{Term $\Cov{\So_n^\msf{A}(t_1, t_2), \So_n^\msf{B}(t_1, t_2)}$.}
    First, note that $\Cov{\So_n^\msf{A}(t_1, t_2), \So_n^\msf{B}(t_1, t_2)} = n^{-1} \Cov{S_n^\msf{A}(t_1, t_2), S_n^\msf{B}(t_1, t_2)}$.
    Then,
    \[ \begin{aligned}
        \Cov[\big]{S_n^\msf{A}(t_1, t_2), S_n^\msf{B}(t_1, t_2)} &= \E[\big]{S_n^\msf{A}(t_1, t_2) S_n^\msf{B}(t_1, t_2)} - \E[\big]{S_n^\msf{A}(t_1, t_2)} \E[\big]{S_n^\msf{B}(t_1, t_2)} \\
        &= \E[\bigg]{\sum_{\substack{P_1 \in \PP \cap \big( \S_n^{u_n \le} \ti \T_{\le t_1}^{t_2 \le} \big) \\ P_2 \in \PP \cap \big( \S_n^{u_n \le} \ti \T_{\le t_1}^{[t_1, t_2]} \big)}} \deg(\set{P_1, P_2}; t_1)},
    \end{aligned} \]
    where we have used that $P_1 \ne P_2$ as they appear in disjoint sets.
    Next, we apply Mecke's formula to the above expression, and then use Lemmas~\ref{lem:int_of_ns}~(\ref{lem:int_of_ns:common_p_minus}) and~\ref{lem:int_of_nt}~(\ref{lem:int_of_nt:common_p}) with $m = 2$ to calculate the covariance term:
    \[ \begin{aligned}
        &\lim_{n \tff} \f{1}{n} \Cov[\big]{S_n^\msf{A}(t_1, t_2), S_n^\msf{B}(t_1, t_2)} = \lim_{n \tff} \f{1}{n} \int_{\S_n^{u_n \le} \ti \T_{\le t_1}^{t_2 \le}} \int_{\S_n^{u_n \le} \ti \T_{\le t_1}^{[t_1, t_2]}} \abs{N(p_1, p_2; t_1)} \mdp[2] \mdp[1] = I_\msf{s} \ti I_\msf{t} \\
        &\qquad I_\msf{s} := \lim_{n \tff} \f{1}{n} \int_{\S_n^{u_n \le}} \int_{\S_n^{u_n \le}} \abs{\Ns(\ps[1], \ps[2])} \d \ps[2] \d \ps[1] = \lim_{n \tff} c (1 - u_n^{1 - \g})^2 = c_3 \\
        &\qquad I_\msf{t} := \int_\R \bigg( \int_{\T_{\le t_1}^{t_2 \le}} \1{r \in \Nt(\pt; t_1)} \dpt \int_{\T_{\le s}^{[t_1, t_2]}} \1{r \in \Nt(\pt; t_1)} \dpt \bigg) \d r \\
        &\qquad \phantom{I_\msf{t}} = \int_{-\ff}^{t_1} \e^{-(t_2 - r)} \big( \e^{-(t_1 - r)} - \e^{-(t_2 - r)} \big) \d r = \e^{-(t_2 - t_1)} \big( 1 - \e^{-(t_2 - t_1)} \big),
    \end{aligned} \]
    where the finiteness of the integral $I_\msf{s}$ requires that $\g' < 1/2$, and we substituted $u_n = n^{-2/3}$.

    \medskip
    \noindent
    \textbf{Term $\Cov{S_n^\msf{A}(t_1, t_2), S_n^\msf{C}(t_1, t_2)}$.}
    In this part of the proof, we use that the covariance term $\Cov{S_n^\msf{A}(t_1, t_2), S_n^\msf{C}(t_1, t_2)}$ is determined by the common $\PP$-points, as the $\PP'$-points in $S_n^\msf{A}(t_1, t_2)$ and $S_n^\msf{C}(t_1, t_2)$ cannot be identical.
    Note that $\Cov{S_n^\msf{A}(t_1, t_2), S_n^\msf{C}(t_1, t_2)} = n^{-1} \Cov{S_n^\msf{A}(t_1, t_2), S_n^\msf{C}(t_1, t_2)}$.
    Then, we have
    \[ \begin{aligned}
        \Cov[\big]{S_n^\msf{A}(t_1, t_2), S_n^\msf{C}(t_1, t_2)} &= \E[\big]{S_n^\msf{A}(t_1, t_2) S_n^\msf{C}(t_1, t_2)} - \E[\big]{S_n^\msf{A}(t_1, t_2)} \E[\big]{S_n^\msf{C}(t_1, t_2)} \\
        &= \E[\bigg]{\sum_{P \in \PP \cap \big( \S_n^{u_n \le} \ti \T_{\le t_1}^{t_2 \le} \big)} \deg(P; t_1) \sum_{P' \in \PP'} \1{P' \in N(P; t_2)} \1{t_1 \le R}},
    \end{aligned} \]
    where we have used that $P'_1 \ne P'_2$ as they appear in disjoint sets.
    As earlier, we apply Mecke's formula to the above expression, and, after factorizing the spatial and the temporal parts, we use Lemma~\ref{lem:int_of_ns}~(\ref{lem:int_of_ns:power_p}) for the spatial part:
    \[ \begin{aligned}
        &\Cov[\big]{S_n^\msf{A}(t_1, t_2), S_n^\msf{C}(t_1, t_2)} = \int_{\S_n^{u_n \le} \ti \T_{\le t_1}^{t_2 \le}} \abs{N(p; t_1)} \int_{\S \ti [t_1, \ff]} \1{p' \in N(p; t_2)} \d p' \mdp = I_\msf{s} \ti I_\msf{t} \\
        &\qquad I_\msf{s} := \int_{\S_n^{u_n \le}} \abs{\Ns(\ps)}^2 \d \ps = c_2 n \big( 1 - u_n^{1 - 2 \g} \big) \\
        &\qquad I_\msf{t} := \int_{\T_{\le t_1}^{t_2 \le}} \abs{\Nt(\pt; t_1)} \int_{[t_1, \ff]} \1{r \in \Nt(p; t_2)} \d r \dpt = \int_{\T_{\le t_1}^{t_2 \le}} (t_1 - b) \int_{[t_1, t_2]} \d r \dpt \\
        &\qquad \quad \phantom{I_\msf{t}} = (t_2 - t_1) \int_{-\ff}^{t_1} (t_1 - b) \e^{-(t_2 - b)} \d b = (t_2 - t_1) \e^{-(t_2 - t_1)}.
    \end{aligned} \]
    Then, we have that $\Cov[\big]{S_n^\msf{A}(t_1, t_2), S_n^\msf{C}(t_1, t_2)} = c_2 n (t_2 - t_1) \e^{-(t_2 - t_1)}$ and thus
    \[ \lim_{n \tff} \f{1}{n} \Cov[\big]{S_n^\msf{A}(t_1, t_2), S_n^\msf{C}(t_1, t_2)} = c_2 (t_2 - t_1) \e^{-(t_2 - t_1)}, \]
    where we substituted $u_n = n^{-2/3}$.
\end{proof}

%% file: proofs/multivariate_normal/lemmas/lem_bounds_of_error_terms.tex
\begin{proof}[Proof of Lemma~\ref{lem:bounds_of_error_terms}]
    First, let us recall from the statement of Proposition~\ref{prop:multivariate_normal_limit} that $\pw$ is a point of the Poisson process~$\wt{\PP} = \PP \sqcup \PP'$ with intensity measure $\wt{\m}$.
    Note that
    \[ D_{\pw} \So_n^\ge(t) = n^{-1/2} D_{\pw} S_n^\ge(t) \qquad \text{and} \qquad D^2_{\pw_1, \pw_2} \So_n^\ge(t) = n^{-1/2} D^2_{\pw_1, \pw_2} S_n^\ge(t). \]
    The expectations $\E{D_p \So_n^\ge(t)}$ and $\E{D_{p'} \So_n^\ge(t)}$ can be calculated by the application of Mecke's formula:
    \[ \E[\big]{D_p S_n^\ge(t)} = \abs{N(p; t)} \qquad \text{and} \qquad \E[\big]{D_{p'} S_n^\ge(t)} = \int_{\S_n^{u_n \le} \ti \T} \1{p' \in N(p; t)} \mdp. \]
    In the following proof, we will need the third and fourth moments of the cost operators.
    As $D_{p} S_n^\ge(t)$ and $D_{p'} S_n^\ge(t)$ are Poisson distributed, we can express their moments using the Touchard polynomials $\t_k(x)$~\cite{touchard} defined by:
    \[ \t_k(x) = \sum_{i = 0}^k \Big\{ \hspace{-0.15cm} \begin{array}{c} k \\ i \end{array} \hspace{-0.15cm} \Big\} x^i, \]
    where the curly brackets denote the Stirling numbers of the second kind.
    Then, the moments of the cost operators can be upper bounded by
    \begin{equation}
        \E[\big]{D_{\pw} S_n^\ge(t)^3} = \t_3 \big( \E[\big]{D_{\pw} S_n^\ge(t)} \big) \qquad \text{and} \qquad \E[\big]{D_{\pw} S_n^\ge(t)^4} = \t_4 \big( \E[\big]{D_{\pw} S_n^\ge(t)} \big) \le 16 \max \big( \E[\big]{D_{\pw} S_n^\ge(t)}^4, 1 \big). \label{eq:poisson_moment_bounds}
    \end{equation}

    \medskip
    \noindent
    \textbf{Error term $E_1(n)$.}
    We substitute $\So_n^\ge(t_{\any}) = n^{-1/2} S_n^\ge(t_{\any})$ and apply the Cauchy--Schwarz inequality:
    \[ \begin{aligned}
        E_1(n)^2 \le n^{-2} \int \Big( &\E[\big]{\big( D^2_{\pw_1, \pw_3} S_n^\ge(t_1) \big)^4} \E[\big]{\big( D^2_{\pw_2, \pw_3} S_n^\ge(t_1) \big)^4} \\
        &\ti \E[\big]{\big( D_{\pw_1} S_n^\ge(t_2) \big)^4} \E[\big]{\big( D_{\pw_2} S_n^\ge(t_2) \big)^4} \Big)^{1/4} \, \wt{\m} (\d (\pw_1, \pw_2, \pw_3)).
    \end{aligned} \]
    Note that the integrand is nonzero only if either $\pw_1, \pw_2 \in \S_n^{u_n \le} \ti \T$ and $\pw_3 \in \S \ti \R$, or $\pw_1, \pw_2 \in \S \ti \R$ and $\pw_3 \in \S_n^{u_n \le} \ti \T$, as otherwise at least one of the cost operators $D^2_{\pw_1, \pw_3}$ and $D^2_{\pw_2, \pw_3}$ is zero.
    In the first case, we have that
    \[ E_1(n)^2 \le n^{-2} \iint_{(\S_n^{u_n \le} \ti \T)^2} \Big( \E[\big]{D_{p_1} S_n^\ge(t_2)^4} \E[\big]{D_{p_2} S_n^\ge(t_2)^4} \Big)^{1/4} \abs{N(p_1, p_2; t_1)} \mdp[1] \mdp[2]. \]
    We bound the fourth moments of $D_{p_1} S_n^\ge(t_2)$ and $D_{p_2} S_n^\ge(t_2)$ using~\eqref{eq:poisson_moment_bounds}, and we extend the integral over $\S_n^{u_n \le} \ti \T$ to $\S_n \ti \T$.
    Then, the integral can be upper bounded by a sum of terms of the form
    \[ \begin{aligned}
        &4 \iint_{(\S_n \ti \T)^2} \abs{N(p_1; t_2)}^{m_1} \abs{N(p_1, p_2; t_1)} \abs{N(p_2; t_2)}^{m_2} \mdp[1] \mdp[2] = 4 I_\msf{s} \ti I_\msf{t} \\
        &\qquad I_\msf{s} := \iint_{(\S_n)^2} \abs{\Ns(\ps[1])}^{m_1} \abs{\Ns(\ps[1], \ps[2])} \abs{\Ns(\ps[2])}^{m_2} \d \ps[1] \d \ps[2] \\
        &\qquad I_\msf{t} := \iint_{\T^2} \abs{\Nt(\pt[1]; t_2)}^{m_1} \abs{\Nt(\pt[1], \pt[2]; t_1)} \abs{\Nt(\pt[2]; t_2)}^{m_2} \dpt[1] \dpt[2],
    \end{aligned} \]
    where $m_1, m_2 \in \set{0, 1}$.
    For the spatial part $I_\msf{s}$, we use Lemma~\ref{lem:int_of_ns}~(\ref{lem:int_of_ns:power_cap_power_0}) to see that $I_\msf{s} \in O(n)$ if $\g, \g' < 1/2$.
    We bound the temporal part~$I_\msf{t} \le 1$ using Lemma~\ref{lem:int_of_nt}~(\ref{lem:int_of_nt:power_cap_power}), thus $E_1(n) \in O(n^{-1/2})$.

    In the second case, when $\pw_1, \pw_2 \in \S \ti \R$ and $\pw_3 \in \S_n^{u_n \le} \ti \T$, we have that
    \[ \begin{aligned}
        E_1(n)^2 &\le n^{-2} \int_{\S_n^{u_n \le} \ti \T} \iint_{(\S \ti \R)^2} \1{\set{p'_1, p'_2} \su N(p; t_1)} \Big( \E[\big]{\big( D_{p'_1} S_n^\ge(t_2) \big)^4} \E[\big]{\big( D_{p'_2} S_n^\ge(t_2) \big)^4} \Big)^{1/4} \d p'_1, \d p'_2 \mdp \\
        &\le n^{-2} \int_{\S_n \ti \T} \bigg( \int_{\S \ti \R} \1{p' \in N(p; t_1)} \E[\big]{\big( D_{p'} S_n^\ge(t_2) \big)^4}^{1/4} \d p' \bigg)^2 \mdp,
    \end{aligned} \]
    where we extended the integral over $\S_n^{u_n \le} \ti \T$ to $\S_n \ti \T$ in the last step.
    We use~\eqref{eq:poisson_moment_bounds} again to bound $\E{D_{p'} S_n^\ge(t_2)^4}$:
    \[ E_1(n)^2 \le 4 n^{-2} \int_{\S_n \ti \T} \bigg( \int_{\S \ti \R} \1{p' \in N(p_1; t_1)} \max \bigg( \int_{\S_n \ti \T} \1{p' \in N(p_2; t_2)} \mdp[2], 1 \bigg) \d p' \bigg)^2 \mdp[1], \]
    the integration domain of the inner integral was extended to $\S_n \ti \T$ again in the last step.
    If the maximum is equal to one, the integral can be upper bounded by $\int_{\S_n \ti \T} \abs{N(p; t_1)}^2 \mdp \in O(n)$ by Lemmas~\ref{lem:int_of_ns}~(\ref{lem:int_of_ns:power_p}) and~\ref{lem:int_of_nt}~(\ref{lem:int_of_nt:power_p}). 
    On the other hand, if the maximum is greater than one, the integral is bounded by
    \[ \begin{aligned}
        \int_{\S_n \ti \T} &\bigg( \int_{\S \ti \R} \1{p' \in N(p_1; t_1)} \int_{\S_n \ti \T} \1{p' \in N(p_2; t_2)} \mdp[2] \d p' \bigg)^2 \mdp[1] = I_\msf{s} \ti I_\msf{t} \\
        I_\msf{s} &:= \int_{\S_n} \bigg( \int_\S \1{\ps' \in \Ns(\ps[1])} \int_{\S_n} \1{\ps' \in \Ns(\ps[2])} \d \ps[2] \d \ps' \bigg)^2 \d \ps[1] \\
        I_\msf{t} &:= \int_\T \bigg( \int_\R \1{r \in \Nt(\pt[1]; t_1)} \int_\T \1{r \in \Nt(\pt[2]; t_2)} \dpt[2] \d r \bigg)^2 \dpt[1].
    \end{aligned} \]
    We begin with the spatial part $I_\msf{s}$:
    \[ \begin{aligned}
        I_\msf{s} &\le \Big( \f{2 \b}{1 - \g} \Big)^2 \int_{\S_n} \bigg( \int_\S \1{\ps' \in \Ns(\ps)} w^{-\g'} \d \ps' \bigg)^2 \d \ps \\
        &= \f{(2 \b)^4}{(1 - \g)^2 (1 - 2 \g')^2} \int_{\S_n} u^{- 2 \g} \d (x, u) = \f{(2 \b)^4 n}{(1 - \g)^2 (1 - 2 \g) (1 - 2 \g')^2} \in O(n),
    \end{aligned} \]
    where we extended the innermost integral to the whole space $\S \ti \R$ and applied Lemmas~\ref{lem:size_of_ns}~(\ref{lem:size_of_ns:p_prime}) and~\ref{lem:size_of_nt}~(\ref{lem:size_of_nt:p_prime}) in the first step.
    Next, we move on to the temporal part $I_\msf{t}$:
    \[ \begin{aligned}
        I_\msf{t} &= \int_\T \bigg( \int_\R \1{r \in \Nt(\pt; t_1)} \e^{-(t_2 - r)} \d r \bigg)^2 \dpt = \int_\T (\e^{-(t_2 - t_1)} - \e^{-(t_2 - b)})^2 \1{b \le t_1 \le b + \ell} \dbdell \\
        &= \int_{-\ff}^{t_1} \int_{t_1 - b}^\ff (\e^{-(t_2 - t_1)} - \e^{-(t_2 - b)})^2 \dell \d b = \f{1}{3} \e^{-(t_2 - t_1)},
    \end{aligned} \]
    Thus, in both cases, we have that $E_1(n) \in O(n^{-1/2})$.

    \medskip
    \noindent
    \textbf{Error term $E_2(n)$.}
    We substitute $\So_n^\ge(t_{\any}) = n^{-1/2} S_n^\ge(t_{\any})$ and apply the Cauchy--Schwarz inequality:
    \[ \begin{aligned}
        E_2(n)^2 \le n^{-2} \int \Big( &\E[\big]{\big( D^2_{\pw_1, \pw_3} S_n^\ge(t_1) \big)^4} \E[\big]{\big( D^2_{\pw_2, \pw_3} S_n^\ge(t_1) \big)^4} \\
        &\ti \E[\big]{\big( D^2_{\pw_1, \pw_3} S_n^\ge(t_2) \big)^4} \E[\big]{\big( D^2_{\pw_2, \pw_3} S_n^\ge(t_2) \big)^4} \Big)^{1/4} \, \wt{\m} (\d (\pw_1, \pw_2, \pw_3)).
    \end{aligned} \]
    As in the case of $E_1$, the integrand is nonzero only if either $\pw_1, \pw_2 \in \S_n^{u_n \le} \ti \T$ and $\pw_3 \in \S \ti \R$, or if $\pw_1, \pw_2 \in \S \ti \R$ and $\pw_3 \in \S_n^{u_n \le} \ti \T$.
    In the first case,
    \[ \begin{aligned}
        E_2(n)^2 &\le n^{-2} \int_{(\S_n^{u_n \le} \ti \T)^2} \int_{\S \ti \R} \1{p' \in N(p_1, p_2; t_1)} \1{p' \in N(p_1, p_2; t_2)} \d p' \, \m^{\otimes 2} (\d (p_1, p_2)) \\
        &\le n^{-2} \iint_{\S_n^2} \abs{\Ns(\ps[1], \ps[2])} \d \ps[2] \d \ps[1] \iint_{\T^2} \abs{\Nt(\pt[1], \pt[2]; t_1)} \dpt[1] \dpt[2],
    \end{aligned} \]
    where we extended the integration domain from $\S_n^{u_n \le} \ti \T$ to $\S_n \ti \T$, and we used that $\abs{\Nt(\pt[1], \pt[2]; t_1) \cap \Nt(\pt[1], \pt[2]; t_2)} \le \abs{\Nt(\pt[1], \pt[2]; t_1)}$.
    To upper bound the spatial and the temporal parts, we use Lemmas~\ref{lem:int_of_ns}~(\ref{lem:int_of_ns:common_p_minus}) and~\ref{lem:int_of_nt}~(\ref{lem:int_of_nt:common_p}) with $m = 2$, respectively:
    \[ E_2(n)^2 \le \f{2 \b^2}{(1 - \g)^2 (1 - 2 \g')} n^{-1}, \]
    and thus $E_2(n) \in O(n^{-1/2})$.

    In the second case, when $\pw_1, \pw_2 \in \S \ti \R$ and $\pw_3 \in \S_n^{u_n \le} \ti \T$, we have that
    \[ \begin{aligned}
        E_2(n)^2 &\le n^{-2} \int_{\S_n^{u_n \le} \ti \T} \iint_{(\S \ti \R)^2} \1{p'_1 \in N(p; t_1) \cap N(p; t_2)} \1{p'_2 \in N(p; t_1) \cap N(p; t_2)} \d (p'_1, p'_2) \mdp \\
        &\le n^{-2} \int_{\S_n \ti \T} \abs{N(p; t_1) \cap N(p; t_2)}^2 \mdp \le n^{-2} \int_{\S_n \ti \T} \abs{N(p; t_1)}^2 \mdp \le \f{2 (2 \b)^2 n^{-1}}{(1 - 2 \g) (1 - \g')^2} \in O(n^{-1}),
    \end{aligned} \]
    where we used Lemma~\ref{lem:int_of_ns}~(\ref{lem:int_of_ns:power_p}) with $\a = 2$, $A = [0, n]$ and $u_- = 0$. 
    Thus, in both cases, we have that $E_2(n) \in O(n^{-1/2})$.

    \medskip
    \noindent
    \textbf{Error term $E_3(n)$.}
    If $p := \pw \in \S_n^{u_n \le} \ti \T$, then
    \[ E_3(n) = n^{-3/2} \int_{\S_n^{u_n \le} \ti \T} \E[\big]{\big( D_p S_n^\ge(t) \big)^3} \mdp = n^{-3/2} \int_{\S_n^{u_n \le} \ti \T} \abs{N(p; t)}^3 + 3 \abs{N(p; t)}^2 + \abs{N(p; t)} \mdp, \]
    where we expressed the third moment of the cost operator with the third Touchard polynomial.
    Next, we apply Lemma~\ref{lem:int_of_nt}~(\ref{lem:int_of_nt:power_p}) for the temporal parts to show that they can be bounded by $\Ga(4)$:
    \[ E_3(n) \le \Ga(4) n^{-3/2} \int_{\S_n^{u_n \le}} \abs{\Ns(\ps)}^3 + 3 \abs{\Ns(\ps)}^2 + \abs{\Ns(\ps)} \d \ps. \]
    For the second and third terms, we extend the integration domain from $\S_n^{u_n \le}$ to $\S_n$, and apply Lemma~\ref{lem:int_of_ns}~(\ref{lem:int_of_ns:power_p}) with $A = [0, n]$ and $u_- = 0$, requiring that $\g < 1/2$, to show that their orders are $O(n^{-1/2})$.
    For the first term, we apply Lemma~\ref{lem:int_of_ns}~(\ref{lem:int_of_ns:power_p}) to get
    \[ \int_{\S_n^{u_n \le}} \abs{N(p; t)}^3 \d (x, u) = \left\{ \begin{array}{ll} \Big( \f{2 \b}{1 - \g'} \Big)^3 \f{n}{1 - 3 \g} (1 - u_n^{1 - 3 \g}) & \text{if } \g \ne 1/3 \\ \Big( \f{2 \b}{1 - \g'} \Big)^3 n \log{u_n^{-1}} & \text{if } \g = 1/3, \end{array} \right. \]
    which is positive regardless of the value of $\g \in (0, 1/2)$, since when the constant $\g > 1/3$, then $(1 - u_n^{1 - 3 \g}) < 0$ as well.
    Substituting $u_n = n^{-2/3}$, we have that the order of the integral is $o(n^{3/2})$, thus $E_3(n) \to 0$ as $n \to \ff$.

    If $p' := \pw \in \S \ti \R$, then
    \[ E_3(n) = n^{-3/2} \int_{\S \ti \R} \E[\big]{\big( D_{p'} S_n^\ge(t) \big)^3} \d p' = n^{-3/2} \int_{\S \ti \R} \E[\big]{D_{p'} S_n^\ge(t)}^3 + 3 \E[\big]{D_{p'} S_n^\ge(t)}^2 + \E[\big]{D_{p'} S_n^\ge(t)} \d p'. \]
    To see that $\lim_{n \tff} E_3(n) = 0$, it is enough to show that $\int_{\S \ti \R} \E{D_{p'} S_n^\ge(t)}^m \d p' \in O(n)$ where $m \in \set{1, 2, 3}$.
    Thus, factorizing the integral as usual to the spatial and the temporal parts,
    \[ \begin{aligned}
        \int_{\S \ti \R} \E[\big]{D_{p'} S_n^\ge(t)}^m \d p' &= \int_{\S \ti \R} \bigg( \int_{\S_n^{u_n \le} \ti \T} \1{p' \in N(p; t)} \mdp \bigg)^m \d p' = I_\msf{s} \ti I_\msf{t} \\
        I_\msf{s} &:= \int_\S \bigg( \int_{\S_n^{u_n \le}} \1{\ps' \in \Ns(\ps)} \d \ps \bigg)^m \d \ps' \in O(n) \\
        I_\msf{t} &:= \int_\R \bigg( \int_\T \1{r \in \Nt(\pt; t)} \dpt \bigg)^m \d r = \f{1}{m},
    \end{aligned} \]
    where, after extending the integration domain from $\S_n^{u_n \le}$ to $\S_n$, we applied Lemma~\ref{lem:int_of_ns}~(\ref{lem:int_of_ns:power_p_prime}) to the spatial part, 
    and we used Lemma~\ref{lem:int_of_nt}~(\ref{lem:int_of_nt:power_prime}) for the temporal part.
    Thus, we have that $\lim_{n \tff} E_3(n) = 0$.
\end{proof}

%% file: proofs/functional_normal/lemmas/lem_plus_minus_decomposition.tex
\begin{proof}[Proof of Lemma~\ref{lem:plus_minus_decomposition}]
We have that
    \[ \begin{aligned}
        S_n^+(t) - S_n^-(t) &= \sum_{P \in \PP \cap (\S_n \ti \T^{0 \le})} \sum_{P' \in \PP'} \1{P' \in N^+(P; t)} - \sum_{P \in \PP \cap (\S_n \ti \T^{0 \le})} \sum_{P' \in \PP'} \1{P' \in N^-(P; t)} \\
        &= \sum_{P \in \PP \cap (\S_n \ti \T^{0 \le})} \sum_{P' \in \PP'} \1{\Ps' \in \Ns(\Ps)} \Big( \1{0 \le B + L} \1{B \le R \le (B + L) \w t} \\
        &\phantom{= \sum_{P \in \PP \cap (\S_n \ti \T)} \sum_{P' \in \PP'} \1{\Ps' \in \Ns(\Ps)}} - \1{B \le R \le B + L \le t} \Big) \\
        &= \sum_{P \in \PP \cap (\S_n \ti \T^{0 \le})} \sum_{P' \in \PP'} \1{\Ps' \in \Ns(\Ps)} \1{R \in \Nt(\Pt; t)} = S_n(t),
    \end{aligned} \]
    as desired.
\end{proof}

%% file: proofs/functional_normal/propositions/prop_finite_dimensional_distributions.tex
\begin{proof}[Proof of Proposition~\ref{prop:davydov_condition_1}]
    To show Condition~(\ref{condition:finite_dimensional}), we follow the same arguments for $\So_n^\pm(t)$ as in Section~\ref{sec:proof_multivariate_normal} for~$S_n(t)$, i.e., we employ~\cite[Theorem~1.1]{normalapproximation2}.
    Following the same steps, we first show that the limiting covariance function of $\So_n^\pm(t)$ is bounded.
    \begin{lemma}[Covariance function of $\So_n^\pm(t)$]\label{lem:limiting_covariance_function_Snt_pm}
        Let $\g, \g' < 1/2$ and $0 \le s \le t \le 1$.
        Then the limit $\lim_{n \tff} \Cov{\So_n^\pm(s), \So_n^\pm(t)} < \ff$, exists.
    \end{lemma}
    \noindent
    The proof of Lemma~\ref{lem:limiting_covariance_function_Snt_pm} and all proofs of subsequent lemmas in this section are given in Section~\ref{sec:proofs_functional_normal}.
    The next step is to bound the error terms $E_1(n)$, $E_2(n)$, $E_3(n)$.
    As in the proof of Lemma~\ref{lem:bounds_of_error_terms}, the cost operators $D_p S^\pm_n(t)$, $D_{p'} S^\pm_n(t)$ and $D^2_{p, p'} S^\pm_n(t)$ are given by
    \[ \begin{aligned}
        D_p S^\pm_n(t) &:= \sum_{P' \in \PP'} \1{P' \in N^\pm(p; t)} \\
        D_{p'} S^\pm_n(t) &:= \sum_{P \in \PP \cap (\S_n \ti \T)} \1{p' \in N^\pm(P; t)} \\
        D^2_{p, p'} S^\pm_n(t) &:= \1{p' \in N^\pm(p; t)},
    \end{aligned} \]
    where $D_p S^\pm_n(t)$ and $D_{p'} S^\pm_n(t)$ are Poisson distributed, and $D^2_{p, p'} S^\pm_n(t)$ is deterministic.
    \begin{lemma}[Bounds of the error terms for $\So_n^\pm(t)$]\label{lem:bounds_of_error_terms_pm}
        Let $\g < 1/2$, $\g' < 1/3$.
        Then, the error terms defined in~\eqref{eq:d3_distance} with $\So_n^\pm$ in place of $\So_n$ satisfy $\lim_{n \tff} (E_1(n) + E_2(n) + E_3(n)) = 0$.
    \end{lemma}
    \noindent
    This shows that each finite-dimensional distribution of $\So_n^\pm$ converge to a multivariate normal distribution.
\end{proof}

%% file: proofs/functional_normal/propositions/prop_tightness.tex
\begin{proof}[Proof of Proposition~\ref{prop:tightness_So_n_pm}]
    We have that
    \[ \E[\big]{\big( \So_n^\pm(t) - \So_n^\pm(s) \big)^4} = n^{-2} \E[\big]{\big( \De_n^\pm(s, t) - \E{\De_n^\pm(s, t)} \big)^4}, \]
    where we wrote $\De_n^\pm(s, t) := S^\pm_n(t) - S^\pm_n(s)$, which is nonnegative for $s \le t$.
    The fourth moment can be expressed with the fourth cumulant $\k_4$~\cite[Appendix~B]{wiener_chaos} as
    \[ \E[\big]{\big( \De_n^\pm(s, t) - \E{\De_n^\pm(s, t)} \big)^4} = \k_4(\De_n^\pm(s, t)) + 3 \Var{\De_n^\pm(s, t)}^2. \]
    Beginning with the variance term, we have the following.
    \begin{lemma}[Variance of $\De_n^\pm(s, t)$]\label{lem:variance_of_De}
        Let $\g, \g' < 1/2$, then, we have $\Var{\De_n^\pm(s, t)} \in O(n(t - s))$.
    \end{lemma}
    \noindent
    Next, we show for the cumulant term that $\abs{\k_4(\De_n^\pm(s, t))} \in O(n (t - s)) \su O(n^2 (t - s)^{1 + \eta})$ for all $s, t \in (0, 1)$ if $t - s \ge a_n$.
    \begin{lemma}[Order of the cumulant term $\k_4(\De_n^\pm(s, t))$]\label{lem:cumulant_term_cases}
        Let $\g, \g' < 1/4$.
        We have for $0 \le s < t \le 1$ the orders of the cumulant term $\k_4(\De_n^\pm(s, t)) \in O(n (t - s))$.
    \end{lemma}
    \noindent
    Before we present the proof of Lemma~\ref{lem:cumulant_term_cases} in Section~\ref{sec:proofs_functional_normal}, let us here provide a rough outline.

    We follow the arguments in~\cite[Lemma~C.3]{cipriani} and partition the window $[0, n]$ into $n$ disjoint intervals $\set{[i - 1, i] \co i \in \set{1, \dots, n}}$ of size $1$, and write $\V_i := \S_{[i - 1, i]}$.
    For $t \in [0, 1]$, analogously to the definition of $S_n^\pm(t)$, we write $V_i^\pm(t)$ for the number of edges whose \quote{$\PP$-endpoint} is in~$\V_i$, i.e.,
    \[ V_i^\pm(t) := \sum_{P \in \PP \cap (\V_i \ti \T^{0 \le})} \deg^\pm(P; t) \qquad \text{and} \qquad S_n^\pm(t) = \sum_{i = 1}^n V_i^\pm(t). \]
    Note that the mappings $t\mapsto V_i^\pm(t)$ are monotone.
    Furthermore, we write
    \[ \De_{i, n}^\pm(s, t) := V_i^\pm(t) - V_i^\pm(s) \qquad \text{and} \qquad \De_n^\pm(s, t) = \sum_{i = 1}^n \De_{i, n}^\pm(s, t) \]
    to denote the change of $V_i^\pm(t)$ in the time interval $[s, t]$.
    Then, using multilinearity of the cumulant,
    \[ \begin{aligned}
        \abs{\k_4(\De_n^\pm(s, t))} &= \abs[\Big]{\sum_{i, j, k, \ell = 1}^n c_{i, j, k, \ell} \k_4(\De_{i, n}^\pm(s, t), \De_{j, n}^\pm(s, t), \De_{k, n}^\pm(s, t), \De_{\ell, n}^\pm(s, t))} \\
        &\le c_1 \sum_{i, j, k, \ell = 1}^n \abs[\big]{\k_4(\De_{i, n}^\pm(s, t), \De_{j, n}^\pm(s, t), \De_{k, n}^\pm(s, t), \De_{\ell, n}^\pm(s, t))},
    \end{aligned} \]
    where $c_{i, j, k, \ell} > 0$ are coefficients depending only on which of the indices $i, j, k, \ell$ are equal, $c_1 := \max(\set{c_{i, j, k, \ell}})$, $\k_4(\De_{i, n}^\pm(s, t), \De_{j, n}^\pm(s, t), \De_{k, n}^\pm(s, t), \De_{\ell, n}^\pm(s, t))$ denotes the joint fourth-order cumulant, and we used the triangle inequality in the last step.
    From now on, we drop the arguments $(s, t)$ for brevity.
    Let $\set{\pi_1, \pi_2} \preceq \set{i, j, k, \ell}$ denote a partition of the set $\set{i, j, k, \ell}$, and define
    \begin{equation}
        \r(i, j, k, \ell) := \max_{\set{\pi_1, \pi_2} \preceq \set{i, j, k, \ell}} \big( \mrm{dist} \big( \bs{V}_{\pi_1}, \bs{V}_{\pi_2} \big) \big),
        \label{eq:rho_definition}
    \end{equation}
    where $\bs{V}_\pi := \set{\mbb{V}_a \co a \in \pi}$.
    Then,
    \[ \abs{\k_4(\De_n^\pm)} = c_1 \sum_{m = 0}^n \sum_{\substack{i, j, k, \ell \\ \r(i, j, k, \ell) = m}}^n \abs[\big]{\k_4(\De_{i, n}^\pm, \De_{j, n}^\pm, \De_{k, n}^\pm, \De_{\ell, n}^\pm)}. \]

    Next, we distinguish three cases.
    \begin{enumerate}[label=Case~(\arabic*), wide=12pt, leftmargin=*]
        \item
            In the first case, all the blocks are the same: $\r(i, j, k, \ell) = 0$, and thus $\V_i = \V_j = \V_k = \V_\ell$. \label{im:cumulant_case_all_in_one}
        \item
            In the second case, not all the blocks are identical.
            Three of them are relatively close to each other, while the fourth is separated from the rest: $\r(i, j, k, \ell) = \mrm{dist} (\V_i, \set{\V_j, \V_k, \V_\ell}) > 0$ and $\mrm{diam} (\set{\V_j, \V_k, \V_\ell}) \le \r(i, j, k, \ell) + 1$. \label{im:cumulant_case_1_3}
        \item
            In the third case, there are two pairs of blocks, with the blocks within each pair being close, but the pairs themselves relatively far apart: $\r(i, j, k, \ell) = \mrm{dist} (\set{\V_i, \V_k}, \set{\V_j, \V_\ell}) > 0$ and $\mrm{dist}(\V_i, \V_j) \vee \mrm{dist}(\V_k, \V_\ell) \le \r(i, j, k, \ell)$, \label{im:cumulant_case_2_2}
    \end{enumerate}
    where the term $+1$ in~\ref{im:cumulant_case_1_3} is due to the size of the blocks.
    Even though \ref{im:cumulant_case_1_3} and \ref{im:cumulant_case_2_2} are not mutually exclusive, we treat them separately for clarity.
    The remainder of the proof is devoted to showing that in all three cases, $\k_4(\De_n^\pm(s, t)) \in O(n (t - s))$, with all the details provided in the proof of Lemma~\ref{lem:cumulant_term_cases} below.

    Then, by Lemma~\ref{lem:cumulant_term_cases}, since $t - s > n^{- 1 / (1 + \eta)}$, we have in all cases that $\k_4(\De_n^\pm(s, t)) \in O(n (t - s)) \su O(n^2 (t - s)^{1 + \eta})$.
\end{proof}

%% file: proofs/functional_normal/lemmas/lem_expectation_of_increments.tex
\begin{proof}[Proof of Lemma~\ref{lem:expectation_of_increments}]
    Note that for some constant $c > 0$, we can bound
    \[ \E[\big]{\De_n^\pm(t_k, t_{k + 1})} = \int_{\S_n \ti \T^{0 \le}} \abs{\de_{t_k, t_{k + 1}}(N^\pm(p))} \d p \le c n (t_{k + 1} - t_k), \]
    as the integral factors into the product of the spatial and the temporal parts, and the spatial part is bounded by Lemma~\ref{lem:int_of_ns}~(\ref{lem:int_of_ns:power_p}) with $u_- = 0$, and the temporal part is bounded by Lemma~\ref{lem:int_of_nt_pm}~(\ref{lem:int_of_nt_pm:power_p_bound}).
    Now, setting $t_k := k n^{-1 / (1 + \eta)}$, we have $\E[\big]{\De_n^\pm(t_k, t_{k + 1})} \le c n^{1 - 1 / (1 + \eta)}$, which is independent of $k$.
    Finally, setting $\eta < 1$ gives
    \[ \max_{k \le \floor{n^{1 / (1 + \eta)}}} \big( n^{-1/2} \E[\big]{\De_n^\pm(t_k, t_{k + 1})} \big) \in O(n^{1/2 - 1 / (1 + \eta)}) \subset o(1), \]
    as desired.
\end{proof}

%% file: proofs/functional_normal/lemmas/lem_tightness_variance_term.tex
\begin{proof}[Proof of Lemma~\ref{lem:variance_of_De}]
    First, note that $D_p(\De_n^\pm(s, t))$ does not depend on the point process~$\PP$, and $D_{p'}(\De_n^\pm(s, t))$ does not depend on the point process~$\PP'$.
    We bound the variance term using the Poincar\'e inequality~\cite[Theorem~18.7]{poisBook}:
    \begin{equation}
        \Var{\De_n^\pm(s, t)} \le \int_{\S_n \ti \T^{0 \le}} \E[\big]{D_p(\De_n^\pm(s, t))^2} \mdp + \int_{\S_n \ti \R} \E[\big]{D_{p'}(\De_n^\pm(s, t))^2} \d p'.
        \label{eq:poincare}
    \end{equation}
    We consider the first term first.
    Note that
    \[ D_p(\De_n^\pm(s, t)) = D_p(S_n^\pm(t)) - D_p(S_n^\pm(s)) \]
    is the number of $\PP'$-points connecting to~$p$ in the time interval $[s, t]$ in both of the \quote{plus} and \quote{minus} cases.
    Thus, $D_p(\De_n^\pm(s, t))$ is Poisson distributed with mean
    \begin{equation}
        \E{D_p(\De_n^\pm(s, t))} = \E[\Big]{\sum_{P' \in \PP'} \big( \1{P' \in \de_{s, t}(N^\pm(p))} \big)} = \abs[\big]{\de_{s, t}(N^\pm(p))}.
        \label{eq:poincare_first_term_integrand}
    \end{equation}
    Let us recall the definition of $\de_{s, t}(N^\pm(p))$ from~\eqref{eq:def_of_de_Nt}.
    Then, we integrate the second moment of this random variable over $\S_n \ti \T^{0 \le}$:
    \begin{equation} \begin{aligned}
        \int_{\S_n \ti \T^{0 \le}} \E[\big]{D_p(\De_n^\pm(s, t))^2} \mdp &= \int_{\S_n \ti \T^{0 \le}} \E[\big]{D_p(\De_n^\pm(s, t))} + \E[\big]{D_p(\De_n^\pm(s, t))}^2 \mdp \\
        &= \int_{\S_n \ti \T^{0 \le}} \abs[\big]{\de_{s, t}(N^\pm(p))} + \abs[\big]{\de_{s, t}(N^\pm(p))}^2 \mdp.
        \label{eq:poincare_first_term}
    \end{aligned} \end{equation}
    As $\abs{\de_{s, t}(N^\pm(p))} = \abs{\Ns(\ps)} \abs{\de_{s, t}(\Nt^\pm(\pt))}$, applying Lemmas~\ref{lem:int_of_ns}~(\ref{lem:int_of_ns:power_p}) and~\ref{lem:int_of_nt_pm}~(\ref{lem:int_of_nt_pm:difference_p}) yields
    \[ \int_{\S_n \ti \T^{0 \le}} \E[\big]{D_p(\De_n^\pm(s, t))^2} \mdp \in O(n (t - s)). \]
    Note that $D_{p'}(\De_n^\pm(s, t))$ is also a Poisson distributed random variable with mean
    \begin{equation}
        \E{D_{p'}(\De_n^\pm(s, t))} = \E[\Big]{\sum_{P \in \PP \cap (\S_n \ti \T^{0 \le})} \1{p' \in \de_{s, t}(N^\pm(P))}} = \int_{\S_n \ti \T^{0 \le}} \1{p' \in \de_{s, t}(N^\pm(p))} \mdp.
        \label{eq:poincare_second_term_integrand}
    \end{equation}
    Then, following a similar calculation as in~\eqref{eq:poincare_first_term}, since $\E{D_{p'}(\De_n^\pm(s, t))^2} = \E{D_{p'}(\De_n^\pm(s, t))}^2 + \E{D_{p'}(\De_n^\pm(s, t))}$, we have
    \begin{equation}
        \int_{\S \ti \R} \E[\big]{D_{p'}(\De_n^\pm(s, t))^2} \d p' = \sum_{m = 1}^2 \int_{\S \ti \R} \bigg( \int_{\S_n \ti \T^{0 \le}} \1{p' \in \de_{s, t}(N^\pm(p))} \mdp \bigg)^m \d p'.
        \label{eq:poincare_second_term}
    \end{equation}
    Both integrals factor into a spatial and a temporal part.
    For the spatial parts,
    \[ \int_\S \bigg( \int_{\S_n} \1{\ps' \in \Ns(\ps)} \mdp \bigg)^m \d \ps' \in O(n), \]
    by Lemma~\ref{lem:int_of_ns}~(\ref{lem:int_of_ns:power_p_prime}), and for the temporal parts,
    \[ \int_\R \bigg( \int_{\T^{0 \le}} \1{r \in \de_{s, t}(\Nt^\pm(\pt))} \dpt \bigg)^m \d r \in O(t - s), \]
    by Lemma~\ref{lem:int_of_nt}~(\ref{lem:int_of_nt:power_prime}).
    Thus, $\int_{\S_n \ti \R} \E[\big]{D_{p'}(\De_n^\pm(s, t))^2} \d p' \in O(n (t - s))$, where we require again that $\g' < 1/2$.
    Then, $\Var{\De_n^\pm(s, t)}^2 \in O \big( n^2 (t - s)^{1 + \eta} \big)$, as desired.
\end{proof}

%% file: proofs/functional_normal/lemmas/lem_tightness_cumulant_term.tex
\begin{proof}[Proof of Lemma~\ref{lem:cumulant_term_cases}]
    In the following, we show the proof for the three cases of the lemma.

    \medskip
    \noindent
    \textbf{\ref{im:cumulant_case_all_in_one}.}
    We begin with \ref{im:cumulant_case_all_in_one}.
    In this case, $i = j = k = \ell$, and using the definition~\cite[Proposition~3.2.1]{wiener_chaos} of the fourth cumulant $\k_4(\De_{i, n}^\pm)$, we have that
    \begin{equation}
        \k_4(\De_{i, n}^\pm) = \sum_{M_1^{(1)}, \dots, M_q^{(1)}} c_{M^{(1)}_1, \dots, M^{(1)}_q} \prod_{b = 1}^q \E[\Big]{\big( \De_{i, n}^\pm \big)^{\abs[\big]{M_b^{(1)}}}}, \qquad \text{with} \qquad c_{M^{(1)}_1, \dots, M^{(1)}_q} := (-1)^{q - 1} (q - 1)!,
        \label{eq:cumulant_formula_for_rho_0}
    \end{equation}
    where $\set{M_1^{(1)}, \dots, M_q^{(1)}} \preceq \set{i, i, i, i}$ denotes a partition of the multiset $\set{i, i, i, i}$ and $\abs{c_{M^{(1)}_1, \dots, M^{(1)}_q}} < \ff$ are coefficients depending on the number of groups~$q$ of the partition $M^{(1)}_1, \dots, M^{(1)}_q$.
    Then, taking absolute values and applying triangle inequality, we have the following upper bound:
    \begin{equation}
        \abs{\k_4(\De_{i, n}^\pm)} \le c_2 \sum_{M_1^{(1)}, \dots, M_q^{(1)}} \prod_{b = 1}^q \E[\Big]{\big( \De_{1, n}^\pm \big)^{\abs[\big]{M_b^{(1)}}}} \le c_2 \sum_{M_1^{(1)}, \dots, M_q^{(1)}} \prod_{b = 1}^q \E[\Big]{\big( \De_{1, n}^\pm \big)^4}^{\abs[\big]{M_b^{(1)}} / 4} \le c_3 \E[\big]{\big( \De_{1, n}^\pm \big)^4},
        \label{eq:cumulant_bound_fourth_moment}
    \end{equation}
    where $c_2 := \max(\set{\abs{c_{M^{(1)}_1, \dots, M^{(1)}_q}}})$, and in the first step we used stationarity of the Poisson processes $\PP, \PP'$, and that $\De_{1, n}^\pm \ge 0$.
    In the third step, recognizing that $\abs{M_b^{(1)}} \in \set{1, \dots, 4}$, we applied Jensen's inequality, and finally, we used that $\sum_{b = 1}^q \abs{M_b^{(1)}} = 4$ with a large enough constant $c_3 := \sum_{M_1^{(1)}, \dots, M_q^{(1)}} c_2$.
    \begin{lemma}[Bound on the fourth moment of $\De_{1, n}^\pm$]
        \label{lem:cumulant_fourth_moment_bound}
        Let $\g < 1/4$.
        Then, we have that
        \[ \E[\big]{\big( \De_{1, n}^\pm \big)^4} \in O(t - s). \]
    \end{lemma}
    \noindent
    Then, we have that in \ref{im:cumulant_case_all_in_one},
    \[ \abs{\k_4(\De_n^\pm(s, t))} \le c_1 c_3 \sum_{i = 1}^n \E[\big]{\big( \De_{1, n}^\pm \big)^4} \in O(n (t - s)). \]

    \medskip
    \noindent
    \textbf{\ref{im:cumulant_case_1_3}.}
    For \ref{im:cumulant_case_1_3}, we employ~\cite[Lemma~5.1]{baryshnikov} to decompose the cumulant to a sum of products, where the terms involve covariances and moments.
    Our strategy is to bound the moments by a constant, and then bound the covariances using the common neighborhood of the points.
    Let us denote the blocks of a partition of the indices $\set{j, k, \ell}$ by $M^{(2)}_1, \dots, M^{(2)}_q$, and let $\abs{c_{M^{(2)}_1, \dots, M^{(2)}_q}} < \ff$ be coefficients depending on the partition $M^{(2)}_1, \dots, M^{(2)}_q$ of the indices $\set{j, k, \ell}$.
    Then, we have that
    \begin{equation} \begin{aligned}
        \abs{\k_4(\De_{i, n}^\pm, \De_{j, n}^\pm, \De_{k, n}^\pm, \De_{\ell, n}^\pm)} &= \abs[\bigg]{\sum_{M^{(2)}_1, \dots, M^{(2)}_q} c_{M^{(2)}_1, \dots, M^{(2)}_q} \Cov[\Big]{\De_{i, n}^\pm, \prod_{m \in M^{(2)}_1} \De_{m, n}^\pm} \prod_{b = 2}^q \E[\Big]{\prod_{m \in M^{(2)}_b} \De_{m, n}^\pm}} \\
        &\le c_4 \sum_{M^{(2)}_1, \dots, M^{(2)}_q} \abs[\Big]{\Cov[\Big]{\De_{i, n}^\pm, \prod_{m \in M^{(2)}_1} \De_{m, n}^\pm}} \prod_{b = 2}^q \E[\Big]{\prod_{m \in M^{(2)}_b} \De_{m, n}^\pm},
    \end{aligned} \label{eq:complex_cumulant_formula} \end{equation}
    where we used the triangle inequality in the second step, set $c_4 := \max(\set{\abs{c_{M^{(2)}_1, \dots, M^{(2)}_q}}})$, and dropped the absolute values on the second term as $\De_{m, n}^\pm \ge 0$.
    Note that $\abs{M^{(2)}_b} \le 2$ as $M^{(2)}_1$ contains at least one of $\set{j, k, \ell}$.
    To upper bound the product, we apply H\"older's inequality:
    \begin{equation} \begin{aligned}
        \prod_{b = 2}^q \E[\Big]{\prod_{m \in M^{(2)}_b} \De_{m, n}^\pm} &\le \prod_{b = 2}^q \prod_{m \in M^{(2)}_b} \E[\big]{\big( \De_{m, n}^\pm \big)^{\abs{M^{(2)}_b}}}^{1/\abs{M^{(2)}_b}} = \prod_{b = 2}^q \E[\big]{\big( \De_{1, n}^\pm \big)^{\abs{M^{(2)}_b}}} \\
        &\le \prod_{b = 2}^q \E[\big]{\big( \De_{1, n}^\pm \big)^2}^{\abs{M^{(2)}_b} / 2} \le \E[\big]{\big( \De_{1, n}^\pm \big)^2}^{1/2 \sum_{b = 2}^q \abs{M^{(2)}_b}},
    \end{aligned} \label{eq:cumulant_product_bound} \end{equation}
    where we used that $\set{\De_{m, n}^\pm}$ are identically distributed in the second step, and employed Jensen's inequality in the third step as $\abs{M^{(2)}_b} \in \set{1, 2}$.
    As we saw in \ref{im:cumulant_case_all_in_one}, $\E{\abs{\De_{1, n}^\pm}^2} \le \sqrt{\E{\abs{\De_{1, n}^\pm}^4}} < \ff$ by Cauchy--Schwarz inequality, thus the product in~\eqref{eq:cumulant_product_bound} is finite.
    Next, we need to bound the covariance term in~\eqref{eq:complex_cumulant_formula}.
    With this, so far we have that
    \begin{equation}
        \abs[\big]{\k_4(\De_n^\pm(s, t))} \le c_5 \sum_{a = 0}^{n - 1} \sum_{\substack{i, j, k, \ell \\ \r(i, j, k, \ell) = a}} \sum_{M^{(2)}_1, \dots, M^{(2)}_q} \abs[\Big]{\Cov[\Big]{\De_{i, n}^\pm, \prod_{m \in M^{(2)}_1} \De_{m, n}^\pm}},
        \label{eq:cumulant_intermediate_result}
    \end{equation}
    where $c_5 > 0$ is a finite constant bounding $c_4$ and the sum of the product terms.
    Next, we bound the covariance term.
    Note that $\abs{M^{(2)}_1} \in \set{1, 2, 3}$.
    As each of these terms is similar, we only consider the case $\abs{M^{(2)}_1} = 3$.
    \begin{lemma}[Bound on the covariance term]
        \label{lem:cumulant_covariance_bound}
        Let $\g < 1/3$.
        Then, we have that
        \[ \abs[\big]{\Cov{\De_{i, n}^\pm, \De_{j, n}^\pm \De_{k, n}^\pm \De_{\ell, n}^\pm}} \le \sum_{\substack{P_1 \in \PP \cap (\V_i \cap \T^{0 \le}), P_2 \in \PP \cap (\V_j \cap \T^{0 \le}) \\ P_3 \in \PP \cap (\V_k \cap \T^{0 \le}), P_4 \in \PP \cap (\V_\ell \cap \T^{0 \le})}} A(\bs{P}_4, \pmb{\s}_4), \]
        where $\bs{P}_4 := \set{P_1, P_2, P_3, P_4}$, $\pmb{\s}_4 := \set{\s_1, \s_2, \s_3, \s_4} \su \set{s, t}$, and
        \[ A(\bs{P}_4, \pmb{\s}_4) := \abs{N^\pm(P_1; \s_1) \cap N^\pm(P_2; \s_2)} (\abs{N^\pm(P_3; \s_3)} + 2) (\abs{N^\pm(P_4; \s_4)} + 2). \]
    \end{lemma}
    Then,
    we have the following bound for $\abs{\k_4(\De_n^\pm(s, t))}$:
    \[ \abs[\big]{\k_4(\De_n^\pm(s, t))} \le c_6 \E[\bigg]{\sum_{a = 0}^{n - 1} \sum_{\substack{i, j, k, \ell \\ \r(i, j, k, \ell) = a}} \sum_{\substack{P_1 \in \PP \cap (\V_i \cap \T^{0 \le}), P_2 \in \PP \cap (\V_j \cap \T^{0 \le}) \\ P_3 \in \PP \cap (\V_k \cap \T^{0 \le}), P_4 \in \PP \cap (\V_\ell \cap \T^{0 \le})}} A(\bs{P}_4, \pmb{\s}_4)} \]
    with a large enough constant $c_6 > 0$.
    We distinguish several cases depending on which of the points $P_2, P_3, P_4$ are identical, and we apply the Mecke formula to each of these cases.
    The integrals with respect to $p_3$, $p_4$ factor, and we combine the sum over the blocks and the integral over a block to a single integral.
    Using $\abs{j - k} \vee \abs{j - \ell} \le a$, we obtain the following bounds for the factors:
    \[ \int_{\S_{[j - a, j + a + 1]} \ti \T^{0 \le}} \abs{N^\pm(p_3, \s_3)}^{m_k} \mdp[3] \le c_7 a \qquad \text{and} \qquad \int_{\S_{[j - a, j + a + 1]} \ti \T^{0 \le}} \abs{N^\pm(p_4, \s_4)}^{m_\ell} \mdp[4] \le c_7 a, \]
    with some large enough constant $c_7 > 0$, where the exponents $m_k, m_\ell \in \set{0, 1, 2}$ depending on the term we are looking at and on the case $P_3 = P_4$.
    We used Lemmas~\ref{lem:int_of_ns}~(\ref{lem:int_of_ns:power_p}) with $u_- = 0$ and~\ref{lem:int_of_nt_pm}~(\ref{lem:int_of_nt_pm:power_p_bound}). 
    Let us set $m_a \in \set{0, 1, 2}$ to be the number of points in $\set{P_3, P_4}$ that appear in the term we consider, and we examine the integral with respect to $p_2$.
    Then, combining again the integral within the block $\V_j$ and the sum of the blocks with respect to $a \le \abs{i - j} \le \abs{x_1 - x_2} + 1$ into a single integral, we have bounds of the form
    \begin{equation} \begin{aligned}
        \abs[\big]{\k_4(\De_n^\pm(s, t))} &\le c_8 \sum_{i = 1}^n \int_{\V_i \ti \T^{0 \le}} \sum_{a = 0}^{n - 1} \int_{\V_j \ti \T^{0 \le}} \abs{N^\pm(p_1, \s_1)}^{m_1} \abs{N^\pm(p_1, \s_1) \cap N^\pm(p_2, \s_2)} \\
        &\hspace{4.5cm} \ti \abs{N^\pm(p_2, \s_2)}^{m_2} a^{m_a} \mdp[2] \mdp[1] \\
        &\le 2^{m_a + 1} c_8 \int_{\S_n \ti \T^{0 \le}} \int_{\S_n \ti \T^{0 \le}} \abs{N^\pm(p_1, \s_1)}^{m_1} \abs{N^\pm(p_1, \s_1) \cap N^\pm(p_2, \s_2)} \\
        &\hspace{4.5cm} \ti \abs{N^\pm(p_2, \s_2)}^{m_2} (\abs{x_1 - x_2}^{m_a} + 1) \mdp[2] \mdp[1],
    \end{aligned} \label{eq:bound_pj_factor} \end{equation}
    where the exponents $m_1 = 0$ and $m_2 \in \set{0, 1, 2}$ depend on the term we are looking at and on the number of the points in $P_3, P_4$ that are identical to $P_2$.
    Even though $m_1 = 0$ in \eqref{eq:bound_pj_factor}, this formula is referenced later in \ref{im:cumulant_case_2_2}, where $m_1$ can be $1$.
    Note also that $m_2 + m_a \le 2$.
    Then, using Lemma~\ref{lem:int_of_ns}~(\ref{lem:int_of_ns:power_cap_power_0}), we have that the spatial part is $O(n)$. 
    The maximum values of the exponents in the application of Lemma~\ref{lem:int_of_ns}~(\ref{lem:int_of_ns:power_cap_power_0}) are summarized in Table~\ref{tab:max_exponent_constraints} for each case, where we used that the expressions are symmetric in the indices $3, 4$.
    Noting that we consider only $\PP$-vertices whose lifetime intersects the interval $[0, 1]$, the temporal part is bounded by a constant using Lemma~\ref{lem:int_of_nt_pm}~(\ref{lem:int_of_nt_pm:power_cap_power}).
    Then, $\abs{\k_4(\De_n^\pm)(s, t)} \in O(n)$. 
    If $\abs{M^{(2)}_1} < 3$, the same calculation can be followed with possibly different exponents $m_j, m_k, m_\ell, m_a$.

    \medskip
    \noindent
    \textbf{\ref{im:cumulant_case_2_2}.}
    We calculate \ref{im:cumulant_case_2_2} similarly.
    Let $M^{(3)}_1, \dots, M^{(3)}_{q_1}$ be the groups of the partition of $\set{i, j}$, and $M^{(3)}_{q_1 + 1}, \dots$, $M^{(3)}_{q_1 + q_2}$ be the groups of the partition of $\set{k, \ell}$, ($q_1, q_2 \in \set{1, 2}$).
    Applying~\cite[Lemma~5.1]{baryshnikov} yields
    \[ \begin{aligned}
        &\abs{\k_4(\De_{i, n}^\pm, \De_{j, n}^\pm, \De_{k, n}^\pm, \De_{\ell, n}^\pm)} \\
        &\qquad = \abs[\bigg]{\sum_{M^{(3)}_1, \dots, M^{(3)}_{q_1 + q_2}} c_{M^{(3)}_1, \dots, M^{(3)}_{q_1 + q_2}} \Cov[\Big]{\prod_{m_1 \in M_1} \De_{m_1, n}^\pm, \prod_{m_{q + 1} \in M_{q + 1}} \De_{m_{q + 1}, n}^\pm} \E[\big]{\De_{m_2, n}^\pm} \E[\big]{\De_{m_{q + 2}, n}^\pm}}.
    \end{aligned} \]
    After the application of the triangle inequality, we bound the product of the expectations by a constant in the same way as in \ref{im:cumulant_case_1_3}, which requires no constraint on $\g$ in this case.
    Then, we only need to show a bound for the covariance term.
    As all partitions of the set $\set{i, j}$ and $\set{k, \ell}$ are considered similar, we consider only the case when $q_1 = q_2 = 2$.
    Then, the covariance term is bounded by $\abs[\big]{\Cov{\De_{i, n}^\pm \De_{j, n}^\pm, \De_{k, n}^\pm \De_{\ell, n}^\pm}} \le \abs[\big]{\E[\big]{\Cov{\De_{i, n}^\pm \De_{j, n}^\pm, \De_{k, n}^\pm \De_{\ell, n}^\pm \given \PP}}}$, where we used the independence property of the Poisson process $\PP$ to see that $\Cov{\E{\De_{i, n}^\pm \De_{j, n}^\pm \given \PP}, \E{\De_{k, n}^\pm \De_{\ell, n}^\pm \given \PP}} = 0$.
    Using bilinearity and the triangle inequality again,
    \[ \begin{aligned}
        &\abs[\big]{\E[\big]{\Cov{\De_{i, n}^\pm \De_{j, n}^\pm, \De_{k, n}^\pm \De_{\ell, n}^\pm \given \PP}}} \\
        &\qquad \le \sum_{\substack{\s_1, \s_2, \s_3, \s_4 \in \set{s, t} \\ P_1 \in \PP \cap (\V_i \ti \T^{0 \le}), P_3 \in \PP \cap (\V_j \ti \T^{0 \le}) \\ P_2 \in \PP \cap (\V_k \ti \T^{0 \le}), P_4 \in \PP \cap (\V_\ell \ti \T^{0 \le})}} \abs[\Big]{\Cov[\Big]{\deg^\pm(P_1; \s_1) \deg^\pm(P_3; \s_3), \deg^\pm(P_2; \s_2) \deg^\pm(P_4; \s_4) \Biggiven \PP}}.
    \end{aligned} \]
    Note the indices of the points $P_1, P_2, P_3, P_4$ in the covariance term.
    We set them so that $P_1, P_2$ are the points corresponding to the indices $i, k$, respectively.
    This is because we would like to use the common neighborhood of the points $P_1, P_2$ to bound the covariance term, so that we can reuse the calculations from \ref{im:cumulant_case_1_3}.
    Expanding one of the individual covariance terms as in \ref{im:cumulant_case_1_3}, we have that
    \[ \begin{aligned}
        &\abs[\Big]{\Cov[\big]{\deg^\pm(P_1; \s_1) \deg^\pm(P_2; \s_2), \deg^\pm(P_3; \s_3) \deg^\pm(P_4; \s_4) \biggiven \PP}} \\
        &\qquad = \mbb{E} \Big[ \sum_{\substack{\set{P_1', P_2', P_3', P_4'} \in \PP'^4\\\set{P_1', P_3'} \cap \set{P_2', P_4'} \ne \es}} \1{P_1' \in N^\pm(P_1; \s_1)} \1{P_3' \in N^\pm(P_3; \s_3)} \\
        &\hspace{4.3cm} \ti \1{P_2' \in N^\pm(P_2; \s_2)} \1{P_4' \in N^\pm(P_4; \s_4)} \Biggiven \PP \Big] \\
        &\qquad = \E[\Big]{\sum_{\set{P_1', P_3', P_4'} \in \PP'^3} \1{P_1' \in N^\pm(P_1; \s_1) \cap N^\pm(P_2; \s_2)} \1{P_3' \in N^\pm(P_3; \s_3)} \1{P_4' \in N^\pm(P_4; \s_4)} \Biggiven \PP},
    \end{aligned} \]
    where in the last step we used that the formula above is symmetric in the indices $i, j$ and $k, \ell$, thus we assume that $P_1' = P_2'$.
    Apart from change of indices, this expression is identical to the one in~\eqref{eq:covterms} in \ref{im:cumulant_case_1_3}.
    Although the application of the Mecke formula for the sum over the points $P_1, P_2, P_3, P_4$ results in cases $P_1 = P_3 \ne P_2 = P_4$, $P_1 = P_3 \ne P_2 \ne P_4$, $P_1 \ne P_3 \ne P_2 = P_4$, $P_1 \ne P_3 \ne P_2 \ne P_4$, all of these terms can be upper bounded with the same calculations as in \ref{im:cumulant_case_1_3}.
    The only difference is the possible values of the exponents $m_1, m_2, m_3, m_4, m_a$.
    In this case, $m_1 = 1$ whenever $P_1 = P_3$ and $P_1' \ne P_3'$.

    We summarized the maximum values of the exponents in the application of Lemma~\ref{lem:int_of_ns}~(\ref{lem:int_of_ns:power_cap_power_0}) in Table~\ref{tab:max_exponent_constraints}.
\end{proof}
    \begin{table} [h] \centering \caption{
            Maximum values of the exponents.
            We put the identical points to the same set for each case.
        } \label{tab:max_exponent_constraints}
        \begin{tabular}{|lccccc|lccccc|} \hline
            \multicolumn{6}{|c|}{\ref{im:cumulant_case_1_3}} & \multicolumn{6}{|c|}{\ref{im:cumulant_case_2_2}} \\
            identical $\PP$-points & $m_1$ & $m_2$ & $m_3$ & $m_4$ & $m_a$ & identical $\PP$-points & $m_1$ & $m_2$ & $m_3$ & $m_4$ & $m_a$ \\ \hline
            $\set{P_1}, \set{P_2, P_3, P_4}$              &  0 & 2 & -- & -- & 0 & $\set{P_1, P_3},       \set{P_2, P_4}$        & 1 & 1 & -- & -- & 0 \\
            $\set{P_1}, \set{P_2}, \set{P_3, P_4}$        &  0 & 0 &  2 & -- & 1 & $\set{P_1, P_3},       \set{P_2}, \set{P_4}$  & 1 & 0 & -- &  1 & 1 \\
            $\set{P_1}, \set{P_2, P_3}, \set{P_4}$        &  0 & 1 & -- &  1 & 1 & $\set{P_1}, \set{P_3}, \set{P_2, P_4}$        & 0 & 1 &  1 & -- & 1 \\
            $\set{P_1}, \set{P_2}, \set{P_3}, \set{P_4}$  &  0 & 0 &  1 &  1 & 2 & $\set{P_1}, \set{P_3}, \set{P_2}, \set{P_4}$  & 0 & 0 &  1 &  1 & 2 \\ \hline
        \end{tabular}
    \end{table}

%% file: proofs/functional_normal/lemmas/lem_cumulant_fourth_moment_bound.tex
\begin{proof}[Proof of Lemma~\ref{lem:cumulant_fourth_moment_bound}]
    Let us introduce the notation
    \[ \de_{s, t}^\pm(P) := \deg^\pm(P; t) - \deg^\pm(P; s) = \sum_{P' \in \PP'} \1{P' \in N^\pm(P; t) \sm N^\pm(P; s)}. \]
    Then,
    \[ \De_{1, n}^\pm = V_1^\pm(t) - V_1^\pm(s) = \sum_{P \in \PP \cap (\V_1 \ti \T^{0 \le})} \de_{s, t}^\pm (P), \]
    and we expand the fourth moment $\E{(\De_{1, n}^\pm)^4}$ as follows:
    \begin{equation} \begin{aligned}
        &\E[\big]{\big( \De_{1, n}^\pm \big)^4} = \E[\Big]{\Big( \sum_{P \in \PP \cap (\V_1 \ti \T^{0 \le})} \de_{s, t}^\pm(P) \Big)^4} = c_1 \E[\Big]{\sum_{\bs{P}_4 \in \PP_{\ne}^4 \cap (\V_1 \ti \T^{0 \le})^4} \prod_{i \le 4} \de_{s, t}^\pm(P_1)} \\
        &\quad + c_2 \E[\Big]{\sum_{\bs{P}_3 \in \PP_{\ne}^3 \cap (\V_1 \ti \T^{0 \le})^3} \de_{s, t}^\pm(P_1)^2 \de_{s, t}^\pm(P_2) \de_{s, t}^\pm(P_3)} + c_3 \E[\Big]{\sum_{\bs{P}_2 \in \PP_{\ne}^2 \cap (\V_1 \ti \T^{0 \le})^2} \de_{s, t}^\pm(P_1)^2 \de_{s, t}^\pm(P_2)^2} \\
        &\quad + c_4 \E[\Big]{\sum_{\bs{P}_2 \in \PP_{\ne}^2 \cap (\V_1 \ti \T^{0 \le})^2} \de_{s, t}^\pm(P_1)^3 \de_{s, t}^\pm(P_2)} + c_5 \E[\Big]{\sum_{P \in \PP \cap \V_1 \ti \T^{0 \le}} \de_{s, t}^\pm(P)^4},
    \end{aligned} \label{eq:fourth_moment_details_1} \end{equation}
    where $\bs{P}_m := (P_1, \dots, P_m)$, $\set{c_i \co i \in \set{4, \dots, 8}} \su (0, \ff)$ are real constants, and we dropped the arguments $t$ for brevity.
    The application of the Mecke formula to the above expression yields
    \begin{equation} \begin{aligned}
        \E[\big]{\big( \De_{1, n}^\pm \big)^4} &= c_1 \int_{(\V_1 \ti \T^{0 \le})^4} \E[\Big]{\prod_{i \le 4} \de_{s, t}^\pm(p_i)} \d \pp_4 + c_2 \int_{(\V_1 \ti \T^{0 \le})^3} \E[\big]{\de_{s, t}^\pm(p_1)^2 \de_{s, t}^\pm(p_2) \de_{s, t}^\pm(p_3)} \d \pp_3 \\
        &\phantom{=} + c_3 \int_{(\V_1 \ti \T^{0 \le})^2} \E[\big]{\de_{s, t}^\pm(p_1)^2 \de_{s, t}^\pm(p_2)^2} \d \pp_2 + c_4 \int_{(\V_1 \ti \T^{0 \le})^2} \E[\big]{\de_{s, t}^\pm(p_1)^3 \de_{s, t}^\pm(p_2)} \d \pp_2 \\
        &\phantom{=} + c_5 \int_{\V_1 \ti \T^{0 \le}} \E[\big]{\de_{s, t}^\pm(p)^4} \d p,
    \end{aligned} \label{eq:fourth_moment_details_2} \end{equation}
    where $\pp_m := (p_1, \dots, p_m)$.
    Each of the expectations in the integrands above can be written in the form $\E[\big]{\prod_{q = 1}^Q \de_{s, t}^\pm(p_q)^{m_q}}$, where $Q \in \set{1, \dots, 4}$ is the number of terms in the products and $m_q \in \set{1, \dots, 4}$ denote the exponents where $\sum_{q = 1}^Q m_q = 4$.
    Then, we bound the terms using H\"older's inequality as follows:
    \begin{equation} \E[\Big]{\prod_{q = 1}^Q \de_{s, t}^\pm(p_q)^{m_q}} \le \prod_{q = 1}^Q \E[\big]{\de_{s, t}^\pm(p_q)^4}^{m_q / 4}. \label{eq:cumulant_holder} \end{equation}
    Recalling the definition of $\de_{s, t}^\pm(p_q)$, to apply Mecke's formula to its fourth moment
    \[ \E[\big]{\de_{s, t}^\pm(p_q)^4} = \E[\bigg]{\bigg( \sum_{P' \in \PP'} \1{P' \in N^\pm(p_q; t) \sm N^\pm(p_q; s)} \bigg)^4}, \]
    we need to examine which of the four points in $P' \in \PP'$ are equal.
    To do so, we generate all the partitions of the four points $P_1', \dots, P_4'$ consisting of $Q'$ groups, so that the points in the same group are considered to be equal.
    Then, apart from combinatorial symmetries, we arrive again to five different partition types of the four points that we also used to distinguish the points in $\PP$ above.
    In each case, Mecke's formula leads to $\abs{\de_{s, t}(N^\pm(p_q; t))}^{Q'}$, where $Q' \in \set{1, \dots, 4}$ denotes the number of groups in the partition, i.e., the number of distinct points in $\set{P_1', \dots, P_4'}$.
    Then,
    \begin{equation} \E[\big]{\de_{s, t}^\pm(p_q)^4} \le c_6 \Big( \abs{\de_{s, t}(N^\pm(p_q))} + \abs{\de_{s, t}(N^\pm(p_q))}^2 + \abs{\de_{s, t}(N^\pm(p_q))}^3 + \abs{\de_{s, t}(N^\pm(p_q))}^4 \Big) \label{eq:fourth_moment_details_3} \end{equation}
    with some large enough constant $c_6 > 0$.
    Substituting the above expressions to the formula \eqref{eq:cumulant_holder}, the integrals in~\eqref{eq:fourth_moment_details_2} can be bounded by
    \[ \int_{\V_1 \ti \T^{0 \le}} \abs{\de_{s, t}(N^\pm(p))}^m \mdp = \int_{\V_1} \abs{\Ns(\ps)}^m \d \ps \int_{\T^{0 \le}} \abs{\de_{s, t}(\Nt^\pm(\pt))}^m \dpt \in O((t - s)^m), \]
    where $m \in \set{1, 2, 3, 4}$, and we applied Lemmas~\ref{lem:int_of_ns}~(\ref{lem:int_of_ns:power_p}) and~\ref{lem:int_of_nt_pm}~(\ref{lem:int_of_nt_pm:difference_p}). 
\end{proof}

%% file: proofs/functional_normal/lemmas/lem_cumulant_covariance_bound.tex
\begin{proof}[Proof of Lemma~\ref{lem:cumulant_covariance_bound}]
    First, note that
    \begin{equation} \begin{aligned}
        \abs[\big]{\Cov{\De_{i, n}^\pm, \De_{j, n}^\pm \De_{k, n}^\pm \De_{\ell, n}^\pm}} &\le \abs[\big]{\Cov[\big]{\E{\De_{i, n}^\pm \given \PP}, \E{\De_{j, n}^\pm \De_{k, n}^\pm \De_{\ell, n}^\pm \given \PP}}} + \abs[\big]{\E[\big]{\Cov[\big]{\De_{i, n}^\pm, \De_{j, n}^\pm \De_{k, n}^\pm \De_{\ell, n}^\pm \biggiven \PP}}} \\
        &= \abs[\big]{\E[\big]{\Cov[\big]{\De_{i, n}^\pm, \De_{j, n}^\pm \De_{k, n}^\pm \De_{\ell, n}^\pm \biggiven \PP}}},
    \end{aligned} \label{eq:conditional_covariance} \end{equation}
    where in the last step we used that the conditional expectations $\E{\De_{i, n}^\pm \given \PP}$ and $\E{\De_{j, n}^\pm \De_{k, n}^\pm \De_{\ell, n}^\pm \given \PP}$ are independent by the independence property of the Poisson process $\PP$ since $\V_i \notin \set{\V_j, \V_k, \V_\ell}$.
    Recalling the definition of $\De_{i, n}^\pm, \De_{j, n}^\pm, \De_{k, n}^\pm, \De_{\ell, n}^\pm$, let us upper bound the covariance using bilinearity:
    \begin{equation} \begin{aligned}
        &\abs[\big]{\Cov{\De_{i, n}^\pm, \De_{j, n}^\pm \De_{k, n}^\pm \De_{\ell, n}^\pm \given \PP}} \le \abs[\bigg]{\sum_{\s_1, \s_2, \s_3, \s_4 \in \set{s, t}} \Cov{V_i^\pm(\s_1), V_j^\pm(\s_2) V_k^\pm(\s_3) V_\ell^\pm(\s_4) \given \PP}} \\
        &\le \sum_{\substack{\s_1, \s_2, \s_3, \s_4 \in \set{s, t} \\ P_1 \in \PP \cap (\V_i \cap \T^{0 \le}), P_2 \in \PP \cap (\V_j \cap \T^{0 \le}) \\ P_3 \in \PP \cap (\V_k \cap \T^{0 \le}), P_4 \in \PP \cap (\V_\ell \cap \T^{0 \le})}} \abs[\Big]{\Cov[\big]{\deg^\pm(P_1; \s_1), \deg^\pm(P_2; \s_2) \deg^\pm(P_3; \s_3) \deg^\pm(P_4; \s_4) \biggiven \PP}}.
    \end{aligned} \label{eq:bilinearity_bound} \end{equation}
    Expanding one of the covariance terms in the sum, we recognize that it is non-zero only if $P_1' \in \set{P_2', P_3', P_4'}$:
    \begin{equation} \begin{aligned}
        &\abs[\Big]{\Cov[\big]{\deg^\pm(P_1; \s_1), \deg^\pm(P_2; \s_2) \deg^\pm(P_3; \s_3) \deg^\pm(P_4; \s_4) \biggiven \PP}} \\
        &= \E[\Big]{\sum_{\substack{\set{P_1', P_2', P_3', P_4'} \in \PP'^4 \\ P_1' \in \set{P_2', P_3', P_4'}}} \!\!\! \1{P_1' \in N^\pm(P_1; \s_1)} \1{P_2' \in N^\pm(P_2; \s_2)} \1{P_3' \in N^\pm(P_3; \s_3)} \1{P_4' \in N^\pm(P_4; \s_4)} \Biggiven \PP} \\
        &= \E[\Big]{\sum_{\set{P_1', P_3', P_4'} \in \PP'^3} \1{P_1' \in N^\pm(P_1; \s_1) \cap N^\pm(P_2; \s_2)} \1{P_3' \in N^\pm(P_3; \s_3)} \1{P_4' \in N^\pm(P_4; \s_4)} \Biggiven \PP},
    \end{aligned} \label{eq:covterms} \end{equation}
    where in the last step we assumed without loss of generality that $P_1' = P_2'$, as the formula above is symmetric in the indices $j, k, \ell$.
    We need to distinguish several cases depending on which of the points $P_1', P_3', P_4'$ are identical: $P_1' \ne P_3' \ne P_4'$, $P_1' = P_3' \ne P_4'$, $P_1' = P_4' \ne P_3'$, $P_1' \ne P_3' = P_4'$ and $P_1' = P_3' = P_4'$.
    Decomposing the covariance accordingly, we have that
    \begin{equation} \begin{aligned}
        &\abs[\Big]{\Cov[\big]{\deg^\pm(P_1; \s_1), \deg^\pm(P_2; \s_2) \deg^\pm(P_3; \s_3) \deg^\pm(P_4; \s_4) \biggiven \PP}} \\
        &\quad = c_1 \prod_{i \in \set{3, 4}} \E{\deg^\pm(P_i; \s_i) \given \PP} \E[\bigg]{\sum_{P' \in \PP} \1[\Big]{P' \in \bigcap_{i \in \set{1, 2}} N^\pm(P_i; \s_i)} \bigggiven \PP} \\
        &\quad \phantom{=} + c_2 \E{\deg^\pm(P_4; \s_4) \given \PP} \E[\bigg]{\sum_{P' \in \PP'} \1[\Big]{P' \in \bigcap_{i \in \set{1, 2, 3}} N^\pm(P_i; \s_i)} \bigggiven \PP} \\
        &\quad \phantom{=} + c_3 \E{\deg^\pm(P_3; \s_3) \given \PP} \E[\bigg]{\sum_{P' \in \PP'} \1[\Big]{P' \in \bigcap_{i = \set{1, 2, 4}} N^\pm(P_i; \s_i)} \bigggiven \PP} \\
        &\quad \phantom{=} + c_4 \E[\bigg]{\sum_{P' \in \PP'} \1[\Big]{P' \in \bigcap_{i \in \set{1, 2}} N^\pm(P_i; \s_i)} \bigggiven \PP} \E[\bigg]{\sum_{P' \in \PP'} \1[\Big]{P' \in \bigcap_{i \in \set{3, 4}} N^\pm(P_i; \s_i)} \bigggiven \PP} \\
        &\quad \phantom{=} + c_5 \E[\bigg]{\sum_{P' \in \PP'} \1[\Big]{P' \in \bigcap_{i \in \set{1, 2, 3, 4}} \!\!\! N^\pm(P_i; \s_i)} \bigggiven \PP},
    \end{aligned} \label{eq:covterms_explicit} \end{equation}
    where $c_1, \dots, c_5 > 0$ are some coefficients.
    Next, we apply the Mecke formula to the above expression:
    \begin{equation} \begin{aligned}
        &\abs[\Big]{\Cov[\big]{\deg^\pm(P_1; \s_1), \deg^\pm(P_2; \s_2) \deg^\pm(P_3; \s_3) \deg^\pm(P_4; \s_4) \biggiven \PP}} \\
        &\qquad = \abs{N^\pm(P_1; \s_1) \cap N^\pm(P_2; \s_2)} \abs{N^\pm(P_3; \s_3)} \abs{N^\pm(P_4; \s_4)} \\
        &\qquad \phantom{=} + \abs{N^\pm(P_1; \s_1) \cap N^\pm(P_2; \s_2)} \abs{N^\pm(P_3; \s_3) \cap N^\pm(P_4; \s_4)} \\
        &\qquad \phantom{=} + \abs{N^\pm(P_1; \s_1) \cap N^\pm(P_2; \s_2) \cap N^\pm(P_3; \s_3)} \abs{N^\pm(P_4; \s_4)} \\
        &\qquad \phantom{=} + \abs{N^\pm(P_1; \s_1) \cap N^\pm(P_2; \s_2) \cap N^\pm(P_4; \s_4)} \abs{N^\pm(P_3; \s_3)} \\
        &\qquad \phantom{=} + \abs{N^\pm(P_1; \s_1) \cap N^\pm(P_2; \s_2) \cap N^\pm(P_3; \s_3) \cap N^\pm(P_4; \s_4)} \\
        &\qquad \le \abs{N^\pm(P_1; \s_1) \cap N^\pm(P_2; \s_2)} (\abs{N^\pm(P_3; \s_3)} + 2) (\abs{N^\pm(P_4; \s_4)} + 2) =: A(\bs{P}_4, \pmb{\s}_4),
    \end{aligned} \label{eq:covterms_Mecke} \end{equation}
    where we bounded by neglecting some intersecting sets, and set $\bs{P}_4 := (P_1, P_2, P_3, P_4)$, $\pmb{\s}_4 := (\s_1, \s_2, \s_3, \s_4)$.
\end{proof}

%% file: proofs/functional_stable/propositions/prop_high_mark_edge_count.tex
\begin{proof}[Proof of Proposition~\ref{prop:high_mark_edge_count}]
    We apply again~\cite[Theorem~2]{davydov1996weak} to the \quote{plus-minus decomposition} of $S_n^\ge(t)$, and verify the three conditions of the theorem.
    We first define
    \[ S_n^{\ge, +}(t) := \sum_{P \in \PP \cap (\S^{u_n \le}_n \ti \T^{0 \le})} \deg^+(P; t) \qquad \text{and} \qquad S_n^{\ge, -}(t) := \sum_{P \in \PP \cap (\S^{u_n \le}_n \ti \T^{0 \le})} \deg^-(P; t), \]
    and we use the notation $S_n^{\ge, \pm}(t)$ whenever we refer to both $S_n^{\ge, +}(t)$ and $S_n^{\ge, -}(t)$ together.

    For Condition~(\ref{condition:finite_dimensional}), we would like to show that for all $t_1, \dots, t_m \in [0, 1]$ and $\eps_2 > 0$,
    \[ \lim_{n \tff} \P[\Big]{\norm[\big]{\big( \So_n^{\ge, \pm}(t_1), \dots, \So_n^{\ge, \pm}(t_m) \big)}_\ff > \eps_2} = \lim_{n \tff} \P[\Big]{\max_{i \in \set{1, \dots, m}} \big( \So_n^{\ge, \pm}(t_i) \big) > \eps_2} = 0. \]
    We bound the probability of the maximum by the sum of the probabilities of the individual events, and then apply Chebyshev's inequality:
    \[ \P[\Big]{\max_{1 \le i \le m} \big( \So_n^{\ge, \pm}(t_i) \big) > \eps_2} \le \sum_{i = 1}^m \P[\Big]{\So_n^{\ge, \pm}(t_i) > \eps_2} \le \f{m n^{-2 \g}}{\eps_2^2} \Var[\big]{S_n^{\ge, \pm}(t)}. \]
    Thus, we need to show the following lemma.
    \begin{lemma}[Variance of the high-mark edge count $S_n^{\ge, \pm}(t)$]\label{lem:variance_high_mark_edge_count}
        If $\g > 1/2$ and $\g' < 1/2$, then
        \[ \Var[\big]{S_n^{\ge, \pm}(t)} \in o(n^{2 \g}). \]
    \end{lemma}
    \noindent
    Then, the finite-dimensional distributions of $\So_n^\ge$ converge to~$0$.

    To verify Condition~(\ref{condition:cumulant}), we follow closely the arguments in the proof of Theorem~\ref{thm:functional_normal} for $\So_n^{\ge, \pm}$ in place of $\So_n^\pm$.
    This time, we set $\chi_1 := 4$, $\chi_2 := 1 + \eta$ with $\eta = 1/3$, and $a_n := n^{-1/2}$ also differs from the thin-tailed case.
    We define $\De_n^{\ge, \pm}(s, t) := S_n^{\ge, \pm}(t) - S_n^{\ge, \pm}(s)$, and we would like to show that if $t - s > n^{-1/2}$, then
    \[ \E[\big]{(\De_n^{\ge, \pm} (s, t) - \E{\De_n^{\ge, \pm} (s, t)})^4} = \k_4 \big( \De_n^{\ge, \pm} (s, t) \big) + 3 \Var[\big]{\De_n^{\ge, \pm} (s, t)}^2 \in O(n^{4 \g} (t - s)^{1 + \eta}). \]
    We state the following lemma to bound the variance term.
    \begin{lemma}[Variance of the change of high-mark edge count $\De_n^{\ge, \pm} (s, t)$]\label{lem:variance_change_of_high_mark_edge_count}
        If $\g > 1/2$ and $\g' < 1/2$, then
        \[ \Var[\big]{\De_n^{\ge, \pm} (s, t)} \in O(n^{2 \g} (t - s)^{(1 + \eta) / 2}). \]
    \end{lemma}
    \noindent
    Then, $\Var{\De_n^{\ge, \pm} (s, t)}^2 \in O(n^{4 \g} (t - s)^{1 + \eta})$.
    Turning our attention to the fourth cumulant, we would like to show that $\k_4(\De_n^{\ge, \pm}(s, t)) \in O(n^{4 \g} (t - s)^{1 + \eta})$ whenever $t - s > n^{-1/2}$.
    We begin by partitioning the domain $\S_n^{u_n \le}$ to~$n$ disjoint blocks $\V_i^{u_n \le} := \S^{u_n \le}_{[i - 1, i]}$ $(i \in \set{1, \dots, n})$, and we introduce
    \begin{alignat*}{2}
        V_i^{\ge, \pm}(t) &:= \sum_{P \in \PP \cap (\V_i^{u_n \le} \ti \T^{0 \le})} \deg^\pm(P; t) & \qquad S_n^{\ge, \pm}(t) &= \sum_{i = 1}^n V_i^{\ge, \pm}(t) \\
        \De_{i, n}^{\ge, \pm}(s, t) &:= V_i^{\ge, \pm}(t) - V_i^{\ge, \pm}(s) & \qquad \De_n^{\ge, \pm}(s, t) &= \sum_{i = 1}^n \De_{i, n}^{\ge, \pm}(s, t),
    \end{alignat*}
    which are again monotone functions of~$t$.
    The multilinearity of the cumulant will lead to
    \[ \abs[\big]{\k_4(\De_n^{\ge, \pm})} \le c_1 \sum_{i, j, k, \ell = 1}^n \abs[\big]{\k_4 \big( \De_{i, n}^{\ge, \pm}, \De_{j, n}^{\ge, \pm}, \De_{k, n}^{\ge, \pm}, \De_{\ell, n}^{\ge, \pm} \big)}, \]
    with some constant $c_1 > 0$, and we dropped the arguments $(s, t)$ for brevity.
    We partition the indices again and define $\r^\ge(i, j, k, \ell)$ as in~\eqref{eq:rho_definition} with the blocks $\set{\V_i^{u_n \le}, \V_j^{u_n \le}, \V_k^{u_n \le}, \V_\ell^{u_n \le}}$ in place of $\set{\V_i, \V_j, \V_k, \V_\ell}$.
    Then, we distinguish three cases based on the distance of the blocks $\V_i^{u_n \le}, \V_j^{u_n \le}, \V_k^{u_n \le}, \V_\ell^{u_n \le}$ the same way as in the proof of Theorem~\ref{thm:functional_normal}, which we treat separately.
    \begin{enumerate}[label=Case~(\arabic*$^\ge$), wide=12pt, leftmargin=*]
        \item $\r^\ge(i, j, k, \ell) = 0$, i.e., $i = j = k = \ell$; \label{im:cumulant_case_all_in_one_ge}
        \item $\r^\ge(i, j, k, \ell) = \mrm{dist} (\V_i^{u_n \le}, \set{\V_j^{u_n \le}, \V_k^{u_n \le}, \V_\ell^{u_n \le}}) > 0$, \\
        and $\mrm{diam} (\set{\V_j^{u_n \le}, \V_k^{u_n \le}, \V_\ell^{u_n \le}}) \le \r^\ge(i, j, k, \ell) + 1$; \label{im:cumulant_case_1_3_ge}
        \item $\r^\ge(i, j, k, \ell) = \mrm{dist} (\set{\V_i^{u_n \le}, \V_k^{u_n \le}}, \set{\V_j^{u_n \le}, \V_\ell^{u_n \le}}) > 0$, \\
        and $\mrm{dist}(\V_i^{u_n \le}, \V_j^{u_n \le}) \vee \mrm{dist}(\V_k^{u_n \le}, \V_\ell^{u_n \le}) \le \r^\ge(i, j, k, \ell)$. \label{im:cumulant_case_2_2_ge}
    \end{enumerate}
    The following lemma summarizes the orders of the cumulant term $\abs{\k_4(\De_n^{\ge, \pm}(s, t))}$.
    \begin{lemma}[Order of the cumulant term $\abs{\k_4(\De_n^{\ge, \pm}(s, t))}$]\label{lem:cumulant_term_cases_ge}
        Let $\g > 1/2$ and $\g' < 1/4$.
        We have for $0 \le s < t \le 1$ the orders of the cumulant term $\abs{\k_4(\De_n^{\ge, \pm}(s, t))} \in O(n^{4 \g} (t - s)^{1 + \eta})$ in the three cases \ref{im:cumulant_case_all_in_one_ge}, \ref{im:cumulant_case_1_3_ge}, and \ref{im:cumulant_case_2_2_ge}.
    \end{lemma}
    \noindent
    Thus, considering all the three cases $\abs[\big]{\k_4(\De_n^{\ge, \pm}(s, t))} \in O(n^{4 \g} (t - s)^{1 + \eta})$ if $t - s > n^{-1/2}$.

    Finally, Condition~(\ref{condition:third}) is fulfilled if the following lemma holds.
    \begin{lemma}[Bound on the expectation of the increments of $\So_n^\pm(t)$]\label{lem:expectation_of_increments_ge}
        Let $t_k := k n^{1 / (1 + \eta)}$, and let $\g > 1/2$ and $\g' \in (0, 1)$.
        Then, for all $\eps > 0$,
        \[ \lim_{n \tff} \max_{k \le \floor{n^{1/2}}} \big\{ n^{-\g} \E[\big]{\De_n^{\ge, \pm}(t_k, t_{k + 1})} \big\} = 0. \]
    \end{lemma}

    We conclude that if $n \to \ff$, we can approximate $\So_n(t)$ by the edge count of the low-mark vertices $\So_n^{(1)}(t)$ for all $t \in [0, 1]$.
\end{proof}

%% file: proofs/functional_stable/propositions/prop_low_mark_edge_count.tex
\begin{proof}[Proof of Proposition~\ref{prop:low_mark_edge_count}]
    Note that $\norm{\So_n^{(1)} - \So_n^{(2)}} = n^{-\g} \norm{S_n^{(1)} - S_n^{(2)}}$.
    We define and bound the error term $E$ as follows:
    \[ \begin{aligned}
        E := \P[\Big]{n^{-\g} \norm{S_n^{(1)} - S_n^{(2)}} \ge \eps_4} &\le \P[\Big]{\sup_{t \in [0, 1]} \sum_{P \in \PP \cap \S^{\le u_n}_n \ti \T} \abs[\Big]{\deg(P; t) - \E[\big]{\deg(P; t) \biggiven \PP}} \ge \eps_4 n^\g} \\
        &\le \eps_4^{-1} n^{-\g} \E[\Big]{\sup_{t \in [0, 1]} \sum_{P \in \PP \cap \S^{\le u_n}_n \ti \T} \abs[\Big]{\deg(P; t) - \E[\big]{\deg(P; t) \biggiven \PP}}},
    \end{aligned} \]
    where we used Markov's inequality.
    We now upper bound the supremum of the sum by the sum of the suprema, and apply the Mecke formula to obtain:
    \[ \begin{aligned}
        E &\le \eps_4^{-1} n^{-\g} \int_{\S^{\le u_n}_n \ti \T} \E[\Big]{\sup_{t \in [0, 1]} \abs[\Big]{\deg(p; t) - \E[\big]{\deg(p; t)}}} \mdp \\
        &\le \eps_4^{-1} n^{-\g} \int_{\S^{\le u_n}_n \ti \T} \E[\Big]{\Big( \sup_{t \in [0, 1]} \abs[\Big]{\deg(p; t) - \E[\big]{\deg(p; t)}} \Big)^2}^{1/2} \mdp,
    \end{aligned} \]
    where we applied Jensen's inequality to the expectation inside the integral in the last step.
    Note that if $b > 1$ or $b + \ell < 0$, then $\deg(p; t) = 0$, and thus the supremum is zero.
    Therefore, we can assume that $b \le 1$ and $b + \ell \ge 0$, and we restrict the supremum to the interval $[0 \vee b, 1 \w (b + \ell)]$:
    \[ E \le \eps_4^{-1} n^{-\g} \int_{\S^{\le u_n}_n \ti \T_{\le 1}^{0 \le}} \E[\Big]{\Big( \sup_{t \in [0 \vee b, 1 \w (b + \ell)]} \abs[\Big]{\deg(p; t) - \E[\big]{\deg(p; t)}} \Big)^2}^{1/2} \mdp. \]
    Since the process $\deg(p; t) - \E{\deg(p; t)}$ is a martingale, the application of Doob's inequality~\cite[Theorem~1.3.8 (iv)]{karatzas} yields
    \[ E \le 2 \eps_4^{-1} n^{-\g} \int_{\S^{\le u_n}_n \ti \T_{\le 1}^{0 \le}} \E[\Big]{\Big( \deg(p; 1 \w (b + \ell)) - \E[\big]{\deg(p; 1 \w (b + \ell))} \Big)^2}^{1/2} \mdp. \]
    As $\deg(p; 1 \w (b + \ell))$ is Poisson distributed with expectation $\E{\deg(p; 1 \w (b + \ell)) \given p}$, we have
    \[ E \le 2 \eps_4^{-1} n^{-\g} \int_{\S^{\le u_n}_n \ti \T_{\le 1}^{0 \le}} \Big( \E[\big]{\deg(p; 1 \w (b + \ell))} \Big)^{1/2} \mdp. \]
    Noting that $\E{\deg(p; 1 \w (b + \ell)) \given p} = \abs{N(p; 1 \w (b + \ell))}$, we calculate the integral:
    \[ E \le 2 \eps_4^{-1} n^{-\g} \int_{\S^{\le u_n}_n} \abs{\Ns(\ps)}^{1/2} \d (x, u) \int_{\T_{\le 1}^{0 \le}} \abs{\Nt(\pt; 1 \w (b + \ell))}^{1/2} \dpt. \]
    We apply Lemma~\ref{lem:int_of_ns}~(\ref{lem:int_of_ns:power_p}) to the spatial integral:
    \[ \int_{\S^{\le u_n}_n} \abs{\Ns(\ps)}^{1/2} \d \ps = \Big( \f{2 \b}{1 - \g'} \Big)^{1/2} \f{2}{1 - \g/2} n u_n^{1 - \g/2} \in O(n^{(1 + \g) / 3}), \]
    where in the last step we do not need to impose any constraints on~$\g$, and we used that $u_n = n^{-2/3}$.
    The temporal integral can be calculated as follows:
    \[ \begin{aligned}
        \int_{\T_{\le 1}^{0 \le}} \abs{\Nt(\pt; 1 \w (b + \ell))}^{1/2} &\dpt = \int_{\T_{\le 1}^{0 \le}} \big( (1 \w (b + \ell)) - b \big)^{1/2} \1{b \le 1 \w (b + \ell) \le b + \ell} \dbdell \\
        &= \int_{\T_{\le 1}^{0 \le}} ((1 - b) \w \ell)^{1/2} \dbdell = \int_0^\ff \int_{\ell}^{1 - \ell} \ell^{1/2} \d b + \int_{1 - \ell}^1 (1 - b)^{1/2} \d b \dell \\
        &\le \int_0^\ff \int_{\ell}^1 \ell^{1/2} \d b \dell = \int_0^\ff (1 + \ell) \ell^{1/2} \dell = \Ga(3/2) + \Ga(5/2),
    \end{aligned} \]
    where we used that the indicator does not affect the integral in the second step, and split the integration domain with respect to $b$ to $[\ell, 1 - \ell]$ and $b \in [1 - \ell, 1]$ in the third step.
    Combining the spatial and temporal integrals, we obtain
    \[ E \le \f{4}{\eps_4 (1 - \g/2)} \Big( \Ga \Big( \f{3}{2} \Big) + \Ga \Big( \f{5}{2} \Big) \Big) \Big( \f{2 \b}{1 - \g'} \Big)^{1/2} n^{- (2 \g - 1) / 3} \in o(1), \]
    as desired.
\end{proof}

%% file: proofs/functional_stable/propositions/prop_convergence_of_Sn_eps_to_Sn.tex
\begin{proof}[Proof of Proposition~\ref{lem:convergence_of_Sn_eps_to_Sn}]
    We prove the two cases separately.

    \medskip
    \noindent
    \textbf{Minus case.}
    We consider the \quote{minus case} first.
    For a point $(X, U, B, L) \in \PP$, we introduce the notation $D := B + L$ for the death coordinate.
    The birth coordinates~$\set{B_i}$ constitute a Poisson point process with intensity measure $\Leb$.
    Since~$\set{L_i}$ is also a Poisson point process with intensity measure~$\mbb{P}_L$, by the displacement theorem~\cite[Exercise~5.1]{poisBook}, $\set{D_i}$ is also a Poisson point process on~$\R$ with intensity measure $\Leb \ast \mbb{P}_L = \Leb$.
    The set of points $(D_i, L_i)$ is a $\PPP(\Leb \otimes \mbb{P}_L)$, and we transform the point process~$\PP$ into a point process $\PP_-$ on $\S \ti \T$ with intensity measure $\Leb \otimes \mbb{P}_L$ by replacing the coordinates $\set{B_i}$ with $\set{D_i}$.
    Note that the total intensity measure of the point process $\PP_-$ is given~$\m$ as it was defined in Section~\ref{sec:model}.
    Using Lemmas~\ref{lem:size_of_ns}~(\ref{lem:size_of_ns:p}) and~\ref{lem:size_of_Nt_pm}~(\ref{lem:size_of_Nt_pm:size}), $\abs{N^-(P; t)} = \wt{c} U^{- \g} L \1{D \le t}$.
    Then, we have
    \[ S_{n, \eps}^{(3), -}(t) = \wt{c} \sum_{P \in \PP_- \cap (\S_n \ti \T)} U^{- \g} L \1{U \le 1 / (\eps n)} \1{D \in [0, t]}. \]
    We proceed similarly as in Step 2:
    \[ \begin{aligned}
        E_1^- &:= \lim_{\eps \downarrow 0} \limsup_{n \tff} \P[\Big]{ \sup_{t \in [0, 1]} \abs[\Big]{S_{n, \eps}^{(3), -}(t) - S_n^{(2), -}(t) - \E{S_{n, \eps}^{(3), -}(t) - S_n^{(2), -}(t)}} > n^\g \de} \\
        &= \lim_{\eps \downarrow 0} \limsup_{n \tff} \P[\bigg]{ \sup_{t \in [0, 1]} \abs[\bigg]{\sum_{P \in \PP \cap (\S_n^{1 / (\eps n) \le} \ti \T^{0 \le})} \abs{N^-(P; t)} - \E[\Big]{\sum_{P \in \PP \cap (\S_n^{1 / (\eps n) \le} \ti \T^{0 \le})} \abs{N^-(P; t)}}} > n^\g \de},
    \end{aligned} \]
    where only the high-mark vertices are considered in the sums.
    The expectation is calculated using the Mecke formula:
    \[ \begin{aligned}
        &\E[\Big]{\sum_{P \in \PP_- \cap (\S_n \ti \T)} U^{- \g} L \1{U \ge 1 / (\eps n)} \1{D \in [0, t]}} \\
        &\qquad = \int_{\S_n \ti \T} u^{- \g} \ell \1{u \ge 1 / (\eps n)} \1{w \in [0, t]} \d x \d u \d b \dell \\
        &\qquad = \int_0^n \int_{1 / (\eps n)}^1 \int_0^t \int_0^\ff u^{-\g} \ell \dell \d w \d u \d x = \f{\wt{c} n t}{1 - \g} \big( 1 - (\eps n)^{-(1 - \g)} \big).
    \end{aligned} \]
    Using the above, we write
    \[ \begin{aligned}
        M_n(t) &:= \sum_{P \in \PP_- \cap (\S_n \ti \T)} U^{- \g} L \1{U \ge 1 / (\eps n)} \1{D \in [0, t]} - \f{n t}{1 - \g} \big( 1 - (\eps n)^{-(1 - \g)} \big) \\
        E_1^- &\le \lim_{\eps \downarrow 0} \limsup_{n \tff} \P[\Big]{\sup_{t \in [0, 1]} \abs[\big]{M_n(t)} > \wt{c}^{-1} n^\g \de},
    \end{aligned} \]
    where only the high-mark vertices are considered in $M_n(t)$.
    Note that $M_n(t)$ is a martingale, since the points are independent.
    To see this, let us introduce $(\FF_t)_{t \ge 0}$ for the filtration generated by the points in $\PP_-$ up to time~$t$, and let $s, t \in [0, 1]$ be two time points such that $s \le t$.
    Then, the expectation of the increments conditioned on the sigma-algebra $\FF_s$ is given by
    \[ \begin{aligned}
        &\E[\big]{M_n(t) - M_n(s) \biggiven \FF_s} \\
        &\qquad = \E[\Big]{\sum_{P \in \PP_- \cap (\S_n \ti \T)} U^{- \g} L \1{U \ge 1 / (\eps n)} \1{D \in [s, t]} \Biggiven \FF_s} - \f{n (t - s)}{1 - \g} \big( 1 - (\eps n)^{-(1 - \g)} \big) = 0,
    \end{aligned} \]
    where in the last step we recognized that the expectation is zero, since the points are independent.
    This shows that $M_n(t)$ is a martingale with respect to the filtration $\FF_t$ generated by the points in $\PP_-$ up to time~$t$.
    Then, Doob's martingale inequality gives that
    \[ E_1^- \le \lim_{\eps \downarrow 0} \limsup_{n \tff} \wt{c} n^{-\g} \de^{-1} \E[\bigg]{\abs[\bigg]{\sum_{P \in \PP_- \cap (\S_n \ti \T)} U^{- \g} L \1{U \ge 1 / (\eps n)} \1{D \in [0, 1]} - \f{n}{1 - \g} \big( 1 - (\eps n)^{-(1 - \g)} \big)}}. \]
    The expectation contains the absolute difference of a random variable and its expectation, thus Jensen's inequality gives that
    \[ E_1^- \le \lim_{\eps \downarrow 0} \limsup_{n \tff} \wt{c} n^{-\g} \de^{-1} \Var[\Big]{\sum_{P \in \PP_- \cap (\S_n \ti \T)} U^{- \g} L \1{U \ge 1 / (\eps n)} \1{D \in [0, 1]}}^{1/2}. \]
    We calculate the variance similarly as in~\eqref{eq:variance_Snt}.
    Then, since $D = B + L$, the Mecke formula gives
    \begin{equation} \begin{aligned}
        &\Var[\Big]{\sum_{P \in \PP_- \cap (\S_n \ti \T)} U^{- \g} L \1{U \ge 1 / (\eps n)} \1{B + L \in [0, 1]}} \\
        &\qquad = \int_{\S_n^{1 / (\eps n) \le} \ti \T} u^{- 2\g} \ell^2 \1{b + \ell \in [0, 1]} \d x \d u \d b \dell \\
        &\qquad = n \int_{1 / (\eps n) \w 1}^1 u^{- 2\g} \d u \int_0^\ff \ell^2 \dell = \f{2 n}{2 \g - 1} \Big( (\eps n \vee 1)^{2 \g - 1} - 1 \Big),
    \end{aligned} \label{eq:variance_with_death} \end{equation}
    where in the first step we used that the covariance part of the expression for two distinct points is zero.
    Then, the bound for~$E_1^-$ becomes
    \[ \begin{aligned}
        E_1^- &\le c_1 \lim_{\eps \downarrow 0} \limsup_{n \tff} n^{-\g} \bigg( n \Big( (\eps n \vee 1)^{2 \g - 1} - 1 \Big) \bigg)^{1/2} \\
        &= c_1 \lim_{\eps \downarrow 0} \limsup_{n \tff} \Big( (\eps \vee 1 / n)^{2 \g - 1} - n^{- (2 \g - 1)} \Big)^{1/2} = c_1 \lim_{\eps \downarrow 0} \eps^{\g - 1/2} = 0,
    \end{aligned} \]
    where $c_1 = \wt{c} \de^{-1} (2 / (2 \g - 1))^{1/2}$, since $\g > 1/2$.

    \medskip
    \noindent
    \textbf{Plus case.}
    Next, we consider the \quote{plus case}.
    As for the \quote{minus case}, we define
    \[ E_1^+ := \lim_{\eps \downarrow 0} \limsup_{n \tff} \P[\Big]{\norm[\Big]{\So_{n, \eps}^{(3), +} - \So_n^{(2), +}} > \de}. \]
    First, we examine the edge count $\So_{n, \eps}^{(3), +}(\any) = n^{-\g} (S_{n, \eps}^{(3), +}(\any) - \E{S_{n, \eps}^{(3), +}(\any)})$.
    Noting that
    \[ \abs{\Nt^+(\Pt; t)} = \big( ((B + L) \w t) - B \big) \1{B \le t} \qquad \text{and} \qquad \f{\d}{\d t} \abs{\Nt^+(\Pt; t)} = \1{t \in [B, B + L]}, \]
    we see that $S_{n, \eps}^{(3), +}(t)$ is a continuous, monotone increasing function of~$t$, which we write as an integral over the interval $[0, t]$ as follows:
    \[ S_{n, \eps}^{(3), +}(t) = \sum_{P \in \PP \cap (\S_n^{\le 1 / (\eps n)} \ti \T^{0 \le})} \abs{N^+(P; t)} = S_{n, \eps}^{(3), +}(0) + \int_0^t \sum_{P \in \PP \cap (\S_n^{\le 1 / (\eps n)} \ti \T^{0 \le})} \f{\d}{\d t'} \abs{N^+(P; t')} \d t'. \]
    We define the edge count $H_{n, \eps}(t)$ that considers the vertices with higher marks than $1 / (\eps n)$:
    \begin{equation}
        H_{n, \eps}(t) := \sum_{P \in \PP \cap (\S_n^{1 / (\eps n) \le} \ti \T^{0 \le})} \f{\d}{\d t} \abs{N^+(P; t)},
        \label{eq:definition_of_H}
    \end{equation}
    where we note that the interval of allowed marks $[1 / (\eps n), 1]$ is the complement of $[0, 1 / (\eps n)]$ taken into account in the edge count $S_{n, \eps}^{(3), +}(t)$.
    Then,
    \[ E_1^+ = \lim_{\eps \downarrow 0} \limsup_{n \tff} \mbb{P} \bigg( \sup_{t \in [0, 1]} \abs[\Big]{\So_{n, \eps}^{(3), +}(0) - \So_n^{(2), +}(0) + \int_0^t H_{n, \eps}(t') - \E{H_{n, \eps}(t')} \d t'} > n^\g \de \bigg). \]
    The application of the triangle inequality and the union bound gives
    \[ \begin{aligned}
        E_1^+ \le \lim_{\eps \downarrow 0} \limsup_{n \tff} \bigg( &\P[\Big]{\abs[\big]{\So_{n, \eps}^{(3), +}(0) - \So_n^{(2), +}(0)} > n^\g \de / 2} \\
        &+ \P[\Big]{\sup_{t \in [0, 1]} \abs[\Big]{\int_0^t H_{n, \eps}(t') - \E{H_{n, \eps}(t')} \d t'} > n^\g \de / 2} \bigg).
    \end{aligned} \]
    We apply Chebyshev's inequality to the first term:
    \[ \begin{aligned}
        &\P[\Big]{\abs[\big]{\So_{n, \eps}^{(3), +}(0) - \So_n^{(2), +}(0)} > n^\g \de / 2} \\
        &\qquad \le 4 \de^{-2} n^{-2 \g} \Var[\big]{S_{n, \eps}^{(3), +}(0) - S_n^{(2), +}(0)} = 4 \de^{-2} n^{-2 \g} \E[\Big]{\sum_{P \in \PP \cap (\S_n^{1 / (\eps n) \le} \ti \T^{0 \le})} \abs{N^+(P; 0)}^2} \\
        &\qquad = 4 \de^{-2} n^{-2 \g} \int_{\S_n^{1 / (\eps n) \le} \ti \T^{0 \le}} \abs{N^+(p; 0)}^2 \mdp = \f{8 \wt{c}^2 \de^{-2}}{2 \g - 1} n^{- (2 \g - 1)} \big( ((\eps n)^{2 \g - 1} \vee 1) - 1 \big)
    \end{aligned} \]
    where we used that the covariance part of the expression for two distinct points is zero in the second step, and applied Lemmas~\ref{lem:int_of_ns}~(\ref{lem:int_of_ns:power_p}) and~\ref{lem:int_of_nt_pm}~(\ref{lem:int_of_nt_pm:power_p_specific}) for the spatial and the temporal parts, respectively, in the third step.
    Then, for the first term, we have
    \[ \lim_{\eps \downarrow 0} \limsup_{n \tff} \P[\Big]{\abs[\big]{\So_{n, \eps}^{(3), +}(0) - \So_n^{(2), +}(0)} > n^\g \de / 2} = \f{8 \wt{c}^2 \de^{-2}}{2 \g - 1} \lim_{\eps \downarrow 0} \eps^{2 \g - 1} = 0. \]
    Next, we consider the second term.
    We apply the triangle inequality to push the absolute value inside the integral, and then we use that $\sup_{t \in [0, 1]} \int_0^t \abs{\any} \d \wt{t} \le \int_0^1 \abs{\any} \d t$:
    \[ \begin{aligned}
        &\P[\Big]{\sup_{t \in [0, 1]} \int_0^t \abs[\big]{H_{n, \eps}(t') - \E{H_{n, \eps}(t')}} \d t' > n^\g \de / 2} \le \P[\Big]{\int_0^1 \abs[\big]{H_{n, \eps}(t') - \E{H_{n, \eps}(t')}} \d t' > n^\g \de / 2} \\
        &\qquad \le 2 \de^{-1} n^{-\g} \int_0^1 \E[\Big]{\abs[\big]{H_{n, \eps}(t') - \E{H_{n, \eps}(t')}}} \d t' \le 2 \de^{-1} n^{-\g} \int_0^1 \Var[\big]{H_{n, \eps}(t')}^{1/2} \d t',
    \end{aligned} \]
    where we applied Markov's inequality in the first step and Jensen's inequality in the second step.
    The variance is calculated using Lemmas~\ref{lem:int_of_ns}~(\ref{lem:int_of_ns:power_p}) and~\ref{lem:int_of_nt_pm}~(\ref{lem:int_of_nt_pm:power_p_specific}) as above:
    \[ \Var[\big]{H_{n, \eps}(t')} = \int_{\S_n^{1 / (\eps n) \le} \ti \T^{0 \le}} \abs{N^+(p; t')}^2 \mdp = \f{2 \wt{c}^2}{2 \g - 1} n \big( ((\eps n)^{2 \g - 1} \vee 1) - 1 \big) (t' + 1). \]
    Then, we have
    \[ \begin{aligned}
        &\lim_{\eps \downarrow 0} \limsup_{n \tff} \P[\Big]{\sup_{t \in [0, 1]} \int_0^t \abs[\big]{H_{n, \eps}(t') - \E{H_{n, \eps}(t')}} \d t' > n^\g \de / 2} \\
        &\qquad \le \lim_{\eps \downarrow 0} \limsup_{n \tff} \bigg( c_2 n^{- (\g - 1/2)} \big( ((\eps n)^{2 \g - 1} \vee 1) - 1 \big)^{1/2} \int_0^1 (t' + 1)^{1/2} \d t' \bigg) = c_3 \lim_{\eps \downarrow 0} \big( \eps^{\g - 1/2} \big) = 0,
    \end{aligned} \]
    where $c_2 = 2 \wt{c} \de^{-1} (2 / (2 \g - 1))^{1/2}$ and $c_3 = c_2 (4 \sqrt{2} - 1) / 3$.
    Thus, we have shown that $E_1^+ = 0$, and thus $E_1 = 0$.
\end{proof}

%% file: proofs/functional_stable/lemmas/lem_convergence_to_levy.tex
\begin{proof}[Proof of Lemma~\ref{lem:convergence_to_levy}]
    First, for $a > 0$, we have
    \[ \begin{aligned}
        \lim_{n \tff} n \P[\big]{n^{-\g} \abs{\Ns(\Ps)} \ge a} &= \lim_{n \tff} n \P[\big]{\wt{c} n^{-\g} U^{-\g} \ge a} = \lim_{n \tff} n \P[\big]{U \le \big( a n^{\g} / \wt{c} \big)^{-1/\g}} \\
        &= \lim_{n \tff} n \big( a n^{\g} / \wt{c} \big)^{-1 / \g} = \wt{c}^{1 / \g} a^{-1 / \g} =: \nu([a, \ff)),
    \end{aligned} \]
    where we applied Lemma~\ref{lem:size_of_ns}~(\ref{lem:size_of_ns:p}).
    Note that $\nu(\any)$ is totally skewed to the right, since $\nu((-\ff, 0]) = 0$.
\end{proof}

\begin{remark}
    The measure $\nu$ is a L\'evy measure, since it is non-negative, $\nu(\any) = 0$ for any set of Lebesgue measure zero, and
    \[ \int_0^1 \abs{a}^2 \nu(\d a) = 2 \wt{c}^{1 / \g} \int_0^1 a a^{-1 / \g} \d a = \f{2 \wt{c}^{1 / \g} \g}{2 \g - 1} < \ff \qquad \text{and} \qquad \int_1^\ff \d \nu = \nu([1, \ff)) = \wt{c}^{1 / \g} < \ff, \]
    where we used that $\g > 1/2$ for the finiteness of the first integral.
\end{remark}

%% file: proofs/functional_stable/lemmas/lem_expectation_of_So_eps.tex
\begin{proof}[Proof of Lemma~\ref{lem:expectation_of_S_eps}]
    Applying the Mecke formula to the expectation of $S_\eps^\ast$, we have
    \[ \E[\big]{S_\eps^\ast(t)} = \E[\Big]{\sum_{(J, B, L) \in \PP_\ff} J (t - B) \1{J \ge \wt{c} \eps^\g} \1{B \le t \le B + L}} = \int_{\wt{c} \eps^\g}^\ff j \, \nu(\d j) \int_{\T_{\le t}^{t \le}} (t - b) \dbdell, \]
    where we used the fact that $J$ is independent of $(B, L)$, and $\nu$ is the intensity measure of $J$.
    The first integral can be calculated using the definition of $\nu$ as follows:
    \[ \int_{\wt{c} \eps^\g}^\ff j \, \nu(\d j) = \f{\wt{c}^{1/\g}}{\g} \int_{\wt{c} \eps^\g}^\ff j^{-1 / \g} \d j = \f{\wt{c}}{1 - \g} \eps^{-(1 - \g)}. \]
    The temporal integral is the integral of the size of the temporal neighborhood at time $t$, thus the application of Lemma~\ref{lem:int_of_nt}~(\ref{lem:int_of_nt:power_p}) gives that the second integral is equal to~$1$, which concludes the proof.
\end{proof}

%% file: proofs/functional_stable/lemmas/lem_convergence_of_So_eps_fixed_time.tex
\begin{proof}[Proof of Lemma~\ref{lem:convergence_of_So_eps_fixed_time}]
    Let us define the strictly decreasing sequence of $\eps_i \to 0$ such that $1 = \eps_0 > \eps_1 > \eps_2 > \cdots$, and let $I_{i + 1} := [\eps_{i + 1}, \eps_i]$.
    We also define the edge count where the jump size is restricted to the interval $I_{i + 1}$ as
    \[ S_{I_{i + 1}}^\ast(t) := \sum_{(J, B, L) \in \PP_\ff} J (t - B) \1{J \in I_{i + 1}} \1{B \le t \le B + L} \qquad \text{and} \qquad \So_{I_{i + 1}}^\ast(t) := S_{I_{i + 1}}^\ast(t) - \E[\big]{S_{I_{i + 1}}^\ast(t)}. \]
    The expectation is calculated using the Mecke formula:
    \[ \E[\big]{S_{I_{i + 1}}^\ast(t)} = \int_{I_{i + 1}} j \, \nu(\d j) \int_{-\ff}^t (t - b) \int_{t - b}^\ff \dell \d b = c_1 \big( \eps_{i + 1}^{-(1/\g - 1)} - \eps_i^{-(1/\g - 1)} \big), \]
    where we used the definition of the intensity measure~$\nu$, and $c_1 := \wt{c}^{1 / \g} / (1 - \g)$.
    Note that the variance of the edge count $S_{I_{i + 1}}^\ast(t)$ is given by
    \[ \Var{S_{I_{i + 1}}^\ast(t)} = \int_{I_{i + 1}} j^2 \nu(\d j) \int_{-\ff}^t (t - b)^2 \int_{t - b}^\ff \dell \d b = \f{2 \wt{c}^{1 / \g}}{2\g - 1} \big( \eps_i^{2 - 1/\g} - \eps_{i + 1}^{2 - 1/\g} \big), \]
    where we used that the points are independent.
    Then, again by independence, $\sum_{i = 0}^\ff \Var{S_{I_{i + 1}}^\ast(t)} = 2 \wt{c}^{1 / \g} / (2 \g - 1) < \ff$, since $\eps_0 = 1$ and $\g > 1/2$.
    By the Kolmogorov convergence criterion~\cite[Lemma~5.16]{kallenberg}, we have that $\sum_{i = 0}^\ff \So_{I_{i + 1}}^\ast(t)$ converges almost surely, and
    \[ \begin{aligned}
        \sum_{i = 0}^\ff \So_{I_{i + 1}}^\ast(t) &= \sum_{i = 0}^\ff \bigg( \sum_{(J, B, L) \in \PP_\ff} J (t - B) \1{J \in I_{i + 1}} \1{B \le t \le B + L} - c_1 \big( \eps_{i + 1}^{-(1/\g - 1)} - \eps_i^{-(1/\g - 1)} \big) \bigg) \\
        &= \lim_{\eps \downarrow 0} \bigg( \sum_{(J, B, L) \in \PP_\ff} J (t - B) \1{J \in [\eps, 1]} \1{B \le t \le B + L} - c_1 \big( \eps^{-(1/\g - 1)} - 1 \big) \bigg).
    \end{aligned} \]
    Next, for large jumps, we set
    \[ S_0^\ast(t) := \sum_{(J, B, L) \in \PP_\ff} J (t - B) \1{J \ge 1} \1{B \le t \le B + L} \qquad \text{and} \qquad \E[\big]{S_0^\ast(t)} = c_1 < \ff, \]
    where we used similar arguments as above to calculate the expectation.
    Since the expectation is finite, the sum is almost surely finite.
    Defining the centered version, we then have $\So_0^\ast(t) := S_0^\ast(t) - \E{S_0^\ast(t)}$.
    Using the above, for a fixed time point~$t$, we have
    \[ \begin{aligned}
        \So(t) &= \lim_{\eps \downarrow 0} \So_\eps^\ast(t) = \So_0^\ast(t) + \sum_{i = 0}^\ff \So_{I_{i + 1}}^\ast(t) \\
        &= \lim_{\eps \downarrow 0} \Big( \sum_{(J, B, L) \in \PP_\ff} J (t - B) \1{J \ge \eps} \1{B \le t \le B + L} - c_1 \eps^{-(1/\g - 1)} \Big),
    \end{aligned} \]
    which converges almost surely.
\end{proof}

%% file: proofs/minor_lemmas/lem_size_of_ns.tex
\begin{proof}[Proof of Lemma~\ref{lem:size_of_ns}]
    Here we calculate the size of the spatial neighborhoods of points.

    \medskip
    \noindent
    \textbf{Part (\ref{lem:size_of_ns:p}).}
    The size of the spatial neighborhood of a point $\ps := (x, u) \in \S$ is given by
    \[ \abs{\Ns(\ps)} = \int_\S \1{\ps' \in N(\ps)} \d \ps' = \int_\S \1{\abs{z - x} \le \b u^{-\g} w^{-\g'}} \d (z, w) = \f{2 \b}{1 - \g'} u^{-\g}, \]
    where we used the notation $p' := (z, w, r)$ as usual.

    \medskip
    \noindent
    \textbf{Part (\ref{lem:size_of_ns:common_p}).}
    We write the size of the common spatial neighborhood of points $p_1$ and $p_2$ as an integral:
    \[ \abs{\Ns(\ps[1], \ps[2])} = \int_\S \1[\big]{\abs{x_1 - z} \le \b u_1^{-\g} w^{-\g'}} \1[\big]{\abs{x_2 - z} \le \b u_2^{-\g} w^{-\g'}} \d (z, w). \]
    Next, we introduce $\r := \abs{x_1 - x_2} \le \abs{x_1 - z} + \abs{x_2 - z}$, where we used the triangle inequality.
    We assume that $\abs{x_1 - z} \ge \r / 2$.
    As the bound is symmetric in the indices $1, 2$, we have the same bound in the case $\abs{x_2 - z} \ge \r / 2$.
    Then,
    \[ \abs{\Ns(\ps[1], \ps[2])} \le \int_\S \1[\big]{\r/2 \le \b u_1^{-\g} w^{-\g'}} \1[\big]{\abs{x_2 - z} \le \b u_2^{-\g} w^{-\g'}} \d (z, w). \]
    The first indicator represents an upper bound on $w$:
    \[ \r/2 \le \b u_1^{-\g} w^{-\g'} \qquad \Longleftrightarrow \qquad w \le (2 \b)^{1/\g'} \r^{-1/\g'} u_1^{-\g / \g'} =: r(\r, u_1). \]
    Using these, we integrate with respect to $z$ and the integral can be bounded as
    \[ \begin{aligned}
        \abs{\Ns(\ps[1], \ps[2])} &\le 2 \b u_2^{-\g} \int_0^{r(\r, u_1)} w^{-\g'} \d w = \f{2 \b}{1 - \g'} u_2^{-\g} r(\r, u_1)^{1 - \g'} \\
        &= \f{(2 \b)^{1/\g'}}{1 - \g'} \r^{-(1/\g' - 1)} u_1^{-(1/\g' - 1) \g} u_2^{-\g} \le \f{(2 \b)^{1/\g'}}{1 - \g'} \r^{-(1/\g' - 1)} u_1^{-\g/\g'} u_2^{-\g/\g'}.
    \end{aligned} \]
    Thus, we obtain the desired bound.

    \medskip
    \noindent
    \textbf{Part (\ref{lem:size_of_ns:p_prime}).}
    The size of the neighborhood of a point $\ps' := (z, w) \in \S$ is given by
    \[ \begin{aligned}
        \int_{\S^{u_- \le}} \1{\ps' \in \Ns(\ps)} \d \ps &= \int_{\S^{u_- \le}} \1{\abs{x - z} \le \b u^{-\g} w^{-\g'}} \d (x, u) \\
        &= 2 \b w^{-\g'} \int_{u_-}^1 u^{-\g} \d u = \f{2 \b}{1 - \g} w^{-\g'} \big( 1 - u_-^{1 - \g} \big),
    \end{aligned} \]
    as desired.
\end{proof}

%% file: proofs/minor_lemmas/lem_integrals_of_ns.tex
\begin{proof}[Proof of Lemma~\ref{lem:int_of_ns}]
    We calculate the integrals one by one.

    \medskip
    \noindent
    \textbf{Part (\ref{lem:int_of_ns:power_p})}.
    The integral is calculated as follows:
    \[ \int_{\S^{u_- \le}_A} \abs{\Ns(\ps)}^\a \d \ps = \Big( \f{2 \b}{1 - \g'} \Big)^\a \int_{\S^{u_- \le}_A} u^{- \a \g} \d (x, u) = \left\{ \begin{array}{ll} \Big( \f{2 \b}{1 - \g'} \Big)^\a \f{\abs{A}}{1 - \a \g} \Big( 1 - u_-^{1 - \a \g} \Big) & \text{if } \g \ne 1/\a \\ \Big( \f{2 \b}{1 - \g'} \Big)^\a \abs{A} \log[\big]{u_-^{-1}} & \text{if } \g = 1/\a. \end{array} \right. \]
    If $u_- = 0$, then the integral with respect to~$u$ requires that $\g < 1/\a$.

    \medskip
    \noindent
    \textbf{Part (\ref{lem:int_of_ns:common_p_0})}.
    If $m = 1$, we apply Lemma~\ref{lem:int_of_ns}~(\ref{lem:int_of_ns:power_p}) to see that lower and the upper bounds match.
    Thus, we assume that $m > 1$.
    Then,
    \begin{equation} I := \int_{\S_n^m} \abs{\Ns(\bsps[m])} \d \bsps[m] = \int_{\S_n} \int_\S \1{\ps' \in \Ns(\ps[1])} \bigg( \int_{\S_n} \1{\ps' \in \Ns(\ps)} \d \ps \bigg)^{m - 1} \d \ps' \d \ps[1]. \label{eq:int_ns_common_p_1} \end{equation}
    To show an upper bound $I_\msf{upper}$, the inner integral can be bounded using Lemma~\ref{lem:size_of_ns}~(\ref{lem:size_of_ns:p_prime}) after extending the integration domain from $\S_n$ to $\S$:
    \begin{equation} \begin{aligned}
        I \le I_\msf{upper} &= \int_{\S_n} \int_\S \1{\ps' \in \Ns(\ps[1])} \bigg( \int_\S \1{\ps' \in \Ns(\ps)} \d \ps \bigg)^{m - 1} \d \ps' \d \ps[1] \\
        &= c_1^{m - 1} \int_{\S_n} \int_\S w^{- (m - 1) \g'} \1{\abs{x_1 - z} \le \b u_1^{-\g} w^{-\g'}} \d (z, w) \d (x_1, u_1) \\
        &= 2 \b c_1^{m - 1} \int_{\S_n} u_1^{-\g} \int_0^1 w^{- m \g'} \d w \d (x_1, u_1) = \f{c_1^m}{1 - m \g'} \, n,
    \end{aligned} \label{eq:int_ns_common_p_2} \end{equation}
    where $c_1 := 2 \b / (1 - \g)$. 
    Then, $\limsup_{n \tff} I / n \le \lim_{n \tff} I_\msf{upper} / n$, which gives the upper bound
    \[ \limsup_{n \tff} \f{1}{n} \int_{\S_n^m} \abs{\Ns(\bsps[m])} \d \bsps[m] \le \f{(2 \b)^m}{(1 - \g)^m (1 - m \g')}, \]
    as mentioned in the lemma.
    For the lower bound, let us set the error term $E > 0$ such that $I = I_\msf{upper} - E$.
    We would like to show that $\lim_{n \tff} E / n = 0$ since then $\lim_{n \tff} I / n = \lim_{n \tff} I_\msf{upper} / n - \lim_{n \tff} E / n = \lim_{n \tff} I_\msf{upper} / n$.
    In $I_\msf{upper}$, let us assume without loss of generality that $p_1 \in \pp_m$ is the point which is furthest from the point $p'$, i.e., $\abs{x_1 - z} \ge \max_{i \in \set{2, \dots, m}} \abs{x_i - z}$.
    Then, due to symmetry,
    \[ \begin{aligned}
        I = m \int_{\S_n} \int_\S &\1{(z, w) \in \Ns(x_1, u_1)} \\
        &\ti \bigg( \int_\S \1{(z, w) \in \Ns(x, u)} \1{\abs{x - z} \le \abs{x_1 - z}} \d (x, u) \bigg)^{m - 1} \d (z, w) \d (x_1, u_1) - E,
    \end{aligned} \]
    and the lower bound matches the upper bound.
    Note that $E \ne 0$ only if there is at least one point $p_\ast \in \pp_m \sm \set{p_1}$ for which $x_\ast \notin [0, n]$.
    Then,
    \[ \begin{aligned}
        \f{E}{m} &\le \int_{\S_n} \int_\S \int_{\S_{[0, n]^C}} \!\!\! \1{(z, w) \in \Ns((x_1, u_1), (x_\ast, u_\ast))} \1{\abs{x_\ast - z} \le \abs{x_1 - z}} \d (x_\ast, u_\ast) \\
        &\hspace{1.3cm} \ti \bigg( \int_{\S_n} \1{(z, w) \in \Ns(\ps)} \d \ps \bigg)^{m - 2} \d (z, w) \d (x_1, u_1) \\
        &\le \int_{\S_n} \int_\S \int_{\S_{[0, n]^C}} \!\!\! \1{(z, w) \in \Ns((x_1, u_1), (x_\ast, u_\ast))} \1{\abs{x_\ast - z} \le \abs{x_1 - z}} \d (x_\ast, u_\ast) \\
        &\hspace{2.5cm} \ti c_1^{m - 2} w^{-(m - 2) \g'} \d (z, w) \d (x_1, u_1),
    \end{aligned} \]
    where we used again the fact that the inner integral can be bounded by Lemma~\ref{lem:size_of_ns}~(\ref{lem:size_of_ns:p_prime}) after extending the integration domain from $\S_n$ to $\S$.
    We distinguish two cases depending on whether the position $x_1$ of the point $p_1$ is close to the boundary of the window $[0, n]$.
    More precisely, with $\eps \in (0, 1/2)$, we consider the cases
    \begin{itemize}
        \item{Case A: $(x_1, u_1) \in \S_{[0, \eps n] \cup [(1 - \eps) n, n]} =: S_A$;}
        \item{Case B: $(x_1, u_1) \in \S_{(\eps n, (1 - \eps) n)} =: S_B$.}
    \end{itemize}
    \begin{center}
        \begin{tikzpicture}
            \begin{axis}[
                    axis lines=center,
                    axis on top,
                    axis equal,
                    hide y axis,
                    xmin=-5, xmax=25,
                    ymin=0, ymax=1.05,
                    xlabel={$\R$},
                    ylabel={mark},
                    xtick={0.01, 5, 15, 20},
                    xticklabels={$0$, $\eps n$, $(1 - \eps) n$, $n$},
                ]
                \draw[pattern=north west lines, draw=none, line width=0mm] ( 0, 0) rectangle ( 5, 1);
                \draw[pattern=north west lines, draw=none, line width=0mm] (15, 0) rectangle (20, 1);
                \node[draw=none] at ( 2.5, 2.5) {Case A};
                \node[draw=none] at (17.5, 2.5) {Case A};
                \node[draw=none] at (10  , 2.5) {Case B};
            \end{axis}
        \end{tikzpicture}
    \end{center}
    In Case~A, we bound the error term~$E$ similarly as in the case of the upper bound above as follows:
    \[ \begin{aligned}
        \f{E}{m} &\le c_1^{m - 1} \int_{\S_A} \int_\S w^{- (m - 1) \g'} \1{(z, w) \in \Ns(\ps[1])} \d (z, w) \d \ps[1] \\
        &= \f{2 \b c_1^{m - 1}}{1 - (m - 1) \g'} \int_{\S_A} u_1^{-\g} \d (x_1, u_1) = \f{2 c_1^m}{1 - (m - 1) \g'} \, \eps n \in O(n),
    \end{aligned} \]
    For Case~B, we have
    \[ \begin{aligned}
        \f{E}{m} &\le c_1^{m - 2} \int_{\S_B} \int_\S w^{-(m - 2) \g'} \\
        &\phantom{\le} \ti \int_{\S_{[0, n]^C}} \1{(z, w) \in \Ns((x_1, u_1), (x_\ast, u_\ast))} \1{\abs{x_\ast - z} \le \abs{x_1 - z}} \d (x_\ast, u_\ast) \d (z, w) \d (x_1, u_1).
    \end{aligned} \]
    As $\abs{x_1 - x_\ast} \ge \eps n$, the triangle inequality gives
    \[ \eps n / 2 \le \abs{x_1 - x_\ast} / 2 \le (\abs{x_1 - z} + \abs{x_\ast - z}) / 2 \le \abs{x_1 - z}. \]
    Extending the integration of domain of the integral with respect to $p_\ast$ from $\S_{[0, n]^C}$ to $\S$, we have
    \[ \begin{aligned}
        \f{E}{m} &\le c_1^{m - 1} \int_{\S_B} \int_\S w^{- (m - 1) \g'} \1{(z, w) \in \Ns((x_1, u_1))} \1{\eps n/2 \le \abs{x_1 - z}} \d (z, w) \d (x_1, u_1) \\
        &= c_1^{m - 1} \int_{\S_B} \int_\S w^{- (m - 1) \g'} \1{\eps n/2 \le \abs{z - x_1} \le \b u_1^{-\g} w^{-\g'}} \d (z, w) \d (x_1, u_1) \\
        &= c_1^{m - 1} \int_{\S_B} \int_0^1 w^{- (m - 1) \g'} \big( 2 \b u_1^{-\g} w^{-\g'} - \eps n \big)_+ \d w \d (x_1, u_1).
    \end{aligned} \]
    The integrand is non-zero only if $w \le (2 \b u_1^{-\g} / (\eps n))^{1 / \g'} \w 1 =: r_n(u_1)$.
    Then, neglecting the term $- \eps n$, we have
    \begin{equation} \begin{aligned}
        \f{E}{m} &\le 2 \b c_1^{m - 1} \int_{\S_B} u_1^{-\g} \int_0^{r_n(u_1)} w^{- m \g'} \d w \d (x_1, u_1) = \f{2 \b c_1^{m - 1}}{1 - m \g'} \int_{\S_B} u_1^{-\g} r_n(u_1)^{1 - m \g'} \d (x_1, u_1) \\
        &\le c_2 n \int_0^1 u_1^{-\g} r_n(u_1)^{1 - m \g'} \d u_1,
    \end{aligned} \label{eq:int_ns_common_p_3} \end{equation}
    where $c_2 = (2 \b)^m / ((1 - \g)^{m - 1} (1 - m \g'))$ is a positive constant if $\g' < 1 / m$, we integrated with respect to $x_1$, and substituted $c_1$.
    Let us examine when $r_n(u_1) \le 1$:
    \[ \Big( \f{2 \b}{\eps n} u_1^{-\g} \Big)^{1/\g'} \le 1 \qquad \Longleftrightarrow \qquad u_1 \ge \Big( \f{2 \b}{\eps n} \Big)^{1 / \g}. \]
    Then,
    \[ \f{E}{m} \le c_2 n \bigg( \int_0^{\big( \f{2 \b}{\eps n} \big)^{1 / \g}} u_1^{-\g} \d u_1 + \Big( \f{2 \b}{\eps n} \Big)^{1 / \g' - m} \int_{\big( \f{2 \b}{\eps n} \big)^{1 / \g}}^1 u_1^{-(1 / \g' - m + 1) \g} \d u_1 \bigg). \]
    For the first integral,
    \[ \int_0^{\big( \f{2 \b}{\eps n} \big)^{1 / \g}} u_1^{-\g} \d u_1 = \f{1}{1 - \g} \Big( \f{2 \b}{\eps n} \Big)^{1 / \g - 1} \in o(n). \]
    For the second integral,
    \[ \Big( \f{2 \b}{\eps n} \Big)^{1 / \g' - m} \int_{\big( \f{2 \b}{\eps n} \big)^{1 / \g}}^1 u_1^{-(1 / \g' - m + 1) \g} \d u_1 = \left\{ \begin{array}{ll} c_3 \Big( \Big( \f{2 \b}{\eps n} \Big)^{1 / \g' - m} - \Big( \f{2 \b}{\eps n} \Big)^{1 / \g - 1} \Big) & \text{if } 1/\g' - 1/\g \ne m - 1 \\ - \f{1}{\g} \Big( \f{2 \b}{\eps n} \Big)^{1 / \g' - m} \log[\Big]{\f{2 \b}{\eps n}} & \text{if } 1/\g' - 1/\g = m - 1, \end{array} \right. \]
    where $c_3 = (1 - (1 / \g' - m + 1) \g)^{-1} \in \R$, which is of order $o(n)$.
    Then, $\limsup_{n \tff} E / n = 0$ for all $\eps > 0$.

    \medskip
    \noindent
    \textbf{Part (\ref{lem:int_of_ns:common_p_minus})}.
    We begin this proof by following the same steps as in Part~(\ref{lem:int_of_ns:common_p_0}) through~\eqref{eq:int_ns_common_p_1} -- \eqref{eq:int_ns_common_p_2}, and we arrive at the following bound:
    \[ \int_{\big( \S_n^{u_n \le} \big)^m} \abs{\Ns(\bsps[m])} \d \bsps[m] \le \f{c_{u_n}^m}{1 - m \g'} \, n, \]
    where $c_{u_n} = 2 \b (1 - u_n^{1 - \g}) / (1 - \g)$, which shows the upper bound since $\lim_{n \tff} u_n^{1 - \g} = 0$.

    For the lower bound, we follow the exact same arguments from Part~(\ref{lem:int_of_ns:common_p_0}) with $\S_{\any}^{u_n \le}$, $c_{u_n}$ in place of $\S_{\any}$, $c_1$, respectively.
    Then, for Case~A, we arrive at the bound
    \[ \f{E}{m} \le \f{2 c_{u_n}^m}{1 - (m - 1) \g'} \, \eps n \in O(n). \]
    For Case~B, following the steps in Part~(\ref{lem:int_of_ns:common_p_0}) until \eqref{eq:int_ns_common_p_3}, we have
    \[ \f{E}{m} \le c_2 (1 - u_n^{1 - \g})^{m - 1} n \int_{u_n}^1 u_1^{-\g} r_n(u_1)^{1 - m \g'} \d u_1, \]
    where $c_2 = (2 \b)^m / ((1 - \g)^{m - 1} (1 - m \g'))$, and we substituted $c_{u_n}$.
    Deviating from the proof of the previous part, we bound $r_n(u_1) \le (2 \b u_1^{-\g} / (\eps n))^{1 / \g'}$ by neglecting the \quote{$\w 1$} part.
    Then,
    \[ \f{E}{m} \le c_2 (1 - u_n^{1 - \g})^{m - 1} n (\eps n)^{- (1 / \g' - m)} \int_{u_n}^1 u_1^{-(1 / \g' - m + 1) \g} \d u_1, \]
    where $c_3 = (2 \b)^{1 / \g' - m} c_2 > 0$.
    The integral can be bounded as follows:
    \[ \int_{u_n}^1 u_1^{-(1 / \g' - m + 1) \g} \d u_1 \le \left\{ \begin{array}{ll} \abs{c_4} u_n^{- ((1 / \g' - m + 1) \g - 1)_+} & \text{if } 1 / \g' - 1 / \g \ne m - 1 \\ \log[\big]{u_n^{-1}} & \text{if } 1 / \g' - 1 / \g = m - 1, \end{array} \right. \]
    where $c_4 = (1 - (1 / \g' - m + 1) \g)^{-1}$.
    Then, if $1 / \g' - 1 / \g \ne m - 1$, we have
    \[ \limsup_{n \tff} \f{E}{n m} \le c_3 \abs{c_4} \eps^{- (1 / \g' - m)} \limsup_{n \tff} \big( n^{- (1/\g' - m)} u_n^{- ((1 / \g' - m + 1) \g - 1)_+} \big). \]
    If the exponent of $u_n$ is $0$, $\limsup_{n \tff} E / n = 0$ for all $\eps > 0$ since $\g' < 1/m$.
    If the exponent of $u_n$ is negative, as $u_n > c_5 n^{-1}$ for some constant $c_5 > 0$, we apply the following bound:
    \[ \limsup_{n \tff} \f{E}{n m} \le c_3 \abs{c_4} c_6 \eps^{- (1 / \g' - m)} \limsup_{n \tff} \big( n^{- (1/\g' - m + 1) (1 - \g)} \big) = 0, \]
    where $c_6 = c_5^{1 - (1 / \g' - m + 1) \g} > 0$.
    On the other hand, if $1 / \g' - 1 / \g = m - 1$, then
    \[ \limsup_{n \tff} \f{E}{n m} \le c_3 \eps^{- (1 / \g' - m)} \limsup_{n \tff} \big( n^{- (1 / \g' - m)} \log{u_n^{-1}} \big) \le c_3 \eps^{- (1 / \g' - m)} \lim_{n \tff} \big( n^{- (1 / \g' - m)} \log{c_5^{-1} n} \big) = 0, \]
    where we used that $u_n > c_5 n^{-1}$ for large $n$.
    Thus, $\lim_{\eps \downarrow 0} \limsup_{n \tff} E/n = 0$ in both cases $A$ and $B$, and the lower bound matches the upper bound in the limit as $n \to \ff$.

    \medskip
    \noindent
    \textbf{Part (\ref{lem:int_of_ns:power_p_prime})}.
    We have that
    \[ \begin{aligned}
        \int_\S \bigg( \int_{\S_n^{u_- \le}} &\1{\ps' \in \Ns(\ps)} \d \ps \bigg)^m \d \ps'
        = \int_\S \int_{\big( \S_n^{u_- \le} \big)^m} \prod_{i = 1}^m \1{\ps' \in \Ns(\ps[i])} \d \bsps[m] \d \ps' \\
        &= \int_{\big( \S_n^{u_- \le} \big)^m} \int_\S \1{\ps' \in \Ns(\bsps[m])} \d \ps' \d \bsps[m] = \int_{\big( \S_n^{u_- \le} \big)^m} \abs[big]{\Ns(\bsps[m])} \d \bsps[m],
    \end{aligned} \]
    where used Fubini's theorem to switch the order of integration, and $\pp_m := (p_1, \ldots, p_m)$.
    This expression can be bounded using Lemma~\ref{lem:int_of_ns}~(\ref{lem:int_of_ns:common_p_minus}) as follows:
    \[ \int_{\big( \S_n^{u_- \le} \big)^m} \abs[big]{\Ns(\bsps[m])} \d \bsps[m] \le \bigg( \f{2 \b}{1 - \g} \bigg)^m \f{n}{1 - m \g'}, \]
    where we used that $1 - u_-^{1 - \g} < 1$. 

    \medskip
    \noindent
    \textbf{Part (\ref{lem:int_of_ns:power_cap_power_0})}.
    We write the common neighborhood as an integral, and then use Fubini's theorem.
    Then, using Lemma~\ref{lem:size_of_ns}~(\ref{lem:size_of_ns:p}) to calculate $\abs{\Ns(\ps[1])}$ and $\abs{\Ns(\ps[2])}$, we have that
    \[ \begin{aligned}
        I &:= \iint_{\S_n^2} \abs{\Ns((x_1, u_1))}^{m_1} \abs{\Ns((x_1, u_1), (x_2, u_2))} \abs{\Ns((x_2, u_2))}^{m_2} \abs{x_1 - x_2}^{m_3} \d (x_1, u_1) \d (x_2, u_2) \\
        &= c_1 \iint_{\S_n^2} u_1^{-m_1 \g} u_2^{-m_2 \g} \abs{\Ns((x_1, u_1), (x_2, u_2))} \abs{x_1 - x_2}^{m_3} \d (x_1, u_1) \d (x_2, u_2) \\
        &= c_1 \int_{\S_n} u_1^{-m_1 \g} \int_\S \1{\ps' \in \Ns((x_1, u_1))} \int_{\S_n} \1{\ps' \in \Ns((x_2, u_2))} u_2^{-m_2 \g} \abs{x_1 - x_2}^{m_3} \d (x_2, u_2) \d \ps' \d (x_1, u_1),
    \end{aligned} \]
    where $c_1 := (2 \b / (1 - \g'))^{m_1 + m_2}$, and we require that $\g < (m_1 \vee m_2)^{-1}$ to apply Lemma~\ref{lem:size_of_ns}~(\ref{lem:size_of_ns:p}).
    Focusing on the inner integral, we use the triangle inequality $\abs{x_1 - x_2}^{m_3} \le (\abs{x_1 - z} + \abs{x_2 - z})^{m_3}$:
    \[ \begin{aligned}
        I_\msf{inner} &:= \int_{\S_n} \1{\ps' \in \Ns((x_2, u_2))} u_2^{-m_2 \g} \abs{x_1 - x_2}^{m_3} \d (x_2, u_2) \\
        &\le 2^{m_3} \int_\S \1{\abs{z - x_2} \le \b u_2^{-\g} w^{-\g'}} u_2^{-m_2 \g} (\abs{z - x_2}^{m_3} + \abs{z - x_1}^{m_3}) \d (x_2, u_2) \\
        &= 2^{1 + m_3} \int_0^1 u_2^{-m_2 \g} \int_0^{\b u_2^{-\g} w^{-\g'}} (x_2')^{m_3} + \abs{z - x_1}^{m_3} \d x_2 \d u_2 \\
        &= c_2 w^{-(1 + m_3) \g'} \int_0^1 u_2^{- (1 + m_2 + m_3) \g} \d u_2 + 2^{1 + m_3} \b w^{-\g'} \abs{z - x_1}^{m_3} \int_0^1 u_2^{-(1 + m_2) \g} \d u_2 \\
        &= c_3 w^{-(1 + m_3) \g'} \big( 1 - u_-^{1 - (1 + m_2 + m_3) \g} \big) + c_4 w^{-\g'} \abs{z - x_1}^{m_3},
    \end{aligned} \]
    where $c_2 := (2 \b)^{1 + m_3} / (1 + m_3)$, $c_3 := c_2 / (1 - (1 + m_2 + m_3) \g)$, $c_4 := 2^{1 + m_3} \b / (1 - (1 + m_2) \g)$, we extended the integration domain from $\S_n$ to $\S$, and substituted for $x_2' := \abs{z - x_2}$.
    Furthermore, if $u_- = 0$, then we require $\g < (1 + m_2 + m_3)^{-1}$ for the finiteness of the integral with respect to $u_2$.
    Then,
    \[ \begin{aligned}
        I &\le c_1 \int_{\S_n} u_1^{-m_1 \g} \int_\S \1{(z, w) \in \Ns((x_1, u_1))} \Big( c_3 w^{-(1 + m_3) \g'} + c_4 w^{-\g'} \abs{z - x_1}^{m_3} \Big) \d (z, w) \d (x_1, u_1) \\
        &= 2 c_1 \int_{\S_n} u_1^{-m_1 \g} \int_0^1 \int_0^{\b u_1^{-\g} w^{-\g'}} \Big( c_3 w^{-(1 + m_3) \g'} + c_4 w^{-\g'} (z')^{m_3} \Big) \d z' \d w \d (x_1, u_1) \\
        &= \int_0^1 w^{-(2 + m_3) \g'} \d w \bigg( c_5 \int_{\S_n} u_1^{-(1 + m_1) \g} \d (x_1, u_1) + c_6 \int_{\S_n} u_1^{- (1 + m_1 + m_3) \g} \d (x_1, u_1) \bigg) \\
        &= c (c' + c'') n,
    \end{aligned} \]
    where we substituted for $z' := \abs{z - x_1}$ in the first step, and $c_5 := 2 \b c_1 c_3$, $c_6 := (2 \b^{1 + m_3} c_1 c_4) / (1 + m_3)$, $c_7 := c_5 / ((1 - (1 + m_1) \g) (1 - (2 + m_3) \g'))$, $c_8 := c_6 / ((1 - (1 + m_1 + m_3) \g) (1 - (2 + m_3) \g'))$.
    Moreover, the finiteness of the integral with respect to $w$ requires $\g' < 1 / (2 + m_3)$, and if $u_- = 0$, then the integral with respect to $u_1$ is finite if $\g < 1 / (1 + m_1 + m_3)$.
    All in all, we require $\g < 1 / (1 + (m_1 \vee m_2) + m_3)$ and $\g' < 1 / (2 + m_3)$.

    \medskip
    \noindent
    \textbf{Part (\ref{lem:int_of_ns:power_cap_power_minus})}.
    In this part, we follow the same strategy as in the previous part, but we consider the case where we have a minimum mark $u_- > 0$.
    Writing the common neighborhood as an integral, the application of Fubini's theorem and Lemma~\ref{lem:size_of_ns}~(\ref{lem:size_of_ns:p}) yields
    \[ \begin{aligned}
        I &:= \iint_{\big( \S_n^{u_- \le} \big)^2} \abs{\Ns((x_1, u_1))}^{m_1} \abs{\Ns((x_1, u_1), (x_2, u_2))} \abs{\Ns((x_2, u_2))}^{m_2} \abs{x_1 - x_2}^{m_3} \d (x_1, u_1) \d (x_2, u_2) \\
        &= c_1 \int_{\S_n^{u_- \le}} u_1^{-m_1 \g} \int_\S \1{\ps' \in \Ns((x_1, u_1))} \\
        &\hspace{3cm} \ti \int_{\S_n^{u_- \le}} \1{\ps' \in \Ns((x_2, u_2))} u_2^{-m_2 \g} \abs{x_1 - x_2}^{m_3} \d (x_2, u_2) \d \ps' \d (x_1, u_1),
    \end{aligned} \]
    where $c_1 := (2 \b / (1 - \g'))^{m_1 + m_2}$.
    For the inner integral, the triangle inequality gives:
    \[ \begin{aligned}
        I_\msf{inner} &:= \int_{\S_n^{u_- \le}} \1{\ps' \in \Ns((x_2, u_2))} u_2^{-m_2 \g} \abs{x_1 - x_2}^{m_3} \d (x_2, u_2) \\
        &= c_2 w^{-(1 + m_3) \g'} \int_{u_-}^1 u_2^{- (1 + m_2 + m_3) \g} \d u_2 + 2^{1 + m_3} \b w^{-\g'} \abs{z - x_1}^{m_3} \int_{u_-}^1 u_2^{-(1 + m_2) \g} \d u_2 \\
        &\le \abs{c_3} w^{-(1 + m_3) \g'} u_-^{- ((1 + m_2 + m_3) \g - 1)_+} + \abs{c_4} w^{-\g'} \abs{z - x_1}^{m_3} u_-^{- ((1 + m_2) \g - 1)_+},
    \end{aligned} \]
    where we followed the exact same steps as in the previous part, and set $c_2 := (2 \b)^{1 + m_3} / (1 + m_3)$, $c_3 := c_2 / (1 - (1 + m_2 + m_3) \g)$ and $c_4 := (2 \b c_3) / (1 - (1 + m_2) \g)$.
    Note that we do not require any bounds on~$\g$ for the finiteness of the integrals.
    Then,
    \[ \begin{aligned}
        I &\le c_1 \int_{\S_n^{u_- \le}} u_1^{-m_1 \g} \int_\S \1{(z, w) \in \Ns((x_1, u_1))} \Big( \abs{c_3} w^{-(1 + m_3) \g'} u_-^{-((1 + m_2 + m_3) \g - 1)_+} \\
        &\hspace{6.25cm} + \abs{c_4} w^{-\g'} \abs{z - x_1}^{m_3} u_-^{- ((1 + m_2) \g - 1)_+} \Big) \d (z, w) \d (x_1, u_1) \\
        &= 2 c_1 \int_{\S_n^{u_- \le}} u_1^{-m_1 \g} \int_0^1 \int_0^{\b u_1^{-\g} w^{-\g'}} \Big( \abs{c_3} w^{-(1 + m_3) \g'} u_-^{-((1 + m_2 + m_3) \g - 1)_+} \\
        &\hspace{5.5cm} + \abs{c_4} w^{-\g'} (z')^{m_3} u_-^{- ((1 + m_2) \g - 1)_+} \Big) \d z' \d w \d (x_1, u_1) \\
        &= \int_0^1 w^{-(2 + m_3) \g'} \d w \bigg( c_5 u_-^{-((1 + m_2 + m_3) \g - 1)_+} \int_{\S_n^{u_- \le}} u_1^{-(1 + m_1) \g} \d (x_1, u_1) \\
        &\hspace{3.5cm} + c_6 u_-^{- ((1 + m_2) \g - 1)_+} \int_{\S_n^{u_- \le}} u_1^{- (1 + m_1 + m_3) \g} \d (x_1, u_1) \bigg) \\
        &\le c \Big( \abs{c'} u_-^{-((1 + m_2 + m_3) \g - 1)_+ - ((1 + m_1) \g - 1)_+} + \abs{c''} u_-^{- ((1 + m_2) \g - 1)_+ -((1 + m_1 + m_3) \g - 1)_+} \Big) n,
    \end{aligned} \]
    where we substituted for $z' := \abs{z - x_1}$ in the first step, and $c_5 := 2 \b c_1 \abs{c_3}$, $c_6 := (2 \b^{1 + m_3} c_1 \abs{c_4}) / (1 + m_3)$, $c_7 := c_5 / ((1 - (1 + m_1) \g) (1 - (2 + m_3) \g'))$, $c_8 := c_6 / ((1 - (1 + m_1 + m_3) \g) (1 - (2 + m_3) \g'))$, and the constants $c, c', c''$ are specified in the statement of Lemma~\ref{lem:int_of_ns}~(\ref{lem:int_of_ns:power_cap_power_0}).
    We also used that $\g \notin \set{(1 + m_2 + m_3)^{-1}, (1 + m_1 + m_3)^{-1}}$ since $\g > 1/2$.
\end{proof}

%% file: proofs/minor_lemmas/lem_size_of_nt.tex
\begin{proof}[Proof of Lemma~\ref{lem:size_of_nt}]
    Here we calculate the size of the temporal neighborhoods of points.

    \medskip
    \noindent
    \textbf{Part (\ref{lem:size_of_nt:p}).}
    The size of the temporal neighborhood of a point $\pt = (b, \ell) \in \T$ is given by
    \[ \abs{\Nt(\pt; t)} = \int_\R \1{b \le r \le t \le b + \ell} \d r = \1{b \le t \le b + \ell} (t - b). \]

    \medskip
    \noindent
    \textbf{Part (\ref{lem:size_of_nt:p_prime}).}
    The size of the neighborhood of a point $r \in \R$ is given by
    \[ \begin{aligned}
        \int_\T \1{r \in \Nt(\pt; t)} \dpt &= \int_\T \1{b \le r \le t \le b + \ell} \dbdell \\
        &= \1{r \le t} \int_{-\ff}^r \int_{t - b}^\ff \dell \d b = \1{r \le t} \int_{-\ff}^r \e^{-(t - b)} \d b = \1{r \le t} \e^{-(t - r)},
    \end{aligned} \]
    as required.
\end{proof}

%% file: proofs/minor_lemmas/lem_integrals_of_nt.tex
\begin{proof}[Proof of Lemma~\ref{lem:int_of_nt}]
    In this proof, we calculate the integrals one by one using Fubini's theorem.

    \medskip
    \noindent
    \textbf{Part (\ref{lem:int_of_nt:power_p}).}
    The integral is given by
    \[ \int_\T \abs{\Nt(\pt)}^\a \dpt = \int_\T (t - b)^\a \1{b \le t \le b + \ell} \dbdell = \int_{-\ff}^t (t - b)^\a \e^{t - b} \d b = \Ga(\a + 1). \]

    \medskip
    \noindent
    \textbf{Part (\ref{lem:int_of_nt:common_p}).}
    We have that
    \[ \int_{\otimes_{i = 1}^m \T_i} \abs[\Big]{\bigcap_{i = 1}^m \Nt(\pt[i]; t_i)} \d \bspt[m] = \int_\R \bigg( \prod_{i = 1}^m \int_{\T_i} \1{r \in \Nt(\pt; t_i)} \dpt \bigg) \d r. \]
    If $\T_i = \T$ and $t_i = t$ for all indices $i \in \set{1, \dots, m}$, then
    \[ \begin{aligned}
        \int_{\otimes_{i = 1}^m \T_i} \abs{\Nt(\bspt[m]; t)} \d \bspt[m] &= \int_\R \bigg( \int_\T \1{r \in \Nt(\pt; t)} \dpt \bigg)^m \d r \\
        &= \int_\R \bigg( \int_\T \1{b \le r \le t \le b + \ell} \dbdell \bigg)^m \d r = \int_{-\ff}^t \e^{-m (t - r)} \d r = \f{1}{m}.
    \end{aligned} \]

    \medskip
    \noindent
    \textbf{Part (\ref{lem:int_of_nt:power_cap_power}).}
    We bound the integral using Lemma~\ref{lem:size_of_nt}~(\ref{lem:size_of_nt:p}) as follows:
    \[ \begin{aligned}
        \int_\T \int_\T &\abs{\Nt(\pt[1]; t_1)}^{\a_1} \abs{\Nt(\pt[1]; t_1) \cap \Nt(\pt[2]; t_2)} \abs{\Nt(\pt[2]; t_2)}^{\a_2} \dpt[2] \dpt[1] \\
        &\le \int_\T \abs{\Nt(\pt[1]; t_1)}^{\a_1} \dpt[1] \int_\T \abs{\Nt(\pt[2]; t_2)}^{\a_2 + 1} \dpt[2] \le \Ga(\a_1 + 1) \Ga(\a_2 + 2),
    \end{aligned} \]
    where we used Part (\ref{lem:int_of_nt:power_p}) of this lemma in the second step.

    \medskip
    \noindent
    \textbf{Part (\ref{lem:int_of_nt:power_prime}).}
    We use Lemma~\ref{lem:size_of_nt}~(\ref{lem:size_of_nt:p_prime}) to calculate the integral:
    \[ \int_\R \bigg( \int_\T \1{p' \in N(p; t)} \mdp \bigg)^\a \d p' = \int_{-\ff}^t \e^{-\a (t - r)} \d r = \f{1}{\a}, \]
    as required.
\end{proof}

%% file: proofs/minor_lemmas/lem_size_of_nt_pm.tex
\begin{proof}[Proof of Lemma~\ref{lem:size_of_Nt_pm}]

    \noindent
    \textbf{Part (\ref{lem:size_of_Nt_pm:size}).}
    Part~(\ref{lem:size_of_Nt_pm:size}) of the lemma is trivial.

    \medskip
    \noindent
    \textbf{Part (\ref{lem:size_of_Nt_pm:size_difference}).}
    For the \quote{plus case}, we have
    \[ \begin{aligned}
        \abs[\big]{\de_{t_1, t_2}(\Nt^+(\pt))} &= \abs{\Nt^+(\pt; t_2)} - \abs{\Nt^+(\pt; t_1)} = \int_\R \1{b \le r \le (b + \ell) \w t_2} - \1{b \le r \le (b + \ell) \w t_1} \d r \\
        &= \int_\R \1{b \vee t_1 \le r \le (b + \ell) \w t_2} \d r = \big( ((b + \ell) \w t) - (b \vee t_1) \big) \1{b \le t_2} \1{t_1 \le b + \ell}.
    \end{aligned} \]
    In the \quote{minus case},
    \[ \begin{aligned}
        \abs[\big]{\de_{t_1, t_2}(\Nt^-(\pt))} &= \abs{\Nt^-(\pt; t_2)} - \abs{\Nt^-(\pt; t_1)} = \int_\R \1{b \le r \le b + \ell \le t_2} - \1{b \le r \le b + \ell \le t_1} \d r \\
        &= \int_\R \1{b \le r \le t_1 \le b + \ell \le t_2} \d r = (t_1 - b) \1{b \le t_1 \le b + \ell \le t_2}.
    \end{aligned} \]

    \medskip
    \noindent
    \textbf{Part (\ref{lem:size_of_Nt_pm:size_p_prime}).}
    For the \quote{plus case}, we have
    \[ \begin{aligned}
        &\int_{\T^{0 \le}} \1{r \in \Nt^+(\pt; t)} \dpt = \int_{\T^{0 \le}} \1{b \le r \le (b + \ell) \w t} \dbdell \\
        &\qquad = \1{r \le 0} \int_{- \ff}^r \int_{- b}^\ff \dell \d b + \1{0 \le r \le t} \int_{-\ff}^r \int_{r - b}^\ff \dell \d b = \1{r \le 0} \e^r + \1{0 \le r \le t}.
    \end{aligned} \]
    For the \quote{minus case},
    \[ \begin{aligned}
        &\int_{\T^{0 \le}} \1{r \in \Nt^-(\pt; t)} \dpt = \int_{\T^{0 \le}} \1{b \le r \le (b + \ell) \w t} \dbdell \\
        &\qquad = \1{r \le 0} \int_{- \ff}^r \int_{- b}^{t - b} \dell \d b + \1{0 \le r \le t} \int_{-\ff}^r \int_{r - b}^{t - b} \dell \d b \\
        &\qquad = \1{r \le 0} \big( \e^r - \e^{-(t - r)} \big) + \1{0 \le r \le t} \big( 1 - \e^{-(t - r)} \big),
    \end{aligned} \]
    as required.
\end{proof}

%% file: proofs/minor_lemmas/lem_integrals_of_nt_pm.tex
\begin{proof}[Proof of Lemma~\ref{lem:int_of_nt_pm}]
    Here we show the integrals of the \quote{plus-minus} temporal neighborhoods.

    \medskip
    \noindent
    \textbf{Part (\ref{lem:int_of_nt_pm:power_p_specific})}.
    For the \quote{plus case}, we have
    \[ \begin{aligned}
        \int_{\T^{0 \le}} &\abs{\Nt^+(\pt; t)}^m \dpt = \int_{\T_{\le t}^{0 \le}} (((b + \ell) \w t) - b)^m \dbdell \\
        &= \int_{-\ff}^0 \int_{- b}^{t - b} \ell^m \dell \d b + \int_0^t \int_0^{t - b} \ell^m \dell \d b + \int_{-\ff}^t \int_{t - b}^\ff (t - b)^m \dell \d b.
    \end{aligned} \]
    Note that
    \[ \f{\d}{\d \ell} \bigg( - \e^{-\ell} m! \sum_{i = 0}^m \f{\ell^i}{i!} \bigg) = \ell^m \e^{-\ell}. \]
    Then,
    \[ \begin{aligned}
        &\int_{\T^{0 \le}} \abs{\Nt^+(\pt; t)}^m \dpt = \int_{-\ff}^0 m! \bigg( \e^b \sum_{i = 0}^m \f{(- b)^i}{i!} - \e^{-(t - b)} \sum_{i = 0}^m \f{(t - b)^i}{i!} \bigg) \d b \\
        &\qquad \phantom{=} + \int_0^t m! \bigg( 1 - \e^{-(t - b)} \sum_{i = 0}^m \f{(t - b)^i}{i!} \bigg) \d b + \int_{-\ff}^t (t - b)^m \e^{-(t - b)} \d b \\
        &\qquad = \sum_{i = 0}^m \f{m!}{i!} \bigg( \int_{-\ff}^0 (- b)^i \e^b \d b - \int_{-\ff}^t (t - b)^i \e^{-(t - b)} \d b \bigg) + m! (t + 1) = m! (t + 1),
    \end{aligned} \]
    where we applied dominated convergence theorem with $\exp{-(t - b)}$ and $\exp{b}$ as the dominating functions to interchange the integral and the summation, and recognized the Gamma functions.
    For the \quote{minus case}, we use the same argument to obtain
    \[ \int_{\T^{0 \le}} \abs{\Nt^-(\pt; t)}^m \dpt = \int_{\T^{[0, t]}} \ell^m \dbdell = \int_{-\ff}^0 \int_{- b}^{t - b} \ell^m \dell \d b + \int_0^t \int_0^{t - b} \ell^m \dell \d b = m! t. \]

    \medskip
    \noindent
    \textbf{Part (\ref{lem:int_of_nt_pm:power_p_bound}).}
    For the \quote{plus case}, we have
    \[ \begin{aligned}
        \int_{\T^{0 \le}} &\abs{\Nt^+(\pt; t)}^\a \dpt = \int_{\T_{\le t}^{0 \le}} (((b + \ell) \w t) - b)^\a \dbdell \\
        &= \int_{-\ff}^0 \int_{- b}^{t - b} \ell^\a \dell \d b + \int_0^t \int_0^{t - b} \ell^\a \dell \d b + \int_{-\ff}^t \int_{t - b}^\ff (t - b)^\a \dell \d b \\
        &\le 2 c \int_{-\ff}^0 \e^b - \e^{-(t - b)} \d b + 2 c \int_0^t 1 - \e^{-(t - b)} \d b + \int_{-\ff}^t (t - b)^\a \e^{-(t - b)} \d b = 2 c t + \Ga(\a + 1),
    \end{aligned} \]
    where we used that $\ell^\a \e^{-\ell} \le c \e^{-\ell / 2}$ with $c := (2 \a)^\a \e^{-\a}$, and substituted for $t - b$ in the last term.
    In the \quote{minus case}, we use the same argument as above to obtain
    \[ \int_{\T^{0 \le}} \abs{\Nt^-(\pt; t)}^\a \dpt = \int_{\T^{[0, t]}} \ell^\a \dbdell = \int_{-\ff}^0 \int_{- b}^{t - b} \ell^\a \dell \d b + \int_0^t \int_0^{t - b} \ell^\a \dell \d b \le 2 c t. \]

    \medskip
    \noindent
    \textbf{Part (\ref{lem:int_of_nt_pm:difference_p})}.
    First, using Lemma~\ref{lem:size_of_Nt_pm}~(\ref{lem:size_of_Nt_pm:size_difference}), we calculate the integral in the \quote{plus case}:
    \[ \begin{aligned}
        &\int_{\T_{\le t_2}^{t_1 \le}} \abs[\big]{\de_{t_1, t_2}(\Nt^+(\pt))}^m \dpt = \int_{\T_{\le t_2}^{t_1 \le}} \big( ((b + \ell) \w t_2) - (b \vee t_1) \big)^m \dbdell \\
        &\qquad = \int_{-\ff}^{t_1} \int_{t_1 - b}^{t_2 - b} (b + \ell - t_1)^m \dell \d b + \int_{-\ff}^{t_1} \int_{t_2 - b}^\ff (t_2 - t_1)^m \dell \d b \\
        &\qquad \phantom{=} + \int_{t_1}^{t_2} \int_0^{t_2 - b} \ell^m \dell \d b + \int_{t_1}^{t_2} \int_{t_2 - b}^\ff (t_2 - b)^m \dell \d b.
    \end{aligned} \]
    We treat the terms one by one.
    For the first integral, we substitute $a := b + \ell - t_1$:
    \[ \int_{-\ff}^{t_1} \int_{t_1 - b}^{t_2 - b} (b + \ell - t_1)^m \dell \d b = \int_{-\ff}^{t_1} \e^{- (t_1 - b)} \int_0^{t_2 - t_1} a^m \e^{-a} \d a \d b \le c_1 (t_2 - t_1) \in O(t_2 - t_1), \]
    where we used that $a^m \e^{-a} \le c_1$ is bounded in the second step.
    For the second integral, we have
    \[ \int_{-\ff}^{t_1} \int_{t_2 - b}^\ff (t_2 - t_1)^m \dell \d b = (t_2 - t_1)^m \e^{-(t_2 - t_1)} \in O(t_2 - t_1), \]
    where in the last step we used that $t_2 - t_1 \le 1$.
    For the third term, we have
    \[ \int_{t_1}^{t_2} \int_0^{t_2 - b} \ell^m \dell \d b \le \Ga(m + 1) (t_2 - t_1) \in O(t_2 - t_1), \]
    where we extended the integration domain from $[0, t_2 - b]$ to $[0, \ff)$.
    For the last integral,
    \[ \int_{t_1}^{t_2} \int_{t_2 - b}^\ff (t_2 - b)^m \dell \d b = \int_{t_1}^{t_2} (t_2 - b)^m \e^{-(t_2 - b)} \d b = \int_0^{t_2 - t_1} a^m \e^{-a} \d a \in O(t_2 - t_1), \]
    where we used the substitution $a := t_2 - b$ in the second step, and used that the integrand is bounded in the last step.
    Then, each of the terms is in $O(t_2 - t_1)$, which completes the proof for the \quote{plus case}.

    For the \quote{minus case},
    \[ \begin{aligned}
        \int_{\T^{0 \le}} &\abs[\big]{\de_{t_1, t_2}(\Nt^-(\pt))}^m \dpt = \int_{\T_{\le t_1}^{[t_1, t_2]}} (t_1 - b)^m \dbdell = \int_{-\ff}^{t_1} (t_1 - b)^m \big( \e^{-(t_1 - b)} - \e^{-(t_2 - b)} \big) \d b \\
        &= \big( 1 - \e^{-(t_2 - t_1)} \big) \int_0^\ff a^m \e^{-a} \d a = \Ga(m + 1) \big( 1 - \e^{-(t_2 - t_1)} \big) \in O(t_2 - t_1),
    \end{aligned} \]
    where we substituted $a := t_1 - b$ in the third step.
    Thus, $\int_{\T^{0 \le}} \abs[\big]{\de_{t_1, t_2}(\Nt^\pm(\pt))}^m \dpt \le (t_2 - t_1)^m$.

    \medskip
    \noindent
    \textbf{Part (\ref{lem:int_of_nt_pm:difference_p_prime})}.
    For the \quote{plus case},
    \[ \begin{aligned}
        \int_\R \bigg( &\int_{\T^{0 \le}} \1{r \in \de_{t_1, t_2}(\Nt^+(\pt))} \dpt \bigg)^m \d r = \int_\R \bigg( \int_{\T_{\le t_2}^{t_1 \le}} \1{b \vee t_1 \le r \le (b + \ell) \w t_2} \dbdell \bigg)^m \d r \\
        &= \int_\R \1{t_1 \le r \le t_2} \bigg( \int_\T \1{b \le r \le b + \ell} \dbdell \bigg)^m \d r \in O(t_2 - t_1),
    \end{aligned} \]
    since the inner integral is equal to $1$ since
    \[ \int_\T \1{b \le r \le b + \ell} \dbdell = \int_{-\ff}^t \int_{t - b}^\ff \dell \d b = \int_{-\ff}^t \e^{-(t - b)} \d b = 1. \]
    For the \quote{minus case}, we have
    \[ \begin{aligned}
        &\int_\R \bigg( \int_{\T^{0 \le}} \1{r \in \de_{t_1, t_2}(\Nt^-(\pt))} \dpt \bigg)^\a \d r = \int_{-\ff}^{t_1} \bigg( \int_{\T_{\le r}^{[t_1, t_2]}} \dpt \bigg)^\a \d r \\
        &\qquad = \int_{-\ff}^{t_1} \big( \e^{-(t_1 - r)} - \e^{-(t_2 - r)} \big)^\a \d r = \big( \e^{-t_1} - \e^{-t_2} \big)^\a \int_{-\ff}^{t_1} \e^{\a r} \d r = \f{1}{\a} \big( 1 - \e^{-(t_2 - t_1)} \big)^\a \in O(t_2 - t_1).
    \end{aligned} \]

    \medskip
    \noindent
    \textbf{Part (\ref{lem:int_of_nt_pm:cap})}.
    Note that by independence of the points $p_i$ and their neighborhoods $\Nt^\pm(\pt[i]; t)$, Fubini's theorem gives
    \[ \iint_{\big( \T^{0 \le} \big)^m} \abs{\Nt^\pm(\bspt[m]; t)} \d \bspt[m] = \int_\R \bigg( \int_{\T^{0 \le}} \1{r \in \Nt^\pm(\pt; t)} \dpt \bigg)^m \d r. \]
    For the \quote{plus case}, using Lemma~\ref{lem:size_of_Nt_pm}~(\ref{lem:size_of_Nt_pm:size_p_prime}), we have:
    \[ \begin{aligned}
        \int_\R &\bigg( \int_{\T^{0 \le}} \1{r \in \Nt^+(\pt; t)} \dpt \bigg)^m \d r = \int_\R \bigg( \int_{\T^{0 \le}} \1{b \le r \le (b + \ell) \w t} \dbdell \bigg)^m \d r \\
        &= \int_\R \bigg( \1{r \le 0} \int_{-\ff}^r \int_{- b}^\ff \dell \d b + \1{0 \le r \le t} \int_{-\ff}^r \int_{r - b}^\ff \dell \d b \bigg)^m \d r \\
        &= \int_{-\ff}^0 \e^{m r} \d r + \int_0^t \d r = \f{1}{m} + t.
    \end{aligned} \]
    For the \quote{minus case}, we have
    \[ \int_\R \bigg( \int_{\T^{0 \le}} \1{r \in \Nt^-(\pt; t)} \dpt \bigg)^m \d r \le \int_\R \bigg( \int_{\T^{0 \le}} \1{r \in \Nt^+(\pt; t)} \dpt \bigg)^m \d r. \]

    \medskip
    \noindent
    \textbf{Part (\ref{lem:int_of_nt_pm:power_cap_power})}.
    For the \quote{plus case}, we have
    \[ \begin{aligned}
        &\iint_{\big( \T_{\le 1}^{0 \le} \big)^2} \abs{\Nt^\pm(\pt[1]; t_1)}^{\a_1} \abs{\Nt^\pm(\pt[1]; t_1) \cap \Nt^\pm(\pt[2]; t_2)} \abs{\Nt^\pm(\pt[2]; t_2)}^{\a_2} \dpt[1] \dpt[2] \\
        &\qquad \le \int_{\T_{\le 1}^{0 \le}} \int_{\T^{0 \le}} \abs{\Nt^\pm(\pt[2]; t_2)}^{\a_2 + 1} \dpt[1] \dpt[2] \le c \int_{\T_{\le 1}^{0 \le}} \dpt \\
        &\qquad = c \int_{-\ff}^0 \int_{- b}^\ff \dell \d b + c \int_0^1 \int_0^\ff \dell \d b = 2 c < \ff
    \end{aligned} \]
    with some constant $c > 0$, and we used Lemma~\ref{lem:int_of_nt_pm}~(\ref{lem:int_of_nt_pm:power_p_bound}) in the second step for the integral with respect to $\pt[2]$.
\end{proof}

%% file: appendix.tex
\appendix
\section{Proofs of Lemmas~\ref{lem:limiting_covariance_function_Snt_pm} and~\ref{lem:bounds_of_error_terms_pm}}\label{sec:appendix_proofs_cov_func_error_terms_pm}

The following two lemmas were necessary to apply Proposition~\ref{prop:multivariate_normal_limit} to show the finite-dimensional convergence of the plus and minus parts.
Here, we follow the same approach as in the proof of Proposition~\ref{prop:multivariate_normal_limit} and use the time interval-based decomposition of the edge counts to calculate the limiting covariance functions of~$\So_n^\pm$.
\input{proofs/functional_normal/lemmas/lem_limiting_covariance_function_pm.tex}
\input{proofs/functional_normal/lemmas/lem_bounds_of_error_terms_pm.tex}

\section{Proofs of Lemmas~\ref{lem:variance_high_mark_edge_count},~\ref{lem:variance_change_of_high_mark_edge_count}, and~\ref{lem:cumulant_term_cases_ge} used in Section~\ref{subsec:thm_stable_part_1}}\label{sec:appendix_proofs_part_1}

The following proofs verify the lemmas used to show that the finite-dimensional distributions of the high-mark edge count $\So_n^\ge$ converge to $0$.
The ideas are similar to the proof of Theorem~\ref{thm:functional_normal}.
\input{proofs/functional_stable/lemmas/lem_variance_of_high_mark_edge_count.tex}
\input{proofs/functional_stable/lemmas/lem_variance_of_change_of_high_mark_edge_count.tex}
To show Condition~\ref{condition:cumulant} of Theorem~\ref{thm:davydov} for the high-mark edge count $\So_n^\ge$, we need to bound the cumulant term $\k_4(\De_n^{\ge, \pm}(s, t))$, which is done in the following proof.
\input{proofs/functional_stable/lemmas/lem_tightness_cumulant_term_ge.tex}
Finally, Condition~\ref{condition:third} of Theorem~\ref{thm:davydov} is about the convergence of the expected increments $\E[\big]{\De_n^{\ge, \pm}(t_k, t_{k + 1})}$, which in proved next.
\input{proofs/functional_stable/lemmas/lem_expectation_of_increments_ge.tex}

\section{Proofs of Propositions~\ref{prop:continuity_of_summation},~\ref{prop:So_cauchy_uniform} and Lemma~\ref{lem:So_eps_cauchy_probability} used in Section~\ref{sec:edge_count_stable}}\label{sec:appendix_proofs_part_2}

Proposition~\ref{prop:continuity_of_summation} stated that the summation functional $\chi(\eta)(t)$ is almost surely continuous with respect to the Skorokhod metric $d_\msf{Sk}$.
In the proof of Proposition~\ref{prop:continuity_of_summation}, we follow the argument of~\cite[Section~7.2.3]{heavytails}.
Our proof is slightly more involved since we have to deal with vertices having potentially large lifetimes, and we need to restrict the domain of the summation functional~$\chi$ to a compact set~$K$.
In the first step, we show that if two point measures are close, then the numbers of points in a compact set are almost surely identical.
In the second step, we show that as the points of the point measures are close, the respective functions the summation functional maps these measures are also close with respect to the Skorokhod metric $d_\msf{Sk}$.
\input{proofs/functional_stable/propositions/prop_continuity_of_summation_functional.tex}
In the next proof, we bound the Skorokhod distance $d_\msf{Sk}(\chi(\eta_n), \chi(\eta))$ of the functionals $\chi(\eta_n)$ and~$\chi(\eta)$ used in the proof of Proposition~\ref{prop:continuity_of_summation}.
\input{proofs/functional_stable/lemmas/lem_bound_of_skorokhod_distance.tex}

The next proofs show that the sequence $\So_{\eps_n}^\ast$ is Cauchy in probability and almost surely with respect to the supremum norm, and we follow a similar approach to what was done in Step~3.
\input{proofs/functional_stable/lemmas/lem_so_cauchy_probability.tex}
In the next proof showing the Cauchy property almost surely, we follow the approach of the proof of Property 2 in~\cite[Proposition~5.7]{heavytails}.
\input{proofs/functional_stable/propositions/prop_so_cauchy_uniform.tex}

%% file: proofs/functional_normal/lemmas/lem_limiting_covariance_function_pm.tex
\begin{proof}[Proof of Lemma~\ref{lem:limiting_covariance_function_Snt_pm}]
    The proof follows the same steps as the proof of Proposition~\ref{prop:covariance_function_of_Snt}.
    Note that
    \[ \Cov{\So_n^\pm(s), \So_n^\pm(t)} = \f{1}{n} \Cov{S_n^\pm(s), S_n^\pm(t)}. \]
    For the covariance functions of $\So_n^\pm(t)$, we assume that $0 \le s \le t \le 1$, and decompose them based on
    \[ \begin{aligned}
        S_n^+(s) &= S_n^{A+}(s) \qquad \qquad S_n^+(t) = S_n^{A+}(s) + S_n^{B+}(s, t) + S_n^{C+}(s, t) \\
        S_n^-(s) &= S_n^{A-}(s) \qquad \qquad S_n^-(t) = S_n^{A-}(s) + S_n^{B-}(s, t) + S_n^{C-}(s, t),
    \end{aligned} \]
    where
    \[ \begin{aligned}
        S_n^{A+}(s) &:= \sum_{P \in \PP \cap (\S_n \ti \T^{0 \le})} \sum_{P' \in \PP'} \1{P' \in N^+(P; s)} \\
        S_n^{B+}(s, t) &:= \sum_{P \in \PP \cap (\S_n \ti \T^{0 \le}_{\le s})} \sum_{P' \in \PP'} \1{P' \in N^+(P; t)} \1{s \le R} \\
        S_n^{C+}(s, t) &:= \sum_{P \in \PP \cap (\S_n \ti \T_{[s, t]})} \sum_{P' \in \PP'} \1{P' \in N^+(P; t)}
    \end{aligned} \]
    for the \quote{plus case} and
    \begin{equation} \begin{aligned}
        S_n^{A-}(s) &:= \sum_{P \in \PP \cap (\S_n \ti \T^{[0, s]})} \sum_{P' \in \PP'} \1{P' \in N^-(P; s)} \\
        S_n^{B-}(s, t) &:= \sum_{P \in \PP \cap (\S_n \ti \T^{[s, t]})} \sum_{P' \in \PP'} \1{P' \in N^-(P; t)} \1{R \le s} \\
        S_n^{C-}(s, t) &:= \sum_{P \in \PP \cap (\S_n \ti \T^{[s, t]})} \sum_{P' \in \PP'} \1{P' \in N^-(P; t)} \1{s \le R}
    \end{aligned} \label{eq:decomp_edge_count_minus} \end{equation}
    for the \quote{minus case}.
    These decompositions of the edge counts are visualized in Figure~\ref{fig:covariance_intervals_plus_minus}.
    \begin{figure}[htb] \centering
        \begin{tikzpicture}[scale = 0.4]
            \draw[->,line width=0.07cm] (0,-3.75) -- (16.5,-3.75);
            \node[draw=none, anchor=north] at (16.5, -3.75) {$\R$};
            \node[draw=none, anchor=north] at (7.5, -3.75) {$s$};
            \node[draw=none, anchor=north] at (15, -3.75) {$t$};
            \node[draw=none] at (-1.5, 4.75) {\Large{$A+$}};
            \draw[line width=0.025cm] ( 0,5.25) -- (7.5,5.25);
            \draw[dashed, line width=0.025cm] (0,4.5) -- (7.5,4.5);
            \draw[fill=white] (7.5,4.5) circle (0.3);
            \draw[fill=white] (7.5,5.25) circle (0.3);
            \node[draw=none, anchor=south] at (3.75, 5.25) {$b$};
            \node[draw=none, anchor=north] at (3.75, 4.5) {$r$};
            \node[draw=none, ] at (-1.5, 1.75) {\Large{$B+$}};
            \draw[line width=0.025cm] (0,2) -- (7.5,2);
            \draw[dashed, line width=0.025cm] (7.5,1.25) -- (15,1.25);
            \draw[fill=white] (7.5,1.25) circle (0.3);
            \draw[fill=white] (15 ,1.25) circle (0.3);
            \draw[fill=white] (7.5,2.) circle (0.3);
            \node[draw=none, anchor=south] at (3.75, 2.) {$b$};
            \node[draw=none, anchor=north] at (11.25, 1.25) {$r$};
            \node[draw=none] at (-1.5, -1.25) {\Large{$C+$}};
            \draw[line width=0.025cm] (7.5,-1) -- (15,-1);
            \draw[dashed, line width=0.025cm] (7.5,-1.75) -- (15,-1.75);
            \draw[fill=white] (7.5,-1) circle (0.3);
            \draw[fill=white] (7.5,-1.75) circle (0.3);
            \draw[fill=white] (15,-1.75) circle (0.3);
            \draw[fill=white] (15,-1) circle (0.3);
            \node[draw=none, anchor=south] at (11.25, -1   ) {$b$};
            \node[draw=none, anchor=north] at (11.25, -1.75) {$r$};

            \draw[->,line width=0.07cm] (23,-3.75) -- (39.5,-3.75);
            \node[draw=none, anchor=north] at (39.5, -3.75) {$\R$};
            \node[draw=none, anchor=north] at (30.5, -3.75) {$s$};
            \node[draw=none, anchor=north] at (38,   -3.75) {$t$};
            \node[draw=none, ] at             (21   , 5   ) {\Large{$A-$}};
            \draw[line width=0.025cm]         (23   , 5.25) -- (30.5, 5.25);
            \draw[dashed, line width=0.025cm] (23   , 4.5 ) -- (30.5, 4.5 );
            \draw[fill=white]                 (30.5 , 4.5 ) circle (0.3);
            \draw[fill=white]                 (30.5 , 5.25) circle (0.3);
            \node[draw=none, anchor=south] at (26.75, 5.25) {$b+\ell$};
            \node[draw=none, anchor=north] at (26.75, 4.5 ) {$r$};
            \node[draw=none, ] at             (21   , 1.75) {\Large{$B-$}};
            \draw[line width=0.025cm]         (30.5 , 2   ) -- (38  , 2   );
            \draw[dashed, line width=0.025cm] (23   , 1.25) -- (30.5, 1.25);
            \draw[fill=white]                 (30.5 , 1.25) circle (0.3);
            \draw[fill=white]                 (30.5 , 2   ) circle (0.3);
            \draw[fill=white]                 (38   , 2   ) circle (0.3);
            \node[draw=none, anchor=south] at (34.25, 2   ) {$b + \ell$};
            \node[draw=none, anchor=north] at (26.75, 1.25) {$r$};
            \node[draw=none] at (21, -1.25) {\Large{$C-$}};
            \draw[line width=0.025cm] (30.5,-1) -- (38,-1);
            \draw[dashed, line width=0.025cm] (30.5,-1.75) -- (38,-1.75);
            \draw[fill=white] (30.5,-1   ) circle (0.3);
            \draw[fill=white] (38  ,-1.75) circle (0.3);
            \draw[fill=white] (38  ,-1   ) circle (0.3);
            \draw[fill=white] (30.5,-1.75) circle (0.3);
            \node[draw=none, anchor=south] at (34.25, -1   ) {$b + \ell$};
            \node[draw=none, anchor=north] at (34.25, -1.75) {$r$};
        \end{tikzpicture}
        \caption{
            Decomposition of the edge count functions $S_n^\pm(s)$ and $S_n^\pm(t)$.
            We require that $b \le r \le b+\ell$.
        }\label{fig:covariance_intervals_plus_minus}
    \end{figure}
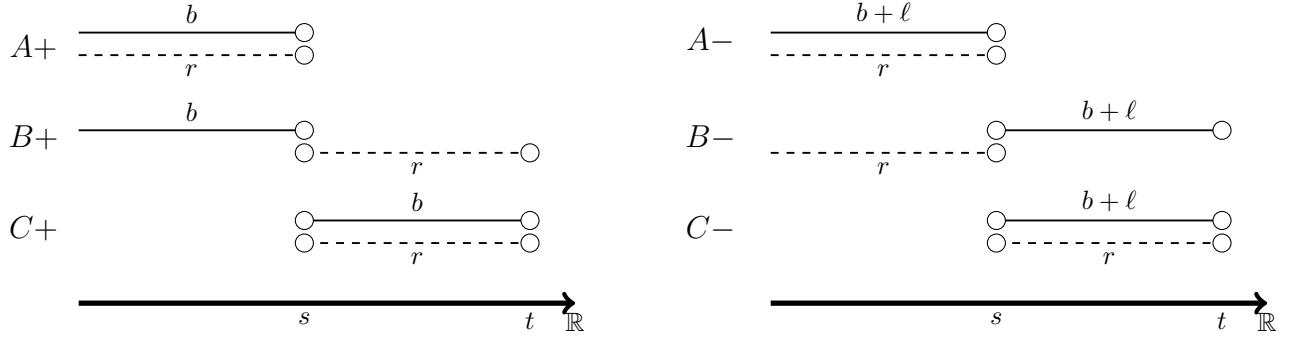
    With these notations, the covariance functions are given as follows:
    \begin{equation} \begin{aligned}
        \Cov{S_n^+(s), S_n^+(t)} &= \Var{S_n^{A+}(s)} + \Cov{S_n^{A+}(s), S_n^{B+}(s, t)} \\
        \Cov{S_n^-(s), S_n^-(t)} &= \Var{S_n^{A-}(s)} + \Cov{S_n^{A-}(s), S_n^{B-}(s, t)},
    \end{aligned} \label{eq:decomp_edge_count_pm} \end{equation}
    as $\Cov{S_n^{A+}(s), S_n^{C+}(s, t)} = \Cov{S_n^{A-}(s), S_n^{C-}(s, t)} = 0$ due to the independence properties of the Poisson processes $\PP$ and~$\PP'$.

    \medskip
    \noindent
    \textbf{Covariance function of $S_n^+(t)$.}
    We begin with the \quote{plus case}.
    The variance term in~\eqref{eq:decomp_edge_count_pm} is
    \[ \begin{aligned}
        &\Var{S_n^{A+}(s)} = \E[\Big]{\sum_{P \in \PP \cap \S_n \ti \T^{0 \le}} \hspace{-0.5cm} \deg^+(P; s)^2} \\
        &\qquad = \int_{\S_n \ti \T^{0 \le}} \E[\big]{\deg^+(p; s)^2} \mdp + \iint_{\big( \S_n \ti \T^{0 \le} \big)^2} \Cov[\big]{\deg^+(p_1; s), \deg^+(p_2; s)} \mdp[1] \mdp[2].
    \end{aligned} \]
    Using the same calculations as in the proof of Lemma~\ref{lem:mean_variance_of_Snt}, we see that $\E{\deg^+(p; s)^2} = \abs{N^+(p; s)} + \abs{N^+(p; s)}^2$, which leads to
    \[ \begin{aligned}
        \lim_{n \tff} \f{1}{n} \int_{\S_n \ti \T^{0 \le}} \E[\big]{\deg^+(p; s)^2} \mdp &= \sum_{k = 1}^2 \lim_{n \tff} \f{1}{n} \int_{\S_n} \abs{\Ns(\ps)}^k \d \ps \int_{\T^{0 \le}} \abs{\Nt^+(\pt; s)}^k \dpt \\
        &= \sum_{k = 1}^2 \Big( \f{2 \b}{1 - \g'} \Big)^k \f{k! (s + 1)}{1 - k \g},
    \end{aligned} \]
    where in the last step we used Lemma~\ref{lem:int_of_ns}~(\ref{lem:int_of_ns:power_p}) for the spatial part and Lemma~\ref{lem:int_of_nt_pm}~(\ref{lem:int_of_nt_pm:power_p_specific}) for the temporal part. 
    Following again the same calculations as in Lemma~\ref{lem:mean_variance_of_Snt}, we get $\Cov{\deg^+(p_1; s), \deg^+(p_2; s)} = \abs{N^+(p_1, p_2; s)}$.
    Using this result, we factorize the second integral:
    \[ \begin{aligned}
        \lim_{n \tff} \f{1}{n} \iint_{\big( \S_n \ti \T^{0 \le} \big)^2} &\abs{N^+(p_1, p_2; s)} \mdp[1] \mdp[2] \\
        &= \lim_{n \tff} \f{1}{n} \iint_{\S_n^2} \abs{\Ns(\ps[1], \ps[2])} \d \ps[1] \d \ps[2] \iint_{\big( \T^{0 \le} \big)^2} \abs{\Nt^+(\pt[1], \pt[2]; s)} \dpt[1] \dpt[2] < \ff,
    \end{aligned} \]
    where we used Lemma~\ref{lem:int_of_ns}~(\ref{lem:int_of_ns:common_p_0}) with $m = 2$ for the spatial part and Lemma~\ref{lem:int_of_nt_pm}~(\ref{lem:int_of_nt_pm:cap}) for the temporal part. 
    Similarly to the proof of Proposition~\ref{prop:covariance_function_of_Snt}, the covariance term $\Cov{S_n^{A+}(s), S_n^{B+}(s, t)}$ is determined by the common $\PP$-points of $S_n^{A+}(s)$ and $S_n^{B+}(s, t)$:
    \[ \begin{aligned}
        \lim_{n \tff} \f{1}{n} \Cov{S_n^{A+}(s), S_n^{B+}(s, t)} &= \lim_{n \tff} \f{1}{n} \int_{\S_n \ti \T_{\le s}^{0 \le}} \abs{N^+(p; s)} \int_{\S \ti [s, t]} \1{p' \in N^+(p; t)} \d p' \mdp = I_\msf{s} \ti I_\msf{t} \\
        I_\msf{s} &:= \lim_{n \tff} \f{1}{n} \int_{\S_n} \abs{\Ns(\ps)}^2 \d \ps \\
        I_\msf{t} &:= \int_{\T_{\le s}^{0 \le}} \abs{\Nt^+(\pt; s)} \int_{[s, t]} \1{r \in \Nt^+(\pt; t)} \d r \dpt,
    \end{aligned} \]
    where we used that the $\PP'$-points connecting to the point $p$ cannot be identical due to the disjoint sets the $\PP'$-points belong to.
    By Lemma~\ref{lem:int_of_ns}~(\ref{lem:int_of_ns:power_p}), we see that $I_\msf{s} < \ff$.
    Using Lemma~\ref{lem:int_of_nt_pm}~(\ref{lem:int_of_nt_pm:power_p_bound}), we see that $I_\msf{t}$ is bounded:
    \[ I_\msf{t} \le (t - s) \int_{\T^{0 \le}} \abs{\Nt^+(\pt; s)} \dpt < \ff. \]
    Thus, for some finite constant~$c_+ > 0$, $\lim_{n \tff} \Cov{\So_n^+(s), \So_n^+(t)} = c_+$.

    \medskip
    \noindent
    \textbf{Covariance function of $S_n^-(t)$.}
    Let us recall the notations in~\eqref{eq:decomp_edge_count_minus}.
    The variance term in~\eqref{eq:decomp_edge_count_pm} is
    \[ \begin{aligned}
        \Var{S_n^{A-}(s)} &= \E[\Big]{\sum_{P \in \PP \cap \S_n \ti \T^{[0, s]}} \hspace{-0.5cm} \deg^-(P; s)^2} \\
        &= \int_{\S_n \ti \T^{[0, s]}} \E[\big]{\deg^-(p; s)^2} \mdp + \iint_{\big( \S_n \ti \T^{[0, s]} \big)^2} \Cov[\big]{\deg^-(p_1; s), \deg^-(p_2; s)} \mdp[1] \mdp[2].
    \end{aligned} \]
    Using the same calculations as in Lemma~\ref{lem:mean_variance_of_Snt}, we have $\E[\big]{\deg^-(p; s)^2} = \abs{N^-(p; s)} + \abs{N^-(p; s)}^2$, thus
    \[ \lim_{n \tff} \f{1}{n} \int_{\S_n \ti \T^{[0, s]}} \E[\big]{\deg^-(p; s)^2} \mdp = \sum_{k = 1}^2 \lim_{n \tff} \f{1}{n} \int_{\S_n} \abs{\Ns(\ps)}^k \d \ps \int_{\T^{0 \le}} \abs{\Nt^-(\pt; s)}^k \dpt < \ff, \]
    where in the last step we used Lemma~\ref{lem:int_of_ns}~(\ref{lem:int_of_ns:power_p}) for the spatial part and Lemma~\ref{lem:int_of_nt_pm}~(\ref{lem:int_of_nt_pm:power_p_bound}) for the temporal part, and we require that $\g < 1/2$.
    As $\Cov{\deg^-(p_1; s), \deg^-(p_2; s)} = \abs{N^-(p_1, p_2; s)}$, factorizing the second integral, and the application of Lemma~\ref{lem:int_of_ns}~(\ref{lem:int_of_ns:common_p_0}) with $m = 2$ for the spatial part and Lemma~\ref{lem:int_of_nt_pm}~(\ref{lem:int_of_nt_pm:cap}) for the temporal part leads to
    \[ \lim_{n \tff} \f{1}{n} \iint_{\big( \S_n \ti \T^{[0, s]} \big)^2} \abs{N^-(p_1, p_2; s)} \mdp[1] \mdp[2] < \ff, \]
    where we require again that $\g' < 1/2$.
    The covariance term $\Cov{S_n^{A-}(s), S_n^{B-}(s, t)}$ is determined by the common $\PP'$-points of $S_n^{A-}(s)$ and $S_n^{B-}(s, t)$:
    \[ \begin{aligned}
        \lim_{n \tff} \f{1}{n} &\Cov{S_n^{A-}(s), S_n^{B-}(s, t)} \\
        &= \lim_{n \tff} \f{1}{n} \int_{\S_n \ti \T^{[0, s]}} \int_{\S_n \ti \T^{[s, t]}} \int_{\S \ti \R} \1{p' \in N^-(p_1; s)} \1{p' \in N^-(p_2; s)} \d p' \mdp[1] \mdp[2] = I_\msf{s} \ti I_\msf{t} \\
        I_\msf{s} &:= \lim_{n \tff} \f{1}{n} \int_{\S_n^2} \abs{\Ns(\ps[1], \ps[2])} \d \ps[1] \d \ps[2] \\
        I_\msf{t} &:= \int_{\T^{[0, s]}} \int_{\T^{[s, t]}} \abs{\Nt^-(\pt[1], \pt[2]; s)} \dpt[2] \dpt[1],
    \end{aligned} \]
    where we used that the $\PP$-points connecting to the point $p'$ cannot be identical due to the disjoint sets the $\PP$-points belong to.
    By Lemma~\ref{lem:int_of_ns}~(\ref{lem:int_of_ns:common_p_0}), $I_\msf{s} < \ff$ if $\g' < 1/2$.
    The temporal integral $I_\msf{t}$ is bounded by Lemma~\ref{lem:int_of_nt_pm}~(\ref{lem:int_of_nt_pm:cap}):
    \[ I_\msf{t} \le \iint_{\big( \T^{0 \le} \big)^2} \abs{\Nt^-(\pt[1], \pt[2]; s)} \dpt[1] \dpt[2] < \ff. \]
    Thus, for some constant $c_- > 0$, the covariance function of $\So_n^-(t)$ is given by $\lim_{n \tff} \Cov{\So_n^-(s), \So_n^-(t)} = c_-$, as desired.
\end{proof}

%% file: proofs/functional_normal/lemmas/lem_bounds_of_error_terms_pm.tex
\begin{proof}[Proof of Lemma~\ref{lem:bounds_of_error_terms_pm}]
    The proof follows the same steps as the proof of Lemma~\ref{lem:bounds_of_error_terms}, and we calculate only the parts that differ from the previous calculations.
    As the spatial parts of the integrals are identical, we only need to show that the temporal integrals are finite.

    \medskip
    \noindent
    \textbf{Error term $E_1(n)$.}
    In case of the error term $E_1(n)$, if $\pw_1, \pw_2 \in \S_n \ti \T^{0 \le}$ and $\pw_3 \in \S \ti \R$, we have
    \[ \int_{\T^{0 \le}} \abs{\Nt^\pm(\pt; t)}^{m_1} \dpt \int_{\T^{0 \le}} \abs{\Nt^\pm(\pt; t)}^{m_2 + 1} \dpt < \ff \]
    due to Lemma~\ref{lem:int_of_nt_pm}~(\ref{lem:int_of_nt_pm:power_p_bound}).
    If $\pw_1, \pw_2 \in \S \ti \R$ and $\pw_3 \in \S_n \ti \T^{0 \le}$, then the argument in the proof of Lemma~\ref{lem:bounds_of_error_terms} applies, if we show first that
    \begin{equation}
        \int_{\T^{0 \le}} \abs{\Nt^\pm(\pt; t)}^2 \dpt < \ff,
        \label{eq:integral_of_Nt2_plus_minus}
    \end{equation}
    which holds due to Lemma~\ref{lem:int_of_nt_pm}~(\ref{lem:int_of_nt_pm:power_p_bound}), and second that
    \[ \int_{\T^{0 \le}} \bigg( \int_\R \1{r \in \Nt^\pm(\pt[1]; s)} \int_{\T^{0 \le}} \1{r \in \Nt^\pm(\pt[2]; t)} \dpt[2] \d r \bigg)^2 \dpt[1] < \ff. \]
    As the inner integral is a finite constant due to Lemma~\ref{lem:int_of_nt_pm}~(\ref{lem:int_of_nt_pm:power_p_bound}), the expression reduces to~\eqref{eq:integral_of_Nt2_plus_minus}, which is finite.

    \medskip
    \noindent
    \textbf{Error term $E_2(n)$.}
    Next, we turn our attention to the error term $E_2(n)$.
    Following along the same lines as in the proof of Lemma~\ref{lem:bounds_of_error_terms}, we need to show that
    \[ \iint_{(\T^{0 \le})^2} \abs{\Nt^\pm(\pt[1], \pt[2]; s)} \dpt[1] \dpt[2] < \ff, \]
    which follows from Lemma~\ref{lem:int_of_nt_pm}~(\ref{lem:int_of_nt_pm:cap}), and that~\eqref{eq:integral_of_Nt2_plus_minus} holds.

    \medskip
    \noindent
    \textbf{Error term $E_3(n)$.}
    For the error term $E_3(n)$,
    \[ \int_{\T^{0 \le}} \abs{\Nt^\pm(\pt; t)}^m \dpt < \ff \qquad \text{and} \qquad \int_\R \bigg( \int_{\T^{0 \le}} \1{r \in \Nt^+(\pt; t)} \dpt \bigg)^m \d r < \ff \]
    by Lemmas~\ref{lem:int_of_nt_pm}~(\ref{lem:int_of_nt_pm:power_p_bound}), and~\ref{lem:int_of_nt_pm}~(\ref{lem:int_of_nt_pm:cap}), respectively.
\end{proof}

%% file: proofs/functional_stable/lemmas/lem_variance_of_high_mark_edge_count.tex
\begin{proof}[Proof of Lemma~\ref{lem:variance_high_mark_edge_count}]
    After the application of Mecke's formula, similarly to~\eqref{eq:variance_Snt}, we have
    \[ \Var{S_n^{\ge, \pm}(t)} = \int_{\S_n^{u_n \le} \ti \T^{0 \le}} \E{\deg^\pm(p; t)^2} \mdp + \iint_{(\S_n^{u_n \le} \ti \T^{0 \le})^2} \Cov{\deg^\pm(p_1; t), \deg^\pm(p_2; t)} \mdp[1] \mdp[2]. \]
    Using~\eqref{eq:expectation_of_degree_square}, we have for the first term
    \[ \begin{aligned}
        \int_{\S_n^{u_n \le} \ti \T^{0 \le}} \E{\deg^\pm(p; t)^2} \mdp &= \int_{\S_n^{u_n \le} \ti \T^{0 \le}} \abs{N^\pm(p; t)} + \abs{N^\pm(p; t)}^2 \mdp \\
        &= n \big( c_1(t) \big( 1 - u_n^{1 - \g} \big) + c_2(t) \big( u_n^{- (2\g - 1)} - 1 \big) \big),
    \end{aligned} \]
    where $\g > 1/2$, $c_1(t), c_2(t) > 0$ are $t$-dependent constants, and we used Lemmas~\ref{lem:int_of_ns}~(\ref{lem:int_of_ns:power_p}) and~\ref{lem:int_of_nt_pm}~(\ref{lem:int_of_nt_pm:power_p_specific}) for the spatial and temporal parts respectively.
    For a large enough~$n$, the terms above can be bounded by $c_2(t) n u_n^{2 \g - 1}$.
    As $u_n = n^{-2/3}$ the first term is in $O(n^{1 + 2/3 (2 \g - 1)}) \subset o(n^{2 \g})$.
    Next, we calculate the covariance term by following again the steps in the proof of Lemma~\ref{lem:mean_variance_of_Snt}:
    \[ \iint_{(\S_n^{u_n \le} \ti \T^{0 \le})^2} \Cov{\deg^\pm(p_1; t), \deg^\pm(p_2; t)} \mdp[1] \mdp[2] = \iint_{(\S_n^{u_n \le} \ti \T^{0 \le})^2} \abs{N^\pm(p_1, p_2; t)} \mdp[1] \mdp[2]. \]
    Requiring $\g' < 1/2$, we bound the spatial part using Lemma~\ref{lem:int_of_ns}~(\ref{lem:int_of_ns:common_p_minus}) and the temporal part by $t + 1/2$ using Lemma~\ref{lem:int_of_nt_pm}~(\ref{lem:int_of_nt_pm:cap}).
    Thus, the covariance term is of order $O(n) \su o(n^{2 \g})$.
\end{proof}

%% file: proofs/functional_stable/lemmas/lem_variance_of_change_of_high_mark_edge_count.tex
\begin{proof}[Proof of Lemma~\ref{lem:variance_change_of_high_mark_edge_count}]
    As in~\eqref{eq:poincare}, we bound the variance term using the Poincar\'e inequality~\cite[Theorem~18.7]{poisBook}, and following~\eqref{eq:poincare_first_term_integrand}, we recognize again that $D_p(\De_n^{\ge, \pm} (s, t))$ is Poisson distributed with mean $\E{D_p(\De_n^{\ge, \pm} (s, t))} = \abs{N^\pm(p; t) \sm N^\pm(p; s)}$.
    We apply the same steps as in~\eqref{eq:poincare_first_term}, and then by the application of Lemmas~\ref{lem:int_of_ns}~(\ref{lem:int_of_ns:power_p}) and~\ref{lem:int_of_nt_pm}~(\ref{lem:int_of_nt_pm:difference_p}), we obtain
    \[ \begin{aligned}
        &\int_{\S_n^{u_n \le} \ti \T^{0 \le}} \abs[\big]{N^\pm(p; t) \sm N^\pm(p; s)} + \abs[\big]{N^\pm(p; t) \sm N^\pm(p; s)}^2 \mdp \\
        &\qquad \le n (t - s) \big( c_1 (1 - u_n^{1 - \g}) + c_2 (u_n^{-(2 \g - 1)} - 1) \big),
    \end{aligned} \]
    where $c_1, c_2 > 0$ are positive constants, and as $u_n > 0$, the bound is valid for~$\g \in (1/2, 1)$.
    To show that the above integral is in $O(n^{2 \g} (t - s)^{(1 + \eta) / 2})$, we check the order of each term:
    \[ \begin{aligned}
        n (t - s)                   \in O(n^{2 \g} (t - s)^{(1 + \eta) / 2}) \qquad &\Longleftrightarrow \qquad (t - s)^{(1 - \eta) / 2} \in O(n^{2 \g - 1}) \\
        n (t - s) u_n^{1 - \g}      \in O(n^{2 \g} (t - s)^{(1 + \eta) / 2}) \qquad &\Longleftrightarrow \qquad (t - s)^{(1 - \eta) / 2} \in O(n^{(4 \g - 1) / 3}) \\
        n (t - s) u_n^{-(2 \g - 1)} \in O(n^{2 \g} (t - s)^{(1 + \eta) / 2}) \qquad &\Longleftrightarrow \qquad (t - s)^{(1 - \eta) / 2} \in O(n^{(2 \g - 1) / 3})
    \end{aligned} \]
    as $\eta = 1/3$, $u_n = n^{-2/3}$. 
    Similarly to~\eqref{eq:poincare_second_term_integrand}, the cost operator $D_{p'}(\De_n^{\ge, \pm}(s, t))$ is also Poisson distributed with mean
    \[ \E{D_{p'}(\De_n^{\ge, \pm}(s, t))} = \int_{\S_n^{u_n \le} \ti \T^{0 \le}} \1{p' \in N^\pm(p; t) \sm N^\pm(p; s)} \mdp. \]
    Next, we follow along~\eqref{eq:poincare_second_term} and Lemmas~\ref{lem:int_of_ns}~(\ref{lem:int_of_ns:power_p_prime}) and~\ref{lem:int_of_nt_pm}~(\ref{lem:int_of_nt_pm:difference_p_prime}) yield
    \[ \int_{\S \ti \R} \E[\big]{D_{p'}(\De_n^{\ge, \pm}(s, t))} + \E[\big]{D_{p'}(\De_n^{\ge, \pm}(s, t))}^2 \d p' \in O(n (t - s)), \]
    where we used Lemmas~\ref{lem:int_of_ns}~(\ref{lem:int_of_ns:power_p_prime}) and~\ref{lem:int_of_nt_pm}~(\ref{lem:int_of_nt_pm:cap}) to bound the spatial and temporal parts, respectively.
\end{proof}

%% file: proofs/functional_stable/lemmas/lem_tightness_cumulant_term_ge.tex
\begin{proof}[Proof of Lemma~\ref{lem:cumulant_term_cases_ge}]
    The proof is similar to the proof of Lemma~\ref{lem:cumulant_term_cases} for the thin-tailed case.
    In this proof, we follow the same steps.

    \medskip
    \noindent
    \textbf{\ref{im:cumulant_case_all_in_one_ge}.}
    In \ref{im:cumulant_case_all_in_one_ge}, as in~\eqref{eq:cumulant_formula_for_rho_0}, we write the cumulant term $\k_4(\De_n^{\ge, \pm}(s, t))$ using~\cite[Proposition~3.2.1]{wiener_chaos}.
    Then, following the steps of~\eqref{eq:cumulant_bound_fourth_moment}, taking absolute values, applying triangle and Jensen's inequalities leads to $\abs{\k_4(\De_{1, n}^{\ge, \pm}(s, t))} \le c_1 n \E[\big]{\big( V_1^\pm(t) - V_1^\pm(s) \big)^4}$ with some $c_1 > 0$.
    Following the fourth moment calculations through~\eqref{eq:fourth_moment_details_1}~--~\eqref{eq:fourth_moment_details_3} with $\V_i^{u_n \le}$ and $\De_n^{\ge, \pm}$ in place of $\V_i$ and $\De_n^\pm$, respectively, we obtain that $\E{(\De_{1, n}^{\ge, \pm})^4}$ can be bounded by sums of products of integrals of the form
    \[ \int_{\V_i^{u_n \le} \ti \T^{0 \le}} \abs{\de_{s, t}(N^\pm(p))}^m \mdp \le c_3 u_n^{- (m \g - 1)_+} (t - s)^m \qquad \qquad m \in \set{1, 2, 3, 4}, \]
    which was bounded using Lemmas~\ref{lem:int_of_ns}~(\ref{lem:int_of_ns:power_p}) and~\ref{lem:int_of_nt_pm}~(\ref{lem:int_of_nt_pm:difference_p}) without imposing any constraints on the parameter~$\g \in (1/2, 1)$.
    This in turn leads to the bounds
    \[ \E[\big]{\big( \De_{1, n}^{\ge, \pm} \big)^4} \le \sum_{H_1 \dots, H_Q} \prod_{q = 1}^Q c_q u_n^{-(m_q \g - 1)_+} (t - s)^{m_q} \qquad \text{and} \qquad \sum_{q = 1}^Q m_q \le 4, \]
    where the sum is over all partitions $\set{H_1 \dots, H_Q} \preceq \set{1, \dots, 4}$ of the indices $\set{1, \dots, 4}$ into~$Q$ groups, $c_q > 0$, and $m_q$ is the number of factors in the $q$th group.
    The number of factors in the terms are denoted by $Q \in \set{1, \dots, 4}$ and $m_q \in \set{1, \dots, 4}$ for all indices $q \in \set{1, \dots, Q}$.
    Expanding the product, we obtain a similar bound regardless of the value of~$Q \le 4$:
    \[ \begin{aligned}
        \E[\big]{\big( \De_{1, n}^{\ge, \pm} \big)^4} \le \sum_{H_1 \dots, H_Q} \prod_{q = 1}^4 \Big( (t - s)^{m_q} \Big) \bigg( &c_4 + c_5 \sum_{q \le 4} u_n^{-(m_q \g - 1)_+} + c_6 \sum_{q_1 < q_2 \le 4} u_n^{-((m_{q_1} + m_{q_2}) \g - 2)_+} \\
        &+ c_7 \sum_{q_1 < q_2 < q_3 \le 4} u_n^{- ((m_{q_1} + m_{q_2} + m_{q_3}) \g - 3)_+} + c_8 u_n^{- (\sum_{q = 1}^4 m_q \g - 4)_+} \bigg),
    \end{aligned} \]
    where $c_4, \dots, c_8 > 0$ are positive constants.
    As $\sum_{q = 1}^Q m_q \le 4$, each of the terms involving~$u_n$ can be bounded by $c_9 u_n^{-(4 \g - 1)} (t - s)$ with some constant $c_9 > 0$.
    For~$n$ blocks, we need to show that
    \[ n u_n^{-(4 \g - 1)} (t - s) \in O(n^{4\g} (t - s)^{1 + \eta}) \quad \Longleftrightarrow \quad n u_n^{-(4 \g - 1)} \in O(n^{4 \g} (t - s)^\eta). \]
    Since $t - s > n^{-1/2}$ and $\g \in (1/2, 1)$, setting $\eta = 1/3$ and $u_n = n^{-2/3}$ yields $n^{1 + 2/3 (4 \g - 1)} \in O(n^{4 \g - 1 / 6})$.

    \medskip
    \noindent
    \textbf{\ref{im:cumulant_case_2_2_ge}.}
    In \ref{im:cumulant_case_2_2_ge}, we follow the same steps as in the proof of Lemma~\ref{lem:cumulant_term_cases} to conclude that $\abs{\k_4(\De_n^{\ge, \pm}(s, t))} \in O(n^{2/3 (3 \g - 1)})$ if $\g' < 1/4$, without any constraints on parameter $\g$.
    The orders of the terms $T_\msf{prod}$ and $T_\msf{cov}$ are presented in Table~\ref{tab:partitions_of_jkl}.

    \medskip
    \noindent
    \textbf{\ref{im:cumulant_case_1_3_ge}.}
    In \ref{im:cumulant_case_1_3_ge}, we introduce the notations
    \[ T_\msf{prod} := \prod_{b = 2}^q \E[\Big]{\prod_{m \in M^{(2)}_b} \De_{m, n}^{\ge, \pm}} \qquad \text{and} \qquad T_\msf{cov} := \abs[\Big]{\Cov[\Big]{\De_{i, n}^{\ge, \pm}, \prod_{m \in M^{(2)}_1} \De_{m, n}^{\ge, \pm}}} \]
    for the product and the covariance terms, respectively, and bound the cumulant $\abs{\k_4(\De_{i, n}^{\ge, \pm}$, $\De_{j, n}^{\ge, \pm}$, $\De_{k, n}^{\ge, \pm}$, $\De_{\ell, n}^{\ge, \pm})}$ as in~\eqref{eq:complex_cumulant_formula} by
    $\abs{\k_4(\De_{i, n}^{\ge, \pm}, \De_{j, n}^{\ge, \pm}, \De_{k, n}^{\ge, \pm}, \De_{\ell, n}^{\ge, \pm})} \le c_{10} \sum_{M^{(2)}_1, \dots, M^{(2)}_q} T_\msf{prod} T_\msf{cov}$.
    To ease understanding, we expanded the formula for each partition of $\set{i, j, k, \ell}$ in Table~\ref{tab:partitions_of_jkl}.
    \begin{table} [htb] \centering \caption{
        Possible partitions of the indices $\set{i, j, k, \ell}$ in \ref{im:cumulant_case_1_3_ge} and \ref{im:cumulant_case_2_2_ge}.
    } \label{tab:partitions_of_jkl}
        \begin{tabular}{|l|l|ll|ll|} \hline
            ~ & partition & $T_\msf{prod}$ & order & $T_\msf{cov}$ & order \\ \hline
            \parbox[t]{4mm}{\multirow{4}{*}{\rotatebox[origin=c]{90}{\ref{im:cumulant_case_1_3_ge}}}}
            & $\set{i}, \set{j, k, \ell}$                      & 1                                                        & $O(1)$                 & $\abs[\big]{\Cov[\big]{\De_{i, n}^{\ge, \pm}, \De_{j, n}^{\ge, \pm} \De_{k, n}^{\ge, \pm} \De_{\ell, n}^{\ge, \pm}}}$  & $O(n^{2/3 (3 \g - 1)})$ \\
            & $\set{i}, \set{j}, \set{k, \ell}$                & $\E{\De_{k, n}^{\ge, \pm} \De_{\ell, n}^{\ge, \pm}}$     & $O(n^{2/3(2 \g - 1)})$ & $\abs[\big]{\Cov[\big]{\De_{i, n}^{\ge, \pm}, \De_{j, n}^{\ge, \pm}}}$                                       & $O(1)$                  \\
            & $\set{i}, \set{j, k}, \set{\ell}$                & $\E{\De_{\ell, n}^{\ge, \pm}}$                           & $O(1)$                 & $\abs[\big]{\Cov[\big]{\De_{i, n}^{\ge, \pm}, \De_{j, n}^{\ge, \pm} \De_{k, n}^{\ge, \pm}}}$                      & $O(n^{2/3 (2 \g - 1)})$ \\
            & $\set{i}, \set{j}, \set{k}, \set{\ell}$          & $\E{\De_{k, n}^{\ge, \pm}} \E{\De_{\ell, n}^{\ge, \pm}}$ & $O(1)$                 & $\abs[\big]{\Cov[\big]{\De_{i, n}^{\ge, \pm}, \De_{j, n}^{\ge, \pm}}}$                                       & $O(1)$                  \\ \hline
            \parbox[t]{4mm}{\multirow{4}{*}{\rotatebox[origin=c]{90}{\ref{im:cumulant_case_2_2_ge}}}}
            & $\set{i, j}, \set{k, \ell}$             & 1                                                        & $O(1)$ & $\abs[\big]{\Cov[\big]{\De_{i, n}^{\ge, \pm}  \De_{j, n}^{\ge, \pm}, \De_{k, n}^{\ge, \pm} \De_{\ell, n}^{\ge, \pm}}}$ & $O(n^{2/3 (3 \g - 1)})$ \\
            & $\set{i, j}, \set{k}, \set{\ell}$       & $\E{\De_{\ell, n}^{\ge, \pm}}$                           & $O(1)$ & $\abs[\big]{\Cov[\big]{\De_{i, n}^{\ge, \pm}  \De_{j, n}^{\ge, \pm}, \De_{k, n}^{\ge, \pm}}}$                     & $O(n^{2/3 (3 \g - 1)})$ \\
            & $\set{i}, \set{j}, \set{k, \ell}$       & $\E{\De_{i, n}^{\ge, \pm}}$                              & $O(1)$ & $\abs[\big]{\Cov[\big]{\De_{j, n}^{\ge, \pm}, \De_{k, n}^{\ge, \pm}  \De_{\ell, n}^{\ge, \pm}}}$                  & $O(n^{2/3 (3 \g - 1)})$ \\
            & $\set{i}, \set{j}, \set{k}, \set{\ell}$ & $\E{\De_{i, n}^{\ge, \pm}} \E{\De_{\ell, n}^{\ge, \pm}}$ & $O(1)$ & $\abs[\big]{\Cov[\big]{\De_{j, n}^{\ge, \pm}, \De_{k, n}^{\ge, \pm}}}$                                       & $O(1)$                  \\ \hline
        \end{tabular}
    \end{table}
    We bound $\E{\De_{k, n}^{\ge, \pm} \De_{\ell, n}^{\ge, \pm}}$ using Cauchy--Schwarz inequality:
    \[ \E{\De_{k, n}^{\ge, \pm} \De_{\ell, n}^{\ge, \pm}} \le \E[\big]{\big( \De_{k, n}^{\ge, \pm} \big)^2}^{1/2} \E[\big]{\big( \De_{\ell, n}^{\ge, \pm} \big)^2}^{1/2} = \E[\big]{\big( \De_{1, n}^{\ge, \pm} \big)^2} \le \E{V_1^{\ge, \pm}(t)^2}, \]
    where we used that $\De_{k, n}^{\ge, \pm}$ and $\De_{\ell, n}^{\ge, \pm}$ are identically distributed in the second step, and that $V_i^{\ge, \pm}(t)$ is monotone in the second step.
    Following the calculations of \ref{im:cumulant_case_all_in_one_ge} for the second moment, we have
    \[ \E{(V_i^{\ge, \pm})^2} \le \prod_{q = 1}^Q c_q u_n^{- (m_q \g - 1)_+} \qquad \qquad \sum_{q = 1}^Q m_q \le 2, \]
    where $Q \in \set{1, 2}$ and $m_q \in \set{1, 2}$ for all indices $q \in \set{1, \dots Q}$.
    If $Q = 2$, then
    \[ \E{(V_i^{\ge, \pm})^2} \le c_{11} \Big( 1 + \sum_{q \le 2} u_n^{- (m_q \g - 1)_+} + \sum_{q_1 < q_2 \le 2} u_n^{- ((m_{q_1} + m_{q_2}) \g - 2)_+} \Big), \]
    with some constant $c_{11} > 0$.
    We bound again each term in the parentheses involving $u_n$ by $c_{12} u_n^{-(2 \g - 1)} \in O(n^{2/3 (2 \g - 1)})$ with some constant $c_{12} > 0$, where $u_n = n^{-2/3}$.
    For the first moments $\E{(V_i^{\ge, \pm})}$, it is easy to see that they are elements of $O(1)$.
    We move on to the covariance term $\abs{\Cov{\De_{i, n}^{\ge, \pm}, \De_{j, n}^{\ge, \pm} \De_{k, n}^{\ge, \pm} \De_{\ell, n}^{\ge, \pm}}}$ in the case $\abs{M^{(2)}_1} = 3$.
    Note that in this case the product term is of order $O(1)$.
    We apply again~\eqref{eq:conditional_covariance}:
    \[ \abs[\big]{\Cov[\big]{\De_{i, n}^{\ge, \pm}, \De_{j, n}^{\ge, \pm} \De_{k, n}^{\ge, \pm} \De_{\ell, n}^{\ge, \pm}}} = \abs[\big]{\E[\big]{\Cov[\big]{\De_{i, n}^{\ge, \pm}, \De_{j, n}^{\ge, \pm} \De_{k, n}^{\ge, \pm} \De_{\ell, n}^{\ge, \pm} \biggiven \PP}}}, \]
    and we bound the conditional covariance using bilinearity as in~\eqref{eq:bilinearity_bound}, where we use $\V_i^{u_n \le}$, $\V_j^{u_n \le}$, $\V_k^{u_n \le}$, $\V_\ell^{u_n \le}$ in place of $\V_i$, $\V_j$, $\V_k$, $\V_\ell$ for the domains of the points $P_1$, $P_2$, $P_3$, $P_4$, respectively.
    The arguments for~\eqref{eq:covterms}~--~\eqref{eq:covterms_Mecke} can be applied without any changes, and we obtain the bound
    \[ \abs[\big]{\k_4(\De_n^{\ge, \pm}(s, t))} \le c_{13} \E[\Bigg]{\sum_{a = 0}^{n - 1} \sum_{\substack{i, j, k, \ell\\ \r^\ge(i, j, k, \ell) = a}} \sum_{\substack{P_1 \in \PP \cap (\V_i^{u_n \le} \cap \T^{0 \le}), P_2 \in \PP \cap (\V_j^{u_n \le} \cap \T^{0 \le}) \\ P_3 \in \PP \cap (\V_k^{u_n \le} \cap \T^{0 \le}), P_4 \in \PP \cap (\V_\ell^{u_n \le} \cap \T^{0 \le})}} A(\bs{P}_4, \pmb{\s}_4)}. \]
    Depending on which of the points $P_2, P_3, P_4$ are identical, we apply the Mecke formula to all the cases.
    The integrals with respect to $p_3$, $p_4$ factor again leading to factors with bounds
    \[ \begin{aligned}
        \int_{\S^{u_n \le}_{[j - a, j + a + 1]} \ti \T^{0 \le}} \abs{N^\pm(p_3, \s_3)}^{m_3} \mdp &\le c_{14} a u_n^{- (m_3 \g - 1)_+} \qquad m_3 \in \set{0, 1, 2}  \\
        \int_{\S^{u_n \le}_{[j - a, j + a + 1]} \ti \T^{0 \le}} \abs{N^\pm(p_4, \s_4)}^{m_4} \mdp &\le c_{14} a u_n^{- (m_4 \g - 1)_+} \qquad m_4 \in \set{0, 1, 2},
    \end{aligned} \]
    with some constant $c_{14} \in \R$, where we used Lemmas~\ref{lem:int_of_ns}~(\ref{lem:int_of_ns:power_p}) and~\ref{lem:int_of_nt_pm}~(\ref{lem:int_of_nt_pm:power_p_bound}) with $u_- = u_n > 0$, which requires no constraint on~$\g \in (1/2, 1)$.
    For the integral with respect to $p_2$, similarly to~\eqref{eq:bound_pj_factor}, the following bound holds:
    \[ \begin{aligned}
        \abs[\big]{\k_4(\De_n^{\ge, \pm}(s, t))} &\le c_{15} \prod_{q \in \set{3, 4}} u_n^{- (m_q \g - 1)_+} \int_{\S_n^{u_n \le} \ti \T^{0 \le}} \int_{\S_n^{u_n \le} \ti \T^{0 \le}} \abs{N^\pm(p_1, \s_1)}^{m_1} \abs{N^\pm(p_1, \s_1) \cap N^\pm(p_2, \s_2)} \\
        &\hspace{5cm} \ti \abs{N^\pm(p_2, \s_2)}^{m_2} (\abs{x_1 - x_2}^{m_a} + 1) \mdp[2] \mdp[1],
    \end{aligned} \]
    where $m_1 = 0$ and $\sum_{q = 1}^{m_a} m_q \le 2$, $m_a \in \set{0, 1, 2}$.
    The temporal part can be bounded with Lemma~\ref{lem:int_of_nt_pm}~(\ref{lem:int_of_nt_pm:power_cap_power}), and the spatial part is bounded by Lemma~\ref{lem:int_of_ns}~(\ref{lem:int_of_ns:power_cap_power_minus}) with $u_- = u_n$.
    Then,
    \[ \begin{aligned}
        \abs[\big]{\k_4(\De_n^{\ge, \pm}(s, t))} &\le c_{16} n \prod_{q \in \set{3, 4}} u_n^{- (m_q \g - 1)_+} \\
        &\quad \ti \Big( u_n^{- ((1 + m_2 + m_a) \g - 1)_+ - ((1 + m_1) \g - 1)_+} + u_n^{- ((1 + m_1 + m_a) \g - 1)_+ - ((1 + m_2) \g - 1)_+} \Big),
    \end{aligned} \]
    where $c_{16} > 0$ is a large enough constant.
    To bound each term in the expansion of the product above, we would like to find the most negative exponent of~$u_n$.
    The exponent $-(m_q \g - 1)_+$ is nonzero only if $m_q \ge 2$.
    Looking into Table~\ref{tab:max_exponent_constraints}, this happens only if $P_3 = P_4$, in which case $m_1 = m_2 = 0$ and $m_a = 1$.
    Then, the exponent of~$u_n$ can be bounded by $-2 (2 \g - 1)$.
    Otherwise, $-(m_q \g - 1)_+ = 0$, and then $m_1 + m_2 + m_a \le 2$.
    In these cases, we bound the exponent of~$u_n$ by $-(3 \g - 1)$.
    Then, as $u_n = n^{-2/3}$, $T_\msf{cov} \in O(n^{2/3 (3 \g - 1)})$, and thus $T_\msf{prod} T_\msf{cov} \in O(n^{2/3 (3 \g - 1)})$ since $\E{\De_{k, n}^{\ge, \pm} \De_{\ell, n}^{\ge, \pm}}$ can only appear in $T_\msf{prod}$ if $M^{(2)}_1 = \set{j}$.
    If we look at the partition $\set{j}$, $\set{k, \ell}$ of the indices $j, k, \ell$, then the product term is or order $O(n^{2/3 (2 \g - 1)})$, and all the exponents $m_1 = m_2 = m_3 = m_4 = m_a = 0$ in the covariance term.
    Then, the covariance term is of order $O(1)$, and $T_\msf{prod} T_\msf{cov} \in O(n^{2/3 (2 \g - 1)})$.
    Following the same train of thought, we arrive to the orders of the terms $T_\msf{prod}$ and $T_\msf{cov}$ as in Table~\ref{tab:partitions_of_jkl}.
    Thus, we conclude that
    \[ \abs[\big]{\k_4(\De_n^{\ge, \pm}(s, t))} = c_{16} n T_\msf{prod} T_\msf{cov} \in O(n^{2 \g + 1/3}), \]
    as desired.
\end{proof}

%% file: proofs/functional_stable/lemmas/lem_expectation_of_increments_ge.tex
\begin{proof}[Proof of Lemma~\ref{lem:expectation_of_increments_ge}]
    Following the same argument as in the proof of Lemma~\ref{lem:expectation_of_increments}, the expectation $\E{S_n^{\ge, \pm}(t)} \le \E{S_n^\pm(t)}$ can be bounded identically by Lemmas~\ref{lem:int_of_ns}~(\ref{lem:int_of_ns:power_p}) and~\ref{lem:int_of_nt_pm}~(\ref{lem:int_of_nt_pm:power_p_specific}) for $\g \in (0, 1)$.
    As $t_k = k n^{-1/2}$, $\E[\big]{\De_n^{\ge, \pm}(t_k, t_{k + 1})} \in O(n^{1/2})$, and thus for $\g > 1/2$,
    \[ \max_{k \le \floor{n^{1/2}}} \big( n^{-\g} \E[\big]{\De_n^{\ge, \pm}(t_k, t_{k + 1})} \big) \in O(n^{- (\g - 1/2)}) \subset o(1), \]
    as required.
\end{proof}

%% file: proofs/functional_stable/propositions/prop_continuity_of_summation_functional.tex
\begin{proof}[Proof of Proposition~\ref{prop:continuity_of_summation}]
    Let $\eta \in \Nl(K)$ denote a point measure on the domain~$K$.
    Let us define $M := M(\eta) := \max_{(j_i, b_i, \ell_i) \in \eta}{\ell_i} + 1$, and set $\overline{K}_{\eps, M} := \set{(j, b, \ell) \in \overline{K}_\eps \co \ell \le M}$, such that $\eta(\any) = \eta(\any \cap \overline{K}_{\eps, M})$.
    We define a subset~$\La_{\eps, M}$ of point measures as follows:
    \[ \begin{aligned}
        \La_{\eps, M} := \big\{ \eta \in \Nl(\overline{K}_{\eps, M}) \co \quad &\eta(\overline{K}_{\eps, M}) < \ff, \quad \eta(\partial \overline{K}_{\eps, M}) = 0, \\
        &\eta \big(\set{(j, b, \ell) \in \overline{K}_{\eps, M} \co b = 0} \big) = 0, \\
        &\eta \big(\set{(j_i, b_i, \ell_i) \in \overline{K}_{\eps, M} \co \exists j \ne i \co b_i = b_j} \big) = 0, \\
        &\eta \big( \set{(j_i, b_i, \ell_i) \in \overline{K}_{\eps, M} \co \exists j \ne i \co b_i + \ell_i = b_j + \ell_j} \big) = 0, \\
        &\eta \big( \set{(j_i, b_i, \ell_i) \in \overline{K}_{\eps, M} \co \exists j \co b_i = b_j + \ell_j} \big) = 0 \big\}.
    \end{aligned} \]
    We now follow two steps:
    \begin{enumerate}
        \item first, we show that $\PP_\ff \in \La_{M(\PP_\ff)}$ almost surely;
        \item then, we show that the functional $\chi$ is continuous on $\La_{\eps, M}$.
    \end{enumerate}

    \medskip
    \noindent
    \textbf{$\PP_\ff \in \La_{M(\PP_\ff)}$ almost surely.}
    Since the expected number of points $\E{\PP_\ff(\overline{K}_\eps)}$ is finite, the number of points~$\PP_\ff(\overline{K}_\eps)$ in the domain~$\overline{K}_\eps$ is almost surely finite.
    Then, $M(\PP_\ff) < \ff$ almost surely, and then $\overline{K}_{\eps, M(\PP_\ff)}$ is almost surely compact containing all the points contributing to~$S_{n, \eps}^{(3)}(\any)$.
    Recall that in the dimension $\overline{\J}$, the set $\overline{K}_{\eps, M}$ is compactified at $\ff$.
    We check each condition in the definition of~$\La_{\eps, M}$ to show that $\P{\eta \cap \overline{K}_{\eps, M} \in \La_{\eps, M}} = 1$.
    The coordinates~$J_i$ of the points $(J_i, B_i, L_i) \in \PP_\ff$ are iid random variables with distribution $\eps \nu([\wt{c} \eps^\g, \ff) \cap \any)$.
    Furthermore, $(B_i, L_i)$ are also iid with joint distribution
    \[ \begin{aligned}
        &\P{B_i \le b, L_i \le \ell} \\
        &\qquad = \left\{ \begin{array}{ll} \tfrac{1}{2} \int_{-\ell}^b \int_{-b'}^\ell \, \mbb{P}_L(\d \ell') \d b' = \tfrac{1}{2} \big( \e^b - (1 + b + \ell) \e^{-\ell} \big) & \text{if } b \in [-\ell, 0] \\[0.4cm] \tfrac{1}{2} \int_{-\ell}^0 \int_{-b'}^\ell \, \mbb{P}_L(\d \ell') \d b' + \tfrac{1}{2} \int_0^b \int_0^\ell \, \mbb{P}_L(\d \ell') \d b' = \tfrac{1}{2} \big( 1 + b - (1 + b + \ell) \e^{- \ell} \big) & \text{if } b \in [0, 1]. \end{array} \right.
    \end{aligned} \]
    Note that at $b = 0$, both parts of the distribution are equal, and thus the distribution of the points $(J_i, B_i, L_i)$ is continuous without any atoms.
    \begin{itemize}
        \item The first condition states that the total number of points we consider is finite.
        This holds almost surely since the expected number of points $\nu([\wt{c} \eps^\g, \ff)) c_1 = \eps c_1$ in the domain~$\overline{K}_{\eps, M}$ is finite, for some constant $c_1 > 0$.
        \item The second and third conditions require that no points can be born or die at the boundaries of the time interval~$[0, 1]$, and no points can have a lifetime of exactly~$M$.
        As the distribution of the points is continuous, the second and third conditions hold almost surely as well since the corresponding sets have measure~$0$.
        \item The last three conditions require that all points are born and die at different times, which hold almost surely since the distributions of the coordinates~$B_i$, $L_i$ are continuous:
        \[ \P[\bigg]{\bigcup_{i < j} \set{B_i = B_j}} \le \sum_{i < j} \P[\big]{\set{B_i = B_j}} = 0, \]
        which can be applied to the final two conditions as well.
    \end{itemize}
    Thus, $\P{\PP_\ff \in \La_{\eps, M}} = 1$.

    \medskip
    \noindent
    \textbf{$\chi$ is continuous on $\La_{\eps, M}$.}
    Next, we show that if $\eta \in \La_{\eps, M}$, then~$\chi$ is continuous at~$\eta$, i.e., if $\eta_n \xrightarrow{v} \eta$ is a converging sequence of point measures in the vague topology, then $\chi(\eta_n) \to \chi(\eta)$ in $D([0, 1], \R)$.
    Note that $\eta_n \xrightarrow{v} \eta$ if and only if for all continuous functions~$f$ with compact support $\abs{\int f \d \eta_n - \int f \d \eta} \to 0$.
    Also note that the function $a \mapsto \nu([a, \ff))$ from $(0, \ff)$ to $(0, \ff)$ is continuous by Lemma~\ref{lem:convergence_to_levy}.
    Thus, the restriction of the jumps using the indicator $\1{J \ge \wt{c} \eps^\g}$ is almost surely continuous.
    Consider two point measures $\eta_1, \eta_2 \in \La_{\eps, M}$.
    We would like to show that if~$\eta_1$,~$\eta_2$ are close to each other, then the summation functionals~$\chi(\eta_1)$,~$\chi(\eta_2)$ are close in the Skorokhod space $D([0, 1], \R)$.
    Note that~$\overline{K}_{\eps, M}$ is compact, and $\eta \in \La_{\eps, M}$ almost surely.
    We label the~$k \in \N$ number of points $(j_i, b_i, \ell_i) \in \mrm{supp}(\eta)$ such that $\eta(\any) = \sum_{i = 1}^k \de_{(J_i, B_i, L_i)}(\any)$ and $B_1 < \cdots < B_k$, which is possible since $\eta \in \La_{\eps, M}$.
    Next, we show that there exists an index~$n_0 \in \N$ such that for all $n \ge n_0$, there exists a labeling of the points $(j_i^{(n)}, b_i^{(n)}, \ell_i^{(n)}) \in \eta_n$ such that
    \[ \eta_n(\any \cap \overline{K}_{\eps, M}) = \sum_{i = 1}^k \de_{(j_i^{(n)}, b_i^{(n)}, \ell_i^{(n)})} (\any) \qquad \text{and} \qquad (j_i^{(n)}, b_i^{(n)}, \ell_i^{(n)}) \to (j_i, b_i, \ell_i), \]
    for all indices $i \in \set{1, \dots, k}$.
    We pick a~$\de > 0$ small enough such that $G_\de(P) \cap G_\de(P') = \es$ for all distinct pair of points $P \ne P'$ in $\eta$ and $G_\de(P) \su \overline{K}_{\eps, M}$, where $G_\de(P)$ is the ball of radius~$\de$ around a point~$P \in \eta$.
    For a large enough~$n$, we have that $(P_i^{(n)}) \in G_\de(P_i)$ for all $i \in \set{1, \dots, k}$.
    Then, by~\cite[Theorem~3.2]{heavytails}, $\eta_n(G_\de(P_i)) \to \eta(G_\de(P_i))$ for all indices~$i$.
    Since $\eta_n, \eta$ are integer valued, $\eta_n(G_\de(P_i)) = \eta(G_\de(P_i))$ for~$n \ge n_0$ with some large enough~$n_0$.
    Thus, the total number of points~$k$ in the set~$\overline{K}_{\eps, M}$ is the same for all measures~$\set{\eta_n}$ and~$\eta$ if $n \ge n_0$, and there is a possible labeling of the temporal coordinates $T := \big( \set{0, 1} \cup \bigcap_{i = 1}^k \set{b_i, b_i + \ell_i} \cap [0, 1] \big)$ such that $0 = \t_1 < \cdots < \t_{k'} = 1$, where $\t_i$ is the $i$th smallest element in the set~$T$, and the number of elements in the set~$T$ is $k'$.
    We denote by $\t_i^{(n)}$ the corresponding elements of the set~$T^{(n)} := \big( \set{0, 1} \cup \bigcap_{i = 1}^k \set{b_i^{(n)}, b_i^{(n)} + \ell_i^{(n)}} \cap [0, 1] \big)$.

    Next, we define a sequence of homeomorphisms $\la_n \co [0, 1] \to [0, 1]$ by $\la_n(\t_i^{(n)}) := \t_i$ for all $i \in \set{1, \dots, k'}$, and $\la_n(\any)$ is defined by linear interpolation between these points.
    Note that the domain of $\la_n$ was chosen so that it is defined for all time instants $\set{b_i} \cup \set{b_i + \ell_i}$, $i \in \set{1, \dots, k'}$.
    The graph of the homeomorphism $\la_n$ is shown in Figure~\ref{fig:lambda}.
    \begin{figure}[htb] \centering
        \begin{tikzpicture}[>=latex, scale=0.2]
            \draw[thick,->] (0, 0) -- (30,  0);
            \draw[thick,->] (0, 0) -- ( 0, 27);

            \draw [-] ( 7, -4) -- (15, -4);
            \node at  ( 7, -4) [circle,fill,inner sep=1pt]{};
            \node at  (15, -4) [circle,fill,inner sep=1pt]{};
            \draw [-] (18, -4) -- (22, -4);
            \node at  (18, -4) [circle,fill,inner sep=1pt]{};
            \node at  (22, -4) [circle,fill,inner sep=1pt]{};
            \draw [-] (12, -5) -- (23, -5);
            \node at  (12, -5) [circle,fill,inner sep=1pt]{};
            \node at  (23, -5) [circle,fill,inner sep=1pt]{};

            \draw [-] (-6,  6) -- (-6, 12);
            \node at  (-6,  6) [circle,fill,inner sep=1pt]{};
            \node at  (-6, 12) [circle,fill,inner sep=1pt]{};
            \draw [-] (-6, 16) -- (-6, 21);
            \node at  (-6, 16) [circle,fill,inner sep=1pt]{};
            \node at  (-6, 21) [circle,fill,inner sep=1pt]{};
            \draw [-] (-7, 11) -- (-7, 23);
            \node at  (-7, 11) [circle,fill,inner sep=1pt]{};
            \node at  (-7, 23) [circle,fill,inner sep=1pt]{};

            \node at ( 3,  0) [font=\Large]{+};
            \node at ( 3, -2) {$-M$};
            \node at ( 8,  0) [font=\Large]{+};
            \node at ( 8, -2) {$0$};
            \node at (20,  0) [font=\Large]{+};
            \node at (20, -2) {$1$};
            \node at (27,  0) [font=\Large]{+};
            \node at (27, -2) {$M + 1$};

            \node at ( 0  ,  8) [font=\Large]{+};
            \node at (-1.5,  8) {$0$};
            \node at ( 0  , 20) [font=\Large]{+};
            \node at (-1.5, 20) {$1$};
            \node at ( 0  ,  4) [font=\Large]{+};
            \node at (-2.5,  4) {$-M$};
            \node at ( 0  , 25) [font=\Large]{+};
            \node at (-3.5, 25) {$M + 1$};


            \draw [dashed] ( 8,  0) -- ( 8,  8);
            \draw [dashed] ( 0,  8) -- ( 8,  8);
            \draw [dashed] (20,  0) -- (20, 20);
            \draw [dashed] ( 0, 20) -- (20, 20);

            \draw [dotted, line width=0.15mm] (-7, 11) -- (12, 11) -- (12, -5);
            \draw [dotted, line width=0.15mm] (-6, 12) -- (15, 12) -- (15, -4);
            \draw [dotted, line width=0.15mm] (-6, 16) -- (18, 16) -- (18, -4);

            \node at ( 8,  8) [circle,fill,inner sep=1pt]{};
            \node at (12, 11) [circle,fill,inner sep=1pt]{};
            \node at (15, 12) [circle,fill,inner sep=1pt]{};
            \node at (18, 16) [circle,fill,inner sep=1pt]{};
            \node at (20, 20) [circle,fill,inner sep=1pt]{};
            \draw [thick] (8, 8) -- (12, 11) -- (15, 12) -- (18, 16) -- (20, 20);

            \node at (10, 5) {$\la_n$};
            \draw [] (10, 6) -- (10, 9.5);
        \end{tikzpicture}
        \caption{
            Graph of the temporal transformation~$\la_n$.
        }
        \label{fig:lambda}
    \end{figure}
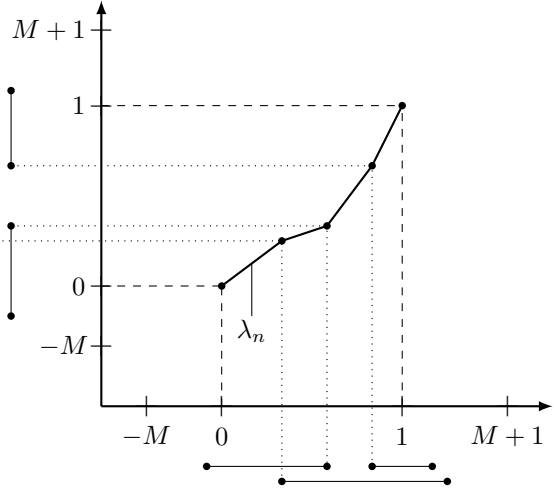
    The next lemma bounds the Skorokhod distance $d_\msf{Sk}(\chi(\eta_n), \chi(\eta))$ of the functionals $\chi(\eta_n)$ and~$\chi(\eta)$, which was defined in Definition~\ref{def:skorokhod_metric}.
    \begin{lemma}[Skorokhod distance between~$\chi(\eta_n)$ and~$\chi(\eta)$]\label{lem:bound_of_skorokhod_distance}
        The Skorokhod distance between~$\chi(\eta_n)$ and~$\chi(\eta)$ is bounded by
        \[ d_\msf{Sk}(\chi(\eta_n), \chi(\eta)) \le (c k + 3) \de, \]
        where~$c$ is a constant depending on the distribution of the points in~$\eta$.
    \end{lemma}
    \noindent
    Thus, by Lemma~\ref{lem:bound_of_skorokhod_distance}, $\chi_{\eta_n} \to \chi(\eta)$ in $D([0, 1], \R)$ almost surely.
\end{proof}

%% file: proofs/functional_stable/lemmas/lem_bound_of_skorokhod_distance.tex
\begin{proof}[Proof of Lemma~\ref{lem:bound_of_skorokhod_distance}]
    The Skorokhod distance can be bounded by
    \[ d_\msf{Sk}(\chi(\eta_n), \chi(\eta)) \le \norm{\la_n - I} \vee \norm{\chi(\eta_n) \circ \la_n^{-1} - \chi(\eta)}. \]
    For the first term, we have
    \[ \begin{aligned}
        \norm{\la_n - I} &= \sup_{i \in \set{1, \dots, k'}} \sup_{s \in [\t_i^{(n)}, \t_{i + 1}^{(n)}]} \abs{\la_n(s) - s} = \sup_{i \in \set{1, \dots, k'}} \sup_{s \in [\t_i^{(n)}, \t_{i+1}^{(n)}]} \abs[\bigg]{\t_i + \f{\t_{i + 1} - \t_i}{\t_{i + 1}^{(n)} - \t_i^{(n)}} (s - \t_i^{(n)}) - s} \\
        &= \sup_{i \in \set{1, \dots, k'}} \sup_{s \in [\t_i^{(n)}, \t_{i+1}^{(n)}]} \abs[\bigg]{\t_i - \t_i^{(n)} + \Big( \f{\t_{i + 1} - \t_i}{\t_{i + 1}^{(n)} - \t_i^{(n)}} - 1 \Big) (s - \t_i^{(n)})}.
    \end{aligned} \]
    Next, we apply the triangle inequality, and use that $s - \t_i^{(n)} \le \t_{i + 1}^{(n)} - \t_i^{(n)}$:
    \[ \norm{\la_n - I} \le \sup_{i \in \set{1, \dots, k'}} \bigg( \abs[\big]{\t_i - \t_i^{(n)}} + \abs[\Big]{\f{\t_{i + 1} - \t_i}{\t_{i + 1}^{(n)} - \t_i^{(n)}} - 1} (\t_{i + 1}^{(n)} - \t_i^{(n)}) \bigg) \le 3 \de. \]
    For the second term,
    \[ \begin{aligned}
        &\norm{\chi(\eta_n) \circ \la_n^{-1} - \chi(\eta)} \\
        &\qquad = \norm[\Big]{\sum_{(j^{(n)}, b^{(n)}, \ell^{(n)}) \in \eta_n} j^{(n)} (\la_n^{-1}(t) - b^{(n)}) \1{j^{(n)} \ge \wt{c} \eps^\g} \1{b^{(n)} \le \la_n^{-1}(t) \le b^{(n)} + \ell^{(n)}} \\
        &\qquad \qquad - \sum_{(j, b, \ell) \in \eta} j (t - b) \1{j \ge \wt{c} \eps^\g} \1{b \le t \le b + \ell}} \\
        &\qquad = \norm[\Big]{\sum_{i = 1}^k j_i^{(n)} (\la_n^{-1}(t) - b_i^{(n)}) \1{b_i^{(n)} \le \la_n^{-1}(t) \le b_i^{(n)} + \ell_i^{(n)}} - j_i (t - b_i) \1{b_i \le t \le b_i + \ell_i}}.
    \end{aligned} \]
    Using the identity $a b c - a' b' c' = (a - a') b c + a' (b - b') c + a' b' (c - c')$ for the above expression, we have
    \[ \begin{aligned}
        \norm{\chi(\eta_n) \circ \la_n^{-1} - \chi(\eta)} = \sup_{t \in [0, 1]} \abs[\bigg]{\sum_{i = 1}^k &\big( j_i^{(n)} - j_i \big) \big( \la_n^{-1}(t) - b_i^{(n)} \big) \1{b_i^{(n)} \le \la_n^{-1}(t) \le b_i^{(n)} + \ell_i^{(n)}} \\
        &+ j_i \Big( (\la_n^{-1}(t) - b_i^{(n)}) - (t - b_i) \Big) \1{b_i \le t \le b_i + \ell_i} \\
        &+ j_i (t - b_i) \Big( \1[\big]{b_i^{(n)} \le \la_n^{-1}(t) \le b_i^{(n)} + \ell_i^{(n)}} - \1[\big]{b_i \le t \le b_i + \ell_i} \Big)}.
    \end{aligned} \]
    Next, we apply the triangle inequality, and examine the three terms separately.
    For the first term, $\la_n^{-1}(t) b_i^{(n)} < \ff$ whenever the indicator is $1$, and $\abs{j_i^{(n)} - j_i} < \de$.
    In the second term, $j_i < \ff$ almost surely, thus we set $c_1 := \sup_{i \in \set{1, \dots, k}} j_i < \ff$.
    Then, $\abs{(\la_n^{-1}(t) - b_i^{(n)}) - (t - b_i)} = \abs{\la_n^{-1}(t) - t - (b_i^{(n)} - b_i)} \le \abs{\la_n^{-1}(t) - t} - \abs{b_i^{(n)} - b_i}$.
    We have seen that $\abs{\la_n^{-1}(t) - t} \le 3 \de$ above, and $\abs{b_i^{(n)} - b_i} \le \de$ by assumption for all indices~$i$, thus the second term is bounded by $4 c_1 \de$.
    In the last term, $j_i (t - b_i) < \ff$, and we need to show that the indicators are the same.
    \begin{itemize}
        \item Note that $b_i^{(n)} < 0$ if and only if $b_i < 0$.
        If $b_i < 0$, then $\1{b_i^{(n)} \le \la_n^{-1}(t)} = \1[\big]{b_i \le t} = 1$.
        Otherwise, $\1{b_i^{(n)} \le \la_n^{-1}(t)} = \1{\la_n(b_i^{(n)}) \le t} = \1{b_i \le t}$.
        \item Similarly, $b_i^{(n)} + \ell_i^{(n)} > 1$ if and only if $b_i + \ell_i > 1$.
        If $b_i + \ell_i > 1$, then $\1{\la_n^{-1}(t) \le b_i^{(n)} + \ell_i^{(n)}} = \1[\big]{t \le b_i + \ell_i} = 1$.
        Otherwise, $\1{\la_n^{-1}(t) \le b_i^{(n)} + \ell_i^{(n)}} = \1{t \le \la_n(b_i^{(n)} + \ell_i^{(n)})} = \1{t \le b_i + \ell_i}$.
    \end{itemize}
    Then, the indicators are the same, and the last term is~$0$ as well.
    Then,
    \[ \norm{\chi(\eta_n) \circ \la_n^{-1} - \chi(\eta)} = \sup_{t \in [0, 1]} \sum_{i = 1}^k c_2 \de = c_2 k \de \]
    for some constant $c_2 > 0$.
    Considering the above two terms, we have
    \[ \inf_\la \Big( \norm{\la - I} \vee \norm{\chi(\eta_n) \circ \la^{-1} - \chi(\eta)} \Big) \le (c_2 k + 3) \de, \]
    as required.
\end{proof}

%% file: proofs/functional_stable/lemmas/lem_so_cauchy_probability.tex
\begin{proof}[Proof of Lemma~\ref{lem:So_eps_cauchy_probability}]
    To ease notations, we set $\eps_n' := \wt{c} \eps_n^\g$ for all indices~$n \ge 0$.
    Furthermore, without loss of generality, we assume that $n \le m$.
    Then, similarly to Step 2, we write
    \[ \begin{aligned}
        E &:= \lim_{N \tff} \sup_{N \le n \le m} \mbb{P} \bigg( \norm[\Big]{\sum_{(J, B, L) \in \PP_\ff} J (t - B) \1{J \in [\eps'_m, \eps'_n]} \1{B \le t \le B + L} \\
        &\hspace{3cm} - c \big( (\eps_m')^{-(1 / \g - 1)} - (\eps_n')^{-(1 / \g - 1)} \big)} \ge 2 \de \bigg),
    \end{aligned} \]
    where $c := \wt{c}^{1 / \g} / (1 - \g)$.
    We use the usual \quote{plus-minus decomposition} again with the notation
    \[ \begin{aligned}
        \So_\eps^{\ast, +}(t) &:= \sum_{(J, B, L) \in \PP_\ff} J (((B + L) \w t) - B) \1{J \ge \eps'} \1{B \le t} \1{B + L \ge 0} - c (\eps')^{-(1 / \g - 1)} (t + 1) \\
        \So_\eps^{\ast, -}(t) &:= \sum_{(J, B, L) \in \PP_\ff} J L \1{J \ge \eps'} \1{B + L \in [0, t]} - c (\eps')^{-(1 / \g - 1)},
    \end{aligned} \]
    where we applied Lemma~\ref{lem:int_of_nt_pm}~(\ref{lem:int_of_nt_pm:power_p_specific}) to calculate the expectations.
    Then, we have that $\So_\eps^\ast = \So_\eps^{\ast, +} - \So_\eps^{\ast, -}$, and the triangle inequality gives $E \le E^+ + E^-$, where
    \[ E^+ := \lim_{N \tff} \sup_{N \le n \le m} \P[\Big]{\norm[\Big]{\So_{\eps_n}^{\ast, +} - \So_{\eps_m}^{\ast, +}} \ge \de} \qquad \text{and} \qquad E^- := \lim_{N \tff} \sup_{N \le n \le m} \P[\Big]{\norm[\big]{\So_{\eps_n}^{\ast, -} - \So_{\eps_m}^{\ast, -}} \ge \de}. \]
    Similarly to Step 3, $\norm{\So_{\eps_n}^{\ast, -} - \So_{\eps_m}^{\ast, -}}$ in the term $E^-$ is a martingale since $\So_{\eps_n}^{\ast, -}$ is a sum of independent random variables, and the expectation is equal to zero by the independence of the points $(J, B, L) \in \PP_\ff$ in the definition of $\So_{\eps_n}^{\ast, -}$.
    Then, by Doob's martingale inequality, we have
    \[ \begin{aligned}
        E^- &\le \lim_{N \tff} \sup_{N \le n \le m} \de^{-1} \E[\Big]{\abs[\big]{\So_{\eps_n}^{\ast, -}(1) - \So_{\eps_m}^{\ast, -}(1)}} \\
        &\le \lim_{N \tff} \sup_{N \le n \le m} \de^{-1} \Var[\Big]{\sum_{(J, B, L) \in \PP_\ff} J L \1{J \in [\eps'_m, \eps'_n]} \1{B + L \in [0, 1]}}^{1/2},
    \end{aligned} \]
    where we applied the Cauchy--Schwarz inequality in the second step.
    The variance is calculated using the independence of the points, and then Mecke's formula gives
    \[ \begin{aligned}
        \Var[\Big]{\sum_{(J, B, L) \in \PP_\ff} J L \1{J \in [\eps'_m, \eps'_n]} \1{B + L \le 1}} &= \int_{\T^{[0, 1]}} \int_{\eps'_m}^{\eps'_n} j^2 \ell^2 \nu(\d j) \dbdell \\
        &= \int_{\eps'_m}^{\eps'_n} j^2 \nu(\d j) = \f{2 \wt{c}^{1 / \g} \g}{2 \g - 1} \big( (\eps_n')^{2 - 1 / \g} - (\eps_m')^{2 - 1 / \g} \big),
    \end{aligned} \]
    where in the second step we used that the temporal part of the integral is~$2$ by Lemma~\ref{lem:int_of_nt_pm:power_p_specific}.
    Note that the variance is similar to the variance given in the proof of Lemma~\ref{lem:convergence_of_So_eps_fixed_time}.
    Then, $E^- = 0$ since $\g > 1/2$.
    For the $E^+$ term, we have
    \[ \begin{aligned}
        &E^+ = \lim_{N \tff} \sup_{N \le n \le m} \mbb{P} \bigg( \sup_{t \in [0, 1]} \abs[\bigg]{\sum_{(J, B, L) \in \PP_\ff} J (t - B) \1{J \in [\eps'_m, \eps'_n]} \1{B \le t} \1{B + L \ge 0} \\
        &\hspace{4.5cm} - c \big( (\eps_m')^{-(1 / \g - 1)} - (\eps_n')^{-(1 / \g - 1)} \big) (t + 1)} \ge \de \bigg)
    \end{aligned} \]
    We use the same strategy as in Step 3, and similarly to the Definition~\eqref{eq:definition_of_H}, we introduce $H_{\eps_m', \eps_n'}'(t)$ and write the edge count as an integral:
    \[ \begin{aligned}
        H_{\eps_m', \eps_n'}'(t) &:= \sum_{(J, B, L) \in \PP_\ff} J \1{J \in [\eps_m', \eps_n']} \1{t \in [B, B + L]} \1{B + L \ge 0} \\
        \So_\eps^{\ast, +}(t) &= \So_\eps^{\ast, +}(0) + \int_0^t H_{\eps_m', \eps_n'}'(t') - \E{H_{\eps_m', \eps_n'}'(t')} \d t',
    \end{aligned} \]
    for and $0 < \eps_m \le \eps_n$.
    Then, following Step 3, we applied triangle inequality twice to bound $E^+$:
    \[ \begin{aligned}
        E^+ &= \lim_{N \tff} \sup_{N \le n \le m} \P[\bigg]{\sup_{t \in [0, 1]} \abs[\bigg]{\So_{\eps_m'}^\ast(0) - \So_{\eps_n'}^\ast(0) + \int_0^t H_{\eps_m', \eps_n'}'(t') - \E[\big]{H_{\eps_m', \eps_n'}'(t')} \d t'} \ge \de} \\
        &\le \lim_{N \tff} \sup_{N \le n \le m} \bigg( \P[\Big]{\abs[\Big]{\So_{\eps_m'}^\ast(0) - \So_{\eps_n'}^\ast(0)} \ge \de / 2} + \P[\Big]{\int_0^1 \abs[\Big]{H_{\eps_m', \eps_n'}'(t') - \E{H_{\eps_m', \eps_n'}'(t')}} \d t' \ge \de / 2} \bigg),
    \end{aligned} \]
    where, following Step 3, we applied triangle inequality twice in the last step, and used that $\sup_{t \in [0, 1]} \int_0^t \abs{\any} \le \int_0^1 \abs{\any}$.
    For the first term, Chebyshev's inequality yields
    \[ \begin{aligned}
        &\P[\Big]{\abs[\big]{\So_{\eps_m}^\ast(0) - \So_{\eps_n}^\ast(0)} \ge \de / 2} \le 4 \de^{-2} \Var[\big]{S_{\eps_m}^\ast(0) - S_{\eps_n}^\ast(0)} \\
        &\qquad = 4 \de^{-2} \int_{\T_{\le 0}^{0 \le}} \int_{\eps_m'}^{\eps_n'} j^2 b^2 \nu(\d j) \dbdell = \f{16 \wt{c}^{1/\g} \g \de^{-2}}{2 \g - 1} \big( (\eps_n')^{2 - 1 / \g} - (\eps_m')^{2 - 1 / \g} \big),
    \end{aligned} \]
    where we applied the Mecke formula to calculate the variance in the second step, and in the last step we used that the integral of the temporal part is $2$ by Lemma \ref{lem:int_of_nt_pm}~(\ref{lem:int_of_nt_pm:power_p_specific}).
    This term converges to~$0$ as $n, m \to 0$ since $\g > 1/2$.
    For the second term, we follow again the same arguments as in Step 3, and first we apply Markov's, and then Jensen's inequality:
    \[ \begin{aligned}
        \P[\Big]{\int_0^1 \abs[\Big]{H_{\eps_m', \eps_n'}'(t') - \E{H_{\eps_m', \eps_n'}'(t')}} \d t' \ge \de / 2} &\le 2 \de^{-1} \E[\bigg]{\int_0^1 \abs[\Big]{H_{\eps_m', \eps_n'}'(t') - \E{H_{\eps_m', \eps_n'}'(t')}} \d t'} \\
        &\le 2 \de^{-1} \int_0^1 \Var{H_{\eps_m', \eps_n'}'(t')}^{1/2} \d t'.
    \end{aligned} \]
    We calculate the variance using Mecke's formula:
    \[ \Var{H_{\eps_m', \eps_n'}'(t')} = \int_{\T^{0 \le}} \int_{\eps_m'}^{\eps_n'} j^2 \1{b \le t \le b + \ell} \, \nu(\d j) \dbdell = \f{2 \wt{c}^{1/\g} \g}{2 \g - 1} \big( (\eps_n')^{2 - 1/\g} - (\eps_m')^{2 - 1/\g} \big) (t' + 1), \]
    where the temporal integral is $t' + 1$ by Lemma~\ref{lem:int_of_nt_pm:power_p_specific}.
    Then,
    \[ \P[\Big]{\int_0^1 \abs[\Big]{H_{\eps_m', \eps_n'}'(t') - \E{H_{\eps_m', \eps_n'}'(t')}} \d t' \ge \de / 2} \le \f{4}{3} \de^{-1} \Big( \f{2 \wt{c}^{1/\g} \g}{2 \g - 1} \big( (\eps_n')^{2 - 1/\g} - (\eps_m')^{2 - 1/\g} \big) \Big)^{1/2} \big( 2^{3/2} - 1 \big), \]
    which converges to~$0$ as $n, m \to \ff$ since $\g > 1/2$.
\end{proof}

%% file: proofs/functional_stable/propositions/prop_so_cauchy_uniform.tex
\begin{proof}[Proof of Proposition~\ref{prop:So_cauchy_uniform}]
    Let~$\wt{\La}$ be a set with $\P{\wt{\La}} = 1$ such that if $\om \in \wt{\La}$, then the sequence $\set{\So_{\eps_k}^\ast(\any, \om), k \ge 1}$ is Cauchy with respect to the uniform convergence on $[0, 1]$.
    We prove the Cauchy property by showing $\sup_{m, n \ge N} \norm[\big]{\So_{\eps_m}^\ast - \So_{\eps_n}^\ast} \to 0$ almost surely as $N \to \ff$.
    Since the supremum is non-increasing in~$N$, it is enough to show that it converges to~$0$ in probability~\cite[Equation~(5.45)]{heavytails}.
    Then, with $\de' := \de / 4$, we have
    \[ \begin{aligned}
        \lim_{N \tff} \P[\Big]{\sup_{m, n \ge N} \norm[\big]{\So_{\eps_m}^\ast - \So_{\eps_n}^\ast} \ge \de} &= \lim_{N \tff} \lim_{M \tff} \P[\Big]{\sup_{m, n \in \set{N, \dots, M}} \norm[\big]{\So_{\eps_m}^\ast - \So_{\eps_n}^\ast} \ge \de} \\
        &\le \lim_{N \tff} \lim_{M \tff} \P[\Big]{\sup_{n \in \set{N, \dots, M}} \norm[\big]{\So_{\eps_n}^\ast - \So_{\eps_N}^\ast} \ge 2 \de'},
    \end{aligned} \]
    where in the last step we used that by the triangle inequality, $\norm[\big]{\So_{\eps_m}^\ast - \So_{\eps_n}^\ast} \le \norm[\big]{\So_{\eps_m}^\ast - \So_{\eps_N}^\ast} + \norm[\big]{\So_{\eps_N}^\ast - \So_{\eps_n}^\ast}$, but then $\sup_{m, n \in \set{N, \dots, M}} \norm[\big]{\So_{\eps_m}^\ast - \So_{\eps_n}^\ast} \le 2 \sup_{n \in \set{N, \dots, M}} \norm[\big]{\So_{\eps_n}^\ast - \So_{\eps_N}^\ast}$.
    Again, by the triangle inequality, we have that $\norm{\So_{\eps_M}^\ast - \So_{\eps_N}^\ast} \ge \norm{\So_{\eps_i}^\ast - \So_{\eps_N}^\ast} - \norm{\So_{\eps_M}^\ast - \So_{\eps_i}^\ast}$, which implies that $\set{\norm{\So_{\eps_i}^\ast - \So_{\eps_N}^\ast} > 2 \de', \norm{\So_{\eps_M}^\ast - \So_{\eps_i}^\ast} \le \de'} \su \set{\norm{\So_{\eps_M}^\ast - \So_{\eps_N}^\ast} > \de'}$ for all~$i$, and then
    \[ \bigcup_{i = N + 1}^M \set[\bigg]{\max_{k \in \set{N, \dots, i - 1}} \Big( \norm[\big]{\So_{\eps_k}^\ast - \So_{\eps_N}^\ast} \Big) \le 2 \de', \norm[\big]{\So_{\eps_i}^\ast - \So_{\eps_N}^\ast} > 2 \de', \norm[\big]{\So_{\eps_M}^\ast - \So_{\eps_i}^\ast} \le \de'} \su \set[\Big]{\norm[\big]{\So_{\eps_M}^\ast - \So_{\eps_N}^\ast} > \de'}. \]
    Note that the union is a disjoint union, since for all $k \in \set{N \dots, i - 1}$, $\So_{\eps_k}^\ast$ is within a distance of $2 \de'$ from $\So_{\eps_N}^\ast$, and $\So_{\eps_i}^\ast$ is further from $\So_{\eps_N}^\ast$ than $\de'$.
    Thus, we have
    \[ \begin{aligned}
        &\P[\Big]{\norm[\big]{\So_{\eps_M}^\ast - \So_{\eps_N}^\ast} > \de'} \\
        &\qquad \ge \sum_{i = N + 1}^M \P[\bigg]{\max_{k \in \set{N, \dots, i - 1}} \Big( \norm[\big]{\So_{\eps_k}^\ast - \So_{\eps_N}^\ast} \Big) \le 2 \de', \norm[\big]{\So_{\eps_i}^\ast - \So_{\eps_N}^\ast} > 2 \de', \norm[\big]{\So_{\eps_M}^\ast - \So_{\eps_i}^\ast} \le \de'} \\
        &\qquad = \sum_{i = N + 1}^M \P[\bigg]{\max_{k \in \set{N, \dots, i - 1}} \Big( \norm[\big]{\So_{\eps_k}^\ast - \So_{\eps_N}^\ast} \Big) \le 2 \de', \norm[\big]{\So_{\eps_i}^\ast - \So_{\eps_N}^\ast} > 2 \de'} \P[\Big]{\norm[\big]{\So_{\eps_M}^\ast - \So_{\eps_i}^\ast} \le \de'},
    \end{aligned} \]
    where in the last step we used the independence of the points whose jumps are in $[\eps_i, \eps_N]$ and those whose jumps are between $[\eps_M, \eps_i]$.
    Then, for the last term, we have $\P[\big]{\norm[\big]{\So_{\eps_M}^\ast - \So_{\eps_i}^\ast} \le \de'} \to 1$ as $N \to \ff$, where we applied Lemma~\ref{lem:So_eps_cauchy_probability}.
    Now, we choose $N \ge N_0$ large enough such that $\P{\max_{i \in \set{N, \dots, M}} (\norm{\So_{\eps_M}^\ast - \So_{\eps_i}^\ast}) \le \de'} \ge 1/2$.
    Then, continuing the lower bound for $\P{\norm{\So_{\eps_M}^\ast - \So_{\eps_N}^\ast > \de'}}$, we have
    \[ \begin{aligned}
        2 \P{\norm{\So_{\eps_M}^\ast - \So_{\eps_N}^\ast} > \de'} &\ge \sum_{i = N + 1}^M \P[\bigg]{\max_{k \in \set{N, \dots, i - 1}} \Big( \norm[\big]{\So_{\eps_k}^\ast - \So_{\eps_N}^\ast} \Big) \le 2 \de', \norm[\big]{\So_{\eps_i}^\ast - \So_{\eps_N}^\ast} > 2 \de'} \\
        &= \P[\bigg]{\max_{n \in \set{N + 1, \dots, M}} \Big( \norm[\big]{\So_{\eps_n}^\ast - \So_{\eps_N}^\ast} \Big) > 2 \de'}.
    \end{aligned} \]
    The second application of Lemma~\ref{lem:So_eps_cauchy_probability} gives that $\lim_{N \tff} \lim_{M \tff} 2 \P{\norm[\big]{\So_{\eps_M}^\ast - \So_{\eps_N}^\ast} > \de'} = 0$.
    Then, $\lim_{N \tff} \lim_{M \tff} \P{\sup_{m, n \in \set{N, \dots, M}} \norm{\So_{\eps_m}^\ast - \So_{\eps_n}^\ast} \ge \de} = 0$.
\end{proof}

%% file: acknowledgement.tex
\medskip
\noindent
\textbf{Acknowledgement.}
This work was supported by the Danish Data Science Academy, which is funded by the Novo Nordisk Foundation (NNF21SA0069429) and Villum Fonden (40516). BJ is supported by the Leibniz Association within the Leibniz Junior Research Group on \emph{Probabilistic Methods for Dynamic Communication Networks} as part of the Leibniz Competition (grant no.\ J105/2020) and the Berlin Cluster of Excellence \emph{MATH+} through the project \emph{EF45-3} on \emph{Data Transmission in Dynamical Random Networks}.